%% file: rationally-not-bph-v2.tex
\pdfoutput=1
\documentclass[11pt,reqno]{amsart}
 
\usepackage[margin=0.95in,headsep=.2in, top=0.9in, bottom=0.9in]{geometry}                  
\usepackage{amssymb}        
\usepackage{mathtools}      
\usepackage{microtype}
\usepackage{array}          
\usepackage{verbatim} 
\usepackage{etoolbox}  
\usepackage[numbers]{natbib}          
\usepackage{graphicx}       
\usepackage{subcaption}     
\usepackage{float}          
\usepackage{tikz}           
\usetikzlibrary{
  decorations.pathreplacing, 
  arrows.meta                
}
 
\usepackage{enumitem}       
\usepackage{refcount}       
\usepackage[pagebackref=true]{hyperref}
\definecolor{linkpurple}{RGB}{136,92,150}   
\hypersetup{colorlinks=true,linkcolor=linkpurple,citecolor=linkpurple,urlcolor=black,
  pdftitle={Rationally slice knots that are not biprojectively H-slice},
  pdfauthor={Elise Brod}}
\usepackage{cleveref}       
 
\makeatletter
\patchcmd{\@maketitle}{\global\topskip42\p@}{\global\topskip20\p@}{}{}  
\patchcmd{\@setauthors}{\@topsep30\p@}{\@topsep20\p@}{}{}   
\patchcmd{\@setabstracta}{\skip@20\p@}{\skip@16\p@}{}{}      
\patchcmd{\@maketitle}{\dimen@34\p@}{\dimen@23\p@}{}{}       
\makeatother
\newtheorem{theorem}{Theorem}[section]
\newtheorem{lemma}[theorem]{Lemma}
\newtheorem{proposition}[theorem]{Proposition}
\newtheorem{corollary}[theorem]{Corollary}
\newtheorem{example}[theorem]{Example}
\newtheorem{family}{Family}
\newtheorem{conj}{Conjecture}
\crefname{family}{family}{families}
\Crefname{family}{Family}{Families}
\crefname{conj}{conjecture}{conjectures}
\Crefname{conj}{Conjecture}{Conjectures}
 
\theoremstyle{definition}
\newtheorem{definition}[theorem]{Definition}
\newtheorem{notation}[theorem]{Notation}
\newtheorem{question}[theorem]{Question}
\crefname{question}{question}{questions}
\Crefname{question}{Question}{Questions}
 
\theoremstyle{remark}
\newtheorem{remark}[theorem]{Remark}

\theoremstyle{plain}
\newtheorem{introthm}{Theorem}

\crefname{introthm}{Theorem}{Theorems}
 
\newenvironment{restatement}[1]
  {\edef\theintrothm{\getrefnumber{#1}}
   \edef\theHintrothm{\getrefnumber{#1}.body} 
   \addtocounter{introthm}{-1}               
   \begin{introthm}}
  {\end{introthm}}

 \newcommand{\claim}[2]{\par\smallskip\noindent\emph{\textbf{Claim~#1:} #2}\par} 
 
\newenvironment{sketch}
  {\begin{proof}[Proof sketch]}
  {\end{proof}}

\newcommand{\Z}{\mathbb{Z}}
\newcommand{\Q}{\mathbb{Q}}
\newcommand{\F}{\mathbb{F}}
\newcommand{\CP}{\mathbb{CP}^{2}} 
\newcommand{\CPbar}{\overline{\mathbb{CP}}^{2}} 
\newcommand{\uCP}{u_{\mathbb{CP}^{2}}} 
\newcommand{\uCPbar}{u_{\overline{\mathbb{CP}}^{2}}} 
\newcommand{\uCPtop}{u_{\mathbb{CP}^{2}}^{\mathrm{top}}} 
\newcommand{\uCPbartop}{u_{\overline{\mathbb{CP}}^{2}}^{\mathrm{top}}} 
\newcommand{\SigmaTwo}{\Sigma_{2}} 
\newcommand{\Kmirror}{\overline{K}} 
\newcommand{\Th}{\operatorname{Th}} 
\newcommand{\diag}{\operatorname{diag}} 
\newcommand{\rank}{\operatorname{rank}} 
\newcommand{\Adj}{\operatorname{Adj}} 
\newcommand{\Fib}{\operatorname{Fib}} 
\newcommand{\Luc}{\operatorname{Luc}} 
\newcommand{\Pell}{\operatorname{Pell}} 

\AddToHook{env/lemma/begin}{\crefalias{theorem}{lemma}}
\AddToHook{env/proposition/begin}{\crefalias{theorem}{proposition}}
\AddToHook{env/corollary/begin}{\crefalias{theorem}{corollary}}
\AddToHook{env/example/begin}{\crefalias{theorem}{example}}
\AddToHook{env/definition/begin}{\crefalias{theorem}{definition}}
\AddToHook{env/notation/begin}{\crefalias{theorem}{notation}}
\AddToHook{env/question/begin}{\crefalias{theorem}{question}}
\AddToHook{env/remark/begin}{\crefalias{theorem}{remark}}
\begin{document}
\title[Rationally slice knots that are not biprojectively H-slice]{Rationally slice knots that are not biprojectively H-slice}
\author[E.\ Brod]{Elise Brod}
\address{Department of Mathematics, University of Texas, Austin, TX 78712}
\email{ebrod@utexas.edu}
\begin{abstract}
We show that rationally slice knots need not be biprojectively H-slice (BPH-slice), answering a question in the study of the rational knot concordance group. Our main tool is a new lattice embedding obstruction to H-sliceness in $\#^n\mathbb{CP}^2$ for knots with vanishing signature, extending work of Kjuchukova--Miller--Ray--Sakall{\i}. We first apply our obstruction to an infinite family of rationally slice knots, beginning with $8_{18}$, the only prime knot with at most nine crossings whose BPH-sliceness was previously unresolved. As additional applications, we obstruct the BPH-sliceness of a family of 2-bridge knots that are not obstructed by the instanton knot concordance invariant $r_0$, we obstruct a non-alternating family of Montesinos knots, and we give the first examples of prime alternating knots that are topologically but not smoothly BPH-slice. We discuss how H-sliceness obstructions could be used to detect exotic smooth structures on $\#^n\mathbb{CP}^2$ and whether our obstruction extends to every smooth structure. 
\end{abstract}
\maketitle
\vspace{-1em}
\section{Introduction}\label{sec:intro}
A knot $K\subset S^3$ is \emph{H-slice} in a smooth closed $4$-manifold $Z$ if it bounds a null-homologous disk in $Z$ minus a ball (where being H-slice in $S^4$ is equivalent to being smoothly slice in $B^4$). Every knot is H-slice in some $\#^a\CP\#^b\CPbar$, but it is interesting to consider what happens when only one orientation is allowed
at a time.
Manolescu and Piccirillo call $K$ \emph{biprojectively H-slice} (\emph{BPH-slice}) if it is H-slice in both $\#^n\CP$ and $\#^n\CPbar$ for some $n$~\cite{MP23}, a notion that turns out to be relevant in the search for exotic smooth structures on closed, simply-connected, definite  $4$-manifolds. They show that nearly every prime knot with vanishing signature and at most nine crossings is BPH-slice; the knot $8_{18}$ is left unresolved.
In the context of studying the rational knot concordance group, Lee asked whether every \emph{rationally slice} knot is BPH-slice~\cite[Question~9]{Lee24}. Rationally slice knots are those that bound a smooth disk in a rational homology ball, and $8_{18}$ is one such example. We answer Lee's question, and settle $8_{18}$, with an infinite family of rationally slice Turk's head knots that begins with the quatrefoil knot.

 \begin{figure}[H]
  \centering
  \begin{subfigure}[b]{0.2\textwidth}
    \centering
    \resizebox{!}{3cm}{\input{figures/Th-3q}}
    \caption{$\Th(3,q)$}\label{fig:Th-3q-intro}
  \end{subfigure}
  \hspace{0.1\textwidth} 
  \begin{subfigure}[b]{0.2\textwidth}
    \centering
    \resizebox{!}{3cm}{\input{figures/Th-3-4}}
    \caption{$\Th(3,4)= 8_{18}$}
  \end{subfigure}
  \hspace{0.1\textwidth} 
  \begin{subfigure}[b]{0.2\textwidth}
    \centering
    \resizebox{!}{3cm}{\input{figures/Th-3-2}}
    \caption{$\Th(3,2)=4_1$}\label{fig:Th-3-2}
  \end{subfigure}
  \caption{The Turk's head knot $\Th(3,q)$ is the braid closure of $(\sigma_1\sigma_2^{-1})^q$. 
  In contrast to \Cref{thm:intro:Th3q}, $\Th(3,2)=4_1$ \emph{is} BPH-slice.
  }
\end{figure}
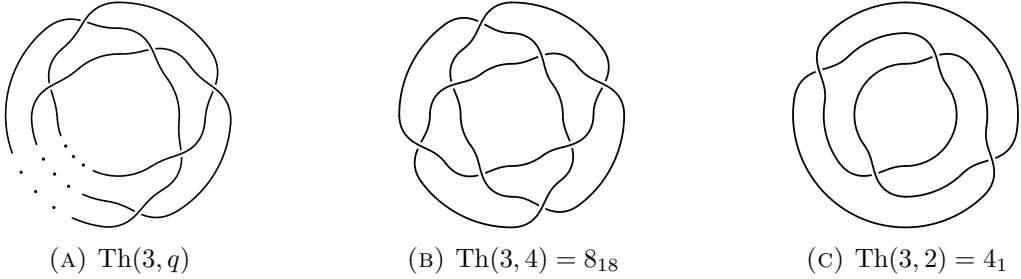

\begin{introthm}\label{thm:intro:Th3q}
Let $q\geq 4$ be even and not divisible by $3$. The Turk's head knot $\Th(3,q)$ is rationally slice, but it is not H-slice in $\#^n\CP$ for any $n$. In particular, $8_{18}=\Th(3,4)$ is not BPH-slice.
\end{introthm}
  The knots $\Th(3,q)$ are rationally slice because they are strongly negative amphichiral~\cite{DPS26,Kaw09}. Several familiar concordance invariants, which could otherwise obstruct BPH-sliceness, vanish for these knots. Since they are rationally slice, they have $\sigma=\sigma_{LT}=\tau=\epsilon=0$, along with other vanishing Heegaard Floer concordance invariants~\cite[\S\S1 and 6]{Lee24}. Since they are also quasi-alternating with vanishing signature, Rasmussen's $s$-invariant and $d(S^3_{\pm1}(\cdot))$ vanish as well (\Cref{lem:alt-qa-invariants}).
Nevertheless, we are able to obstruct them from being BPH-slice by studying how the intersection form of an associated $4$-manifold is constrained by a hypothetical H-slice disk. 

\begin{remark}
    Together with~\cite[\S2]{MP23}, \Cref{thm:intro:Th3q} completes the classification of BPH-sliceness for prime knots with at most nine crossings. See \Cref{subsubsec:classify} for a brief discussion on classifying knots with 10-12 crossings.
\end{remark}
The primary goal of this paper is to introduce a new lattice embedding obstruction to H-sliceness in $\#^n\CP$ for knots with vanishing signature (\Cref{thm:intro:decomp}). We demonstrate the utility of our obstruction by applying it to rationally slice knots, 2-bridge knots that are not obstructed by the instanton knot concordance invariant $r_0$, non-alternating Montesinos knots, and topologically BPH-slice knots.
\subsection{The lattice embedding obstruction}\label{subsec:intro:lattice}
Our obstruction relies on a construction of Kjuchukova--Miller--Ray--Sakall{\i}~\cite{KMRS24}. They show that if a knot $K$ with vanishing signature is H-slice in $\#^n\CP$, then its double branched cover $\SigmaTwo(K)$ bounds a smooth positive-definite $4$-manifold $X$ with $b_2(X)=2n$ whose intersection lattice is of \emph{half-integer surgery type}~\cite[Thm.~1.1]{KMRS24}. 

If we glue $X$ to an arbitrary positive-definite filling $W$ of $-\SigmaTwo(K)$, Donaldson's theorem embeds the intersection lattices of $X$ and $W$
orthogonally
in a standard diagonal lattice~\cite{Don87}. The half-integer surgery structure of $X$ forces the embedding of the intersection lattice of $W$, when represented by a matrix, to have $n$ pairs of opposite columns, together with $r=b_2(W)$ remaining columns. Grouping these columns gives the following result.
\begin{introthm}\label{thm:intro:decomp}
Let $K$ be a knot with $\sigma(K)=0$, and let $W$ be a positive-definite filling of $-\SigmaTwo(K)$ with $b_2(W)=r$ and intersection lattice $\Lambda_W$ whose form is given by a matrix $Q_W$. If $K$ is H-slice in $\#^n\CP$, then $\Lambda_W\hookrightarrow(\Z^n,2I)\oplus(\Z^r,I)$. Equivalently, $Q_W$ admits a decomposition 
\[ Q_W=2CC^T+DD^T, \quad\text{where}\quad C\in M_{r\times n}(\Z)\quad\text{and}\quad D\in M_r(\Z).\]
\end{introthm}
The proof of \Cref{thm:intro:decomp} follows lattice arguments similar to~\cite[Prop.~4.1]{KMRS24} and~\cite[Lem.~3.3]{Owe10}. When $n=0$, this is the classical lattice embedding obstruction to $K$ being slice (see, e.g.,~\cite{Lis07,Owe10}).

We now discuss how to apply \Cref{thm:intro:decomp} to obstruct BPH-sliceness in practice (see \Cref{prop:finiteob}). Since there are knots that are H-slice in $\#^n\CP$ but not in $\#^{n-1}\CP$ for arbitrarily large $n$~\cite[Thm.~1.2]{KMRS24}, an obstruction to BPH-sliceness must rule out every $n$ at once. 
The number $n$ of $\CP$ summands enters \Cref{thm:intro:decomp} only as the number of columns of $C$, so we set $T=CC^T$, whose size does not depend on $n$. One can readily argue that the matrix $DD^T=Q_W-2T$ is positive definite and the diagonal entries of $T$ are bounded. This leaves a finite set of \emph{candidates} for $T$, determined by $Q_W$ alone. If no candidate gives $Q_W-2T=DD^T$ with $D$ square and integral, then, by~\Cref{thm:intro:decomp}, $K$ is not H-slice in $\#^n\CP$ for any $n$. 

The candidate $T=0$ asks whether $Q_W$ itself is of the form $DD^T$ for some square $D$, which again recovers the classical lattice embedding slice obstruction. We can therefore interpret \Cref{thm:intro:decomp} as saying that adding any number of copies of $\CP$ can only relax this obstruction to sliceness in finitely many prescribed ways. In particular,  \Cref{thm:intro:decomp} can only obstruct knots that the classical obstruction already shows are not slice (see~\Cref{rem:limitations,rem:T0}).

The figure-eight knot, $4_1=\Th(3,2)$, illustrates why excluding $T=0$ alone does not obstruct BPH-sliceness. A checkerboard coloring of the diagram in \Cref{fig:Th-3-2} gives a positive-definite filling of $-\SigmaTwo(4_1)$ with $Q_W=\left(\begin{smallmatrix}3&-2\\-2&3\end{smallmatrix}\right)$. Since $\det(Q_W)=5$ is not a square, we know $Q_W\neq DD^T$ for any square integral $D$. This recovers the familiar fact that $4_1$ is not slice. But $T=\left(\begin{smallmatrix}1&-1\\-1&1\end{smallmatrix}\right)$ gives us $Q_W-2T=I_2=I_2I^T_2$. Indeed, $4_1$ is H-slice in $\CP$, since changing a single positive crossing turns it into the unknot. By amphichirality, or by changing a negative crossing to get the unknot, we see that $4_1$ is also H-slice in $\CPbar$, hence BPH-slice (cf.~\cite[Example~2.7]{MP23}). 
 
In contrast, for the knots $\Th(3,q)$ in \Cref{thm:intro:Th3q}, a checkerboard coloring of the  diagram in \Cref{fig:Th-3q-intro} gives a definite filling for which no candidate $T$ admits an integral factorization $Q_W-2T=DD^T$ with $D$ square.
\subsection{Additional applications of \texorpdfstring{\Cref{thm:intro:decomp}}{Theorem B}}\label{subsec:intro:applications}
\subsubsection{Classifying BPH-sliceness of low-crossing number knots}\label{subsubsec:classify}
Using \Cref{thm:intro:decomp} and data from \cite{KnotInfo}, it is straightforward to classify the BPH-sliceness of 10-crossing knots. There are 34 10-crossing knots with $\sigma=\sigma_{LT}=s=0$ that are not smoothly slice \cite{KnotInfo} (see \Cref{prop:0pos-obstructions,prop:Hslice-obstructions}). Using \Cref{thm:intro:decomp} one can show that the following 7 knots are not BPH-slice: $10_{28},\:10_{34},\:10_{68},\:10_{86},\:10_{90},\: 10_{104},\: 10_{109}$ (see \Cref{prop:topvsm-10_109-smooth} as an example).
Using directed crossing changes, one can show that the remaining 27 knots are BPH-slice. 

One can apply a similar strategy for classifying 11- and 12-crossing knots, but there are still some knots whose BPH-slice status we did not determine. For 11-crossings, there are three unknown knots: $11a_{24}$, $11n_{153}$, and $11n_{182}$.
For 12-crossing knots there were three alternating knots ($12a_{101}$, $12a_{1167}$, and $12a_{830}$) and 43 non-alternating knots we were unable to determine.
\subsubsection{Comparison with the instanton knot concordance invariant \texorpdfstring{$r_0$}{r0}.} 
A knot is called \emph{$0$-positive} if it is H-slice in a smooth closed $4$-manifold homeomorphic to some $\#^n\CP$, and \emph{$0$-bipolar} if both it and its mirror are $0$-positive~\cite{CHH13}. Note in particular that a BPH-slice knot is $0$-bipolar. 

Daemi, Imori, Sato, Scaduto, and Taniguchi adapted the instanton homology cobordism invariants $r_s$ of Nozaki, Sato, and Taniguchi~\cite{NST24} to define a knot concordance invariant $r_0(K)=r_0(S^3,K)$, and showed that if $\sigma(K)=0$ and $r_0(\Kmirror)<\infty$, then $K$ is not H-slice in any smooth closed positive-definite $4$-manifold with $b_1=0$~\cite[Prop.~6.6]{DISST25}.
In particular, such a knot is not $0$-positive, so a $0$-bipolar knot with $\sigma(K)=0$ has $r_0(K)=r_0(\Kmirror)=\infty$. They compute $r_0$ for families of 2-bridge knots, obstructing some of them from being $0$-bipolar, hence obstructing them from being BPH-slice.

For the following family of 2-bridge knots with vanishing signature, we show that $r_0$ does not obstruct $0$-bipolarity, but we use \Cref{thm:intro:decomp} to show that the knots are not BPH-slice.
\begin{introthm}\label{thm:intro:2bridgenew}
For every odd $h\geq 3$, the 2-bridge knot $L_h=K(8h^2+4h+1,8h^2-4h+1)$
is not H-slice in $\#^n\CP$ for any $n\geq0$. Moreover, $r_0(L_h)=r_0(\overline{L_h})=\infty$.
\end{introthm}
We compute $r_0(L_h)$ and $r_0(\overline{L_h})$ in Appendix~\ref{app:Lh-r0} using the 2-bridge formulas of Daemi--Scaduto \cite[\S3.1]{DS24}, as in~\cite[\S6.4]{DISST25}.

In \Cref{thm:2bridgeDISST}, we apply \Cref{thm:intro:decomp} to a family of 2-bridge knots, $K_{m,n}$, to recover a consequence of~\cite[Thm.~1.15]{DISST25} using Donaldson's theorem and elementary lattice arguments rather than $r_0$.
\subsubsection{A non-alternating family}
\Cref{thm:intro:decomp} also applies to non-alternating knots, for which checkerboard surfaces need not give definite fillings. For the following family of Montesinos knots studied by Williams~\cite{Wil08}, we instead use a star-shaped plumbing to get a definite filling of $\SigmaTwo(J_q)$.
\begin{introthm}\label{thm:intro:Jq}
For every odd $q\geq1$, the non-alternating Montesinos knot $J_q=M\bigl(1;\frac{2q+3}{q+2},\frac{8}{3},\frac{3}{2}\bigr)$ is not H-slice in $\#^n\CP$ for any $n\geq0$.
\end{introthm}
Like $\Th(3,q)$, the knots $J_q$ are quasi-alternating with $\sigma(J_q)=0$, so the invariants in \Cref{lem:alt-qa-invariants} vanish. Their Levine--Tristram signature functions also vanish (\Cref{prop:Jq-signatures}). Williams obstructs smooth sliceness for this family~\cite{Wil08} (i.e., the $T=0$ candidate is excluded), and \Cref{thm:intro:Jq} extends this result to obstruct H-sliceness in $\#^n\CP$ by excluding all the nonzero candidates.
\subsubsection{Topological versus smooth BPH-sliceness}
Following~\cite{KMRS24}, let $\uCP(K)$ be the smallest $n$ for which $K$ is H-slice in $\#^n\CP$ ($\infty$ if there is none), set $\uCPbar(K)=\uCP(\Kmirror)$, and let $\uCPtop$ and $\uCPbartop$ be the analogues for locally flat disks. These are the \emph{(topological) projective slicing numbers} of $K$. Dunfield and Gong conjecture that a prime alternating knot is smoothly slice if and only if it is topologically slice, and they verify their conjecture for knots with up to $19$ crossings~\cite[Conj.~1.15 and Thm.~1.7]{DG}. Combining our obstruction with Seifert matrix decompositions that bound $\uCPtop$ from above (using~\cite[Thm.~1.3]{KMRS24}),  we show in \Cref{thm:intro:topvsm} that the corresponding statement for BPH-sliceness fails already at low crossing numbers, even for knots whose topological projective slicing numbers are as small as possible.
\begin{introthm}\label{thm:intro:topvsm}
The prime alternating knots $10_{28}$, $10_{109}$, $12a_{792}$, and $13a_{4237}$ are topologically BPH-slice but not smoothly BPH-slice. Moreover, the knot $13a_{4237}$ satisfies the Fox--Milnor criterion,
\[
\uCPtop(13a_{4237})=\uCPbartop(13a_{4237})=\uCP(13a_{4237})=1,\qquad\text{and}\qquad \uCPbar(13a_{4237})=\infty.
\]
\end{introthm}
 The projective slicing numbers for all four knots are collected in \Cref{thm:topvsm-examples}. To our knowledge, these are the first examples in the literature of alternating knots that are topologically but not smoothly BPH-slice. Since $10_{109}$ is known to be rationally slice, it gives an additional answer to Lee's question~\cite[Question~9]{Lee24}.
\subsection{Using H-sliceness obstructions to detect exotic definite manifolds}\label{subsec:intro:exotica}
A smooth $4$-manifold homeomorphic but not diffeomorphic to $\#^n\CP$ is called an \emph{exotic} $\#^n\CP$. It is not known if an exotic $\#^n\CP$ exists for any $n$ (the case $n=0$ is the smooth 4-dimensional Poincar\'e conjecture).
Every BPH-slice knot is $0$-bipolar, so a $0$-bipolar knot that is not BPH-slice would be H-slice in an exotic $\#^n\CP$ or $\#^n\CPbar$~(see \cite[\S2]{MP23}).

To detect an exotic $\#^n\CP$ using H-sliceness, one would need a manifold homeomorphic to $\#^n\CP$ together with an H-slice disk for some knot $K$ that is not H-slice in the standard $\#^n\CP$. One way to construct such a manifold and disk is through the zero-surgery strategy of Manolescu--Piccirillo~\cite{MP23}. If a knot $K'$ is H-slice in $\#^n\CP$ and $S^3_0(K')\cong S^3_0(K)$, then $K$ is H-slice in a manifold homeomorphic to $\#^n\CP$~\cite[Lem.~3.3]{MP23}.

Many familiar concordance invariants that obstruct BPH-sliceness, including the $r_0$ obstruction discussed above, also obstruct $0$-bipolarity, so they cannot distinguish $0$-bipolar knots from BPH-slice knots. Hence, these invariants cannot be used to detect an exotic $\#^n\CP$ in this way. 

Prior work on using H-sliceness to detect an exotic $\#^n\CP$ proposed using Rasmussen's $s$-invariant~\cite[Cor.~1.9]{MMSW23} and its generalizations (which are not known to obstruct $0$-positivity; see \Cref{rem:CHH-s}) as the obstruction to H-sliceness \cite{FGMW10,MP23,Nak23,Qin25RBG}. All of our examples have $s=0$, so $s$ cannot serve as our obstruction. As with $s$, we do not know if \Cref{thm:intro:decomp} is an obstruction to $0$-positivity, since its proof uses the standard smooth structure on $\#^n\CP$ via the half-integer surgery filling of \cite{KMRS24} (\Cref{q:exotic-half-integer}).

If any knot in Theorems~\ref{thm:intro:Th3q},~\ref{thm:intro:2bridgenew}, or~\ref{thm:intro:Jq}, or either of the knots $\overline{13a_{4237}}$ or $10_{109}$ from \Cref{thm:intro:topvsm}, were H-slice in a smooth closed $4$-manifold $Z$ homeomorphic to $\#^n\CP$, then $Z$ would be exotic. On the other hand, if the conclusion of \Cref{thm:intro:decomp} holds for every smooth manifold homeomorphic to $\#^n\CP$, then \Cref{thm:intro:decomp} also obstructs $0$-positivity and these knots are not $0$-positive. In either case, none of the existing obstructions to $0$-positivity that we compute detect these knots, including $r_0$ for $L_h$. We record these computations and their limitations in \Cref{subsec:exotica-candidates}.

Since $s$ vanishes for all of our examples, our lattice obstruction supplies information about BPH-sliceness that $s$ does not. This raises the question of whether other constructions from Khovanov homology, such as skein lasagna modules~\cite{MWW22}, can obstruct rationally slice knots from being BPH-slice (\Cref{q:khovanov}).
\subsection{Further directions}
We conjecture that \Cref{thm:intro:Th3q} extends to all Turk's head knots $\Th(p,q)$ with $p\geq3$ odd and $q\geq4$ even (\Cref{conj:Thpq}), and we expect that \Cref{thm:intro:decomp} can be used to make progress on the problem. To support future work, Appendix~\ref{app:Gpq} computes their Goeritz matrices and determinants, recovering a formula proposed in~\cite{DMMS10} (cf. the simultaneous determinant result in \cite{Sta26}). No analogue of \Cref{conj:Thpq} holds for odd $q$ or $q=2$ (see \Cref{rem:Thpq-even}).

Our applications of \Cref{thm:intro:decomp} never actually use the requirement that each candidate matrix $T$ be decomposable as $T=CC^T$. In Appendix~\ref{app:D2}, we show that this requirement holds automatically for large $n$ (where $n$ is both the number of columns of $C$ and the number of $\CP$ summands in \Cref{thm:intro:decomp}) whenever fewer than six diagonal entries of $Q_W$ are at least $5$. We ask whether the requirement $T=CC^T$ can ever obstruct BPH-sliceness when checking if $Q_W-2T=DD^T$ (for some square integral $D$) alone does not (\Cref{q:D2}). For the Goeritz matrix of a reduced alternating diagram, such an example would need at least $16$ crossings.
\subsection{Conventions}\label{subsec:conv}
Manifolds are smooth, compact, and oriented. Maps are orientation-preserving. We use $\simeq$ to mean isotopic and $\cong$ to mean diffeomorphic. Knots are oriented, with $-K$ and $\overline{K}$ denoting the orientation reverse and the mirror of $K$, respectively. $\Sigma_2(K)$ denotes the double cover of $S^3$ branched over $K$. The signature of the right-handed trefoil is $-2$. We say that the 2-bridge knot $K(p,q)$ has $\SigmaTwo(K(p,q))\cong S^3_{p/q}(U)$, in agreement with \cite{DISST25} and disagreement with many other sources. Our rational tangle diagrams and Montesinos links agree with those in \cite{Iss18, OwS}. Continued fractions are taken with the convention
\[
[a_1,\ldots,a_\ell]
=a_1-\cfrac{1}{a_2-\cfrac{1}{\ddots-\cfrac{1}{a_\ell}}}.
\]
\subsection{Organization}\label{subsec:org}
\Cref{sec:prelim} contains background on slicing notions, lattices, 2-bridge and Montesinos knots, and definite fillings of double branched covers.
\Cref{sec:lattice} proves \Cref{thm:intro:decomp} and turns it into the finite candidate test of \Cref{prop:finiteob}.
Sections~\ref{sec:Th}, \ref{sec:2bridge}, and~\ref{sec:montesinos} apply this test to rationally slice Turk's head knots, 2-bridge knots, and non-alternating Montesinos knots, respectively.
\Cref{sec:topvsm} proves a more detailed version of \Cref{thm:intro:topvsm}.
\Cref{sec:exotica} discusses whether \Cref{thm:intro:decomp} holds for every smooth structure on $\#^n\CP$ (\Cref{q:exotic-half-integer}), shows that several obstructions to $0$-bipolarity vanish for our examples, and compares our obstruction with invariants from Khovanov homology.
The computations of $r_0$ for $L_h$ and its mirror appear in Appendix~\ref{app:Lh-r0}. 
Appendix~\ref{app:Gpq} computes Goeritz matrices and determinants for Turk's head knots. 
Appendix~\ref{app:D2} examines condition~\ref{cond:finiteob:gram} of \Cref{prop:finiteob}.
Appendix~\ref{app:comps} contains signature and Seifert matrix computations.
\begin{remark}
On a first read, we recommend reading the introduction; skimming \Cref{sec:prelim} for conventions and unfamiliar tools; reading \Cref{sec:lattice}; and reading the opening to at least one of Sections~\ref{sec:Th},~\ref{sec:2bridge}, or \ref{sec:montesinos} (each less than 2 pages). For a gentler introduction to the material and to provide intuition for the lattice embedding obstruction, read the proof of \Cref{prop:topvsm-10_109-smooth}, which shows that the knot $10_{109}$ is not BPH-slice, after reading \Cref{subsubsec:checkerboard} and \Cref{sec:lattice}.
\end{remark}
\subsection{AI Declaration}
The author used ChatGPT Work (most recently 6.1 Sol) as a reference librarian and proofreader. The inclusions of \cite{KimOh05} and \cite{Wil08}, and the lack of spelling mistakes, are the most visible impact of this use-case.
Beyond reference-related tasks, the author used Claude (Opus 4 and 5) as a calculator and data processor to: speed up the search for and verification of two examples in \Cref{sec:topvsm} (filter databases from~\cite{KnotInfo} and~\cite{DG}, apply \Cref{prop:finiteob} to Goeritz lattices of low-crossing alternating knots, and apply~\cite[Thm.~1.3]{KMRS24}); help compute the instanton knot concordance invariant $r_0$ for a family of 2-bridge knots using formulas from \cite{DS24} (see Appendix~\ref{app:Lh-r0}); and check the computations in \Cref{thm:Gpq-det}.
The author also used Claude (Opus 5) to convert hand-drawn images into TikZ figures. 
The tools used by the author are freely provided to all students at The University of Texas at Austin.
\subsection{Acknowledgments}\label{subsec:ack}
 The author would like to thank her advisor, Lisa Piccirillo, for her support and insightful mathematical discussions; Jaewon Lee, for helpful discussions about rationally slice knots; and Brendan Owens, for his advice and for frequently having already written down the answers to her questions years in advance. The author also thanks O\u{g}uz \c{S}avk and Abhishek Shivkumar for their feedback on earlier drafts. The author has been supported by the Simons Foundation Collaboration Grant on New Structures in Low-Dimensional Topology, and indirectly supported by the National Science Foundation as a participant in numerous mathematical gatherings. 
 \section{Preliminaries}\label{sec:prelim}
For a knot $K \subset S^3$, we denote by $\Sigma_2(K)$ the double cover of $S^3$ branched over $K$. The $3$-manifold $\Sigma_2(K)$ is an oriented rational homology 3-sphere with $|H_1(\Sigma_2(K);\mathbb{Z})| = |\det(K)|$. By a \emph{filling} of a $3$-manifold $Y$ we mean a compact oriented $4$-manifold with oriented boundary $Y$.
\subsection{Slicing notions}\label{subsec:prelim:hslice}
\begin{definition}[Rationally slice]
    A knot $K\subset S^3$ is \emph{rationally slice} if it bounds a smoothly and properly embedded disk in a compact, oriented smooth $4$-manifold $B$ with $\partial B=S^3$ and $H_*(B;\mathbb Q)\cong H_*(B^4;\mathbb Q)$.
\end{definition}
\begin{definition}[H-slice]
A knot $K\subset S^3$ is \emph{H-slice} in a closed oriented $4$-manifold $Z$ if it bounds a smooth, properly embedded disk $\Delta$ in $Z^{\circ}=Z\setminus\mathring B^4$ with $[\Delta]=0$ in $H_2(Z^\circ,\partial Z^\circ)$.
\end{definition}
When $Z=S^4$, this agrees with the definition of $K$ being smoothly slice.
\begin{definition}[Projective slicing numbers] \label{def:slicing-numbers}
Following Kjuchukova--Miller--Ray--Sakall\i{} \cite{KMRS24}, define the $\CP$- and $\CPbar$-\emph{slicing numbers} of a knot $K\subset S^3$ as
\begin{align*}
\uCP(K)&=\min\bigl\{m\geq 0:K\text{ is H-slice in }\#^{m}\CP\bigr\},\\
\uCPbar(K)&=\min\bigl\{m\geq 0:K\text{ is H-slice in }\#^{m}\CPbar\bigr\},
\end{align*}
with the minimum of the empty set understood to be $\infty$. To suppress orientations, we refer to the \emph{projective slicing numbers} of $K$.
\end{definition}
\begin{definition}[BPH-slice]
 $K$ is \emph{biprojectively H-slice} (or \emph{BPH-slice}) if $K$ is H-slice in $\#^n\CP$ and $\#^n\CPbar$ for some $n$. Equivalently, $K$ is \emph{BPH-slice} if $\uCP(K)$ and $\uCPbar(K)$ are both finite. 
\end{definition}
To show that a knot $K$ is \emph{not} BPH-slice, we need only show that \emph{one} projective slicing number of $K$ is infinite. Note that $\uCP(K)=\uCP(-K)=\uCPbar(\Kmirror)$, so when $K$ is isotopic to its mirror (or mirror reverse), its two projective slicing numbers are equal. To compare H-sliceness in $\#^n\CP$ with H-sliceness in potentially exotic $\#^n\CP$, we use the following notions of Cochran--Harvey--Horn.
\begin{definition}[$0$-bipolar, see~\cite{CHH13}]
A knot $K$ is \emph{$0$-positive} if it is H-slice in some closed, smooth, oriented $4$-manifold homeomorphic to $\#^n\CP$, for some $n\geq0$. It is \emph{$0$-negative} if its mirror is $0$-positive, and \emph{$0$-bipolar} if it is both $0$-positive and $0$-negative.
\end{definition}
In \Cref{sec:topvsm} we will also use the following analogous topological definitions from \cite{KMRS24}:
\begin{definition}[Topological projective slicing numbers]\label{def:topological-hslice}
A knot $K\subset S^3$ is \emph{topologically H-slice} in a closed, oriented $4$-manifold $Z$ if it bounds a locally flat, properly embedded, null-homologous disk in $Z\setminus\mathring B^4$. Its \emph{topological $\CP$-slicing number} is denoted $\uCPtop(K)$, and we define $\uCPbartop(K)$ as $\uCPtop(\Kmirror)$. We say $K$ is \emph{topologically BPH-slice} if both topological projective slicing numbers are finite.
\end{definition}
\subsubsection{Invariants obstructing H-sliceness}\label{subsubsec:invariants}
Every knot that is H-slice in $\#^n\CP$ is $0$-positive, and every BPH-slice knot is $0$-bipolar. We gather some known obstructions, separated according to whether they are known to obstruct $0$-positivity (\Cref{prop:0pos-obstructions}) or only H-sliceness in the standard $\#^n\CP$ (\Cref{prop:Hslice-obstructions}). A knot that is $0$-positive but not H-slice in $\#^n\CP$ for any $n$ would give an exotic smooth structure on some $\#^n\CP$.
\begin{proposition}[{\cite{CHH13,DISST25}}]\label{prop:0pos-obstructions}
Let $K$ be a $0$-positive knot. Then:
\begin{enumerate}
    \item\label{obs:LT} $\sigma_{LT}(\omega)\leq0$ for all $\omega\in S^1$, where
      $\sigma_{LT}(\omega)$ is averaged at roots of $\Delta_K$;
    \item\label{obs:tau} $\tau(K)\geq0$;
    \item\label{obs:delta} if the $p^r$-signatures of $K$ vanish and the $p^r$-fold branched cover of $K$ is an integral homology sphere, then $\delta_{p^r}(K)\leq0$;
    \item\label{obs:d} $d(S^3_{\pm1}(K))\leq0$;
    \item\label{obs:HK} if the $2^r$-signatures of $K$ vanish, then the corresponding Hedden--Kirk definite-filling obstructions vanish;
    \item\label{obs:r0} if $\sigma(K)=0$, then $r_0(\Kmirror)=\infty$.
\end{enumerate}
If $K$ is $0$-bipolar, then in addition $\epsilon(K)=0$.
\end{proposition}
\begin{proof}
Items~\eqref{obs:LT}--\eqref{obs:HK} and the final assertion are \cite[Prop.~1.2(1),(2),(4)--(7)]{CHH13}. Item~\eqref{obs:r0} is the contrapositive of~\cite[Prop.~6.6]{DISST25} applied to $\Kmirror$. We refer the reader to \cite{CHH13} and \cite{DISST25} for details.
\end{proof}
\begin{remark}\label{rem:CHH-s}
 \cite[Prop.~1.2(3)]{CHH13} asserts that $s(K)\geq0$ for every $0$-positive knot, and hence that $s(K)=0$ for every $0$-bipolar knot. However, its proof uses the claimed equality between $s$ and the instanton invariant $s^{\sharp}$ from~\cite{KM13}, which was subsequently retracted. Gong showed that $s$ and $s^{\sharp}$ need not agree~\cite{Gong}. It is not known whether $s$ obstructs $0$-positivity (see~\cite[Rmk.~2.3]{MP23}). 
\end{remark}
\begin{proposition}\label{prop:Hslice-obstructions}
Let $K$ be H-slice in $\#^n\CP$ for some $n$. Then:
\begin{enumerate}\setcounter{enumi}{6}
    \item\label{obs:s} $s(K)\geq0$ \cite[Cor.~1.9 applied to $\Kmirror$]{MMSW23}; 
    \item\label{obs:decomp} if $\sigma(K)=0$, then every positive-definite filling $W$ of $-\SigmaTwo(K)$ satisfies $Q_W=2CC^T+DD^T$ where $D$ is square and $C$ has $n$ columns (\Cref{thm:intro:decomp}).
\end{enumerate}
\end{proposition}
The proofs of \cite[Cor.~1.9]{MMSW23} and \Cref{thm:intro:decomp} both use the standard smooth structure on $\#^n\CP$. 

If $K$ is $0$-bipolar, applying \Cref{prop:0pos-obstructions} to $K$ and $\Kmirror$ shows that they both have 
\[
\sigma_{LT}=\sigma=\tau=\epsilon=d(S^3_{\pm 1}(\cdot))=0;\quad  r_0=\infty;
\]
as well as $\delta_{p^r}=0$ when the $p^r$-fold branched cover is an integral homology sphere, and vanishing Hedden--Kirk obstructions to $\Sigma_{2^r}(K)$ bounding a positive-definite manifold $W$ with $H_1(W;\Z/2\Z)=0$. If $K$ is BPH-slice, then $K$ and its mirror must also have $s=0$.

The knots discussed in \Cref{sec:Th,sec:montesinos,sec:2bridge,sec:topvsm} are quasi-alternating with vanishing signature, so many of the conditions above are satisfied automatically. 
\begin{lemma}\label{lem:alt-qa-invariants}
    If $K$ is a quasi-alternating knot with $\sigma(K)=0$, then for both $K$ and $\Kmirror$,
    \[s=\tau=\epsilon=\delta_2=d\left(S^3_{\pm 1}(\cdot)\right)=0.\]
\end{lemma}
\begin{proof}
Since $K$ is quasi-alternating, so is $\Kmirror$, and it suffices to show the result for $K$. 

Quasi-alternating knots are $\sigma$-thin in Khovanov homology and in knot Floer homology by \cite{ManolescuOzsvath08}, so $-\sigma(K)=2\tau(K)$. Then $s(K)=-\sigma(K)$ by \cite{Ras10} and $\epsilon(K)=\operatorname{sgn}\tau(K)$ by \cite[Prop.~3.6(5)]{Hom14Cable}. Moreover, $\delta_2(K)=-\sigma(K)/2$~\cite{LO15}. Since $\sigma(K)=0$, we have $\tau(K)=s(K)=\epsilon(K)=\delta_2(K)=0$.

Since $K$ is Floer-thin with $\tau(K)=0$, $CFK^\infty(K)$ admits a filtered model consisting of the unknot complex and acyclic square summands~\cite[Lem.~7]{Pet13}, so $V_0(K)=0$ and $d(S^3_{1}(K))=-2V_0(K)=0$~\cite[Prop.~1.6]{NW15}. Applying this to $\Kmirror$ gives $d(S^3_{-1}(K))=0$.
\end{proof}
Note that the knots considered in this paper also have vanishing Levine--Tristram signature function, but that this is not implied by $\sigma(K)=0$. We refer the reader to~\cite{Con21} for a survey on $\sigma_{LT}(\omega)$.
\subsection{Lattices and Donaldson's theorem}\label{subsec:prelim:lattice}
\begin{definition}[Gram matrix]
A matrix $G$ is called a \emph{Gram matrix} of the rows of $M$ if $G=MM^T$. Similarly, $G$ is a Gram matrix of the columns of $M$ if $G=M^TM$.
\end{definition}
\begin{definition}[\cite{KMRS24, Owe20}]
An \emph{integral lattice} is a pair $(L, Q)$, where $L$ is a finitely generated free abelian group and $Q \colon L \times L \to \mathbb{Z}$ is a symmetric bilinear form. An \emph{isometric embedding} of lattices $(L_1, Q_1) \hookrightarrow (L_2, Q_2)$ is an injective homomorphism $\varphi \colon L_1 \to L_2$ preserving the form, i.e., $Q_2(\varphi(x), \varphi(y)) = Q_1(x,y)$ for all $x,y \in L_1$.
\end{definition}
\begin{definition}[Intersection lattice]
 If $X$ is a smooth $4$-manifold (possibly with boundary), its \emph{intersection lattice} is $\Lambda_X := (H_2(X)/\mathrm{Tors}, Q_X)$. 
\end{definition}
Fixing an isomorphism $H_2(X)/\mathrm{Tors} \cong \mathbb{Z}^m$ corresponds to choosing an ordered basis $\{h_1,\dots,h_m\}$ mapped to the standard basis vectors $\{e_1, \dots, e_m\}$. Under this identification, $Q_X$ is encoded by the $m \times m$ symmetric matrix $G$ where $G_{ij} = Q_X(h_i, h_j)$. $G$ is called the matrix representation of $Q_X$ with respect to this basis. A change of basis corresponds to a choice of matrix $P \in \operatorname{GL}_m(\mathbb{Z})$ that transforms the representation of $Q_X$ from $G$ to $P^T G P$.
\begin{notation}
 All lattice embeddings we consider are isometric, so we often drop this descriptor and simply say \emph{embedding}, denoted $\hookrightarrow$. When the choice of basis is clear or implicit, we identify an intersection form $Q_X$ with its matrix $G$. For example, we write $Q_W=DD^T$ instead of $G=DD^T$.
\end{notation}
To construct lattice embeddings from $4$-manifolds, we rely on Donaldson's Diagonalization Theorem applied to the following standard gluing construction:
\begin{theorem}[Donaldson's Diagonalization Theorem~\cite{Don87}]
\label{thm:donaldson}
If $Z$ is a smooth, closed, oriented, positive-definite $4$-manifold, then its intersection lattice is standard, i.e.,
\[
(H_2(Z)/\mathrm{Tors}, Q_Z) \cong (\mathbb{Z}^{b_2(Z)}, I).
\]
\end{theorem}
\begin{proposition}[A standard application of \Cref{thm:donaldson}]
\label{prop:standard-lattice}
Let $Y$ be an oriented rational homology $3$-sphere. Suppose $X$ and $W$ are smooth, compact, oriented, positive-definite $4$-manifolds with $\partial X = Y$ and $\partial W = -Y$. Then gluing $X$ and $W$ together along $Y$ yields a closed, positive-definite $4$-manifold $Z = X \cup_Y W$ with $b_2(Z) = b_2(X) + b_2(W)$. Moreover, there exists a finite-index isometric embedding of lattices
\begin{equation}
\label{eq:standard-embedding}
\Lambda_X \oplus \Lambda_W \ \hookrightarrow\ \Lambda_Z\cong \left(\mathbb{Z}^{b_2(X)+b_2(W)}, I\right).
\end{equation} In particular, the images of $\Lambda_X$ and $\Lambda_W$ in
$\Lambda_Z$ are mutually orthogonal.
\end{proposition}
\begin{proof} 
Gluing along $Y$ produces a smooth, closed, oriented $4$-manifold $
Z=X\cup_Y W$.
Since $Y$ is a rational homology 3-sphere ($H_1(Y;\mathbb{Q}) = H_2(Y;\mathbb{Q})= 0$), the Mayer--Vietoris sequence with rational coefficients gives $b_2(Z) = b_2(X) + b_2(W)$, and Novikov additivity implies $Z$ is positive definite (see \cite[Prop.~7]{IM18}). By
Donaldson's theorem (\Cref{thm:donaldson}), $(H_2(Z)/\mathrm{Tors}, Q_Z) \cong (\mathbb{Z}^{b_2(Z)}, I)$.

Since $H_2(Y;\mathbb Z)=0$, the Mayer--Vietoris sequence with integral coefficients contains
\[
0 \longrightarrow H_2(X;\Z) \oplus H_2(W;\Z) \xrightarrow{\ j_*\ } H_2(Z;\Z) \longrightarrow H_1(Y) \longrightarrow \cdots
\]
Since $H_1(Y;\Z)$ is finite, $j_*$ is injective with finite-index image. Classes from $X$ and $W$ may be represented by cycles contained in the disjoint interiors of the two pieces, hence the two images in $H_2(Z)/\operatorname{Tors}$ are orthogonal. Restricting the diagonalization $Q_Z\cong I$ to these two subgroups gives the isometric embedding \eqref{eq:standard-embedding}.
\end{proof}
Several elementary facts about matrices and lattices recur in our applications. We state and prove them where they are used and reference them here for convenience: \Cref{prop:Dinvert} (odd $\det(2T+DD^T)$ implies $D$ is invertible over $\F_2$), \Cref{prop:2bridge-root-count} (normalizing consecutive square-two rows of such a $D$ to $e_i-e_{i+1}$), and the characteristic-vector argument in the proof of \Cref{prop:Th3qdecomp}.
\subsection{2-bridge and Montesinos knots}\label{subsec:prelim:knot}
We recall the constructions of rational tangles, 2-bridge knots, and Montesinos knots, along with our conventions for their double branched covers. 
\begin{definition}\label{def:rationaltangle}
A diagram of the \emph{rational tangle} $R$ corresponding to the continued fraction $\alpha/\beta=[a_1,\dots,a_{\ell}]$ is given by the 4-strand braid $\sigma_2^{a_1}\sigma_1^{a_2}\sigma_2^{a_3}\cdots\sigma_i^{a_{\ell}}$ (where $i=1$ if $\ell$ is even and $2$ if $\ell$ is odd) using the conventions shown below. Any two continued fraction expansions of $\alpha/\beta$ represent isotopic (relative to the four boundary points) rational tangle diagrams (e.g.,~\cite[Thm.~1]{KL04}).
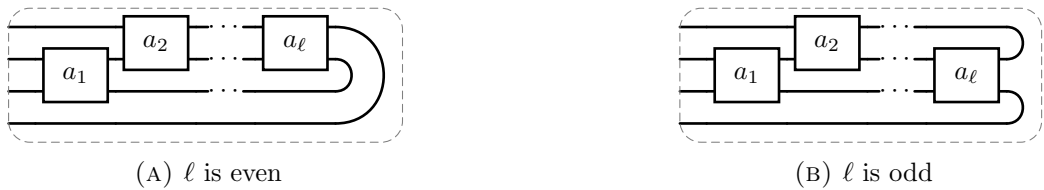
\begin{figure}[H]
  \centering
  \begin{subfigure}[b]{0.50\textwidth}
    \centering
    {\tikzset{every picture/.append style={scale=0.709}}\input{figures/tangle-even}}
    \caption{$\ell$ is even}
    \label{fig:tangle-even}
  \end{subfigure}
  \hfill
  \begin{subfigure}[b]{0.45\textwidth}
    \centering
    {\tikzset{every picture/.append style={scale=0.709}}\input{figures/tangle-odd}}
    \caption{$\ell$ is odd}
    \label{fig:tangle-odd}
  \end{subfigure}
  \caption{Rational tangles corresponding to $\alpha/\beta=[a_1,a_2,\dots,a_{\ell}]$}
\end{figure}
\end{definition}
\begin{definition}\label{def:2bridge}
 The 2-bridge knot $K(p,q)$, where $p>0$ is odd and $\gcd(p,q)=1$, is the numerator closure of a rational tangle determined by the continued-fraction expansion $p/q=[a_1,\dots,a_{\ell}]$. When $0<q<p$, $\overline{K(p,q)}\simeq K(p,p-q)$.
\begin{figure}[H]
  \centering
  \begin{subfigure}[b]{0.45\textwidth}
    \centering
    {\tikzset{every picture/.append style={scale=0.709}}\input{figures/twobridge-even}}
    \caption{$\ell$ is even}
    \label{fig:twobridge-even}
  \end{subfigure}
  \hfill
  \begin{subfigure}[b]{0.45\textwidth}
    \centering
    {\tikzset{every picture/.append style={scale=0.709}}\input{figures/twobridge-odd}}
    \caption{$\ell$ is odd}
    \label{fig:twobridge-odd}
  \end{subfigure}
  \caption{2-bridge knot $K(p,q)$ where $p/q=[a_1,a_2,\dots,a_{\ell}]$}
\end{figure}
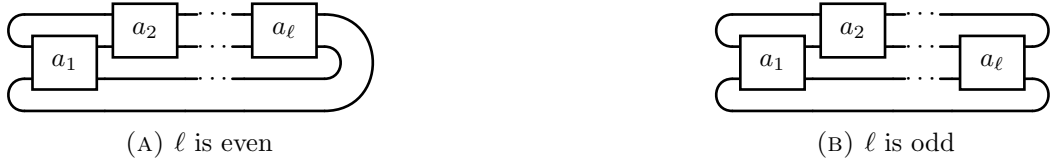
\end{definition}
We follow the convention of~\cite{DISST25}, where $K(3,1)$ is the right-handed trefoil and $\SigmaTwo(K(p,q))\cong S^3_{p/q}(U)$, where $U$ is the unknot. 
\begin{definition}[{cf. \cite[Def.~3.1]{OwS}}]\label{def:montesinos} The diagram of a \emph{Montesinos link} $M(e;\alpha_1/\beta_1,\ldots,\alpha_k/\beta_k)$ is obtained by joining rational tangles $R_i$ corresponding to $\alpha_i/\beta_i$ together with an integer number $e$ of half-twists in the manner depicted below. 
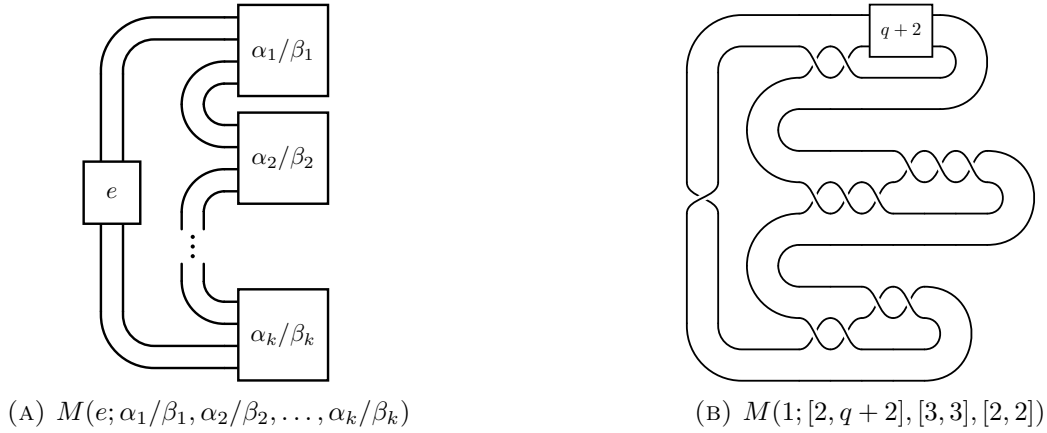
\begin{figure}[H]
    \centering
    \begin{subfigure}[b]{0.47\textwidth}
        \centering
        \resizebox{!}{4.8cm}{\input{figures/montesinos-general}}
    \caption{$M(e;\alpha_1/\beta_1,\alpha_2/\beta_2,\dots,\alpha_k/\beta_k)$}
        \label{fig:montesinos-general}
    \end{subfigure}
    \hfill
    \begin{subfigure}[b]{0.47\textwidth}
        \centering
        \resizebox{!}{5cm}{\input{figures/Jq}}
        \caption{$M(1;[2,q+2],[3,3],[2,2])$}
        \label{fig:montesinos-example}
    \end{subfigure}
    \caption{Montesinos links}
    \label{fig:montesinos}
\end{figure}
\end{definition}
The following three Montesinos links are isotopic (e.g.,~\cite[Rmk.~3.5]{OwS}):
\begin{equation}\label{eq:montesinos}
    M(e;\dots,\frac{\alpha_i}{\beta_i},\dots);\quad M(e+1;\dots,\frac{\alpha_i}{\beta_i+\alpha_i},\dots);\quad M(e-1;\dots,\frac{\alpha_i}{\beta_i-\alpha_i},\dots).
\end{equation}

We follow the conventions of \cite{Iss18,OwS}, where $\Sigma_2(M(e;\dots,\alpha_i/\beta_i,\dots))$ is the Seifert-fibered space $Y(-e; \dots,(\alpha_i,\beta_i),\dots)$. In the conventions of \cite{NR}, whose work we also reference in \Cref{sec:montesinos}, this same Seifert-fibered space is described as $M(g=0, (1,-e),\dots,(\alpha_i,\beta_i),\dots)$.

\begin{definition}[{cf.~\cite{LT88}}]\label{def:reducedM}
 Consider a diagram for a Montesinos link $L=M(e;\alpha_1/\beta_1,\ldots,\alpha_k/\beta_k)$ (with $k\geq 3$) composed of rational tangle diagrams $R_1,\dots, R_k$ assembled as in \Cref{def:montesinos}. We say this is a \emph{reduced Montesinos diagram} of $L$ if either \begin{enumerate}
 \item the diagram is alternating; or
 \item $e=0$, each rational tangle diagram $R_i$ corresponding to $\alpha_i/\beta_i=[a_1,\dots,a_\ell]$ is alternating with at least two crossings, and $a_1\neq 0$.
 \end{enumerate}
\end{definition}
\begin{remark}[Conventions]
A rational tangle $R$ given by $\alpha/\beta=[a_1,a_2,\dots,a_{\ell}]$ as in \Cref{def:rationaltangle} is isotopic to the rational tangle $\tilde{R}$ in ``twist form'' given by the \emph{positive} continued fraction $-\beta/\alpha=[0, -a_1,a_2, -a_3,\dots,\pm a_{\ell}]_+$, where the final term is $-a_{\ell}$ if $\ell$ is odd. Consequently, $R$ is alternating if the terms of its negative continued fraction expansion \emph{alternate} signs. 
The definition of \emph{reduced} in~\cite[\S4]{LT88} requires that the ``lower'' two strands of $\tilde{R}$ be incident to a common crossing within $\tilde{R}$. These two strands correspond to the ``center'' two strands of $R$, so this requirement is satisfied when $a_1\neq 0$. See Kauffman--Lambropoulou~\cite{KL04} for a detailed comparison of rational tangle conventions.
\end{remark}
\begin{lemma}\label{lem:nonaltM}
 If $K$ admits a non-alternating reduced Montesinos diagram, then $K$ is non-alternating. 
\end{lemma}
\begin{proof}
 See the discussion in~\cite[\S4]{T91} connecting the Lickorish--Thistlethwaite definition of reduced Montesinos diagrams to Theorems 1 and 2 of \cite{T87}.
\end{proof}
\subsection{Fillings of double branched covers}\label{subsec:prelim:fillings}
If $W$ is a filling of $\pm \SigmaTwo(K)$, then $\det(K)=t^2|\det (Q_W)|$ for some integer $t$ (see, e.g., \cite[Lem.~2.1]{OwS}), so $\det(Q_W)$ is always odd. Combined with the fact that $\det(K)\equiv (-1)^{\sigma(K)/2}\pmod 4$~\cite{Mur65}, for knots with vanishing signature we have that $|\det(Q_W)|\equiv 1\pmod 4$.
\subsubsection{Checkerboard colorings and Goeritz matrices}\label{subsubsec:checkerboard}
We recall the classical Gordon--Litherland checkerboard construction \cite{GL78}, and largely follow their conventions. A checkerboard coloring of a knot diagram assigns black and white to regions of the diagram so that two regions meeting along an arc have different colors. Each crossing $c$ in the diagram is given an incidence sign $\eta(c)$ and a type according to the following figure:
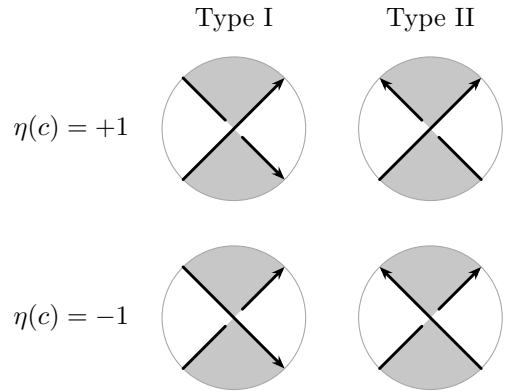
\begin{figure}[H]
    \centering
    \resizebox{!}{4.5cm}{\input{figures/gl-eta-type}}
    \caption{Crossing incidence signs and types of a checkerboard coloring}
    \label{fig:checkerboard-conventions}
\end{figure}

Label the white regions $X_1,\dots,X_{N+1}$, and for $1\leq i,j\leq N+1$ define
\[g_{ij}=\begin{cases}-\sum\eta(c) \text{ where $c$ is incident to }X_i \text{ and } X_j &i\neq j\\
-\sum\limits_{1\leq k\leq N+1; k\neq i}g_{ik}&i=j\end{cases}\]
Then the Goeritz matrix $G$ associated to this coloring of $K$ is given by $G_{ij}=g_{ij}$ where $1\leq i,j\leq N$. $G$ is the intersection form of the double cover of $B^4$ branched over a pushed-in copy of the \emph{black} checkerboard surface~\cite[Thm.~3]{GL78}, which has boundary $\SigmaTwo(K)$. Moreover, defining the Gordon--Litherland correction term $\mu$ as the sum of $\eta(c)$ over the type II crossings, the signature of $K$ is $\sigma(K)=\text{sign}G-\mu$ \cite[Thm.~6]{GL78}.

For alternating knots, one way to compute the Goeritz matrix of a checkerboard coloring is by using a Tait graph of a reduced alternating diagram. The \emph{Tait graph} $\Gamma_C$ associated to each color $C$ has a vertex for each region of that color and an edge for each crossing. Its \emph{graph Laplacian}, $\Delta(\Gamma_C)$, has diagonal entries equal to vertex degrees and off-diagonal entries equal to minus the numbers of edges between vertices. Deleting one row and column of the matrix, corresponding to a vertex in the Tait graph, gives the \emph{reduced Laplacian} $\Delta_{\mathrm{red}}(\Gamma_C)$.
For a black and white checkerboard coloring of an alternating diagram, all crossings have the same incidence sign, which we call $\eta$. The coloring with $\eta=+1$ gives a positive-definite Goeritz matrix $G_+$, and the other coloring gives a negative-definite Goeritz matrix $G_-$. Thus, the reduced Laplacian $\Delta_{\mathrm{red}}(\Gamma_W)$ obtained by deleting the $({N+1})$st row and column of $\Delta(\Gamma_W)$ gives the intersection form of a positive-definite filling of $\eta\SigmaTwo(K)$. For more on the construction of Goeritz matrices from signed checkerboard graphs, see \cite[\S2.1]{Greene13Spanning}. 
\subsubsection{Plumbings over graphs}
A weighted tree $\Gamma$ specifies a plumbing $P_\Gamma$ of oriented disk bundles over $S^2$. Each vertex $v$ represents a bundle with Euler number $w(v)$, and each edge specifies a plumbing of the corresponding bundles. We call $\Gamma$ the \emph{plumbing graph} and $P_\Gamma$ the \emph{plumbing over $\Gamma$}.

The manifold $P_\Gamma$ is simply connected. The zero sections give a basis for its intersection lattice; since $\Gamma$ is a tree, we can orient them so that the matrix has diagonal entries $w(v)$ and off-diagonal entries $-1$ for adjacent vertices and $0$ otherwise. We use the same labels for vertices and their corresponding basis vectors. See Gompf--Stipsicz for more details~\cite[Example 4.6.2]{GS99}.
\section{The lattice embedding obstruction}\label{sec:lattice}
A lattice embedding obstruction argument typically takes two definite smooth $4$-manifolds $X$ and $W$ glued together along their common boundary $Y$ (a rational homology 3-sphere), applies Donaldson's theorem (\Cref{thm:donaldson}) to produce an isometric lattice embedding of $\Lambda_X\oplus\Lambda_W$ into the standard diagonal lattice, and uses properties of $\Lambda_X$ and $\Lambda_W$ to produce a contradiction to the existence of one of the fillings. 

We apply this strategy to obstruct a knot $K$ with vanishing signature from being H-slice in $\#^n\CP$ for any $n$ (i.e., we show $\uCP(K)=\infty$), supposing for the sake of contradiction that $K$ \emph{is} H-slice in $\#^n\CP$. In this section, we prove \Cref{thm:decomp}, which describes a necessary decomposition of the intersection form of \emph{any} positive-definite filling $W$ of $-\Sigma_2(K)$. In \Cref{prop:finiteob} we give a practical way to use \Cref{thm:decomp} as an obstruction. If we can find some filling $W$ whose intersection form does not meet the criteria of \Cref{prop:finiteob}, we have our contradiction and we conclude that $\uCP(K)=\infty$.

To produce \Cref{thm:decomp}, we start with a construction of Kjuchukova--Miller--Ray--Sakall{\i}~\cite{KMRS24}, which extends work of Owens on unknotting numbers (particularly \cite[Thm.~2]{Owe10}) and provides a positive-definite filling $X$ of $\SigmaTwo(K)$ whose intersection lattice is of \emph{half-integer surgery type}. 

\begin{definition}[Half-integer surgery type]\label{def:half}
    A positive-definite integral lattice $\Lambda_X$ of rank $2n$ is of \emph{half-integer surgery type} if $\Lambda_X$ has a basis $u_1,\dots,u_n,v_1,\dots,v_n$ with
\begin{equation}\label{eq:half}
    u_i\cdot u_j=2\delta_{ij},\qquad u_i\cdot v_j=\delta_{ij}.
\end{equation}
%
%
\end{definition}
%
%
\begin{theorem}[{\cite[Thm.~1.1]{KMRS24}}]
\label{thm:KMRS1}
Let $K$ be a knot with $\sigma(K)=0$. If $K$ is H-slice in $\#^{n}\CP$ for some $0\le n<\infty$, then $\SigmaTwo(K)$ bounds a smooth, compact, oriented, positive-definite $4$-manifold $X$ with $b_2(X)=2n$ whose intersection lattice $\Lambda_X=(H_2(X)/\mathrm{Tors},Q_X)$ is of half-integer surgery type.
\end{theorem}
\begin{remark}
We refer the reader to \cite{KMRS24} for the detailed construction of $X$, but we emphasize that the proof of \Cref{thm:KMRS1} uses the hypothesis that $K$ is H-slice in the standard $\#^n\CP$. In particular, the construction involves building a twice-punctured $\#^n\CP$ by attaching 2-handles to $S^3\times[0,1]$ along a $+1$-framed $n$-component unlink in $S^3\times\{1\}$. It remains of interest to ask whether the conclusion of \Cref{thm:KMRS1} holds for all $0$-positive knots with vanishing signature (\Cref{q:exotic-half-integer}).
\end{remark}
 One can then obtain a closed definite $4$-manifold by gluing $X$ to an arbitrary positive-definite filling $W$ of $-\SigmaTwo(K)$, which also gives an isometric embedding of $\Lambda_X\oplus\Lambda_W$ into the standard diagonal lattice. Our \Cref{thm:intro:decomp} then places a necessary condition on the intersection form of $W$. 
\begin{remark}
Kjuchukova--Miller--Ray--Sakall{\i} give examples of gluing specific plumbings to their $X$ with half-integer surgery lattice and applying Donaldson's theorem (e.g., \cite[Prop.~4.1, Prop.~4.4, Cor.~5.5]{KMRS24}). Our theorem extends the lattice analysis used by~\cite{KMRS24} for specific plumbings to arbitrary definite fillings, and our applications focus on obstructing BPH-sliceness.
\end{remark}
We now prove \Cref{thm:intro:decomp} from the introduction, restated here:
\begin{restatement}{thm:intro:decomp}\label{thm:decomp}
Let $K$ be a knot with $\sigma(K)=0$, and let $W$ be a positive-definite filling of $-\SigmaTwo(K)$ with $b_2(W)=r$ and intersection lattice $\Lambda_W$ whose form is given by a matrix $Q_W$. If $K$ is H-slice in $\#^n\CP$, then $\Lambda_W\hookrightarrow(\Z^n,2I)\oplus(\Z^r,I)$. Equivalently, $Q_W$ admits a decomposition 
\begin{equation}\label{eq:paired-decomposition}
 Q_W=2CC^T+DD^T, \quad\text{where}\quad C\in M_{r\times n}(\Z)\quad\text{and}\quad D\in M_r(\Z). 
\end{equation}
\end{restatement}
\begin{remark}\label{rem:decomp-embedding}
    When $n=0$, this gives the lattice embedding condition in the classical Donaldson obstruction to smooth sliceness: $\Lambda_W\hookrightarrow (\Z^r,I)$~\cite{Lis07}. For our applications, it is useful to express this embedding condition as a decomposition of $Q_W$, so that we can enumerate and exclude possible candidates for $T=CC^T$ (as discussed in \Cref{subsec:intro:lattice}).
\end{remark}
\begin{sketch}After gluing $W$ to $X$ as described above (where $\Lambda_X$ has a basis as in \Cref{def:half} and $\Lambda_W$ has some basis $w_1,\dots,w_r$), we show that $Q_W$ admits a decomposition of the form \eqref{eq:paired-decomposition}: first we show that the square-two basis elements $u_i$ of $\Lambda_X$ can be written as the sum of standard basis vectors $e_{2i-1}+e_{2i}$ for $1\leq i\leq n$. Then we let $Q_W=AA^T$ where $w_k$ is the $k$th row of $A$ and use the fact that $w_k\cdot u_i=w_k\cdot (e_{2i-1}+e_{2i})=0$ to describe the columns of $A$. The columns of $A$ provide our decomposition of $Q_W$. Finally, we show \eqref{eq:paired-decomposition} is equivalent to $\Lambda_W\hookrightarrow(\Z^n,2I)\oplus(\Z^r,I)$.
\end{sketch}
\begin{proof}

Suppose $K$ is H-slice in $\#^n\CP$. Since $\sigma(K)=0$, \Cref{thm:KMRS1} provides a positive-definite filling $X$ of $\Sigma_2(K)$ whose intersection lattice $\Lambda_X=(H_2(X)/\mathrm{Tors}, Q_X)$ is of half-integer surgery type with rank $2n$. That is, $\Lambda_X$ admits a basis $u_1,\dots, u_n, v_1,\dots,v_n$ with 
\[
u_i\cdot u_j=2\delta_{ij},
\qquad
u_i\cdot v_j=\delta_{ij}.
\]
We also fix an arbitrary basis $w_1,\dots, w_r$ of $\Lambda_W$.

Since $\Sigma_2(K)$ is a rational homology sphere, we glue $W$ and $X$ together along their common boundary, and \Cref{prop:standard-lattice} gives an isometric embedding of lattices
\begin{equation}\label{eq:thm:embedding}
 \Lambda_X\oplus \Lambda_W\hookrightarrow(\Z^{2n+r},I),
\end{equation} 
in which the two summands have orthogonal images. From now on, we identify the basis vectors of $\Lambda_X$ and $\Lambda_W$ with their images in the standard lattice $(\Z^{2n+r},I)$. In particular, after the embedding we still have $u_i\cdot u_j=2\delta_{ij}$ and $ u_i\cdot v_j=\delta_{ij}$.

\claim{1}{We can assume $u_i=e_{2i-1}+e_{2i}$ for $1\leq i\leq n$.}

In $(\Z^{2n+r},I)$, an integral vector of square $2$ is a sum of two distinct signed standard basis vectors, so for each $u_i$ we have $u_i=\pm e_a\pm e_b$ with $a\ne b$. For distinct $u_i,u_j$ we check that the locations of their nonzero coordinates do not overlap. If they overlapped in exactly one nonzero coordinate location, we would have $u_i\cdot u_j=\pm 1$, contradicting the fact that $u_i\cdot u_j=0$. If they overlapped in both locations, then $u_i\equiv u_j\pmod 2$ so $u_i\cdot v_i \equiv u_j\cdot v_i\pmod 2$. But $u_i\cdot v_i=1$ and $u_j\cdot v_i=0$, a contradiction. 

So we can assume that, up to automorphism of $(\Z^{2n+r},I)$ to relabel and sign-flip coordinates, our embedding \eqref{eq:thm:embedding} satisfies the following property (cf. \cite[proof of Prop.~4.1; Step 1]{KMRS24} and \cite[Lem.~3.3]{Owe10}): 
\[
u_i=e_{2i-1}+e_{2i}\qquad(1\le i\le n).\]

\claim{2}{Let $A\in M_{r\times(2n+r)}(\Z)$ be the matrix whose $k$th \emph{row} is the basis vector $w_k$. Then the \emph{columns} of $A$ are \[(c_1,-c_1,\dots,c_n,-c_n, d_1,\dots,d_r).\]}

Since the embedding is isometric, $\Lambda_W$ is represented by the matrix $Q_W=AA^{T}$. Since the two summands in \eqref{eq:thm:embedding} have orthogonal images, each row $w_k$ of $A$ is orthogonal to every $u_i$, so
\[
0=w_k\cdot u_i=A_{k,2i-1}+A_{k,2i}. \qquad (1\leq i\leq n\:;\:1\leq k\leq r)
\]
That means the $(2i-1)$st and $(2i)$th \emph{columns} of $A$ are negatives of each other. We write them as
$\{c_i,-c_i\}_{1\leq i\leq n}$ and label the remaining $r$ columns $d_1,d_2,\dots,d_r$.

\claim{3}{$Q_W=2CC^T+DD^T$.}

Notice that $AA^T$ is the sum of the outer products of its columns. Now we use the fact that $c_ic_i^T+(-c_i)(-c_i)^T=2c_ic_i^T$, and we set $C=(c_1,\dots,c_n)\in M_{r\times n}(\Z)$ and $D=(d_1,\dots,d_r)\in M_r(\Z)$ to get the desired decomposition \eqref{eq:paired-decomposition}:
\begin{equation*}
Q_W\;=AA^T\;=\;2\sum_{i=1}^{n}c_ic_i^T\;+\;\sum_{j=1}^{r}d_jd_j^T=2CC^T+DD^T.
\end{equation*}
Finally, the decomposition \eqref{eq:paired-decomposition} is equivalent to the embedding $\Lambda_W\hookrightarrow(\Z^n,2I)\oplus(\Z^r,I)$. Sending $w_k$ to the $k$th row of $(C\ D)$ preserves the form, since the pairing of the $k$th and $\ell$th rows in $(\Z^n,2I)\oplus(\Z^r,I)$ is $(2CC^T+DD^T)_{k\ell}$. Conversely, the images of a basis of $\Lambda_W$ under any such embedding form the rows of a matrix $(C\ D)$ satisfying \eqref{eq:paired-decomposition}.
\end{proof}
\begin{remark}
    The decomposition \eqref{eq:paired-decomposition} does not require $W$ to be simply connected. 
\end{remark}
The matrix $C$ in~\eqref{eq:paired-decomposition} has $n$ columns, where $n$ is the number of $\CP$ summands. In particular, the decomposition depends on $n$. To obstruct such a decomposition for all $n$, we instead consider the matrix $T=CC^T\in M_r(\Z)$, which is symmetric and positive semidefinite for any $n$. Recall from \Cref{subsec:prelim:fillings} that since $W$ is a filling of $-\SigmaTwo(K)$, we have that $\det(Q_W)$ is odd, so we can use the following:
\begin{proposition}
\label{prop:Dinvert}
 Let $Q,T,D\in M_r(\Z)$ with $Q=2T+DD^T$. If $\det(Q)$ is odd, then $D$ is invertible over $\F_2$, so $D$ is nonsingular over $\Z$ and $Q-2T$ is positive definite.
\end{proposition}
\begin{proof}
 Reduce $Q=2T+DD^T$ modulo $2$. The even term $2T$ vanishes, leaving $Q\equiv DD^T\pmod 2$. Then since $\det(Q)$ is odd, so is $\det(DD^T)=\det(D)^2$, which means $\det(D)$ is odd too. It follows that $D$ is invertible over $\F_2$ and nonsingular over $\Z$. Consequently, the positive semidefinite matrix $DD^T=Q-2T$ is nonsingular, so it is positive definite.
\end{proof}
In particular, no row $\rho_i$ of $D$ vanishes, so $1\leq \rho_i\cdot\rho_i$. Since $\rho_i\cdot\rho_i=(Q_W)_{ii}-2T_{ii}$, we have $T_{ii}\leq((Q_W)_{ii}-1)/2$. Together with $T_{ii}\geq 0$, the matrices $T=CC^T$ coming from a decomposition as in \Cref{thm:decomp} are restricted to the following set, which depends on $Q_W$ but not on $n$.
\begin{definition}[Candidate set $\mathcal{T}(Q)$]\label{def:candidates}
 Let $Q$ be a positive-definite rank-$r$ integral matrix with odd determinant. Define the \emph{candidate set of $Q$}, denoted $\mathcal T(Q)$, to be the set of integral symmetric positive-semidefinite \emph{candidate} matrices $T\in M_r(\Z)$ satisfying
\[
0\leq T_{ii}\leq \frac{Q_{ii}-1}{2}.
\]
\end{definition}
\begin{remark}
 For any $T\in \mathcal{T}(Q_W)$ the diagonal entries are bounded, and $T^2_{ij}\leq T_{ii}T_{jj}$ because $T$ is positive semidefinite. So the size of $\mathcal{T}(Q_W)$ depends only on the diagonal entries of $Q_W$. For example, if all diagonal entries of $Q_W$ are less than 3, $\mathcal{T}(Q_W)$ contains only $T=0$. Similarly, if every diagonal entry of $Q_W$ is at most $4$, then every candidate has $T_{ii}\in\{0,1\}$.
\end{remark}
Now we can apply \Cref{thm:decomp} as an obstruction using the following proposition, which characterizes the necessary decomposition in terms of the candidate set of $Q_W$:
\begin{proposition}\label{prop:finiteob}
Let $Q$ be a positive-definite integral matrix of rank $r$ and odd determinant, and let $n\geq0$. 
Then $Q$ admits an integral decomposition \[Q=2CC^T+DD^T,\qquad C\in M_{r\times n}(\Z),\quad D\in M_r(\Z),\]if and only if there exists some candidate $T\in\mathcal T(Q)$ such that
\begin{enumerate}[label=\textup{(D\arabic*)}, ref=\textup{(D\arabic*)}]
\item\label{cond:finiteob:embedding} $Q-2T$ is the Gram matrix of $r$ integer vectors in $\Z^r$;
\item\label{cond:finiteob:gram} $T$ is the Gram matrix of $r$ integer vectors in $\Z^n$.
\end{enumerate}
Whenever~\ref{cond:finiteob:embedding} holds, $\det(Q-2T)$ is a positive odd square.
\end{proposition}
\begin{proof}
 Suppose $Q$ admits such a decomposition. Setting $T=CC^T$ satisfies~\ref{cond:finiteob:gram} and gives $Q-2T=DD^T$, which satisfies~\ref{cond:finiteob:embedding} since $D$ is square. Since $\det(Q)$ is odd, \Cref{prop:Dinvert} gives that $\det(D)$ is odd (hence nonzero), so $\det(Q-2T)$ is a positive odd square. The discussion preceding \Cref{def:candidates} shows that $T\in\mathcal{T}(Q)$.
 
 On the other hand, given a $T\in \mathcal{T}(Q)$ passing both conditions, condition~\ref{cond:finiteob:gram} provides some $C\in M_{r\times n}(\Z)$ such that $T=CC^T$, and then $Q-2T=Q-2CC^T=DD^T$ for some $D\in M_r(\Z)$ by condition~\ref{cond:finiteob:embedding}.
\end{proof}
\begin{remark}\label{rem:T0}
The candidate $T=0$ always satisfies condition~\ref{cond:finiteob:gram}, and condition~\ref{cond:finiteob:embedding} becomes $Q_W=DD^T$ with $D$ square. This is precisely the classical lattice embedding obstruction to $\SigmaTwo(K)$ bounding a rational homology ball (and $K$ being smoothly slice), obtained by gluing $W$ to a such a ball and applying Donaldson's theorem. Excluding the nonzero candidates extends this obstruction to H-sliceness in every $\#^n\CP$. Since every candidate must fail for \Cref{thm:decomp} to give an obstruction, a filling $W$ can show that $\uCP(K)=\infty$ only if it already shows, through the classical obstruction, that $\SigmaTwo(K)$ does not bound a rational homology ball and $K$ is not slice. 
\end{remark}
\begin{remark}\label{rem:limitations}
Since our obstruction uses fillings of $\SigmaTwo(K)$, it cannot distinguish two signature-zero knots that have the same oriented double branched cover, such as a positive Conway mutant pair (see~\cite[\S1]{Picc} and \cite{GP26}). For example, the Conway knot is a positive mutant of the slice Kinoshita--Terasaka knot~\cite{Picc}, so its double branched cover bounds a rational homology ball and \Cref{thm:decomp} cannot obstruct it from being BPH-slice. However, the Conway knot is not BPH-slice. Piccirillo constructed a knot $K'$ with the same $0$-trace as the Conway knot and computed $s(K')=2$~\cite{Picc}. By the trace embedding lemma (see \cite[Lem.~3.5]{MP23}), the Conway knot is H-slice in a given $4$-manifold exactly when $K'$ is, so applying \Cref{prop:Hslice-obstructions}\eqref{obs:s} to $K'$ shows that the Conway knot is not H-slice in any $\#^n\CPbar$. 
\end{remark}
 In practice, we start by checking $\det(Q-2T)$ because this is simple to compute, and in some settings it is sufficient to show that $\uCP(K)=\infty$. If $\det(Q-2T)$ is a positive square, we then check condition~\ref{cond:finiteob:embedding}, which does not depend on $n$. In all of our applications, this condition obstructs BPH-sliceness. That means condition~\ref{cond:finiteob:gram}, the only one that depends on $n$, is never used. 

The following criterion shows that condition~\ref{cond:finiteob:gram} is automatic for sufficiently large $n$ in all of our applications. We prove the result in Appendix~\ref{app:D2}, where we also show that its numerical threshold is sharp for individual candidates, and ask whether this condition can provide an obstruction for every $n$ when condition~\ref{cond:finiteob:embedding} alone does not.
\begin{proposition}[proved as {\Cref{cor:D2-auto}}]\label{prop:cor:D2-auto}
Let $Q$ be as in \Cref{prop:finiteob}. If fewer than six diagonal entries of $Q$ are at least $5$, then every $T\in\mathcal T(Q)$ satisfies condition~\ref{cond:finiteob:gram} for all $n\geq\operatorname{rank}(T)+3$. 
\end{proposition}
\subsection{Outline of the applications}\label{subsec:sketch}
The applications of \Cref{thm:decomp} as a lattice embedding obstruction appearing in \Cref{sec:Th,sec:montesinos,sec:2bridge,sec:topvsm} all follow the same structure, which we sketch here.
\begin{enumerate}
 \item For each knot $K$ with vanishing signature that we consider, we describe some simply-connected positive-definite filling $W$ of $-\SigmaTwo(K)$. 
 \item Since $W$ is simply connected, $\det(Q_W)=\det(K)$, so we satisfy the odd determinant hypothesis of \Cref{prop:finiteob}. 
 \item Using elementary lattice arguments inspired by work of Lisca and of Owens, we show that each candidate $T\in \mathcal{T}(Q_W)$ fails condition~\ref{cond:finiteob:embedding} (often by showing that $\det(Q_W-2T)$ is not square).
 \item We conclude that $K$ is not H-slice in $\#^n\CP$ for any $n$ (i.e., $\uCP(K)=\infty$).
\end{enumerate}
The fillings come from checkerboard surfaces for the alternating knots in \Cref{sec:Th,sec:topvsm}, linear plumbings for the 2-bridge knots in \Cref{sec:2bridge}, and star-shaped plumbings for the non-alternating Montesinos knots in \Cref{sec:montesinos}.

We sometimes treat the candidate $T=0$ separately in step (3), since it recovers the classical sliceness obstruction (\Cref{rem:T0}). In \Cref{sec:Th} we use a classical determinant computation for $T=0$. For the knots in \Cref{sec:montesinos} we exclude $T=0$ with a square-two lattice argument, recovering Williams's non-sliceness result for these knots~\cite{Wil08}.

\Cref{prop:cor:D2-auto} applies to the fillings in all of our applications, which is why we never use condition~\ref{cond:finiteob:gram}. The Goeritz matrices in \Cref{sec:Th} have diagonal entries $3$ and $4$, the linear plumbing matrices in \Cref{sec:2bridge} have at most one diagonal entry larger than $4$, and the plumbing matrix in \Cref{sec:montesinos} has diagonal entries $2$ and $3$. The Goeritz matrices used in \Cref{sec:topvsm} have diagonal entries at most $4$.
\section{Obstructing rationally slice Turk's head knots}\label{sec:Th}
In this section we apply \Cref{thm:decomp} and \Cref{prop:finiteob} to \emph{Turk's head knots}, denoted $\Th(p,q)$ where $p$ and $q$ are coprime. We show that when $q\geq 4$ is even, the knot $\Th(3,q)$ is not H-slice in $\#^n\CP$ for any $n$. $\Th(3,q)$ is rationally slice, so we answer a question of Lee~\cite[Question~9]{Lee24}, asking if there exists a rationally slice knot that fails to be BPH-slice. Since $\Th(3,4)$ is the knot $8_{18}$, we simultaneously resolve an example of Manolescu--Piccirillo~\cite{MP23} and complete their classification of the BPH-sliceness of prime knots up to 9 crossings.
\begin{definition}
For coprime $p,q\ge 2$, the \emph{Turk's head knot} $\Th(p,q)$ is the closure of the braid $\beta^q$ where $\beta=\sigma_1\sigma_2^{-1}\cdots\sigma_{p-1}^{(-1)^p}$ on $p$ strands. When $p$ and $q$ are not coprime, the Turk's head construction is a link. We will always assume $\Th(p,q)$ is a knot.
\begin{figure}[H]
    \centering
    \resizebox{!}{4cm}{\input{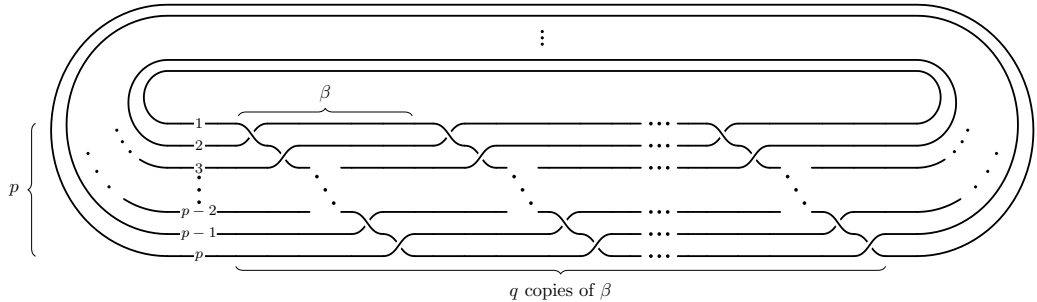}}
    \caption{$\Th(p,q)$, shown for odd $p$}
    \label{fig:Th-p-q}
\end{figure}
\end{definition}
When $p$ is odd, $\Th(p,q)$ has vanishing signature and is strongly negative amphichiral (see~\cite[Thm.~C, Thm.~D]{DPS26}). Strongly negative amphichiral knots are rationally slice, by Kawauchi~\cite{Kaw09}. We focus on the following family of rationally slice Turk's head knots with vanishing signature:
\begin{family}\label{fam:Th3q}
Let $q\geq 4$ be even and not divisible by $3$, and consider the family $\Th(3,q)=\widehat{(\sigma_1\sigma_2^{-1})^q}$ and its associated alternating diagram with $2q$ crossings. The quatrefoil knot, $8_{18}$, is the first member.

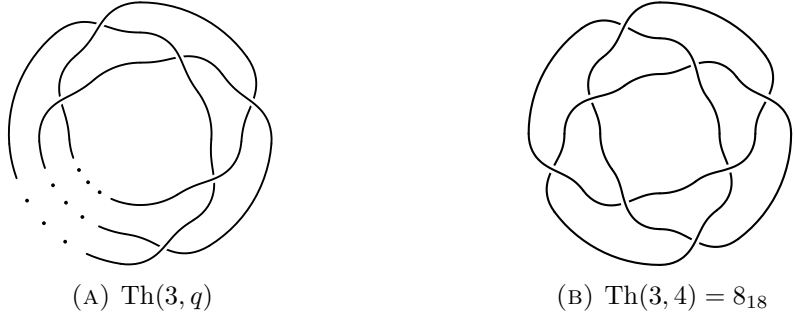
\begin{figure}[H]
  \centering
  \begin{subfigure}[b]{0.25\textwidth}
    \centering
    \resizebox{!}{3.5cm}{\input{figures/Th-3q}}
    \caption{$\Th(3,q)$}
    \label{fig:Th-3q}
  \end{subfigure}
  \hspace{0.15\textwidth} 
  \begin{subfigure}[b]{0.25\textwidth}
    \centering
    \resizebox{!}{3.5cm}{\input{figures/Th-3-4}}
    \caption{$\Th(3,4)=8_{18}$}
    \label{fig:Th-3-4}
  \end{subfigure}
  \caption{Turk's head knots are also known as \emph{rosette} knots}
\end{figure}
\end{family}
We restate \Cref{thm:intro:Th3q} from the introduction here in slightly greater generality and sketch the proof. The full proof is contained in \Cref{subsec:pf:Th3q}.
\begin{restatement}{thm:intro:Th3q}\label{thm:Th3q}
Let $q\geq 4$ be even and not divisible by $3$. The Turk's head knot $\Th(3,q)$ is rationally slice but is not H-slice in $\#^n\CP$ or $\#^n\CPbar$ for any $n$. In particular, $8_{18}=\Th(3,4)$ is not BPH-slice.
\end{restatement}
\begin{sketch}
Using the standard alternating diagram of $K=\Th(3,q)$, we apply the checkerboard construction of Gordon--Litherland to obtain a simply-connected positive-definite filling $W$ of $-\SigmaTwo(K)$ with intersection form given by a matrix $G_q$ (\Cref{lem:Th3qgoeritz}). We then use \Cref{prop:finiteob} to contradict the necessary decomposition described in \Cref{thm:decomp}. In particular, the diagonal entries of $G_q$ are all $3$, which greatly restricts the candidate matrices $T\in \mathcal{T}(G_q)$ we need to check against the conditions of \Cref{prop:finiteob}. We first show in \Cref{prop:Th3qdecomp} that any $T\neq 0$ fails condition~\ref{cond:finiteob:embedding}. A classical computation shows that $\det(G_q)$ is not a perfect square when $q$ is even, so \Cref{prop:Th3q0decomp} shows that $T=0$ fails condition~\ref{cond:finiteob:embedding}. Since $\sigma(K)=0$, we conclude that $\uCP(K)=\infty$ by \Cref{thm:decomp}. Since $K$ is isotopic to $-\Kmirror$, we also have $\uCPbar(K)=\infty$. The assertion about $8_{18}$ follows by setting $q=4$ and noticing that $\Th(3,4)=8_{18}$. 
\end{sketch}
\begin{remark}
 Our choice to study this particular family of knots was originally motivated by interest in the knot $8_{18}$, because of its presence in \cite{MP23}. Showing that $8_{18}$ is not BPH-slice also shows that its positive-slicing degree is in the range $[1,4]$ \cite[\S5 Table~1]{Qin24}. To generalize results to an infinite family, we benefited from discussions with Jaewon Lee and from Di Prisa and \c{S}avk's thorough survey on Turk's head knots~\cite{DPS26}. The survey contains a section dedicated to the family $\Th(3,q)$, which helped us realize that the determinant of these knots is essentially a classical result in graph theory. We record one version of this computation in \Cref{lem:Th3qdet}.
\end{remark}
\subsection{The Goeritz matrix of \texorpdfstring{$\Th(3,q)$}{Th(3,q)}}\label{subsec:Th3qlattice}
We use the classical Gordon--Litherland checkerboard construction, recalled in \Cref{subsubsec:checkerboard}, to describe definite fillings of $\SigmaTwo(\Th(3,q))$ whose intersection forms are given by \emph{Goeritz matrices}. Recall that for a checkerboard coloring of an alternating knot, the reduced Laplacian of the Tait graph corresponding to the white regions represents the positive-definite intersection form of a filling of $\pm\SigmaTwo(K)$. We denote the wheel graph on $q+1$ vertices by $\mathcal{W}_{q+1}$, the cycle on $q$ vertices by $\mathcal{C}_{q}$, and the adjacency matrix of a graph by $\Adj(\cdot)$.
\begin{lemma}\label{lem:Th3qgoeritz}
Consider the standard alternating diagram of
$K=\Th(3,q)$ for $q\geq 3$. A checkerboard coloring of the diagram gives us the following positive-definite Goeritz matrix:
\[
G_q \;=\; 3I_q-\Adj(\mathcal{C}_{q})\;=\;
\begin{pmatrix}
3&-1&0&\cdots&-1\\
-1&3&-1&\ddots&0\\
0&-1&3&\ddots&\vdots\\
\vdots&\ddots&\ddots&3&-1\\
-1&0&\cdots&-1&3
\end{pmatrix}.
\]
Moreover, $G_q$ is the intersection form of a positive-definite filling of $-\Sigma_2(K)$.
\end{lemma}
\begin{proof}
 Consider the checkerboard coloring of the diagram of $\Th(3,q)$ for $q\geq3$ below. There are $q$ labeled regions arranged in a cycle and one central region meeting all of the others, so the Tait graph corresponding to the white regions, $\Gamma_W$, is the wheel $\mathcal{W}_{q+1}$ with the hub vertex $v_{q+1}$ corresponding to the central region (cf. \cite[proof of Thm.~4]{DMMS10}).
\begin{figure}[H]
  \centering
  \begin{subfigure}[b]{0.3\textwidth}
    \centering
    {\tikzset{every picture/.append style={scale=0.73}}\input{figures/Th-3q-checkerboard}}
    \caption{$\Th(3,q)$}
    \label{fig:Th-3q-color}
  \end{subfigure}
  \hspace{0.1\textwidth}
  \begin{subfigure}[b]{0.3\textwidth}
    \centering
    {\tikzset{every picture/.append style={scale=0.83}}\input{figures/Th-3q-tait-white-overlay}}
    \caption{$\mathcal{W}_{q+1}$}
    \label{fig:Th-3q-wheel}
  \end{subfigure}
  \caption{A checkerboard coloring of $\Th(3,q)$ and its Tait graph, $\mathcal{W}_{q+1}$}
\end{figure}
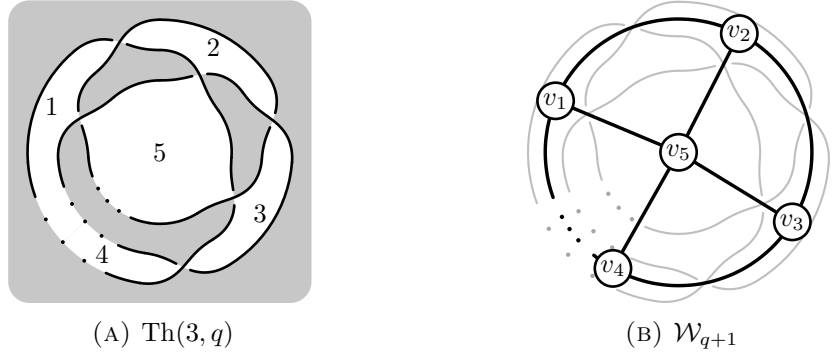
Recall that a graph Laplacian has the degrees of each vertex on the diagonal, $-1$ in the $ij$ entry when $v_i$ and $v_j$ are adjacent, and $0$ elsewhere. Then $\Delta(\mathcal{W}_{q+1})=\diag(3,\dots,3,q)-\Adj(\mathcal{W}_{q+1}).$ Deleting the row and column corresponding to the hub $v_{q+1}$ gives the reduced Laplacian $\Delta_{\mathrm{red}}(\mathcal W_{q+1})
=3I_q-\Adj(\mathcal C_q)$. Then $\Delta_{\mathrm{red}}(\mathcal{W}_{q+1})=G_q$ is a Goeritz matrix for this coloring. Using the checkerboard filling construction, we have that $G_q$ is the intersection form of a positive-definite filling of $\pm\SigmaTwo(K)$. Since $K$ is amphichiral, $\SigmaTwo(K)\cong -\SigmaTwo(K)$, so $G_q$ is both a Goeritz matrix corresponding to \Cref{fig:Th-3q-color} and the intersection form of a positive-definite filling of $-\Sigma_2(K)$.
\end{proof}
\begin{remark}
 Using the sign conventions from \Cref{subsubsec:checkerboard}, computing the Goeritz matrix directly from \Cref{fig:Th-3q-color} also gives $G_q$, but the purpose of the proof above is to emphasize the graph structure of $\Th(3,q)$.
We assume $q\geq3$ because when $q=2$, the Tait graph is not the wheel $\mathcal{W}_3$, and $G_2=\left(\begin{smallmatrix}3&-2\\-2&3\end{smallmatrix}\right)$. The knot $\Th(3,2)=4_1$ is BPH-slice.
\end{remark}
\begin{lemma}[{\cite[\S5.5]{DPS26}}]\label{lem:Th3qdet}
Let $\Luc_n$ and $\Fib_n$ denote the Lucas and Fibonacci numbers. Then
\[
\det(G_q)\;=\;\Luc_{2q}-2
\;=\;
\begin{cases}
5\Fib_q^{2}, & q \text{ even},\\[2pt]
\Luc_q^{2}, & q \text{ odd}.
\end{cases}
\]
Moreover, when $3\nmid q$, $\det(G_q)$ is odd.
\end{lemma}
\begin{proof}
By \Cref{lem:Th3qgoeritz}, $G_q$ is a reduced Laplacian of the wheel $\mathcal{W}_{q+1}$. The Matrix--Tree Theorem gives us $\det(G_q)=\mathrm{Trees}(\mathcal{W}_{q+1})$, where $\mathrm{Trees}(\cdot)$ is the number of spanning trees of a graph. The spanning-tree formula $\mathrm{Trees}(\mathcal{W}_{q+1})=\Luc_{2q}-2$ was computed by Sedl\'a\v{c}ek \cite{Sed70} and Myers \cite{Mye71}. The two determinant formulas now follow from the identities
\[
\Luc_{2q}=\Luc_q^2-2(-1)^q \qquad\text{and}\qquad \Luc_q^2-5\Fib_q^2=4(-1)^q.
\]
Finally, $\Fib_q$ and $\Luc_q$ are odd if and only if $3\nmid q$, so $\det(G_q)$ is then odd.
\end{proof}
\subsection{Proof of \texorpdfstring{\Cref{thm:Th3q}}{Theorem A}}\label{subsec:pf:Th3q}
First we gather the two remaining ingredients needed in the proof. 
\begin{proposition}\label{prop:Th3qdecomp} Let $q\geq 4$ with $3\nmid q$. If $G_q=2T+DD^T$, where $D,T\in M_q(\Z)$ and $T$ is positive semidefinite, then $T=0$.
\end{proposition}
\begin{sketch}
Since the diagonal entries of $G_{q}$ are all $3$, we show that $T$ must be a diagonal matrix with nonzero entries $T_{ii}=1$ exactly when the $i$th row of $D$ has square 1. We let $v$ denote the sum of the rows of $D$ and show that $v$ is characteristic, so every coordinate of $v$ is odd and $v\cdot v\geq q$. But $v\cdot v=q-2(\#\text{nonzero entries of }T)$, so we must have $T=0$.
\end{sketch}
\begin{proof}
Suppose that such a decomposition exists. Since $3\nmid q$, $\det(G_q)$ is odd by \Cref{lem:Th3qdet}. Then \Cref{prop:Dinvert} shows that $D$ is nonsingular and invertible over $\F_2$. We write $\rho_1,\dots,\rho_q\in\mathbb Z^q$ for the rows of $D$. 

\claim{1}{$T=I_\mathcal{N}$, where $I_{\mathcal{N}}$ denotes the matrix with $(I_{\mathcal{N}})_{ii}=1$ for $i$ in the set ${\mathcal{N}}$ and zeros elsewhere.}

We have $0\leq T_{ii}$ since $T$ is positive semidefinite. Since $D$ is nonsingular, \[1\leq \rho_i\cdot\rho_i=(DD^T)_{ii} = 3-2T_{ii}, \]so $T_{ii}\leq 1$. Hence $T_{ii}\in\{0,1\}$. Define the set indexing the nonzero diagonal entries of $T$ as
\[
{\mathcal{N}}:=\{\,i \;:\; T_{ii}=1\,\}=\{\,i\;:\;\rho_i\cdot\rho_i=1\,\}.
\]
Now we show that the remaining entries of $T$ are zero. If $i\not\in \mathcal{N}$, then the $i$th row and column of $T$ vanish because $T_{ii}=0$ and $T$ is positive semidefinite. So $T_{ij}=T_{ji}=0$ when $i\not\in \mathcal{N}$ or $j\not\in \mathcal{N}$. If $i,j\in \mathcal{N}$ and $i\neq j$, we have $\rho_i\cdot\rho_i=\rho_j\cdot\rho_j=1$, so $\rho_i$ and $\rho_j$ are each a signed standard basis vector. Since $D$ is nonsingular, $\rho_i\neq \pm\rho_j$, so $\rho_i\cdot\rho_j=0$. Then $(DD^T)_{ij}=0$, and because $G_q=2T+DD^T$, we have
\[
2T_{ij}=(G_q)_{ij}.
\]
Since $2T_{ij}$ is even, and the off-diagonal entries of $G_q$ are either $0$ or $-1$, we must have $(G_q)_{ij}=0$ and hence $T_{ij}=0$. Thus $T=I_{\mathcal{N}}$.

\claim{2}{$\mathcal{N}=\emptyset$.}

Recall that a vector $v$ in an integral lattice $L$ is \emph{characteristic} if $v\cdot x\equiv x\cdot x \pmod 2$ for every $x\in L$. In the standard lattice $(\Z^q,I)$, a vector is characteristic exactly when each of its coordinates is odd.

Let $v=\rho_1+\cdots+\rho_q$. Because every row of $G_q-2I_{\mathcal{N}}$ has odd sum, $v\cdot \rho_i \equiv \rho_i\cdot\rho_i \pmod 2$ for every $i$. Since $D$ is invertible over $\F_2$, the rows of $D$ form a basis of $(\Z/2)^q$, so $v$ is characteristic in $(\Z^q, I)$. Hence every coordinate of $v$ is odd, so $v\cdot v\geq q$.
 Moreover, $v\cdot v$ is the sum of all the entries of $DD^T$, so $v\cdot v=q-2|{\mathcal{N}}|$ (because $q$ is the sum of the entries of $G_q$ and $2|{\mathcal{N}}|$ is the sum of the entries of $2T$). We now have $v\cdot v=q-2|{\mathcal{N}}|\geq q$. Thus, ${\mathcal{N}}=\emptyset$ and $T=0$.
\end{proof}
\begin{remark}
 The preceding characteristic-vector argument is modeled on Lisca's use of the Wu element of a circular subset in the proof of \cite[Lem.~3.2]{Lis17}.
\end{remark}
\begin{proposition}\label{prop:Th3q0decomp}
 Let $q\geq 4$. If $q$ is even, then $G_q$ does not admit a decomposition $G_q=DD^T$, where $D\in M_q(\Z)$.
\end{proposition}
\begin{proof}
 Suppose $G_q$ did admit such a decomposition. Then $\det(G_q)=(\det(D))^2$ is a perfect square. Since $q$ is even, \Cref{lem:Th3qdet} gives $\det(G_q)=5\Fib_q^2$. But the exponent of $5$ in the prime factorization of $5\Fib_q^2$ is odd, so $5\Fib_q^2$ is not a square, contradicting $G_q=DD^T$.
\end{proof}
\begin{remark}[Cyclic branched covers of the figure-eight knot]\label{rem:Thq-figure8}
The double branched cover $\SigmaTwo(\Th(3,q))$ is homeomorphic to the $q$-fold cyclic branched cover $\Sigma_q(4_1)$ of the figure-eight knot (see~\cite[\S4.7]{DPS26}; this also follows from~\cite[Prop.~2.1]{AMMMPS21}). For even $q\geq4$, the determinant computation in \Cref{prop:Th3q0decomp} shows that $G_q\neq DD^T$, so $\Sigma_q(4_1)$ does not bound a rational homology ball (see \Cref{rem:T0}). When moreover $3\nmid q$, \Cref{prop:Th3qdecomp,prop:Th3q0decomp} together with the proof of \Cref{thm:decomp} show that $\Sigma_q(4_1)$ does not bound a positive-definite $4$-manifold whose intersection lattice is of half-integer surgery type. Since $4_1$ is amphichiral, $\Sigma_q(4_1)\cong-\Sigma_q(4_1)$, so the same holds for negative-definite $4$-manifolds. In contrast, $\Sigma_2(4_1)$ bounds such a manifold by \Cref{thm:KMRS1}, since $4_1$ is H-slice in $\CP$. On the other hand, when $q$ is an odd prime power not divisible by $3$, $\Sigma_q(4_1)$ bounds a rational homology ball~\cite[Thm.~1.1 and Prop.~2.1]{AMMMPS21}. So by \Cref{rem:T0}, our obstruction cannot detect whether $\Th(3,q)$ is BPH-slice for these $q$, no matter which filling we use.
\end{remark}
We now prove \Cref{thm:Th3q}. 

\begin{proof}[Proof of \Cref{thm:Th3q}]
Suppose that $q\geq 4$ is even and not divisible by $3$, and define $K=\Th(3,q)$. By \Cref{lem:Th3qgoeritz}, $G_q$ is the intersection form of a positive-definite filling of $-\Sigma_2(K)$. Since $\det(G_q)$ is odd, \Cref{prop:finiteob} applies. By \Cref{prop:Th3qdecomp}, all nonzero candidates in $\mathcal{T}(G_q)$ fail condition~\ref{cond:finiteob:embedding}. By \Cref{prop:Th3q0decomp}, the candidate $T=0$ does too. It follows from \Cref{prop:finiteob} that $G_q$ does not admit an integral decomposition $G_q=2CC^T+DD^T$ where $C\in M_{q\times n}(\Z)$ and $D\in M_q(\Z)$. Since $\sigma(K)=0$, we apply \Cref{thm:decomp} to conclude that $\uCP(K)=\infty$. Hence $K$ is not BPH-slice.

Since $K$ is isotopic to $-\Kmirror$, mirroring exchanges $\CP$ and $\CPbar$, and reversing knot orientation does not change projective slicing numbers, we have $\uCPbar(K)=\infty$.
\end{proof}
\subsection{BPH-sliceness of other rationally slice Turk's head knots}\label{subsec:Thpq}
When $p$ is odd, $\Th(p,q)$ is rationally slice, so it may be of interest to determine when these knots are BPH-slice. 
When $q$ is even, this question appears to be approachable using \Cref{prop:finiteob} applied to the Goeritz matrix $G_{p,q}$ from the standard alternating diagram of $\Th(p,q)$, although it does not follow directly from the arguments in the $p=3$ case. We make the following conjecture: 
\begin{conj}\label{conj:Thpq}
Let $p\geq3$ be odd and $q\geq4$ be even, with
$\gcd(p,q)=1$. Then $\Th(p,q)$ is not BPH-slice.
\end{conj}
\begin{remark}\label{rem:Thpq-even}
When $q$ is odd, no general analogue of \Cref{thm:Th3q} holds (e.g., $\Th(3,5)=10_{123}$ and $\Th(5,3)=12a_{1019}$ are smoothly slice \cite{KnotInfo}). Moreover, \Cref{prop:Gpq-factorization} shows that the Goeritz matrix $G_{p,q}$, defined in \Cref{lem:Gpq-goeritz}, admits an integral factorization $G_{p,q}=DD^T$ with $D$ square. Thus $T=0$ satisfies both conditions of \Cref{prop:finiteob}, so our lattice obstruction applied to $G_{p,q}$ cannot answer the question in this case. In the case $q=2$, it is known to Golla, Lee, and \c{S}avk that $\Th(p,2)$ is BPH-slice, and we thank O{\u{g}}uz \c{S}avk for bringing this to our attention (see \Cref{lem:Thp2}).
\end{remark}
Although we are primarily interested in the case where $q$ is even, to make this conjecture more approachable we compute the Goeritz matrix $G_{p,q}$ of the standard alternating diagram of $\Th(p,q)$ for odd $p$ and any $q\geq 3$ (\Cref{lem:Gpq-goeritz}), together with its determinant (\Cref{thm:Gpq-det}) in Appendix~\ref{app:Gpq}. 
For $p>3$, the larger rank of $G_{p,q}$ allows more nonzero candidates than in the proof of \Cref{thm:Th3q}. Nevertheless, all diagonal entries are $3$ or $4$, so every candidate has $T_{ii}\in \{0,1\}$, and \Cref{prop:cor:D2-auto} shows that condition~\ref{cond:finiteob:gram} never provides an additional obstruction to BPH-sliceness. The small diagonal entries of $G_{p,q}$ contribute to our suspicion that this conjecture is tractable in the even-$q$ case using the same techniques as when $p=3$. 

In \Cref{thm:Gpq-det} we see that when $q\geq 4$ is even, $\det(\Th(p,q))$ is the $p$th Pell number times some perfect square (cf.~\cite{Sau26}). Since $\Pell_7=169$ is the only square Pell number greater than $1$, $\det(G_{p,q})$ is a perfect square for even $q$ exactly when $p=7$. So excluding the candidate $T=0\in\mathcal T(G_{7,q})$ using condition~\ref{cond:finiteob:embedding} requires a lattice nonembedding argument rather than just the square determinant condition.
\section{Obstructing two families of 2-bridge knots}\label{sec:2bridge}
In this section, we apply our lattice obstruction to two families of 2-bridge knots. We are motivated by work of Daemi, Imori, Sato, Scaduto, and Taniguchi, who define a knot concordance invariant $r_0$ from instanton Floer theory (adapting the invariants $r_s$ of~\cite{NST24}) and use it to obstruct H-sliceness in definite $4$-manifolds: if $\sigma(K)=0$ and $r_0(\Kmirror)<\infty$, then $K$ is not H-slice in any positive-definite closed smooth $4$-manifold with $b_1=0$~\cite[Prop.~6.6, applied to $\overline K$]{DISST25}. We show that \Cref{thm:decomp} obstructs BPH-sliceness in settings where $r_0$ does not obstruct $0$-bipolarity, and that it can be used to recover consequences of this $r_0$ obstruction.

We obstruct each member of the following family of 2-bridge knots from being BPH-slice, and we show that $r_0$ does not obstruct them from being $0$-bipolar. We discuss the significance of this distinction further in \Cref{sec:exotica}.
\begin{family}\label{fam:Lh}
For all odd $h\geq 3$, define the 2-bridge knot $L_h$ as
\begin{equation}\label{eq:2bridge-Lh}
L_h=K(8h^2+4h+1,8h^2-4h+1).
\end{equation}
Using the continued-fraction expansion $(8h^2+4h+1)/(8h^2-4h+1)=[2^{[h-1]},3,2h+1,2]$ gives the following 2-bridge diagram for $L_h$:
\begin{figure}[H]
    \centering
    \resizebox{!}{2.5cm}{\input{figures/Lh}}
    \captionsetup{skip=-4pt}
    \caption{$L_h$}
    \label{fig:Lh}
\end{figure}
\end{family}
We restate \Cref{thm:intro:2bridgenew} from the introduction here and sketch its proof. The full proof is contained in \Cref{subsec:new2b}.
\begin{restatement}{thm:intro:2bridgenew}\label{thm:2bridgenew}
For every odd $h\geq 3$, the 2-bridge knot $L_h=K(8h^2+4h+1,8h^2-4h+1)$
is not H-slice in $\#^n\CP$ for any $n\geq0$. Moreover, $r_0(L_h)=r_0(\overline{L_h})=\infty.$
\end{restatement}
\begin{sketch}
We construct a positive-definite filling of $-\SigmaTwo(L_h)$ with intersection matrix $A_h$ of rank $2h+1$ and odd determinant (\Cref{prop:2bridge-Lh-form}). Then we show that if $A_h-2T=DD^T$, the consecutive $2h-2$ diagonal $2$'s of $A_h$ force the first row $\rho_1$ of $D$ to be supported only on the final two coordinates. When $h$ is odd, a contradiction in the parity of $\rho_1\cdot \rho_2$ shows that every candidate $T\in\mathcal T(A_h)$ fails condition~\ref{cond:finiteob:embedding} of \Cref{prop:finiteob}. Since $\sigma(L_h)=0$, \Cref{thm:decomp} gives $\uCP(L_h)=\infty$. \Cref{prop:Lh-r0} gives $r_0(L_h)=r_0(\overline{L_h})=\infty$.
\end{sketch}

This second family of 2-bridge knots is treated by \cite[Thm.~1.15]{DISST25}. Using $r_0$, they show that for each $K$ in this family, $K$ is not $0$-positive. We recover the consequence that $K$ is not H-slice in $\#^n\CP$ for any $n$. 
\begin{family}\label{fam:DISST}
Following \cite[Thm.~1.15]{DISST25}, for all $m\geq 1$ and $n\geq 0$, define the 2-bridge knot 
\begin{equation}\label{eq:2bridge-Kmn}
K_{m,n}=K(212mn-68n+53,106m-34).
\end{equation}
 Using the continued-fraction expansion $(212mn-68n+53)/(106m-34)=[2n+1,2^{[2m-2]},4,5,4]$ (where $2^{[2m-2]}$ means $2m-2$ consecutive entries equal to $2$, with $2^{[0]}$ denoting no entries) gives the following 2-bridge diagram for $K_{m,n}$. When $n=0$, notice that $K_{m,0}=K(53,106m-34)\simeq K(53,19)$ for every $m$, and the corresponding diagram is isotopic to the one given by $[3,5,4]=53/19$.
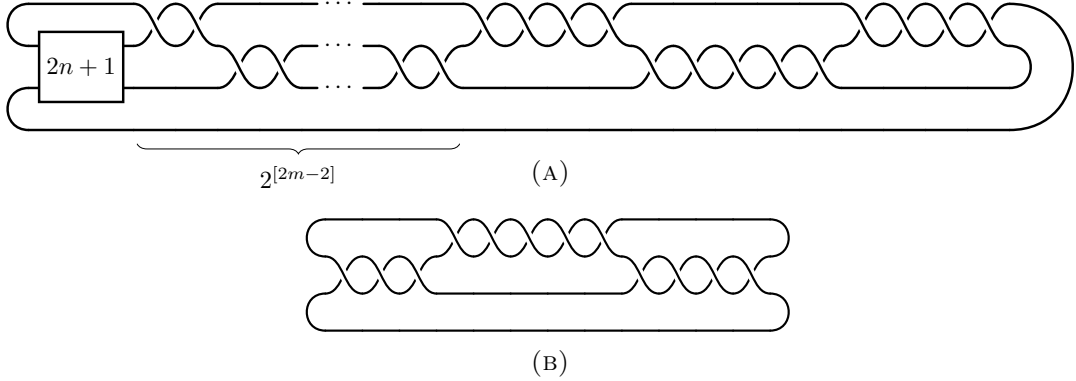
\begin{figure}[H]
\centering
\begin{subfigure}[b]{0.9\textwidth}
    \centering
    \resizebox{!}{2.75cm}{\input{figures/Kmn}}
    \captionsetup{skip=-15pt}
    \caption{}
    \label{fig:Kmn}
\end{subfigure}

\vspace{10pt}

\begin{subfigure}[b]{0.6\textwidth}
    \centering
    \resizebox{!}{1.5cm}
    {\input{figures/K354}}
    \caption{}
    \label{fig:K354}
\end{subfigure}
\caption{(A) $K_{m,n}$ when $n\geq 1$; (B) $K_{m,0}=K(53,19)$}
\end{figure}
\end{family}
\begin{theorem}\label{thm:2bridgeDISST}
For every $m\geq1$ and $n\geq0$, the 2-bridge knot $K_{m,n}$ is not H-slice in $\#^j\CP$ for any $j\geq0$. 
\end{theorem}
\begin{sketch}
We construct a positive-definite filling of $-\SigmaTwo(K_{m,n})$ with intersection matrix $B_{m,n}$ of rank $2n+7$ and odd determinant (\Cref{prop:2bridge-Kmn-form}). If $B_{m,n}-2T=DD^T$, then when $(B_{m,n})_{ii}=2$, the $i$th row of $D\in M_{2n+7}(\Z)$ has square two, and we show that these square-two rows would require $2n+8$ coordinates. This contradiction shows that every candidate $T\in\mathcal T(B_{m,n})$ fails condition~\ref{cond:finiteob:embedding} of \Cref{prop:finiteob}. Since $\sigma(K_{m,n})=0$, \Cref{thm:decomp} gives $\uCP(K_{m,n})=\infty$. The full proof is contained in \Cref{subsec:rec2b}.
\end{sketch}
\subsection{Linear plumbings}\label{subsec:2bridge-plumbings}
Recall that $K(p,q)$ is a knot when $p$ is odd and coprime to $q$, and that we follow the conventions of~\cite{DISST25}, so
\[
\SigmaTwo(K(p,q))\cong S^3_{p/q}(U)\quad\text{and}\quad -\SigmaTwo(K(p,q))\cong S^3_{p/(p-q)}(U).
\]
\begin{definition}\label{def:linear-plumbing}
For integers $a_1,\ldots,a_r\geq2$, let $\Lambda[a_1,\ldots,a_r]$ be the positive-definite, symmetric, tridiagonal matrix with diagonal entries $a_1,\ldots,a_r$ and adjacent off-diagonal entries $-1$:
\[
\Lambda[a_1,\ldots,a_r]=
\begin{pmatrix}
a_1&-1&0&\cdots&0\\
-1&a_2&-1&\ddots&\vdots\\
0&-1&\ddots&\ddots&0\\
\vdots&\ddots&\ddots&a_{r-1}&-1\\
0&\cdots&0&-1&a_r
\end{pmatrix}.
\]
If $p/q=[a_1,\ldots,a_r]$, we also write $\Lambda(p,q)=\Lambda[a_1,\ldots,a_r]$. We call $\Lambda(p,q)$ a \emph{linear plumbing matrix} associated to $p/q$.
\end{definition}
\begin{lemma}\label{lem:2bridge-plumbing}
Let $p>q>0$ be coprime. Then $\Lambda(p,q)$ is positive definite, has determinant $p$, and is the intersection form of a simply-connected linear plumbing with boundary $S^3_{p/q}(U)$. 
\end{lemma}
\begin{proof}
Take the canonical simply-connected negative-definite $4$-manifold plumbing with boundary $S^3_{-p/q}(U)$, as in~\cite{Lis07}. Its intersection form is $-\Lambda(p,q)$, so reversing orientation gives a simply-connected positive-definite plumbing with intersection form $\Lambda(p,q)$ and boundary $ S^3_{p/q}(U)$.
Since $\Lambda(p,q)$ presents $H_1(S^3_{p/q}(U);\Z)$, which has order $p$, $\det(\Lambda(p,q))=p$.
\end{proof}
 $\Lambda(p,q)$ represents the positive-definite linear lattice $(\Z^r,\Lambda(p,q))$. When $p/q=[a_1,\dots,a_r]$, the indices $1,\dots,r$ determine the order of the basis elements of the plumbing lattice. 
\begin{corollary}\label{cor:2bridge-plumbing}
 For coprime $p>q>0$ with $p$ odd, $-\SigmaTwo(K(p,q))$ bounds a positive-definite simply-connected manifold with intersection form given by $\Lambda(p,p-q)$, which has $\det(\Lambda(p,p-q))=p$.
\end{corollary}
We will use the following proposition to normalize square-two rows in our lattice embedding arguments in the proofs of \Cref{thm:2bridgeDISST} and \Cref{thm:2bridgenew}, as well as the proofs of \Cref{prop:JqT0} and \Cref{prop:13a4237sm-CPbar}. Recall that $2^{[k]}$ denotes $k$ consecutive entries equal to $2$. 
\begin{proposition}\label{prop:2bridge-root-count}
Let $D\in M_r(\Z)$ have odd determinant, and suppose that $v_1,\ldots,v_k$ are rows of $D$ with Gram matrix $\Lambda[2^{[k]}]$, where $k\geq1$. Then after permuting and changing signs of the standard basis vectors, we have
\[
v_i=e_i-e_{i+1}\qquad(1\leq i\leq k).
\]
In particular, $D$ has at least $k+1$ distinct columns containing nonzero entries, so $r\geq k+1$.
Moreover, orthogonal square-two rows of $D$ have disjoint supports. Consequently, mutually orthogonal blocks of square-two rows can be normalized in this way on disjoint sets of coordinates.
\end{proposition}
\begin{proof}
Since $v_i\cdot v_i=2$ for all $1\leq i\leq k$, we use the following observation about orthogonal square-two vectors from the proof of \Cref{thm:decomp}: each $v_i$ has exactly two nonzero coordinates, both equal to $\pm1$. When $v_i\cdot v_j=0$, they either have disjoint supports or have the same support. In the latter case, $v_i\equiv v_j\pmod2$. But $\det(D)$ is odd, so $D$ is invertible over $\F_2$ and its rows are linearly independent modulo $2$. In particular, two distinct rows cannot be congruent modulo $2$. Thus $v_i$ and $v_j$ have disjoint supports.

Since $v_i\cdot v_j=0$ when $|i-j|\geq2$, it follows that nonconsecutive rows have disjoint supports. On the other hand, $v_i\cdot v_{i+1}=-1$ implies that consecutive rows share exactly one coordinate. Since nonconsecutive rows have disjoint supports, the support of $v_{i+1}$ contains exactly one coordinate outside of the union of the supports of $v_1,\ldots,v_{i}$.

Label these coordinates so that $v_i$ is supported on the $i$th and $(i+1)$st coordinates. Choose their signs so that $v_1=e_1-e_2$. Once $v_i=e_i-e_{i+1}$, the pairing $v_i\cdot v_{i+1}=-1$ forces the coefficient of $e_{i+1}$ in
$v_{i+1}$ to be $1$. Then we choose the sign of its second coordinate so that $v_{i+1}=e_{i+1}-e_{i+2}$. Continue until we have $v_k=e_k-e_{k+1}$.

The union of the supports therefore has cardinality $k+1$, so $D$ has at least $k+1$ distinct columns with nonzero entries and we must have $r\geq k+1$.

Finally, rows belonging to different, mutually orthogonal blocks have disjoint supports by the first paragraph of the argument, so the coordinate changes for each block can be made independently.
\end{proof}
\subsection{Proof of \texorpdfstring{\Cref{thm:2bridgenew}}{Theorem C}}\label{subsec:new2b}
First we gather the remaining ingredients needed in our proof. 
\begin{proposition}\label{prop:Lh-properties}
For every odd $h\geq3$, $\sigma(L_h)=0$ and $r_0(L_h)=r_0(\overline{L_h})=\infty$.
\end{proposition}
\begin{proof}
The signature is the ordinary-signature case of \Cref{prop:Lh-LT-signature}, and the values of $r_0$ are \Cref{prop:Lh-r0}.
\end{proof}
\begin{proposition}\label{prop:2bridge-Lh-form}
For odd $h\geq3$, the rank-$(2h+1)$ matrix
\begin{equation}\label{eq:2bridge-Ah}
A_h=\Lambda[h+1,3,2^{[2h-2]},3]
\end{equation}
is the intersection form of a simply-connected positive-definite filling of $-\SigmaTwo(L_h)$. Moreover,
\[
\det(A_h)=\det(L_h)=8h^2+4h+1.
\]
\end{proposition}
\begin{proof}
By \Cref{cor:2bridge-plumbing}, $-\SigmaTwo(L_h)$ bounds a positive-definite simply-connected manifold with intersection form given by $\Lambda(8h^2+4h+1,8h)$ with determinant $8h^2+4h+1$. The result follows from expressing this linear plumbing matrix using a continued fraction to show that $A_h=\Lambda(8h^2+4h+1,8h)$.

From our description of \Cref{fam:Lh}, we have the following continued fraction:
\[\frac{8h^2+4h+1}{8h^2-4h+1}=[2^{[h-1]},3,2h+1,2].\]
Using Riemenschneider's point rule (e.g., as in \cite{Lis07}) to compute $\left[\frac{p}{p-q}\right]$ from $\left[\frac{p}{q}\right]$, we have 
\[
\frac{8h^2+4h+1}{8h}=[h+1,3,2^{[2h-2]},3].
\]
\end{proof}
\begin{proposition}\label{prop:Lh-lattice}
Let $h\geq3$ be odd and $r=2h+1$. For every $T\in\mathcal T(A_h)$, the matrix $A_h-2T$ is not the Gram matrix of integer vectors in $\Z^r$. That is, $T$ fails condition~\ref{cond:finiteob:embedding} of \Cref{prop:finiteob}.
\end{proposition}
\begin{proof}
Suppose $A_h-2T=DD^T$, where $D\in M_r(\Z)$. Since $\det(A_h)$ is odd, \Cref{prop:Dinvert} shows that $\det(D)$ is odd. Let $\rho_1,\ldots,\rho_r$ be the rows of $D$. For $3\leq i\leq r-1$, we have $(A_h)_{ii}=2$, so $T_{ii}=0$.
Since $T$ is positive semidefinite, the corresponding rows and columns of $T$ vanish. Thus
\[
\rho_i\cdot\rho_j=(A_h)_{ij}
\qquad(3\leq i\leq r-1,\ 1\leq j\leq r).
\]
In particular, the rows $\rho_3,\ldots,\rho_{r-1}$ have Gram matrix $\Lambda[2^{[2h-2]}]$. Since $\det(D)$ is odd, \Cref{prop:2bridge-root-count} allows us to permute and change signs of the standard basis vectors so that
\[
\rho_3=e_1-e_2,\quad
\rho_4=e_2-e_3,\quad\ldots,\quad
\rho_{r-1}=e_{2h-2}-e_{2h-1}.
\]
We use these coordinates throughout the remainder of the proof.

Since $\rho_1$ is orthogonal to $\rho_3,\ldots,\rho_{r-1}$, its first $2h-1$ coordinates are equal. But
\[
\rho_1\cdot\rho_1=h+1-2T_{11}\leq h+1<2h-1,
\]
so these coordinates must vanish. Therefore $\rho_1=ae_{2h}+be_{2h+1}$ for some $a,b\in\Z$.

Similarly, $\rho_2\cdot\rho_3=-1$ and $\rho_2\cdot\rho_i=0$ for $4\leq i\leq r-1$, so the first $2h-1$ coordinates of $\rho_2$ have the form
\[
(k-1,k,\ldots,k)
\]
for some $k\in\Z$. If $k\neq0$, then $\rho_2\cdot\rho_2\geq(2h-2)k^2\geq4$, contradicting $\rho_2\cdot\rho_2=3-2T_{22}\leq3$. Hence $k=0$, and $\rho_2=-e_1+ce_{2h}+de_{2h+1}$ for some $c,d\in\Z$.

We now reach a contradiction from the parity of $\rho_1\cdot\rho_2$. Since $h$ is odd,
\[
a^2+b^2=h+1-2T_{11}\equiv0\pmod2,
\]
so $a\equiv b\pmod2$. Likewise,
\[
c^2+d^2=\rho_2\cdot\rho_2-1=2-2T_{22} \equiv0\pmod2,
\]
so $c\equiv d\pmod2$. It follows that $\rho_1\cdot\rho_2=ac+bd\equiv0\pmod2$. However, we also have $\rho_1\cdot\rho_2=(A_h)_{12}-2T_{12}=-1-2T_{12}$, which is odd. 

This contradiction shows that $A_h-2T$ is not the Gram matrix of integer vectors in $\Z^{r}$. 
\end{proof}
\begin{proof}[Proof of \Cref{thm:2bridgenew}]
By \Cref{prop:2bridge-Lh-form}, $A_h$ is the intersection form of a positive-definite filling of $-\SigmaTwo(L_h)$ and has odd determinant so \Cref{prop:finiteob} applies. Then \Cref{prop:Lh-lattice} shows that every $T\in\mathcal{T}(A_h)$ fails condition~\ref{cond:finiteob:embedding} of \Cref{prop:finiteob}. Since $\sigma(L_h)=0$ by \Cref{prop:Lh-properties}, \Cref{thm:decomp} gives $\uCP(L_h)=\infty$. Also by \Cref{prop:Lh-properties}, $r_0(L_h)=r_0(\overline{L_h})=\infty$. 
\end{proof}
\subsection{Proof of \texorpdfstring{\Cref{thm:2bridgeDISST}}{Theorem 5.1}}\label{subsec:rec2b}
Again we start by gathering the remaining ingredients needed in our proof. 
\begin{lemma}[{\cite[Lem.~6.5]{DISST25}}]\label{lem:2bridge-DISST-invariants}
For $m\geq1$ and $n\geq0$, $\sigma(K_{m,n})=0$.
\end{lemma}
\begin{proposition}\label{prop:2bridge-Kmn-form}
For $m\geq1$ and $n\geq0$, the rank-$(2n+7)$ matrix
\begin{equation}\label{eq:2bridge-Bmn}
B_{m,n}=\begin{cases}
\Lambda[2,3,2^{[2]},3,2^{[2]}],&n=0,\\
\Lambda[2^{[2n-1]},2m+1,2,3,2^{[2]},3,2^{[2]}],&n\geq1
\end{cases}
\end{equation}
is the intersection form of a simply-connected positive-definite filling of $-\SigmaTwo(K_{m,n})$, and
\[
\det(B_{m,n})=\det(K_{m,n})=212mn-68n+53.
\]
\end{proposition}
\begin{proof}
By \Cref{cor:2bridge-plumbing}, recalling that $K_{m,0}=K(53,19)$, we have that $-\SigmaTwo(K_{m,n})$ bounds a positive-definite simply-connected manifold with intersection form given by $\Lambda(53,53-19)$ when $n=0$ and by $\Lambda(212mn-68n+53,(212mn-68n+53)-(106m-34))$ when $n\geq 1$. In both cases, we have determinant $212mn-68n+53$. The result follows from expressing these two linear plumbing matrices using continued fractions.

From our description of \Cref{fam:DISST}, we have the following continued fractions:
\[\frac{53}{19}=[3,5,4], \quad \text{and }\quad\frac{212mn-68n+53}{106m-34}=[2n+1,2^{[2m-2]},4,5,4]\quad (n\geq 1).\]
Using Riemenschneider's point rule (e.g., as in \cite{Lis07}) to compute $\left[\frac{p}{p-q}\right]$ from $\left[\frac{p}{q}\right]$, we have 
\begin{align*} 
\frac{53}{53-19}&=[2,3,2^{[2]},3,2^{[2]}],\quad\text{and}\\
\frac{212mn-68n+53}{212mn-68n+53-(106m-34)} &= [2^{[2n-1]},2m+1,2,3,2^{[2]},3,2^{[2]}].
\end{align*}
\end{proof}
\begin{proposition}\label{prop:Kmn-lattice}
Let $m\geq1$, $n\geq0$, and $r=2n+7$. For every $T\in\mathcal T(B_{m,n})$, the matrix $B_{m,n}-2T$ is not the Gram matrix of integer vectors in $\Z^r$. That is, $T$ fails condition~\ref{cond:finiteob:embedding} of \Cref{prop:finiteob}.
\end{proposition}
\begin{proof} 

Suppose $B_{m,n}-2T=DD^T$, where $D\in M_r(\Z)$. Since $\det(B_{m,n})$ is odd, \Cref{prop:Dinvert} shows that $\det(D)$ is odd. Let $\rho_1,\ldots,\rho_r$ be the rows of $D$. If $(B_{m,n})_{ii}=2$, then $T_{ii}=0$. Since $T$ is positive semidefinite, the $i$th row and column of $T$ vanish, so
\[
\rho_i\cdot\rho_j=(B_{m,n})_{ij}
\qquad(1\leq j\leq r).
\]
In particular, $\rho_i$ has square two.

We will now show that the blocks of square-two rows of $D$ require $r+1$ distinct coordinate positions, contradicting $D\in M_r(\Z)$.

When $B_{m,n}$ has a consecutive string of $k$ diagonal entries equal to $2$, $D$ has a block of $k$ square-two rows. Then the maximal strings of twos in $\diag(B_{m,0})$ have lengths $1$, $2$, and $2$. When $n\geq1$, the maximal strings of twos in $\diag(B_{m,n})$ have lengths $2n-1$, $1$, $2$, and $2$. Rows of $D$ corresponding to different maximal blocks are orthogonal, because their indices differ by at least two, and their pairings agree with those of the tridiagonal matrix $B_{m,n}$.

Since $\det(D)$ is odd, \Cref{prop:2bridge-root-count} shows that a block of length $k$ uses exactly $k+1$ coordinates, and that different blocks have disjoint supports. The union of these supports therefore has cardinality
\[
\begin{cases}
2+3+3=8,&n=0,\\
2n+2+3+3=2n+8,&n\geq1.
\end{cases}
\]
In both cases, the square-two rows require $2n+8=r+1$ distinct coordinate positions, contradicting $D\in M_r(\Z)$.
\end{proof}
\begin{proof}[Proof of \Cref{thm:2bridgeDISST}]
By \Cref{prop:2bridge-Kmn-form}, $B_{m,n}$ is the intersection form of a positive-definite filling of $-\SigmaTwo(K_{m,n})$ and has odd determinant so \Cref{prop:finiteob} applies. By \Cref{prop:Kmn-lattice}, every $T\in \mathcal{T}(B_{m,n})$ fails condition \ref{cond:finiteob:embedding}
of \Cref{prop:finiteob}. Since $\sigma(K_{m,n})=0$ by \Cref{lem:2bridge-DISST-invariants}, \Cref{thm:decomp} gives $\uCP(K_{m,n})=\infty$.
\end{proof}
\cite[Thm.~1.15]{DISST25} proves the stronger statement that these knots are not H-slice in any positive-definite closed smooth $4$-manifold with $b_1=0$.
%
\section{Obstructing a non-alternating family of Montesinos knots}\label{sec:montesinos}
In this section, we apply our lattice obstruction to a non-alternating family of Montesinos knots described by Williams~\cite[\S5]{Wil08}, where the filling comes from a plumbing over a star-shaped graph rather than a checkerboard surface or a linear plumbing. Williams obstructed smooth sliceness for this family, and we extend his result to show that this family is not H-slice in $\#^n\CP$ for any $n$. We are motivated by showing that the utility of our obstruction is not tied to alternating knots, and by the many vanishing obstructions to $0$-bipolarity for this family (discussed in \Cref{sec:exotica}).

Recall that diagrams of Montesinos knots are constructed by connecting rational tangle diagrams together with some number of twists (see \Cref{subsec:prelim:knot} for diagram conventions, which agree with \cite{Iss18, OwS}). 
\begin{family}\label{fam:Jq}
For each odd integer $q\geq1$, define the Montesinos knot \[J_q=M\left(1;\frac{2q+3}{q+2},\frac{8}{3},\frac{3}{2}\right),\]
where the diagram below uses these continued fraction expansions of each rational tangle parameter: $\frac{2q+3}{q+2}=[2,q+2]$; $\frac{8}{3}=[3,3]$; and $\frac{3}{2}=[2,2]$.

\begin{figure}[H]
    \centering
    \resizebox{!}{5.25cm}{\input{figures/Jq}}
    \caption{$J_q=M\left(1;[2,q+2],[3,3],[2,2]\right)$}
    \label{fig:Jq}
\end{figure}
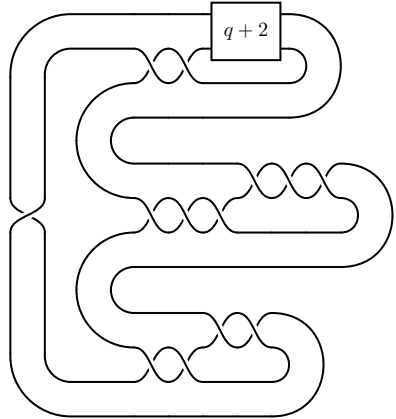
\end{family}
We restate \Cref{thm:intro:Jq} from the introduction and sketch its proof. The full proof is contained in \Cref{subsec:pf:Jq}.
\begin{restatement}{thm:intro:Jq}\label{thm:Jq}
For every odd $q\geq 1$, the non-alternating Montesinos knot $J_q=M\bigl(1;\frac{2q+3}{q+2},\frac{8}{3},\frac{3}{2}\bigr)$ is not H-slice in $\#^n\CP$ for any $n\geq0$.
\end{restatement}
\begin{sketch}
$J_q$ is non-alternating (\Cref{prop:Jq-properties}) and $-\SigmaTwo(J_q)$ bounds a simply-connected positive-definite plumbing with intersection matrix $S_q\in M_{q+6}(\Z)$ (\Cref{prop:Jq-lattice}). We then use \Cref{prop:finiteob} to contradict the necessary decomposition described in \Cref{thm:decomp}. The diagonal entries of $S_q$ are $3$ in exactly three locations, and are $2$ elsewhere, which greatly restricts the candidate matrices $T\in \mathcal T(S_q)$. First we show in \Cref{prop:JqTneq0} that every nonzero $T\in\mathcal T(S_q)$ fails condition~\ref{cond:finiteob:embedding} of \Cref{prop:finiteob}. \Cref{prop:JqT0} shows that $T=0$ fails condition~\ref{cond:finiteob:embedding}, reproving Williams's lattice embedding obstruction to $S_q\hookrightarrow\Z^{q+6}$ (although when $\det(J_q)$ is not a square, $T=0$ fails condition~\ref{cond:finiteob:embedding} automatically). Since $\sigma(J_q)=0$ (\Cref{prop:Jq-properties}), we conclude that $\uCP(J_q)=\infty$.
\end{sketch}
\begin{proposition}\label{prop:Jq-properties}
 Let $q\geq 1$ be odd. Then $J_q$ is quasi-alternating, non-alternating, $\det(J_q)=26q+51$, and $\sigma(J_q)=0$.
\end{proposition}
\begin{proof}
Using our original description of $J_q$, we have $e=1$, each rational tangle parameter is greater than $1$, and $\frac{3}{3-2}=3>\frac{8}{3}$, so by \cite[Thm.~1]{Iss18}, $J_q$ is quasi-alternating.

Recall from \Cref{subsec:prelim:knot}, equation~\eqref{eq:montesinos}, that a Montesinos knot with $e$ twists and a rational tangle $R=\alpha/\beta$ is isotopic to the knot with $e-1$ twists where $R$ is replaced by $R'=\alpha/(\beta-\alpha)$, and that different continued fraction expansions of the same rational number yield an isotopic rational tangle diagram. Then we have
\[M\left(1;\frac{2q+3}{q+2},\frac{8}{3},\frac{3}{2}\right)\simeq M\left(0;\frac{2q+3}{q+2},\frac{8}{3},\frac{3}{2-3}\right)\simeq M\left(0; [1,-1,q+1],[2,-1,2],-3\right)\]
The new diagram is non-alternating, and we highlight two consecutive over-crossings below:
\begin{figure}[H]
    \centering
    \resizebox{!}{5.5cm}{\input{figures/Jq-nonalt-highlight}}
    \caption{$J_q=M\left(0; [1,-1,q+1],[2,-1,2],-3\right)$}
    \label{fig:Jq-nonalt-highlight}
\end{figure}
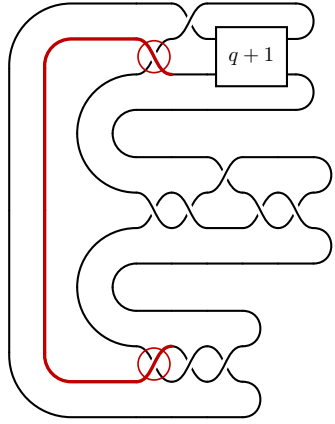
Since this diagram is non-alternating, and each of its rational tangle diagrams is alternating with at least two crossings, this is a \emph{reduced Montesinos diagram} by \Cref{def:reducedM}. By \Cref{lem:nonaltM}, $J_q$ is non-alternating. 

Using the determinant formula for Montesinos links (see \cite[proof of Thm.~1]{Iss18}), we have 
\[\det(J_q)=\left|(2q+3)(8)(3)\left(1-\frac{q+2}{2q+3}-\frac{3}{8}-\frac{2}{3}\right)\right|=26q+51.\]

\Cref{prop:Jq-signatures} shows that $J_q$ has vanishing Levine--Tristram signature function, so $\sigma(J_q)=0$. 
\end{proof}
\begin{remark}\label{rem:Jq-signature}
    We can also compute the signature and determinant of $J_q$ using a checkerboard coloring. Following the conventions in \Cref{subsubsec:checkerboard} and letting the point at infinity be white, we get a Goeritz matrix $G$ with $|\det(G)|=26q+51$ and $\text{sign}(G)=-2$. When $q$ is odd, the correction term is $\mu=-2$, so $\sigma(J_q)=-2-(-2)=0$. Note that when $q$ is even, the correction term is $-4$, so those Montesinos knots do not have vanishing signature. 
\end{remark}
\begin{remark}
    To see $J_q$ in the notation of~\cite{Wil08}, take $n_1=q$, $m_1=n_2=1$, $m_2=2$, and the final parameter to be $-3$. That gives $M\left(0;\frac{2q+3}{q+2},\frac{8}{3},-3\right)$. Note that Williams claims this knot has vanishing signature for all $q\geq 1$, but \Cref{rem:Jq-signature} shows this only holds when $q$ is odd. 
\end{remark}
To illustrate the sensitivity of condition~\ref{cond:finiteob:embedding} compared to the square determinant condition, we point out the infinite subfamily indexed by
\begin{equation}
\mathcal Q=\bigl\{26m^2-10m-1,\ 26m^2+10m-1\;:\;m\geq1\bigr\}.
\end{equation}
These are precisely the integers $q$ for which $26q+51$ is a perfect square~\cite[\S5]{Wil08}.
%
%
If $q\in\mathcal Q$, condition~\ref{cond:finiteob:embedding} excludes a candidate ($T=0$) that passes the square determinant condition.
\subsection{The star-shaped plumbing lattice of \texorpdfstring{$J_q$}{Jq}}\label{subsec:Jq-lattice}
\begin{proposition}\label{prop:Jq-lattice}
For every odd $q\geq1$, let $P_q$ be the plumbing over the weighted tree below. Then $P_q$ is simply connected and positive definite, with boundary $-\SigmaTwo(J_q)$ and intersection matrix $S_q$, defined below. 
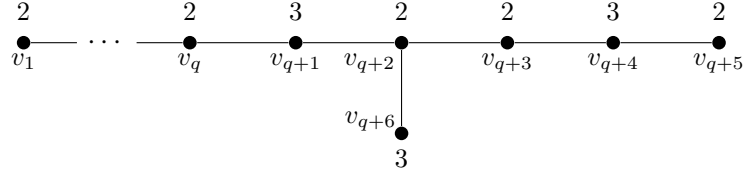
\[
\input{figures/sq-plumbing}
\]
There are $q$ vertices of weight $2$ preceding $v_{q+1}$, and
\[(S_q)_{ij}=\begin{cases}
    w(v_i),&i=j\\
    -1, &\text{if }v_i \text{ and }v_j \text{ are adjacent} \\
    0, & \text{otherwise}.
\end{cases}
\]
\end{proposition}
\begin{proof}
Recall (as in the proof of \Cref{prop:Jq-properties}) that we can replace all three of our original rational tangle parameters $\alpha_i/\beta_i$ with $\alpha_i/(\beta_i-\alpha_i)$ and lower the twist parameter by three to get 
\[
J_q\simeq M(-2;-\frac{2q+3}{q+1},-\frac{8}{5},-3)\simeq M(-2;[-3,(-2)^{[q]}],[-2,-3,-2],-3).
\]
Then $\SigmaTwo(J_q)$ bounds the $4$-manifold plumbing over a weighted star-shaped graph obtained by negating the weight of each vertex in the graph above~\cite[\S2, Fig.~4]{Iss18}. We check that the plumbing is negative definite by computing its rational Euler number, $\mathcal{E}_{\Q}$: 
\[
\mathcal{E}_{\Q}\left(\SigmaTwo(J_q)\right)=-2-\left(-\frac{q+1}{2q+3}-\frac{5}{8}-\frac{1}{3}\right)=\frac{q+1}{2q+3}-\frac{25}{24}<0.
\] 
By \cite[Thm.~5.2]{NR}, the plumbing is negative definite. Reversing orientations, we have that $-\SigmaTwo(J_q)=\partial P_q$. Using its vertex basis ordered as $v_1,\ldots,v_{q+6}$, the positive-definite intersection form of $P_q$ is given by the matrix $S_q$.
\end{proof}
%
%
\subsection{Proof of \texorpdfstring{\Cref{thm:Jq}}{Theorem D}}\label{subsec:pf:Jq}
To show that the nonzero candidate matrices $T\in\mathcal{T}(S_q)$ fail condition~\ref{cond:finiteob:embedding} we contradict the positive-definiteness of $S_q-2T$. When $T=0$, that strategy is not viable because $S_q$ \emph{is} positive definite. Instead, we use the square-two vectors of the plumbing lattice to reach a contradiction in the parity of $v_{q+1}\cdot v_{q+4}$. Although a more uniform argument (similar to the $T=0$ case) is possible, we separate the cases to emphasize that more tools are available to obstruct nonzero candidates. 
\begin{proposition}\label{prop:JqTneq0}
Let $q\geq 1$ be odd and $r=q+6$. If $T\in\mathcal T(S_q)$ and $S_q-2T$ is the Gram matrix of integer vectors in $\Z^r$, then $T=0$.
\end{proposition}
\begin{sketch}
We first show that $T$ must be a diagonal matrix with nonzero entries $T_{ii}=1$ only if $i\in \{q+1,q+4,q+6\}$ (cf. the proof of \Cref{prop:Th3qdecomp}). Next we show that, for each remaining candidate $T\neq 0$, $S_q-2T$ is not positive definite. We conclude that $T=0$.
\end{sketch}
\begin{proof}
Suppose $T$ passes condition~\ref{cond:finiteob:embedding} of \Cref{prop:finiteob}, so $S_q-2T=DD^T$ with $D\in M_r(\Z)$. By \Cref{prop:Dinvert}, $S_q-2T$ is positive definite.

\claim{1}{$T=I_{\mathcal{N}}$ where $I_{\mathcal{N}}$ denotes the diagonal matrix with $(I_{\mathcal{N}})_{ii}=1$ for $i\in {\mathcal{N}}$ and $(I_{\mathcal{N}})_{ii}=0$ otherwise.}

The only diagonal entries of $S_q$ greater than $2$ are the three $3$'s on the diagonal, located at indices $q+1$, $q+4$, and $q+6$. Since $0\leq T_{ii}\leq ((S_q)_{ii}-1)/2$, we can define the set indexing the nonzero diagonal entries of $T$ as
\[
{\mathcal{N}}=\{i\::\:T_{ii}=1\}\subseteq\{q+1,q+4,q+6\}.
\]
Now we show the remaining entries of $T$ are zero. Positive semidefiniteness gives $T_{ij}^2\leq T_{ii}T_{jj}$, so every row and column with zero diagonal entry vanishes.

Let $\rho_i$ be the $i$th row of $D$. If $T_{ii}=1$, then $\rho_i\cdot\rho_i=3-2=1$, so $\rho_i$ is a signed standard basis vector. If $i\neq j$ and $T_{ii}=T_{jj}=1$, nonsingularity of $D$ gives $\rho_i\neq\pm\rho_j$, hence $\rho_i\cdot\rho_j=0$. The vertices indexed by ${\mathcal{N}}$ are pairwise nonadjacent, so $(S_q)_{ij}=0$. Then we have 
\[
0=\rho_i\cdot\rho_j=(S_q)_{ij}-2T_{ij}=-2T_{ij}.
\]
It follows that $T=I_{\mathcal{N}}$.

\claim{2}{If ${\mathcal{N}}\neq\emptyset$, then $S_q-2T$ is not positive definite.}

Since $S_q-2T$ is positive definite, we have $x^T(S_q-2T)x>0$ for every nonzero $x\in \Z^r$. Since $T=I_{\mathcal{N}}$ by Claim 1, we have \[x^T(S_q-2T)x=x^TS_qx-2\sum_{i\in {\mathcal{N}}}x_i^2.\]
Now for each $i\in\{q+1,q+4,q+6\}$, we construct a vector $x=\sum x_ke_k$ so that $x^TS_qx=2x_i^2$. 

To make the construction more geometrically natural, consider the plumbing tree for $P_q$, where $w(v_k)$ denotes the weight of the vertex $v_k$ and $(v_a,v_b)\in E$ if $a<b$ and there is an edge between $v_a$ and $v_b$. Then we have
\[
x^TS_qx=\sum_kw(v_k)x_k^2-2\sum_{(v_a,v_b)\in E}x_ax_b.
\]
This interpretation makes the entries of the following table particularly easy to check.
\[
\begin{array}{c|l|c|c}
i & x & x^TS_qx & 2x_i^2\\ \hline
q+1 & e_q+2e_{q+1}+e_{q+2} & 8 & 8\\
q+4 & e_{q+3}+2e_{q+4}+e_{q+5} & 8 & 8\\
q+6 & e_q+2e_{q+1}+5e_{q+2}+3e_{q+3}+e_{q+4}+5e_{q+6}
&50&50
\end{array}
\]
Suppose ${\mathcal{N}}\neq\emptyset$ and choose $i\in {\mathcal{N}}\subseteq \{q+1,q+4,q+6\}$. Take $x$ from the table above. In each case, $x^TS_qx=2x_i^2$. Since $i\in {\mathcal{N}}$,
\[
x^T(S_q-2T)x=2x_i^2-2\sum_{j\in {\mathcal{N}}}x_j^2\leq 0 .
\]
But $S_q-2T$ is positive definite by \Cref{prop:Dinvert}, so $x^T(S_q-2T)x>0$ for every nonzero $x$. This contradiction shows
${\mathcal{N}}=\emptyset$, and therefore $T=0$.
\end{proof}
\begin{remark}
 To construct a vector $x$ such that $x^TS_qx=2x_i^2$ by hand, use the plumbing tree. Take $i\in {\mathcal{N}}$, set the weight of $v_i$ to $1$, blow down $v_i$, and then continue with a series of blowdowns (doing a change of basis to $S_q-2I_{\{i\}}$ each time) until the corresponding matrix $S$ is simple enough that we can see some $\tilde x$ such that $\tilde x^TS\tilde x=0$. Then reverse the change of basis to get $x$. See \cite[\S5]{NR} for examples.
\end{remark}
\begin{proposition}\label{prop:JqT0}
For every odd integer $q\geq1$, the lattice $(\Z^{q+6}, S_q)$ does not embed in the standard diagonal lattice of the same rank. That is, $T=0$ fails condition~\ref{cond:finiteob:embedding} of \Cref{prop:finiteob}.
\end{proposition}
\begin{proof}
Suppose such an embedding exists, and identify the ordered vertex basis $v_1,\ldots,v_{q+6}$ with its image in $\Z^{q+6}$. Let $D\in M_{q+6}(\Z)$ be the matrix whose $i$th row is $v_i$. Then $S_q=DD^T$. Since $\det(S_q)=26q+51$ is odd, $\det(D)$ is odd.

We first normalize the square-two rows of $D$ using a familiar argument, and then we use the pairings and squares of $v_{q+1}$ and $v_{q+4}$ to obtain a contradiction.

\claim{1}{After permuting and changing signs of the standard basis vectors, we may assume that our square-two rows of $D$ have the form $v_i=e_{i}-e_{i+1}$.}

The square-two rows form three mutually orthogonal blocks:
\[
v_1,\ldots,v_q;\qquad
v_{q+2},v_{q+3};\qquad
v_{q+5}.
\]
Their Gram matrices are $\Lambda[2^{[q]}]$, $\Lambda[2^{[2]}]$, and $\Lambda[2]$ (respectively). By \Cref{prop:2bridge-root-count}, these blocks can be normalized on disjoint sets of coordinates. Together they use $(q+1)+3+2=q+6$ coordinates, and we can arrange them as $v_i=e_i-e_{i+1}$ for $i\in \{1, \dots, q, q+2, q+3, q+5\}$.

\claim{2}{The pairing $v_{q+1}\cdot v_{q+4}$ given by $S_q$ is incompatible with Claim 1.}

The pairings of $v_{q+1}$ and $v_{q+4}$ with each of the vectors in Claim~1 give us
\[
\begin{aligned}
v_{q+1}
  &=(\underbrace{t,\ldots,t}_{q},t+1;\
      a-1,a,a;\
      c,c),\\
v_{q+4}
  &=(\underbrace{b,\ldots,b}_{q+1};\
      d,d,d+1;\
      f,f+1)
\end{aligned}
\]
for integers $t,a,b,c,d,f$, where the semicolons separate coordinate blocks of sizes $q+1$, $3$, and $2$.

The first two blocks of $v_{q+1}$ are both nonzero. Since $v_{q+1}\cdot v_{q+1}=3$, and an integral vector of square three has exactly three nonzero coordinate (each $\pm1$), the final block $(c,c)$ must be $(0, 0)$.
Similarly, since $v_{q+4}$ is square three, it has three nonzero entries equal to $\pm 1$. The last two blocks of $v_{q+4}$ are both nonzero, so we must have $b=0$. Since the final two coordinates of $v_{q+4}$ differ by one, exactly one of them is nonzero. Then the middle block $(d,d,d+1)$ has exactly two nonzero coordinates, giving us $d=-1$. Thus
\[
v_{q+1}\cdot v_{q+4}
=(a-1)(-1)+a(-1)+a\cdot0
=1-2a\neq0.
\]
This contradicts the orthogonality of $v_{q+1}$ and $v_{q+4}$.
\end{proof}
\begin{remark}
    The preceding proposition also recovers Williams's obstruction to smooth sliceness for these knots \cite[\S4.2]{Wil08}.
\end{remark}
\begin{proof}[Proof of \Cref{thm:Jq}]
By \Cref{prop:Jq-lattice}, $S_q$ represents the intersection form of a positive-definite filling of $-\SigmaTwo(J_q)$, and $\det(S_q)$ is odd, so \Cref{prop:finiteob} applies. \Cref{prop:JqTneq0} shows that every nonzero $T\in\mathcal T(S_q)$ fails condition~\ref{cond:finiteob:embedding}, and \Cref{prop:JqT0} rules it out for $T=0$. So \Cref{prop:finiteob} excludes every possible decomposition $S_q=2CC^T+DD^T$. Since $\sigma(J_q)=0$ by \Cref{prop:Jq-properties}, \Cref{thm:decomp} shows that $\uCP(J_q)=\infty$. 
\end{proof}
\section{Topological versus smooth BPH-sliceness}\label{sec:topvsm}
In this section we compare the topological and smooth notions of BPH-sliceness. Recall that we write $\uCPtop$ and $\uCPbartop$ for the topological counterparts of $\uCP$ and $\uCPbar$, using locally flat disks in place of smooth disks, and we say a knot is topologically BPH-slice when both $\uCPtop$ and $\uCPbartop$ are finite. 

Combining \Cref{thm:decomp} with Seifert matrix decompositions that give upper bounds on the topological projective slicing numbers, we prove the following refinement of \Cref{thm:intro:topvsm}. We give examples of prime alternating knots that are topologically BPH-slice but not smoothly BPH-slice. To the author's knowledge, these are the first such alternating examples pointed out in the literature. Each example has topological projective slicing numbers at most two. One example is rationally slice, incidentally providing an additional answer to \cite[Question~9]{Lee24}. Our examples extend the comparisons of differing topological and smooth projective slicing numbers in~\cite{KMRS24}.
\begin{theorem}\label{thm:topvsm-examples}
The following prime alternating knots with vanishing signature have the indicated projective slicing numbers:
\[
\begin{array}{c|cc|cc}
K&\uCPtop(K)&\uCPbartop(K)&\uCP(K)&\uCPbar(K)\\ \hline
13a_{4237}&1&1&1&\infty\\
10_{109}&a&a&\infty&\infty\\
10_{28}&1&1&\infty&1\\
12a_{792}&b&c&\text{not determined here}&\infty
\end{array}
\]
where $a,b,c\in\{1,2\}$. Moreover:
\begin{enumerate}
\item\label{item:13a4237} $13a_{4237}$ satisfies the Fox--Milnor criterion.
\item\label{item:10_109} $10_{109}$ is rationally slice.
\item\label{item:2bridge-top}$r_0(\overline{10_{28}})<\infty$, while $r_0(12a_{792})=r_0(\overline{12a_{792}})=\infty$.
\end{enumerate}
\end{theorem}
\begin{sketch}
None of these knots are topologically slice, so all of their projective slicing numbers must be nonzero. We show the finite smooth projective slicing numbers using crossing changes, and the remaining topological projective slicing numbers using the Seifert matrix decomposition criterion of \Cref{thm:KMRS-topological}~(\cite[Thm.~1.3]{KMRS24}). The infinite smooth projective slicing numbers of $13a_{4237}$ and $10_{109}$ follow from applying our lattice obstruction to checkerboard fillings. Observing that $K_{m,0}\simeq 10_{28}$ and $\overline{L_3}\simeq 12a_{792}$ gives us $\uCP(10_{28})$, $\uCPbar(12a_{792})$, and \eqref{item:2bridge-top}. The proofs appear in \Cref{subsec:13a4237,subsec:10_109,subsec:2bridge-top}, with the Seifert matrix computations deferred to Appendix~\ref{subsec:seifert}.
\end{sketch}

Many known obstructions to BPH-sliceness vanish for the knots $13a_{4237}$, $10_{109}$, and $12a_{792}$ (while $10_{28}$ is already obstructed from being $0$-bipolar by \cite{DISST25}). We discuss these invariants and their implications in \Cref{sec:exotica}.
\begin{remark}
 These examples are not unique in their projective slicing properties or their listed vanishing invariants. We chose them because they nicely illustrate the separation of topological and smooth notions for prime alternating knots. 
\end{remark}
Our interest in prime alternating knots is motivated by the following conjecture of Dunfield--Gong.
\begin{conj}[{\cite[Conj.~1.15]{DG}}]\label{conj:DG}
A prime alternating knot is smoothly slice if and only if it is topologically slice.
\end{conj}
Dunfield--Gong establish this equivalence for prime alternating knots with at most $19$ crossings~\cite[Thm.~1.7]{DG}. It is therefore natural to ask whether topological and smooth BPH-sliceness coincide for prime alternating knots. Our examples show that the corresponding agreement between categories already fails for BPH-sliceness at low crossing number, even with small topological slicing numbers and many vanishing invariants.

For non-alternating knots, the distinction between smooth and topological BPH-sliceness already appears in the literature.
\begin{example}
Let $K$ be the positive-clasped untwisted Whitehead double of the right-handed trefoil. As discussed
in~\cite[\S1.2]{KMRS24}, $K$ is topologically BPH-slice (in fact it is topologically slice) but not smoothly BPH-slice. It is prime because it has Seifert genus one, and non-alternating because it is nontrivial with Alexander polynomial one. Here $\tau(K)=1$ obstructs smooth BPH-sliceness.
\end{example}
For prime alternating knots, \cite{KMRS24} provides examples where the topological and smooth projective slicing numbers differ, but all four quantities are finite in their examples. 
\subsection{Topological upper bounds}\label{subsec:topvsm-seifert}
We use the following result to bound the topological projective slicing numbers of our examples.
\begin{theorem}[{\cite[Thm.~1.3]{KMRS24}}]\label{thm:KMRS-topological}
Let $A$ be a $2g\times2g$ Seifert matrix for a knot $K$. Suppose there are an integral matrix $B\in M_{2g}(\Z)$ and column vectors $c_1,\ldots,c_n\in\Z^{2g}$ such that
\[
A=B-\sum_{i=1}^n c_ic_i^T, \qquad \det(tB-B^T)=\pm t^k
\]
for some integer $k$. Then $\uCPtop(K)\leq n$.
\end{theorem}
Using the Seifert matrix $-A^T$ for the mirror gives the corresponding upper bound on $\uCPbartop(K)$. Each application below can be checked from a Seifert matrix and a list of integral vectors. A knot obtained from a smoothly slice knot by a negative-to-positive crossing change is H-slice in $\CP$~\cite[Lem.~2.4]{MP23}, which also gives a topological upper bound.
\subsection{\texorpdfstring{$13a_{4237}$}{13a4237}} \label{subsec:13a4237}
This example separates the two categories even when both topological projective slicing numbers are as small as possible for a nonslice knot. It also satisfies the Fox--Milnor criterion and is smoothly H-slice in $\CP$, while our lattice obstruction rules out H-sliceness in every $\#^n\CPbar$.
The following four propositions prove the projective slicing assertions for
$13a_{4237}$ and show~\eqref{item:13a4237}. 
%
%
\begin{proposition}\label{prop:13a4237sm-CP}
$\uCP(13a_{4237})=1$.
\end{proposition}
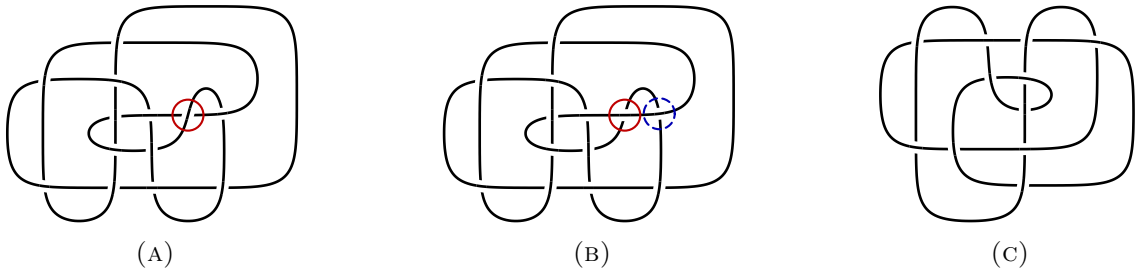
\begin{figure}[H]
  \centering
  \tikzset{every picture/.append style={scale=0.4}}
  \begin{subfigure}[b]{0.33\textwidth}
    \centering
    \input{figures/cc3-1-13a4237}
    \caption{}\label{fig:cc-a}
  \end{subfigure}\hfill
  \begin{subfigure}[b]{0.33\textwidth}
    \centering
    \input{figures/cc3-2-switched}
    \caption{}\label{fig:cc-b}
  \end{subfigure}\hfill
  \begin{subfigure}[b]{0.30\textwidth}
    \centering
    \input{figures/cc3-3-snappy-mirror-11a169}
    \caption{}\label{fig:cc-c}
  \end{subfigure}
  \caption{(A) SnapPy's diagram of $13a_{4237}$, with the crossing to be changed circled. (B) $K'$ is the diagram after the crossing change. The two circled crossings cancel by a Reidemeister II move. (C) SnapPy's diagram of $\overline{11a_{169}}$.}
  \label{fig:cc}
\end{figure}
\begin{proof}
Changing the marked positive crossing in \Cref{fig:cc-a} gives $\overline{11a_{169}}$, which is smoothly slice~\cite{KnotInfo}. Reversing this change recovers $13a_{4237}$ by a negative-to-positive crossing change, so  $\uCP(13a_{4237})\leq1$. Since $13a_{4237}$ is not smoothly slice (e.g., \Cref{prop:13a4237sm-CPbar}), we have $\uCP(13a_{4237})=1$.

To see that $K'\simeq\overline{11a_{169}}$, note that the two circled crossings in \Cref{fig:cc-b} cancel by a Reidemeister~II move, leaving a reduced alternating $11$-crossing diagram of $K'$. As a diagram on $S^2$, it differs from SnapPy's diagram of $\overline{11a_{169}}$ in \Cref{fig:cc-c} by a single flype, but rather than exhibit this isotopy, we identify $K'$ by its Jones polynomial. With KnotInfo's normalization,
\[
V_{K'}(t) = -t^{7}+3t^{6}-7t^{5}+12t^{4}-16t^{3}+19t^{2}-19t+18-13t^{-1}+8t^{-2}-4t^{-3}+t^{-4},
\]
which is $V_{11a_{169}}(t^{-1})$. Since $K'$ is alternating and $t^{4}V_{K'}(t)$ is irreducible over $\mathbb{Z}$, $K'$ is prime \cite{Men84}. Among prime knots with at most $11$ crossings and their mirror images, $\overline{11a_{169}}$ is the only one with this Jones polynomial \cite{KnotInfo}, so $K'\simeq\overline{11a_{169}}$.
\end{proof}
\begin{proposition}\label{prop:13a4237sm-CPbar}
    $\uCPbar(13a_{4237})=\infty$.
\end{proposition}
\begin{proof}
To show $\uCPbar(13a_{4237})=\infty$, we use the following checkerboard coloring, which has $\eta=+1$:
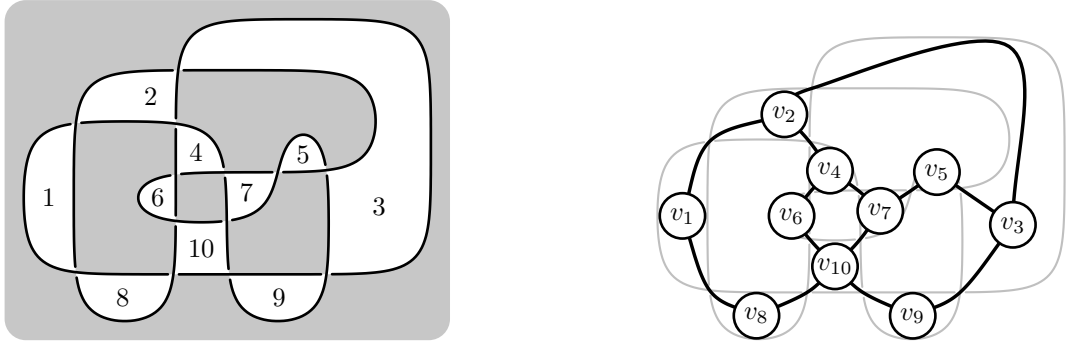
\begin{figure}[H]
  \centering
  \begin{subfigure}[b]{0.4\textwidth}
    \centering
    {\tikzset{every picture/.append style={scale=0.5625}}\input{figures/13a4237-checkerboard}}
  \end{subfigure}
  \hspace{0.1\textwidth}%
  \begin{subfigure}[b]{0.4\textwidth}
    \centering
    {\tikzset{every picture/.append style={scale=0.5625}}\input{figures/13a4237-tait-white-overlay}}
  \end{subfigure}
  \caption{A checkerboard coloring of $13a_{4237}$ and its Tait graph $\Gamma_W$}
  \label{fig:13a4237-checkerboard-graph}
\end{figure}
Since $\eta=+1$, the checkerboard construction gives a positive-definite filling of $\SigmaTwo(13a_{4237})$ whose intersection form is given by $G=\Delta_{\mathrm{red}}(\Gamma_W)$. Labeling the vertices of the Tait graph as in the figure and deleting $v_{10}$ gives $\det(G)=289$ and diagonal entries $(2,3,3,3,2,2,3,2,2)$. The remaining entries can be read off from $\Gamma_{W}$ or the checkerboard coloring directly.

Suppose $G-2T=DD^T$ for $T\in\mathcal T(G)$ and $D\in M_9(\Z)$, and write $\rho_1,\ldots,\rho_9$ for the rows of $D$. The rows and columns of $T$ indexed by $1,5,6,8,9$ vanish. Consequently, these rows of $D$ have square $2$ and retain all their pairings from $G$.

By \Cref{prop:Dinvert}, $\det(D)$ is odd. Among the five specified square-two rows, only $\rho_1$ and $\rho_8$ have nonzero pairing (because $v_1$ and $v_8$ are adjacent in $\Gamma_W$, $\rho_1\cdot\rho_8=-1$). Thus they form mutually orthogonal blocks of lengths $1,1,1,2$. By \Cref{prop:2bridge-root-count}, after permuting and changing signs of coordinates, we may write
\[
\rho_5=e_1-e_2,\qquad
\rho_9=e_3-e_4,\qquad
\rho_6=e_5-e_6,\qquad
\rho_1=e_7-e_8,\qquad
\rho_8=e_8-e_9.
\]

Write $\rho_3=(x_1,\ldots,x_9)$. Its square is $3-2T_{33}\in\{1,3\}$, so every coordinate belongs to $\{0,\pm1\}$. Its pairings with the five displayed rows give
\[
x_1-x_2=x_3-x_4=-1,\qquad
x_5=x_6,\qquad
x_7=x_8=x_9.
\]
The first four coordinates contribute exactly $2$ to its square. The last three must vanish, since otherwise they contribute at least $3$ more. The remaining two coordinates contribute an even number, so $\rho_3$ has even square, a contradiction. Thus every candidate, including $T=0$, fails condition~\ref{cond:finiteob:embedding}. Since $\sigma(13a_{4237})=0$, applying \Cref{thm:decomp} to $\overline{13a_{4237}}$ gives $\uCPbar(13a_{4237})=\infty$.
\end{proof}
\begin{proposition}\label{prop:13a4237top}
$\uCPtop(13a_{4237})=\uCPbartop(13a_{4237})=1$.
\end{proposition}
\begin{proof}
\Cref{prop:13a4237sm-CP} gives $\uCPtop(13a_{4237})\leq\uCP(13a_{4237})=1$.
For the other orientation, \Cref{prop:13a4237-seifert} supplies a Seifert matrix decomposition with one vector, so \Cref{thm:KMRS-topological} gives $\uCPbartop(13a_{4237})\leq1$. By \Cref{prop:13a4237sm-CPbar}, this knot is not smoothly slice. Then \cite[Thm.~1.7]{DG} implies that it is not topologically slice, so neither topological slicing number is zero.
\end{proof}
\begin{proposition} \label{prop:13a4237-alexander-signature}
The Alexander polynomial of $K=13a_{4237}$ satisfies the Fox--Milnor criterion and $K$ has vanishing Levine--Tristram signature function.
\end{proposition}
\begin{proof}
From KnotInfo, we have $\Delta_K(t)=-6+ 30t-66t^2+ 85t^3-66t^4+ 30t^5-6t^6$. Let $f(t)=2-6t+6t^2-3t^3$. Then
\[f(t)f(t^{-1})=t^{-3}\left(-6+ 30t-66t^2+ 85t^3-66t^4+ 30t^5-6t^6\right)\doteq\Delta_K(t),
\]
so $\Delta_K(t)$ satisfies the Fox--Milnor criterion. KnotInfo also gives $\sigma_{LT}=0$.
\end{proof}
\subsection{\texorpdfstring{$10_{109}$}{10.109}}\label{subsec:10_109}
 We show that $10_{109}$ is topologically BPH-slice, but not smoothly BPH-slice. Since $10_{109}$ is negative amphichiral~\cite{KnotInfo}, it is rationally slice by~\cite[Thm.~A]{DPLS25}, so this provides an additional answer to \cite[Question~9]{Lee24}.
\begin{proposition}\label{prop:topvsm-10_109-upper}
$\uCPtop(10_{109})=\uCPbartop(10_{109})\in\{1,2\}$.
\end{proposition}
\begin{proof}
\Cref{prop:10_109-seifert} supplies a Seifert matrix decomposition with two vectors, so \Cref{thm:KMRS-topological} gives $\uCPtop(10_{109})\leq2$. Since $10_{109}$ is negative amphichiral, its topological projective slicing numbers are equal. Since $\det(10_{109})=85$ is not a square, $10_{109}$ fails the Fox--Milnor criterion and is not topologically slice, so its topological slicing numbers are either $1$ or $2$. 
\end{proof}
\begin{remark}
As mentioned in \Cref{subsec:org}, the following example is meant to give a gentle introduction to the strategy used to prove \Cref{thm:Th3q}. We illustrate how we can use determinant computations to obstruct BPH-sliceness. The proof uses the checkerboard filling construction from \Cref{subsec:prelim:fillings} and the results of \Cref{sec:lattice}, but does not depend on the preceding applications. We use these two consequences from condition~\ref{cond:finiteob:embedding} of \Cref{prop:finiteob}: If a candidate $T$ satisfies condition~\ref{cond:finiteob:embedding}, then $G-2T$ is positive definite (\Cref{prop:Dinvert}), and $\det(G-2T)$ is a perfect square (\Cref{prop:finiteob}).
\end{remark}
\begin{proposition}\label{prop:topvsm-10_109-smooth}
    $\uCP(10_{109})=\uCPbar(10_{109})=\infty$.
\end{proposition}
\begin{proof}
\emph{Step 1: build a definite filling.} Consider the following checkerboard coloring of $K=10_{109}$ and its corresponding Goeritz matrix:
\[
  \vcenter{\hbox{%
    \tikzset{every picture/.append style={scale=0.5}}%
    \input{figures/10_109-checkerboard}%
  }}
  \qquad
  G=\begin{pmatrix}
    4&-1&0&-1&0\\
    -1&2&0&0&-1\\
    0&0&4&-1&-1\\
    -1&0&-1&2&0\\
    0&-1&-1&0&3
  \end{pmatrix}
\]
Since $K$ is alternating, the checkerboard construction (\Cref{subsec:prelim:fillings}) gives a positive-definite filling $W$ of $\pm\SigmaTwo(K)$ with intersection matrix $G$. Since $K$ is amphichiral, $\SigmaTwo(K)\cong-\SigmaTwo(K)$, so $W$ is a positive-definite filling of $-\SigmaTwo(K)$. We have $\sigma(K)=0$, and $\det(G)=85$ is odd, so \Cref{prop:finiteob} applies to $Q=G$.

\emph{Step 2: describe the candidate set.} A candidate $T\in\mathcal T(G)$ satisfies $0\leq T_{ii}\leq (G_{ii}-1)/2$ (\Cref{def:candidates}). The diagonal of $G$ is $(4,2,4,2,3)$, so
\[
T_{22}=T_{44}=0\qquad\text{and}\qquad T_{11},T_{33},T_{55}\in\{0,1\}.
\]
Since $T$ is positive semidefinite, $T_{ij}^2\leq T_{ii}T_{jj}$. Hence the second and fourth rows and columns of $T$ vanish, and each of $T_{13}$, $T_{15}$, $T_{35}$ lies in $\{-1,0,1\}$ and can be nonzero only if both of the corresponding diagonal entries of $T$ equal $1$.

\emph{Step 3: when could $G-2T$ be positive definite?} Suppose $T$ satisfies condition~\ref{cond:finiteob:embedding}, so that $G-2T$ is positive definite and in particular its principal $2\times2$ minors are positive. If $T_{11}=T_{33}=1$, then the minor of $G-2T$ on rows and columns $1$ and $3$ is
\[
\det\begin{pmatrix}2&-2T_{13}\\-2T_{13}&2\end{pmatrix}=4-4T_{13}^2,
\]
so $T_{13}=0$. In the same way, if $T_{11}=T_{55}=1$, then the minor on rows $1$ and $5$ is $2-4T_{15}^2$, so $T_{15}=0$; and if $T_{33}=T_{55}=1$, then the minor on rows $3$ and $5$ is $2-(1+2T_{35})^2$, so $T_{35}\in\{0,-1\}$. These conditions leave us with ten possibilities: the eight diagonal matrices with $T_{11},T_{33},T_{55}\in\{0,1\}$, and the two matrices with $T_{33}=T_{55}=1$ and $T_{35}=T_{53}=-1$.

\emph{Step 4: compute some determinants.} For example, the candidate with $T_{33}=1$ as its only nonzero entry gives
\[
G-2T=\begin{pmatrix}
    4&-1&0&-1&0\\
    -1&2&0&0&-1\\
    0&0&2&-1&-1\\
    -1&0&-1&2&0\\
    0&-1&-1&0&3
  \end{pmatrix},\qquad \det(G-2T)=27,
\]
which is not a perfect square. The determinants of all ten candidates are as follows, where $T_{35}=0$ unless otherwise indicated:
\[
\begin{array}{c|ccccc}
(T_{11},T_{33},T_{55})&000&100&010&001&110\\ \hline
\det(G-2T)&85&23&27&3&5
\end{array}
\qquad
\begin{array}{c|ccccc}
(T_{11},T_{33},T_{55})&101&011&111&011&111\\
T_{35}&0&0&0&-1&-1\\ \hline
\det(G-2T)&-3&-7&-5&-3&-1
\end{array}
\]
For the five candidates on the left, $\det(G-2T)$ is positive but not a perfect square. The first of these is $T=0$, where this is the classical obstruction (\Cref{rem:T0}): $\det(10_{109})=85$ is not a square, so $10_{109}$ is not slice. For the five candidates on the right, $\det(G-2T)<0$, so $G-2T$ is not even positive definite.

\emph{Step 5: conclusion.} No candidate satisfies condition~\ref{cond:finiteob:embedding}, so by \Cref{prop:finiteob}, $G$ admits no decomposition $G=2CC^T+DD^T$ for any $n$. Since $\sigma(K)=0$, \Cref{thm:decomp} gives $\uCP(K)=\infty$. Since $K$ is amphichiral, its two projective slicing numbers are equal (\Cref{subsec:prelim:hslice}), so $\uCPbar(K)=\infty$ as well.
\end{proof}
\subsection{\texorpdfstring{$10_{28}$ and $12a_{792}$}{10.28 and 12a792}}\label{subsec:2bridge-top}
$10_{28}$ is $K_{m,0}$ from \Cref{fam:DISST} (see \Cref{fig:K354}) and is obstructed from being $0$-positive by \Cref{prop:0pos-obstructions}\eqref{obs:r0} because its mirror has $r_0<\infty$ (by the proof of \cite[Thm.~1.15]{DISST25}). $12a_{792}$ is $\overline{L_3}$ from \Cref{fam:Lh} (shown in \Cref{fig:12a792}), so we obstruct it from being H-slice in connected sums of the standard $\CPbar$ (\Cref{thm:2bridgenew}) and show that $r_0$ does not obstruct it from being $0$-negative (\Cref{prop:Lh-r0}). For the remaining values in the table, we show that both of these knots are topologically BPH-slice, and that $\uCPbar(10_{28})=1$.

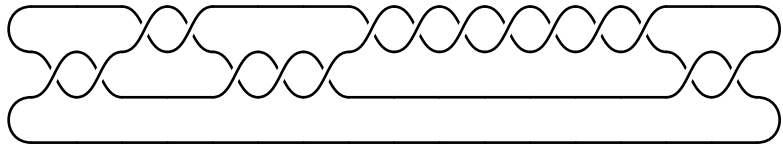
\begin{figure}[h]
    \centering
    \input{figures/L3-mirror}
    \caption{$12a_{792}$ is isotopic to $\overline{L_3}$, given by $[-2,-2,-3,-7,-2]$}
    \label{fig:12a792}
\end{figure}
\begin{proposition}\label{prop:1028-sm}
$\uCPbar(10_{28})=1.$
\end{proposition}
\begin{proof}
Changing the marked negative crossing in \Cref{fig:cc1028-1} gives us the knot $8_8$, which is smoothly slice~\cite{KnotInfo}. Since $10_{28}$ is not slice, but it is obtained from a slice knot by changing a positive crossing to a negative one, $\uCPbar(10_{28})=1$.
\end{proof}
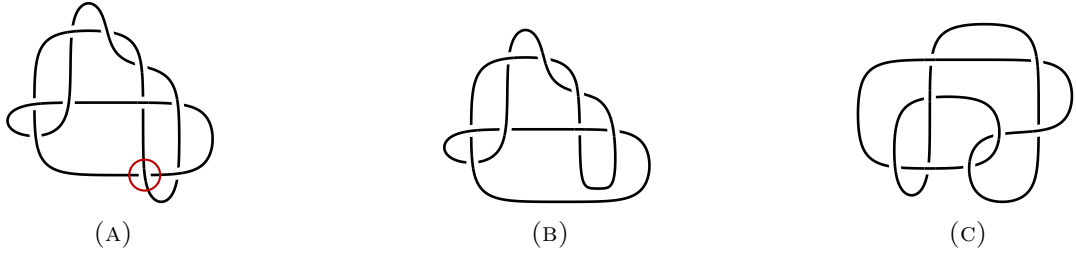
\begin{figure}[H]
  \centering
  \tikzset{every picture/.append style={scale=0.4}}
  \begin{subfigure}[b]{0.33\textwidth}
    \centering
    \input{figures/cc1028-1-10_28}

    \caption{}\label{fig:cc1028-1}
  \end{subfigure}\hfill
  \begin{subfigure}[b]{0.33\textwidth}
    \centering
    \input{figures/cc1028-2-afterR2}
    \caption{}\label{fig:cc1028-2}
  \end{subfigure}\hfill
  \begin{subfigure}[b]{0.30\textwidth}
    \centering
    \input{figures/cc1028-3-8_8}
    \caption{}\label{fig:cc1028-3}
  \end{subfigure}
  \caption{(A) KnotInfo's diagram of $10_{28}$, with the crossing to be changed circled; (B) The diagram after the crossing change and a Reidemeister II move; (C) KnotInfo's diagram of $8_8$, which is isotopic to (B)}
  \label{fig:cc1028}
\end{figure}
\begin{proposition}\label{prop:1028-top}
$\uCPtop(10_{28})=\uCPbartop(10_{28})=1$.
\end{proposition}
\begin{proof}
In \Cref{prop:1028-seifert} we give a Seifert matrix decomposition using one vector, so by~\Cref{thm:KMRS-topological}, $\uCPtop(10_{28})\leq 1$ and $\uCPbartop(10_{28})\leq 1$ (the latter also follows from \Cref{prop:1028-sm}). Since $10_{28}$ is not topologically slice (it fails the Fox--Milnor criterion), we have $\uCPtop(10_{28})=\uCPbartop(10_{28})=1$.
\end{proof}
\begin{proposition}\label{prop:topvsm-12a792-upper}
\[\uCPtop(12a_{792})\in \{1,2\};\quad \uCPbartop(12a_{792})\in\{1,2\}.\]
\end{proposition}
\begin{proof}
\Cref{prop:12a792-seifert} supplies a Seifert matrix decomposition with two vectors for $12a_{792}$ and $\overline{12a_{792}}$, so $\uCPtop(12a_{792})\leq2$ and $\uCPbartop(12a_{792})\leq 2$ by \Cref{thm:KMRS-topological}. Its determinant is $\det(12a_{792})=85$, so $12a_{792}$ is not topologically slice and its topological slicing numbers are either $1$ or $2$. 
\end{proof}
\section{Exotica, \texorpdfstring{$0$}{0}-bipolarity, and comparisons with other invariants}\label{sec:exotica}
As discussed in \Cref{subsec:intro:exotica}, a $0$-bipolar knot that is not BPH-slice would detect an exotic $\#^n\CP$ or $\#^n\CPbar$. \Cref{thm:decomp}, like Rasmussen's $s$-invariant, is not known to obstruct $0$-positivity, because the construction underlying the proof of \Cref{thm:decomp} uses the standard smooth structure through the half-integer surgery filling of \Cref{thm:KMRS1}.

On the other hand, the conclusion of \Cref{thm:KMRS1} only involves the intersection lattice of a double cover of $Z^\circ$ branched along an H-slice disk for $K$, and such a branched cover exists whenever $K$ is H-slice in a smooth manifold $Z$ homeomorphic to $\#^n\CP$. Moreover, most slice obstructions from classical topology, gauge theory, and Heegaard Floer homology also obstruct sliceness in homotopy $4$-spheres ($s$ is a notable potential exception) \cite[\S7.3]{Man26}, so we might expect obstructions of this kind to behave similarly for $\#^n\CP$. Since \Cref{thm:decomp} is proved using Donaldson's theorem (which applies to any smooth closed definite $4$-manifold) and lattice arguments, we would not be surprised by an affirmative answer to the following question:
\begin{question}\label{q:exotic-half-integer}
Let $Z$ be a closed smooth oriented $4$-manifold homeomorphic to $\#^n\CP$, and let $K\subset S^3$ be a knot with $\sigma(K)=0$. If $K$ is H-slice in $Z$, must $\SigmaTwo(K)$ bound a smooth positive-definite $4$-manifold with $b_2=2n$ whose intersection lattice is of half-integer surgery type?
\end{question}
An affirmative answer would strengthen \Cref{thm:decomp} to a $0$-positivity (hence $0$-bipolarity) obstruction for knots with vanishing signature. In that case, the applications throughout this paper demonstrate that \Cref{thm:decomp} can detect knots that are not $0$-bipolar in settings where $\sigma, \sigma_{LT}, \tau, \epsilon, \delta_2, d(S^3_{\pm 1}(\cdot))$, and (in the case of \Cref{fam:Lh}) $r_0$ do not.

Finding an H-slice disk for one of our obstructed knots (or any knot obstructed from being BPH-slice by \Cref{thm:decomp}) in such a manifold $Z$ would instead detect an exotic smooth structure on $\#^n\CP$ and give a negative answer. See \Cref{subsec:intro:exotica} for one possible approach through the zero-surgery constructions of Manolescu--Piccirillo~\cite{MP23}.

In \Cref{subsec:exotica-candidates}, we record a selection of invariants that fail to obstruct our examples from being $0$-bipolar.
In \Cref{subsec:exotica-khovanov}, we compare our obstruction with invariants from Khovanov homology.
\subsection{Knots with vanishing obstructions to \texorpdfstring{$0$-bipolarity}{0-bipolarity}}\label{subsec:exotica-candidates}
The knots in \Cref{fam:Th3q,fam:Lh,fam:Jq}, together with $13a_{4237}$ and $10_{109}$, are quasi-alternating with vanishing signature and have an infinite projective slicing number. The following computations leave open the possibility that they could be $0$-bipolar. The proposition also covers the larger Turk's head family in \Cref{conj:Thpq}. In \Cref{rem:families-invariants}, we record the obstructions in \Cref{prop:0pos-obstructions}\eqref{obs:delta} and~\eqref{obs:HK} that we have only partially determined.
\begin{proposition}
    \label{prop:exotica-invariants}
Let $K$ be one of the Turk's head knots $\Th(p,q)$ with odd $p\geq3$ and even $q\geq4$, the knots $L_h$ with odd $h\geq3$, the knots $J_q$ with odd $q\geq 1$, $13a_{4237}$, or $10_{109}$. Then:
\begin{enumerate}
\item \label{item:exotica-QA}$\sigma(K)=s(K)=\tau(K)=\epsilon(K)=\delta_2(K)=d(S^3_{\pm1}(K))=0.$
\item\label{item:exotica-LT} $\sigma_{LT}(K)(\omega)=0$ if $\omega\in S^1\setminus\{1\},\Delta_K(\omega)\ne0$.
\item \label{item:exotica-r0}$r_0(L_h)=r_0(\overline{L_h})=\infty.$
\item \label{item:exotica-delta}For every prime $\ell$ and integer $r\geq1$,
$\delta_{\ell^r}(\Th(p,q))=\delta_{\ell^r}(10_{109})=0$.
\end{enumerate}
\end{proposition}
\begin{proof}
Each knot has vanishing signature. Since they are all quasi-alternating, item~\eqref{item:exotica-QA} follows from \Cref{lem:alt-qa-invariants}. For item~\eqref{item:exotica-LT}, the Turk's head knots with odd $p$ and the knot $10_{109}$ are negative amphichiral so $K$ is isotopic to $-\Kmirror$. Since $\sigma_{LT}(-\Kmirror)=-\sigma_{LT}(K)$, we have $\sigma_{LT}(K)=0$ (see \cite[Prop.~2.3(4),(5)]{Con21}). \Cref{prop:Lh-LT-signature} gives $\sigma_{LT}(L_h)=0$, \Cref{prop:Jq-signatures} gives $\sigma_{LT}(J_q)=0$, and \Cref{prop:13a4237-alexander-signature} gives $\sigma_{LT}(13a_{4237})=0$.

Item~\eqref{item:exotica-r0} is \Cref{prop:Lh-r0}. 
For item~\eqref{item:exotica-delta}, Jabuka proves that each $\delta_{\ell^r}$ is an integer-valued smooth concordance homomorphism \cite[Thm.~1.2]{Jab12}. The knots in question are negative amphichiral, so their concordance classes satisfy $[K]=-[K]$. Consequently, $2\delta_{\ell^r}(K)=0$.
\end{proof}
Item~\eqref{item:exotica-QA} of \Cref{prop:exotica-invariants} holds for every quasi-alternating knot with vanishing signature (\Cref{lem:alt-qa-invariants}). Those whose Levine--Tristram signature functions also vanish are therefore natural targets for applying \Cref{thm:decomp} (or other new obstructions to H-sliceness in $\#^n\CP$) and may be of interest in the search for detecting exotic definite manifolds.
\begin{remark}[Other obstructions from branched covers]
\label{rem:families-invariants}
There are some obstructions to $0$-bipolarity which we only partially determine. For quasi-alternating $K$, $\SigmaTwo(K)$ bounds smooth definite manifolds of both signs with $H_1(-;\Z)=0$ \cite[\S1, property~(iii)]{Iss18}, so the Hedden--Kirk obstruction in \Cref{prop:0pos-obstructions}\eqref{obs:HK} does not obstruct our examples. We have not checked the Hedden--Kirk obstructions for covers of degree $2^r>2$. Since these knots all have $\det(K)>1$, $\Sigma_{2^r}(K)$ is not an integral homology sphere \cite[p.~190, \S2]{Jiang81}. Thus \Cref{prop:0pos-obstructions}\eqref{obs:delta} does not apply to covers of degree $2^r$. We have not checked the $\delta_{\ell^r}$ obstructions for odd primes $\ell$, beyond the vanishings established in \Cref{prop:exotica-invariants}\eqref{item:exotica-delta}.
\end{remark}
\subsection{Comparison with Khovanov homology}\label{subsec:exotica-khovanov}
We compare our lattice obstruction with $s$ and related invariants from Khovanov homology by asking which knots they can obstruct from being BPH-slice. Negative amphichiral knots provide a natural setting for this comparison. Their concordance classes have order at most two, so every integer-valued concordance homomorphism, including Rasmussen's $s$-invariant over any field~\cite[Lem.~3.20]{DLS26}, vanishes.
They are also rationally slice~\cite{DPLS25}. 

Our Turk's head examples and $10_{109}$ show that our lattice obstruction can detect failure of BPH-sliceness among rationally slice knots with vanishing $s$-invariant. Since these knots are alternating with vanishing signature, the refinements of $s$ introduced by Dunfield--Lipshitz--Sch\"utz vanish as well~\cite[Prop.~6.30]{DLS26}. 

It is then natural to ask whether constructions from Khovanov homology can obstruct rationally slice knots from being BPH-slice, especially since it is not known whether $s$ vanishes for all rationally slice knots~\cite[\S6]{Lee24}. 
\begin{question}\label{q:khovanov}
Can constructions from Khovanov homology or skein lasagna modules be used to show that some rationally slice knot is not BPH-slice? 
\end{question}
\appendix
\section{Computation of \texorpdfstring{$r_0(L_h)$}{r0(Lh)}}\label{app:Lh-r0}
For the knots $L_h$ from \Cref{fam:Lh}, we prove the following result:
\begin{proposition}\label{prop:Lh-r0}
For every odd $h\geq3$, $r_0(L_h)=r_0(\overline{L_h})=\infty$.
\end{proposition}
We follow~\cite[\S6.4]{DISST25}, where the authors use the 2-bridge formulas of~\cite{DS24} to explicitly compute the I-graded S-complex of $\overline{10_{28}}$ (see~\cite[Fig.~11 and Prop.~6.21]{DISST25}) and they point out that similar computations can be carried out for other 2-bridge knots. They explain that, for 2-bridge knots, one can avoid perturbations and work with I-graded S-complexes over $\Z[T^{\pm1},U^{\pm1}]$. We need two components of this complex for our computation: the map from the reducible generator and the differential of one irreducible generator.

Following their notation, we write $(C_*,d)$ for the irreducible chain complex and $\delta_2$ for the map from the reducible generator $1$ to $C_*$. The maps $d$ and $\delta_2$ have integer degrees $-1$ and $-2$, respectively. The real grading $\deg_I$ records the Chern--Simons filtration, with $\deg_I(1)=0$. For a chain $\alpha$, $\deg_I(\alpha)$ is the largest degree of a nonzero term, where $\deg_I(0)=-\infty$. Multiplication by $U$ raises the integer and real gradings by $4$ and $1$ (respectively), and multiplying by $T$ preserves both gradings~\cite[\S2.1 and \S2.6]{DS24}.
\begin{lemma}\label{lem:2bridge-r0-criterion}
Let $K$ be a 2-bridge knot. If there is a chain $\alpha\in C_{-1}$ such that $d\alpha=\delta_2(1)$ and $\deg_I(\alpha)\leq0$,
then $r_0(K)=\infty$.
\end{lemma}
\begin{proof}
By \cite[Eq.~(118)]{DISST25} with $s=0$, we have
\[
r_0=-\min\bigl\{\inf\{\deg_I(d\alpha-\delta_2(1)): \alpha\in C_{-1},\ \deg_I(\alpha)\leq0\},0\bigr\}.
\]
The stated chain makes the infimum $-\infty$.
\end{proof}
\begin{definition}
For coprime $p,q$ with $p>1$ odd and a pair of positive odd integers $(a,b)$, we say that the rectangle $[-a,a]\times[-b,b]$ is \emph{empty} if it contains no nonzero point of
\[
\mathcal{L}_{p,q}=\{(x,y)\in\Z^2:x+qy\equiv0\pmod p\},
\]
except possibly at its corners (e.g., $(-a,b)\in \mathcal{L}_{p,q}\setminus \{0\}$ is allowed). 
\end{definition}
This is the condition $N_1=1$, $N_2=0$ in~\cite[Eqs.~(3.3)--(3.4)]{DS24}. Reflection across the horizontal axis shows that this condition is unchanged when $q$ is replaced with $-q$. Recall that our 2-bridge knot conventions agree with \cite{DISST25}, so the parameters used in~\cite{DS24} differ from ours by replacing $q$ with $-q$ (see \cite[Rmk.~6.20]{DISST25}).
\begin{proposition}\label{prop:2bridge-rectangles}
Let $h\geq3$ be odd, and set
\[
p=8h^2+4h+1,\qquad q=8h,\qquad
u=h(4h+1),\qquad w=h(4h-1).
\]
Among pairs of positive odd integers $(a,b)$ whose rectangle $[-a,a]\times[-b,b]$ is empty for $(p,q)$, each of the following congruences $\pmod p$ has exactly the solution indicated in the table:
\[
\begin{array}{c|c}
\text{congruence}\pmod p&(a,b)\\ \hline
a\equiv qb&(1,u)\\
a\equiv-qb&(4h+1,h)\\
a\equiv qb\pm q&(4h-1,h)\\
a\equiv-qb\pm(q+4)&(1,w)
\end{array}
\]%
\end{proposition}
We defer the proof until after we compute $r_0$.
\begin{proof}[Proof of \Cref{prop:Lh-r0}]
With $p$ and $q$ as in \Cref{prop:2bridge-rectangles}, $L_h=K(p,p-q)\simeq K(p,-q)$ and $\overline{L_h}=K(p,q)$. Write $q'=\mp q$ for their respective second parameters, so we can work with $K(p,q')$. 

The irreducible generators of $C_*$ are indexed by $1,\ldots,(p-1)/2$, and the reducible generator has index $0$.
For $K(p,q')$, \cite[Eqs.~(3.1)--(3.5)]{DS24} (converted to our conventions) show the component of $d$ from source index $i$ to target index $j$ is nonzero precisely when there is a pair of positive odd integers $(a,b)$ whose rectangle is empty for $(p,q')$ (equivalently, for $(p,q)$) and
\begin{equation}\label{eq:2bridge-r0-congruences}
a\equiv\varepsilon_1i+\varepsilon_2j,\qquad -q'b\equiv-\varepsilon_1i+\varepsilon_2j\pmod p,
\qquad \varepsilon_1,\varepsilon_2\in\{\pm1\}.
\end{equation}

Setting $i=0$ in \eqref{eq:2bridge-r0-congruences} determines the nonzero components of $\delta_2(1)$ (see after~\cite[Eq. (3.5)]{DS24}), so the two conditions become $a\equiv -q'b$ and $j\equiv\varepsilon_2a$. Since $1\leq j\leq(p-1)/2$, the first two rows of the table in \Cref{prop:2bridge-rectangles} determine the unique pair $(a_0,b_0)$ and target index $j$ for $L_h$ and $\overline{L_h}$.

 Now we want to find an irreducible generator whose differential has a single nonzero component at the same target index. Subtracting the two congruences in \eqref{eq:2bridge-r0-congruences} gives $a+q'b\equiv\varepsilon_1(2i)$, and choosing $i=4h$ for $L_h$ and $i=4h+2$ for $\overline{L_h}$ allows us to apply the last two rows of \Cref{prop:2bridge-rectangles}. To see that the targets agree, take $\varepsilon_1=1$ and $\varepsilon_2=-1$.
 
 In the table below, $(a_0,b_0)$ and $(a_i,b_i)$ are the pairs corresponding to the nonzero components of $\delta_2(1)$ and the differential of the generator at index $i$, respectively.
\[
\begin{array}{c|c|cc|c}
&\delta_2&\multicolumn{2}{c|}{d}&\\
\text{knot}
&(a_0,b_0)
&\text{source }i
&(a_i,b_i)
&\text{common target }j\\ \hline
L_h
&(1,u)&4h&(4h-1,h)&1\\
\overline{L_h}
&(4h+1,h)&4h+2&(1,w)&4h+1
\end{array}
\]

Thus, for each knot, $\delta_2(1)$ and the differential of the chosen generator each have exactly one nonzero component with the same target. Each coefficient is $\pm U^k(T^2-T^{-2})$ for some $k\in\Z$, and the component associated to $(a,b)$ lowers the real grading by $ab/p$ (see~\cite[Eqs.~(3.2),~(3.8) and the subsequent discussion]{DS24}). Multiplying the source and target generators by signs
and powers of $U$ gives us $\alpha$ and $\beta$ with
\[
d\alpha=\delta_2(1)=(T^2-T^{-2})\beta.
\]
This uses only units of $\Z[T^{\pm1},U^{\pm1}]$ (i.e., we never divide by $T^2-T^{-2}$). The map degrees give $\alpha\in C_{-1}$ and $\beta\in C_{-2}$. For both knots, $a_0b_0=u$ and $a_ib_i=w$, so
\[
\deg_I(\beta)=-\frac{u}{p}\quad\text{and}\quad
\deg_I(\alpha)=\frac{w-u}{p}=-\frac{2h}{p}<0.
\]
The result follows from \Cref{lem:2bridge-r0-criterion}.
\end{proof}
\begin{remark}
    For $h=3$, \Cref{prop:Lh-r0} is consistent with~\cite[Table~1]{DISST25}: $L_3=K(85,61)$ and its mirror have vanishing signature but do not appear there, so $\Gamma_{L_3}(0)=\Gamma_{\overline{L_3}}(0)=0$. 
\end{remark}
\begin{proof}[Proof of \Cref{prop:2bridge-rectangles}]
Since $qu=(4h-1)p+1$, we have $q^{-1}\equiv u\pmod p$. The points
\[
(-q,1),\qquad (-1,u),\qquad (2,2h+1)
\]
belong to $\mathcal L_{p,q}$. Emptiness and the oddness of $a,b$ therefore imply
\begin{equation}\label{eq:rectangle-bounds}
1\leq a<q,\qquad b\leq u,\qquad b\geq2h+1\Longrightarrow a=1.
\end{equation}

First suppose $b<2h+1$. For the four congruences, respectively, consider
\[
qb-a,\qquad qb+a,\qquad q(b\pm1)-a,\qquad qb+a\pm(q+4).
\]
Each corresponding expression is an odd multiple of $p$. By \eqref{eq:rectangle-bounds}, each is in the range $(-p,3p)$, and therefore equals $p$. Since $p=qh+4h+1$ and $0<a<q$, the first and fourth equations force $b$ to be even. The second one gives $(a,b)=(4h+1,h)$, and the third one gives $(4h-1,h)$ or $(4h-1,h+2)$. But $(4h-1,h+2)$ is excluded because $(-(4h-1),h+1)\in\mathcal L_{p,q}$, which lies on an edge of its rectangle away from the corners. 

If $2h+1\leq b\leq u$, then $a=1$. Multiplying the congruences by $q^{-1}\equiv u$ gives, in their order,
\[
b\equiv u;\qquad p-u;\qquad u\pm1;\qquad
w\ \text{or}\ u+6h+2\pmod p.
\]
Since $u<p/2$, the only odd representatives in the required range are $b=u$ in the first row and $b=w$ in the fourth.

Finally, all four listed rectangles are empty. For $|x|\leq1$, the nonzero points of $\mathcal L_{p,q}$ with smallest $|y|$ are $(\pm1,\mp u)$, proving the assertion for $(1,u)$ and $(1,w)$. For the rectangles corresponding to $(4h+1,h)$ and $(4h-1,h)$, we have $|x+qy|\leq p$ and $|x+qy|\leq p-2$, respectively (where equality to $p$ is only possible at corners of the rectangles). Away from the corners, a lattice point must therefore satisfy $x+qy=0$. Since $|x|<q$, that must be the origin.

\end{proof}
\section{More on Turk's head knots}\label{app:Gpq}
In this appendix we compute the Goeritz matrix $G_{p,q}$ and its determinant, proving the determinant formula for $\Th(p,q)$ proposed in \cite{DMMS10}. The purpose of these computations is to make \Cref{conj:Thpq} more approachable, and to explain why \Cref{prop:finiteob} applied to the Goeritz matrix of the standard alternating diagram of $\Th(p,q)$ for odd $q$ does not obstruct BPH-sliceness. We also give a short proof (already known to Golla, Lee, and \c{S}avk) that $\Th(p,2)$ is BPH-slice when $p$ is odd.
\begin{lemma}\label{lem:Gpq-goeritz}
Let $p=2m+1\geq3$ and $q\geq3$. The standard alternating diagram of $\Th(p,q)$ has a Goeritz matrix
\[
G_{p,q}=M_m\otimes I_q-I_m\otimes \Adj(\mathcal{C}_{q})=
\begin{pmatrix}
    B_q &      -I_q &       &\\
    -I_q&  \ddots   & \ddots&\\
        &   \ddots  &B_q    &-I_q\\
        &           &-I_q   &G_q
\end{pmatrix},
\qquad B_q=4I_q-\Adj(\mathcal{C}_{q}),
\]
of rank $mq$, where $M_m$ is the $m\times m$ tridiagonal matrix with diagonal entries $4,\ldots,4,3$ and off-diagonal entries $-1$.
\end{lemma}
\begin{proof}
Use the same checkerboard construction as in \Cref{lem:Th3qgoeritz}. 
\begin{figure}[H]
  \centering
  \resizebox{!}{7.65cm}{\input{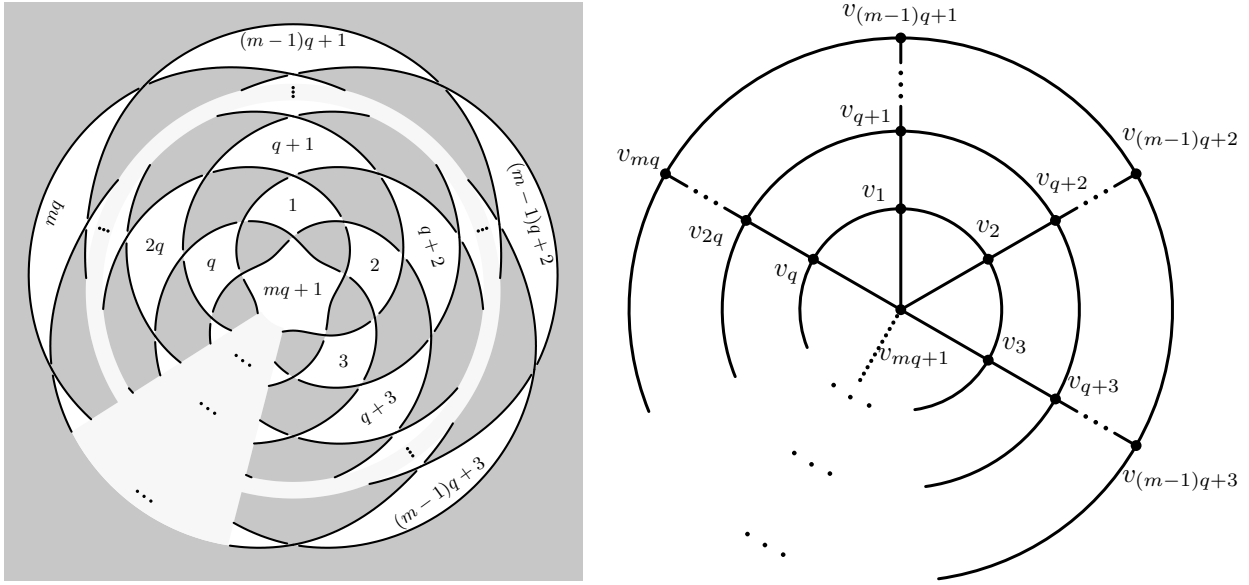}}\hfill
  \resizebox{!}{7.65cm}{\input{figures/Th-podd-tait-target}}
  \caption{A checkerboard coloring of $\Th(p,q)$ and its Tait graph $\Gamma_W$}
  \label{fig:Thpq-checkerboard}
\end{figure}
Think of the colored diagram as a target with $p-1$ concentric rings of alternating colors surrounding the white bullseye, where each ring is made up of $q$ regions. The white regions are the bullseye and the $q$ regions in each of the $m$ white rings. Then the Tait graph $\Gamma_W$ consists of a hub vertex $v_{mq+1}$ (the bullseye) and $m$ concentric $q$-cycles (ordered innermost to outermost). We label the vertices so that the $i$th cycle is the $q$-cycle on $(v_{(i-1)q+1}, v_{(i-1)q+2},\dots,v_{iq})$ and for $1\leq i\leq m-1$ and $1\leq j\leq q$, vertex $v_{(i-1)q+j}$ in the $i$th cycle is adjacent to $v_{iq+j}$ in the $(i+1)$st cycle. Each of the vertices in the first (innermost) cycle is also adjacent to the hub vertex. Then all vertices in $\Gamma_W$ have degree $4$ except for the $q$ vertices in the $m$th (outermost) cycle (which have degree $3$), and the hub vertex (which has degree $q$). Deleting the hub gives $G_{p,q}$, where the diagonal of $M_m$ lists the degrees of the remaining vertices. 
\end{proof}
The following proposition explains why our lattice obstruction, applied to $G_{p,q}$, cannot obstruct BPH-sliceness when $q$ is odd.
\begin{proposition}
    \label{prop:Gpq-factorization}
Let $p=2m+1\geq3$ and $q\geq3$ be odd. Then $G_{p,q}=DD^T$ for some $D\in M_{mq}(\Z)$.
\end{proposition}
\begin{proof}
Write $q=2s+1$. Let $\Pi$ be the cyclic permutation matrix of size $q$, so that $\Adj(\mathcal C_q)=\Pi+\Pi^{-1}$, and set $S=\Pi^s-\Pi^{-s}$. Since $\Pi^T=\Pi^{-1}$ and $\Pi^q=I_q$,
\[
S^T=-S,\qquad S^2=\Adj(\mathcal C_q)-2I_q.
\]

Let $U$ be the $m\times m$ upper-bidiagonal matrix with diagonal entries $1$ and superdiagonal entries $-1$, and let $R$ be the permutation matrix reversing the standard basis. Since $RUR=U^T$, the matrix $A=UR$ is symmetric and
\[
A^2=AA^T=UU^T=M_m-2I_m.
\]
Consequently, the integral matrix $D=A\otimes I_q+I_m\otimes S$ satisfies
\[
DD^T=(M_m-2I_m)\otimes I_q  -I_m\otimes\bigl(\Adj(\mathcal C_q)-2I_q\bigr)=G_{p,q}.
\]
\end{proof}
Now we compute $\det(G_{p,q})$, which we expect to be a necessary step in proving \Cref{conj:Thpq}. 
\begin{remark}\label{rem:Gpq-det-other-results}
This result is not necessarily novel, but the way we compute it in terms of the Goeritz matrix is relevant to \Cref{conj:Thpq}. The formula was known to the authors of \cite{DMMS10}, and proving it was no longer needed for the Turk's head case of the Kauffman--Harary conjecture after Mattman--Solis~\cite{MS09} proved the conjecture in general. Independently,~\cite{Sta26} gives a different version of this computation, generated by Claude, that was posted to the arXiv within days of our initial submission of this paper. For even $q$, the factorization of $\det(\Th(p,q))$ as $\Pell_p$ times a square also follows from Saurabh's Alexander-polynomial factorization \cite[Thm.~7.4 and Rmk.~7.5]{Sau26}. The argument below identifies the square factor explicitly and verifies the determinant formula proposed by Dowdall--Mattman--Meek--Solis~\cite{DMMS10}.
\end{remark}
Let $\Pell_n$ denote the Pell numbers, defined by
\[
\Pell_0=0,\qquad \Pell_1=1,\qquad \Pell_{n+2}=2\Pell_{n+1}+\Pell_n.
\]
For the even and odd cases of $q$, respectively, define polynomials $V_j,O_j\in\Z[x]$ by
\[
\begin{aligned}
V_0&=0,\qquad V_1=1,
& V_{j+1}&=(x+2)V_j-V_{j-1},\\
O_0&=1,\qquad O_1=x+3,
& O_{j+1}&=(x+2)O_j-O_{j-1},
\end{aligned}
\qquad j\geq1.
\]
Set $H_m=M_m-2I_m$.
\begin{theorem}\label{thm:Gpq-det}
Let $p=2m+1\geq3$, $q\geq3$, and $\lambda_k=2\cos(2\pi k/q)$. Then
\[
\det(G_{p,q})
=\prod_{k=0}^{q-1}\det(M_m-\lambda_kI_m)
=
\begin{cases}
\Pell_p\,\det\bigl(V_{q/2}(H_m)\bigr)^2,
  &q\text{ even},\\[2pt]
\det\bigl(O_{(q-1)/2}(H_m)\bigr)^2,
  &q\text{ odd}.
\end{cases}
\]
In particular, $\det(\Th(p,q))$ is composite, and it is a perfect square if and only if $q$ is odd or $p=7$.
\end{theorem}
\begin{proof}
The adjacency matrix of $\mathcal{C}_q$ has eigenvalues $\lambda_0,\ldots,\lambda_{q-1}$. Diagonalizing this $\Adj(\mathcal{C}_q)$ in
\[
G_{p,q}=M_m\otimes I_q-I_m\otimes\Adj(\mathcal{C}_q)
\]
gives the first equality.
The factor indexed by $k=0$ is $\det(H_m)=1$, since $H_m=UU^T$ and $\det(U)=1$ (see the proof of \Cref{prop:Gpq-factorization}).
When $q$ is even, the factor indexed by $k=q/2$ is $\det(M_m+2I_m)$. Its trailing principal determinants satisfy
\[
t_0=1,\qquad t_1=5,\qquad t_j=6t_{j-1}-t_{j-2}\quad(j>1).
\]
The odd-indexed Pell numbers $\Pell_{2j+1}$ have the same initial values and recurrence, so this factor is $\Pell_{2m+1}=\Pell_p$.

Since $\lambda_k=\lambda_{q-k}$, the remaining determinant factors indexed by $k$ and $q-k$ are equal, so we have 
\[
\det(G_{p,q})
=
\begin{cases}
\Pell_p\left(\prod_{k=1}^{(q/2)-1}\det(M_m-\lambda_kI_m)\right)^2
  &q\text{ even},\\[2pt]
\left(\prod_{k=1}^{(q-1)/2}\det(M_m-\lambda_kI_m)\right)^2
  &q\text{ odd}.
\end{cases}
\]

Now we identify the remaining product expression. The polynomial recurrences give
\[
V_j(2\cos\theta-2)=\frac{\sin(j\theta)}{\sin\theta},
\qquad
O_j(2\cos\theta-2)
=\frac{\sin((j+\tfrac12)\theta)}{\sin(\theta/2)}
\quad(0<\theta<\pi).
\]
Both sides have the same initial values and satisfy the same recurrence. Since $V_j$ and $O_j$ are monic of degrees $j-1$ and $j$ (for $j\geq 1$), respectively, these identities give
\[
\begin{aligned}
V_{q/2}(x)
  &=\prod_{k=1}^{q/2-1}(x+2-\lambda_k)
  &&(q\text{ even}),\\
O_{(q-1)/2}(x)
  &=\prod_{k=1}^{(q-1)/2}(x+2-\lambda_k)
  &&(q\text{ odd}).
\end{aligned}
\]
Since $H_m+2I_m=M_m$, substituting $H_m$ and taking determinants gives, in the even case,
\[
\det\bigl(V_{q/2}(H_m)\bigr)
=\prod_{k=1}^{q/2-1}\det(M_m-\lambda_kI_m).
\]
For the odd case, use $O_{(q-1)/2}$. Substituting both into the preceding expression for $\det(G_{p,q})$ proves the formula.

Since $H_m$ is positive definite with determinant one, for $1\leq k<q$ we have
\[
\det(M_m-\lambda_kI_m)
=\det\bigl(H_m+(2-\lambda_k)I_m\bigr)>1.
\]
In particular, $\det\bigl(V_{q/2}(H_m)\bigr)$ for even $q$ and $\det\bigl(O_{(q-1)/2}(H_m)\bigr)$ for odd $q$ are both greater than $1$, so $\det(\Th(p,q))$ is composite. Ljunggren's theorem states that the only square Pell numbers are $\Pell_0=0$, $\Pell_1=1$, and $\Pell_7=169$ (see \cite[p.~19]{Coh96}), so $\det(\Th(p,q))$ is a perfect square if and only if $q$ is odd or $p=7$.
\end{proof}
For even $q$ and $p=3$, our computations recover $\det(G_q)=5\Fib_q^2$. 
\begin{remark}\label{rem:Gpq-det}Dowdall--Mattman--Meek--Solis proposed a formula for $\det(\Th(p,q))$ when $p$ is odd, which they verified numerically for $p\leq23$, $q\leq29$ \cite[\S6]{DMMS10}. They define
\[
    d_1=1,\qquad d_2=1-x,\qquad d_{n+1}=(1-x)d_n-xd_{n-1}.
\]
For $p=2m+1$, the roots of $d_p$ occur in positive reciprocal pairs $\alpha_i^{\pm1}$, $1\leq i\leq m$. Choosing $\alpha_i>1$, their proposed determinant is
\[
    \tilde{G}_{p,q} =\Bigl[\prod_{i=1}^m \bigl(\alpha_i^{q/2}-\alpha_i^{-q/2}\bigr)\Bigr]^2.
\]
\Cref{thm:Gpq-det} proves this formula. Taking every second term of their recurrence gives
\[
    d_{n+2}=(x^2-4x+1)d_n-x^2d_{n-2}.
\]
Set $f_0(y)=1$ and $f_j(y)=\det(M_j-yI_j)$ for $j\geq1$. Then $f_1(y)=3-y$, and expansion along the first row gives
\[
    f_j(y)=(4-y)f_{j-1}(y)-f_{j-2}(y)\qquad(j\geq2).
\]
Thus $(-1)^jx^jf_j(x+x^{-1})$ satisfies the recurrence for $d_{2j+1}(x)$ and has initial terms $d_1=1$ and $d_3=1-3x+x^2$. Hence
\[
    d_{2m+1}(x)=(-1)^mx^mf_m(x+x^{-1}),
\]
so the numbers $\alpha_i+\alpha_i^{-1}$ are the eigenvalues of $M_m$. Since
\[
\bigl(\alpha^{q/2}-\alpha^{-q/2}\bigr)^2 =\alpha^{-q}(\alpha^q-1)^2=\prod_{k=0}^{q-1}\bigl(\alpha+\alpha^{-1}-\lambda_k\bigr),
\]
we obtain
\[
\tilde{G}_{p,q}=\prod_{k=0}^{q-1}\det(M_m-\lambda_kI_m)=\det(G_{p,q}).
\]
\end{remark}

\begin{lemma}\label{lem:Thp2}
    When $p$ is odd, the Turk's head knot $\Th(p,2)$ is BPH-slice. 
\end{lemma}
\begin{figure}[H]
    \centering
    \resizebox{!}{4cm}{\input{figures/Th-p2-crossings-flat-labelled}}
    \caption{$\Th(p,2)$, shown for odd $p\geq 7$}
    \label{fig:Thp2}
\end{figure}
\begin{proof}
    When $p=1$, $\Th(1,2)$ is the unknot. For $p\geq 3$, consider the standard alternating diagram of $\Th(p,2)$, shown for $p\geq 7$ in \Cref{fig:Thp2}. Changing the circled positive crossing to a negative crossing yields a diagram of $\Th(p-2,2)$. It follows that a sequence of $n=\frac{p-1}{2}$ positive-to-negative crossing changes gives the unknot, so $\Th(p,2)$ is H-slice in $\#^n\CPbar$. Similarly, by changing the circled negative crossing to a positive crossing, we also get a diagram of $\Th(p-2,2)$. It follows that $\Th(p,2)$ is H-slice in $\#^n\CP$. Hence, $\Th(p,2)$ is BPH-slice.

    Equivalently, we can describe our two crossing changes using the braid notation for Turk's heads knots. Define three braids as follows: \[\beta=\sigma_1\sigma_2^{-1}\dots \sigma_{p-2}\sigma_{p-1}^{-1}; \quad \beta^{+\mapsto -} =\sigma_1^{-1}\sigma_2^{-1}\dots \sigma_{p-2}\sigma_{p-1}^{-1};\quad \beta^{-\mapsto+}=\sigma_1\sigma_2^{-1}\dots \sigma_{p-2}\sigma_{p-1}.\]
    Then the closure of $\beta^{+\mapsto-}\beta$ is isotopic to the closure of $(\sigma_3\sigma_4^{-1}\dots\sigma_{p-2}\sigma_{p-1}^{-1})^2$. Similarly, the closure of $\beta^{-\mapsto+}\beta$ is isotopic to that of $(\sigma_1\sigma_2^{-1}\dots \sigma_{p-4}\sigma_{p-3}^{-1})^2$. So we have $\Th(p,2)=\widehat{\beta^2}$, and $\Th(p-2,2)\simeq \widehat{\beta^{+\mapsto-}\beta}\simeq\widehat{\beta^{-\mapsto+}\beta}$.
\end{proof}
\section{The role of condition \texorpdfstring{\ref{cond:finiteob:gram}}{(D2)}}\label{app:D2}
Recall that, for a positive-definite matrix $Q$ with odd determinant, condition~\ref{cond:finiteob:gram} asks that a candidate $T\in\mathcal T(Q)$ be the Gram matrix of $r$ integer vectors in $\Z^n$, that is, $T=CC^T$ for some $C\in M_{r\times n}(\Z)$. Condition~\ref{cond:finiteob:gram} is the only part of \Cref{prop:finiteob} involving $n$, the number of $\CP$ summands in \Cref{thm:decomp}. Condition~\ref{cond:finiteob:gram} is not needed in any of our applications. We prove that this condition provides no additional obstruction to BPH-sliceness for the applications in this paper and ask whether it can be essential in other settings.

To determine when condition~\ref{cond:finiteob:gram} is satisfied automatically for sufficiently large $n$, we use the following classical results of Mordell and Ko (appearing in several forms across overlapping manuscripts published between 1932 and 1938, summarized
in~\cite[\S1]{KimOh05}).
\begin{theorem}[Mordell and Ko]\label{thm:mordell-ko}
Every positive-definite integral symmetric matrix $T$ of rank $m\leq5$ is the Gram matrix of integer vectors in $\Z^{m+3}$. In contrast, the rank-$6$ Cartan matrix $E_6$ is not the Gram matrix of integer vectors in $\Z^N$ for any $N$.
\end{theorem}
The following proposition allows us to apply this theorem even when the rank of $T$ is large (as it often is in our applications), and its corollary explains why we never use condition~\ref{cond:finiteob:gram}.
\begin{proposition}\label{prop:D2-auto}
Let $T\in M_r(\Z)$ be symmetric and positive semidefinite of rank $m$, with fewer than six diagonal entries greater than or equal to $2$. Then $T$ is the Gram matrix of integer vectors in $\Z^{m+3}$.
\end{proposition}
\begin{proof}
Consider the positive-definite integral lattice $L=\Z^r/\ker T$ with the form induced by $T$. It suffices to embed $L$ in $(\Z^{m+3},I)$.

The images $g_1,\ldots,g_r$ of the standard basis vectors generate $L$ and satisfy $g_i^2=T_{ii}$. Whenever a generator $g$ has square one, we have the orthogonal splitting $ L=\Z g\oplus g^\perp, $ where $g^\perp$ is generated by the vectors $g_i-(g_i\cdot g)g$. Their squares are $g_i^2-(g_i\cdot g)^2\leq g_i^2$.

Repeat this process of splitting off square-one generators, and discard zero generators. This gives
\[
L\cong(\Z^\ell,I)\oplus L',
\]
where $\ell$ is the number of square-one summands we split off. Since projection never increases squares, each remaining generator comes from an original $g_i$ with $T_{ii}\geq2$. There are at most five such generators, so $\rank L'\leq5$.

If $L'=0$, the conclusion is immediate. Otherwise, \Cref{thm:mordell-ko} embeds $L'$ in $(\Z^{\rank L'+3},I)$, giving an embedding of $L$ in $(\Z^{\ell+\rank L'+3},I)=(\Z^{m+3},I)$.
\end{proof}
\begin{corollary}\label{cor:D2-auto}
Let $Q$ be as in \Cref{prop:finiteob}. If fewer than six diagonal entries of $Q$ are at least $5$, then every $T\in\mathcal T(Q)$ satisfies condition~\ref{cond:finiteob:gram} for all $n\geq\operatorname{rank}(T)+3$. 
\end{corollary}
\begin{proof}
For $T\in\mathcal T(Q)$, $T_{ii}\geq2$ only if $Q_{ii}\geq5$. \Cref{prop:D2-auto} therefore gives a Gram representation in $\Z^{\operatorname{rank}(T)+3}$, and appending zero coordinates gives one in every larger dimension.
\end{proof}
As discussed in \Cref{subsec:sketch}, \Cref{cor:D2-auto} applies to the fillings in all of our applications, so we never use condition~\ref{cond:finiteob:gram}.

At the level of lattices, a candidate passing condition~\ref{cond:finiteob:embedding} does not necessarily pass condition~\ref{cond:finiteob:gram} for any $n$. The following example demonstrates this behavior, although the lattice does still admit a decomposition using a different candidate. 
\begin{example}\label{ex:2E6}
Let $Q=2E_6+I_6$, where the vertices $v_1,\ldots,v_5$ of the $E_6$ Dynkin diagram form a chain and $v_6$ is adjacent only to $v_3$. This matrix is positive definite with $\det(Q)=5061$. The matrix $E_6$ belongs to $\mathcal T(Q)$, and $Q-2E_6=I_6$ satisfies condition~\ref{cond:finiteob:embedding}. But $E_6$ is not the Gram matrix of integer vectors in $\Z^n$ for any $n$, so it fails condition~\ref{cond:finiteob:gram}. $E_6$ has six diagonal entries equal to $2$, so it also demonstrates that the bound in \Cref{prop:D2-auto} is sharp.

Nevertheless, $Q$ admits a decomposition with $n=3$:
\[
Q=2CC^T+DD^T,\qquad
C=\begin{pmatrix}
0&0&0\\0&0&0\\1&0&0\\-1&0&0\\0&1&0\\0&0&1
\end{pmatrix},\qquad
D=\begin{pmatrix}
1&1&1&1&1&0\\
-1&-1&-1&0&1&1\\
0&0&1&-1&0&-1\\
-1&0&0&1&0&-1\\
0&1&0&-1&0&1\\
0&0&0&1&-1&1
\end{pmatrix}.
\]
\end{example}
It is not clear whether, even at the level of lattices, condition~\ref{cond:finiteob:gram} can exclude \emph{every} candidate passing condition~\ref{cond:finiteob:embedding} (when at least one such candidate exists). The issue we encountered in trying to construct such an example is that when $Q$ has diagonal entries large enough to allow for a candidate failing condition~\ref{cond:finiteob:gram}, there are also many candidates satisfying it. In all of our attempts, at least one of these candidates also passed condition~\ref{cond:finiteob:embedding}.

Even if such a $Q$ does exist, in the context of obstructing BPH-sliceness, it remains to show that $Q$ is realized as the intersection form of a filling of $-\SigmaTwo(K)$ for some knot $K$ with $\sigma(K)=0$. Recall that $\det(Q)\equiv 1\pmod 4$ is one necessary condition.
\begin{question}\label{q:D2}
Is there a positive-definite integral matrix $Q$ of odd determinant and rank $r$ that admits no decomposition $Q=2CC^T+DD^T$, with $D\in M_r(\Z)$ and $C\in M_{r\times n}(\Z)$, for any $n\geq0$, but for which some $T\in\mathcal T(Q)$ satisfies condition~\ref{cond:finiteob:embedding}? If so, can such a $Q$ be the intersection form of a positive-definite filling of $-\SigmaTwo(K)$ for a knot $K$ with $\sigma(K)=0$?
\end{question}
\begin{remark}
By \Cref{cor:D2-auto}, such a $Q$ must have at least six diagonal entries at least $5$, because every candidate $T$ passing condition~\ref{cond:finiteob:embedding} must have at least six diagonal entries at least $2$ (otherwise \Cref{prop:D2-auto} would give a decomposition for some $n$). If $Q$ is also the Goeritz matrix of a reduced alternating diagram of $K$, then $K$ must have at least $16$ crossings because the diagonal entries of $Q$ are the degrees of all but one vertex of the Tait graph, and each vertex has degree at least 2.
\end{remark}
\section{Signature and Seifert matrix computations}\label{app:comps}
We compute the Levine-Tristram signature functions for the knots in \Cref{fam:Lh} and \Cref{fam:Jq}, and we collect the Seifert matrix calculations used in \Cref{sec:topvsm}.
\subsection{Levine--Tristram signature functions}%
\begin{proposition}\label{prop:Lh-LT-signature}
For every odd $h\geq3$, the symmetrically normalized Alexander polynomial of $L_h$ is
\[
\Delta_{L_h}(t)
=\frac{h(h+1)}2(t+t^{-1}-2)^2
 +h(t+t^{-1}-2)+1.
\]
It has no roots on the unit circle, and the Levine--Tristram signature function of $L_h$ vanishes identically.
\end{proposition}
\begin{proof}
The mirror of $L_h$ has the even continued fraction
\[
\frac{8h^2+4h+1}{8h}=[h+1,2,-2h,-2].
\]
The standard Seifert surface for this even continued fraction gives, up to mirroring, the following Seifert matrix
\[
F_h=
\begin{pmatrix}
(h+1)/2&1&0&0\\
0&1&1&0\\
0&0&-h&1\\
0&0&0&-1
\end{pmatrix}.
\]
Since mirroring preserves the Alexander polynomial, we can use this matrix to compute $\Delta_{L_h}$. Writing $a=(h+1)/2$, expansion of the tridiagonal determinant gives
\(
\det(F_h-tF_h^T)
=ah(1-t)^4+ht(1-t)^2+t^2.
\)
Dividing by $t^2$ gives the stated normalization and formula.

For $t\in S^1$, set $x=2-t-t^{-1}\in[0,4]$. Then
\[
\Delta_{L_h}(t)=\frac{h(h+1)}2x^2-hx+1.
\]
This quadratic has positive leading coefficient and negative discriminant, so $\Delta_{L_h}$ has no roots on the unit circle. The Levine--Tristram signature is locally constant away from Alexander roots and vanishes near $1$, so the signature vanishes identically (see, e.g.,~\cite[Rmks.~2.1 and 2.2]{Con21}).
\end{proof}
\begin{proposition}\label{prop:Jq-signatures}
For odd $q\geq1$,
\[
\nabla_{J_q}(z)=1+z^2+z^4-z^6-\tfrac{q-1}{2}\,z^2(1+z^2+z^4).
\]
In particular, $\Delta_{J_q}$ has no roots on the unit circle, so $\sigma_{LT}(\omega)=0$ for every $\omega\in S^1\setminus\{1\}$, and $\sigma(J_q)=0$.
\end{proposition}
\begin{proof} First we describe $\nabla_{J_q}$ relative to a 2-component link.

Based on the diagram for $J_q$ given in \Cref{fig:Jq}, consider $J_{q+2}$ and fix a crossing $c$ in the $(q+4)$-twist region of the first tangle $[2,q+4]$ (on the top two strands). The two oriented strands incident to $c$ are antiparallel, so changing $c$ produces two adjacent crossings of opposite sign, which cancel by a Reidemeister II move and leave us with $J_q$. 

The oriented smoothing at $c$ leaves crossings in the twist region which can all be removed by Reidemeister I moves. The result is the two-component link \[L=M(1;2,[3,3],[2,2])\simeq M(0;2,[3,3],-3)=L10n41,\] which does not depend on $q$. All crossings of the region have the same sign, so the skein relation gives
\[
\nabla_{J_{q+2}}-\nabla_{J_q}=\pm z\,\nabla_L ,
\]
with a sign independent of $q$. \cite{LinkInfo} records $\nabla_{L10n41}=\pm z(1+z^2+z^4)$ for both orientations, so $\nabla_{J_q}=\nabla_{J_1}+\varepsilon\,\tfrac{q-1}{2}\,z^2(1+z^2+z^4)$ for some $\varepsilon=\pm1$.

To determine the sign of $\varepsilon$, we consider $J_1=12n_{300}$, which has $\nabla_{12n_{300}}=1+z^2+z^4-z^6$~\cite{KnotInfo}. Evaluating at $z=2i$ gives $\det(J_q)=|\nabla_{J_q}(2i)|=|77-26\,\varepsilon\,(q-1)|$. Since $\det(J_3)=129$, we get $\varepsilon=-1$, and the formula follows.

Finally, $\Delta_{J_q}(t)=\nabla_{J_q}(t^{1/2}-t^{-1/2})$, and for $\omega\in S^1$ we have $z^2=\omega+\omega^{-1}-2=-u$ with $u=|1-\omega|^2\geq0$. Substituting gives
\[
\Delta_{J_q}(\omega)=u^3+\Bigl(1+\tfrac{q-1}{2}u\Bigr)(u^2-u+1)>0,
\]
so $\Delta_{J_q}$ has no roots on the unit circle, and the Levine--Tristram signature (which vanishes near $\omega=1$ and is locally constant away from the roots of $\Delta_{J_q}$) vanishes identically (see~\cite[Rmks.~2.1 and 2.2]{Con21}).
\end{proof}
\subsection{Seifert matrices for topological projective slicing numbers}\label{subsec:seifert}
\begin{proposition}\label{prop:13a4237-seifert}
$\uCPbartop(13a_{4237})\leq1$.
\end{proposition}
\begin{proof}
Applying Seifert's algorithm to the diagram of $13a_{4237}$ shown in \Cref{fig:cc-a} (obtained using the PD code on KnotInfo) gives a genus-three Seifert surface $F$. For a choice of basis $\gamma_1,\ldots,\gamma_6$ of $H_1(F;\Z)$, using the convention $A_{ij}=\operatorname{lk}(\gamma_i^+,\gamma_j)$, we obtain the following Seifert matrix:
\begin{equation}\label{eq:13a4237-seifert}
A=\begin{pmatrix}
1&1&1&0&1&1\\
0&1&1&0&1&1\\
0&0&1&1&1&0\\
0&0&0&-2&-1&0\\
0&0&0&-1&-2&0\\
0&0&0&1&0&-2
\end{pmatrix}.
\end{equation}
Using the convention with the positive push-off in the second argument instead gives $A^T$ and does not change the result.

Now set
\[
c_1=\begin{pmatrix}-2&1&3&-2&2&5\end{pmatrix}^T, \quad \text{and}
\quad B_-=-A^T+c_1c_1^T.
\]
Then $\det(tB_--B_-^T)=t^3$. Since $-A^T=B_--c_1c_1^T$ is a Seifert matrix for the mirror, \Cref{thm:KMRS-topological} applies with one vector, so $\uCPbartop(13a_{4237})\leq1$.
\end{proof}
%
%
\begin{proposition}\label{prop:10_109-seifert} 
$\uCPtop(10_{109})\leq 2$.
\end{proposition}
\begin{proof}
\cite{KnotInfo} records the following Seifert matrix for $10_{109}$:
\[
A=\begin{pmatrix}
-1&-1& 0& 0&-1& 0& 0& 0\\
 0&-1& 0& 0& 0& 0& 0& 0\\
-1&-1&-1& 0&-1& 0& 0& 0\\
-1& 0&-1& 1& 0& 0& 1& 0\\
 0&-1& 0& 0&-1& 0& 0& 0\\
-1&-1&-1& 1&-1& 1& 1& 1\\
 0& 0& 0& 0& 0& 0& 1& 0\\
-1& 0&-1& 1& 0& 0& 1& 1
\end{pmatrix}
\]
We apply \Cref{thm:KMRS-topological} with $n=2$ by setting
\[
c_1=(0,0,0,0,0,0,0,1)^T,\quad c_2=(1,0,1,-1,-1,-1,-1,-1)^T, \quad \text{and}\quad B=A+c_1c_1^T+c_2c_2^T.
\]
Then we compute $\det(tB-B^T)=t^4$, which gives us $\uCPtop(10_{109})\leq2$.
\end{proof}
\begin{proposition}\label{prop:1028-seifert}
$\uCPtop(10_{28})\leq 1$ and $\uCPbartop(10_{28})\leq 1$.
\end{proposition}
\begin{proof}
\cite{KnotInfo} records the following Seifert matrix for $10_{28}$:
\[
A=\begin{pmatrix}
-1&0&0&0\\
0&2&0&0\\
-1&0&-1&0\\
1&1&1&2
\end{pmatrix},
\qquad \det(A-tA^T)=4-13t+19t^2-13t^3+4t^4 .
\]
We apply \Cref{thm:KMRS-topological} with $n=1$ by setting $c=(-2,1,0,-2)^T$ and $B=A+cc^T$. Then we have $\det(tB-B^T)=t^2$, which gives $\uCPtop(10_{28})\leq 1$.

Similarly, $-A^T$ is a Seifert matrix for $\overline{10_{28}}$. Setting $c'=(-2,1,-2,-2)^T$ and $B'=-A^T+c'c'^T$, we have $\det(tB'-B'^T)=t^2$, which gives $\uCPbartop(10_{28})\leq 1$.
\end{proof}
\begin{proposition}\label{prop:12a792-seifert}
$\uCPtop(12a_{792})\leq2,\qquad \uCPbartop(12a_{792})\leq2$.
\end{proposition}
\begin{proof}
Take the Seifert matrix $F_h$ from \Cref{prop:Lh-LT-signature} with $h=3$. Then 
set
\[
C_+=\begin{pmatrix}1&-1\\0&0\\-1&0\\0&-1\end{pmatrix},
\qquad
C_-=\begin{pmatrix}1&0\\-1&-1\\1&0\\-1&-1\end{pmatrix},
\qquad
B_+=V_3+C_+C_+^T,\quad B_-=-V_3^T+C_-C_-^T.
\]
Expanding the two determinants gives $\det(tB_+-B_+^T)=\det(tB_--B_-^T)=t^2$. Applying \Cref{thm:KMRS-topological} to the two columns of each $C_\pm$ gives the stated bounds.
\end{proof}
%
%
%
{\footnotesize \bibliographystyle{alphamod-b}
\bibliography{references}}
\end{document}

%% file: figures/Th-3q.tex
\begin{tikzpicture}[line width=1.00pt, line cap=round, line join=round, every node/.style={font=\small}]
\draw (0.000,2.400) -- (0.151,2.395) -- (0.302,2.381) -- (0.451,2.357) -- (0.599,2.324) -- (0.744,2.282) -- (0.887,2.230) -- (1.025,2.170) -- (1.160,2.101) -- (1.290,2.024) -- (1.415,1.938) -- (1.534,1.845) -- (1.648,1.745) -- (1.754,1.638) -- (1.854,1.524) -- (1.946,1.404);
\draw (0.000,1.300) -- (0.074,1.301) -- (0.143,1.303) -- (0.207,1.307) -- (0.268,1.314) -- (0.327,1.322) -- (0.384,1.332) -- (0.440,1.343) -- (0.495,1.355) -- (0.550,1.368) -- (0.605,1.380) -- (0.661,1.392) -- (0.717,1.402) -- (0.775,1.410) -- (0.834,1.416) -- (0.895,1.417) -- (0.958,1.413) -- (1.022,1.404) -- (1.089,1.388) -- (1.157,1.364) -- (1.226,1.331) -- (1.296,1.287) -- (1.366,1.232) -- (1.434,1.164) -- (1.500,1.082);
\draw[white, line width=4.0pt, line cap=butt] (0.417,1.744) -- (0.454,1.720) -- (0.489,1.695) -- (0.521,1.669) -- (0.550,1.641) -- (0.577,1.613) -- (0.602,1.584) -- (0.625,1.555) -- (0.647,1.525) -- (0.666,1.495) -- (0.685,1.464) -- (0.702,1.433) -- (0.717,1.402) -- (0.733,1.371) -- (0.747,1.340) -- (0.761,1.309) -- (0.774,1.277) -- (0.787,1.246) -- (0.800,1.215) -- (0.814,1.184) -- (0.827,1.153) -- (0.841,1.122) -- (0.855,1.091) -- (0.870,1.059) -- (0.886,1.028);
\draw (0.000,1.850) -- (0.105,1.844) -- (0.200,1.828) -- (0.286,1.804) -- (0.362,1.773) -- (0.430,1.736) -- (0.489,1.695) -- (0.541,1.651) -- (0.586,1.604) -- (0.625,1.555) -- (0.660,1.505) -- (0.690,1.454) -- (0.717,1.402) -- (0.742,1.350) -- (0.765,1.298) -- (0.787,1.246) -- (0.809,1.194) -- (0.832,1.142) -- (0.855,1.091) -- (0.881,1.039) -- (0.908,0.986) -- (0.939,0.933) -- (0.973,0.878) -- (1.012,0.821) -- (1.054,0.761);
\draw (1.054,0.761) -- (1.089,0.710) -- (1.121,0.659) -- (1.150,0.606) -- (1.177,0.551) -- (1.202,0.496) -- (1.224,0.439) -- (1.243,0.381) -- (1.259,0.323) -- (1.273,0.264) -- (1.284,0.204) -- (1.292,0.144) -- (1.297,0.083) -- (1.300,0.022) -- (1.299,-0.038) -- (1.296,-0.099);
\draw (1.946,1.404) -- (2.001,1.319) -- (2.044,1.237) -- (2.076,1.157) -- (2.097,1.080) -- (2.111,1.007) -- (2.116,0.937) -- (2.115,0.870) -- (2.109,0.806) -- (2.097,0.745) -- (2.083,0.687) -- (2.065,0.630) -- (2.046,0.576) -- (2.024,0.522) -- (2.003,0.469) -- (1.981,0.416) -- (1.959,0.364) -- (1.938,0.310) -- (1.919,0.255) -- (1.901,0.198) -- (1.886,0.138) -- (1.872,0.075) -- (1.861,0.008) -- (1.852,-0.063) -- (1.845,-0.141);
\draw[white, line width=4.0pt, line cap=butt] (1.716,0.831) -- (1.743,0.807) -- (1.770,0.784) -- (1.797,0.761) -- (1.825,0.740) -- (1.852,0.719) -- (1.880,0.698) -- (1.907,0.678) -- (1.935,0.658) -- (1.963,0.638) -- (1.990,0.617) -- (2.018,0.597) -- (2.046,0.576) -- (2.073,0.554) -- (2.100,0.532) -- (2.127,0.508) -- (2.153,0.484) -- (2.178,0.458) -- (2.203,0.431) -- (2.227,0.402) -- (2.251,0.371) -- (2.273,0.339) -- (2.294,0.304) -- (2.314,0.268) -- (2.332,0.228);
\draw (1.500,1.082) -- (1.547,1.020) -- (1.592,0.963) -- (1.637,0.912) -- (1.681,0.866) -- (1.725,0.823) -- (1.770,0.784) -- (1.815,0.747) -- (1.861,0.712) -- (1.907,0.678) -- (1.953,0.644) -- (2.000,0.610) -- (2.046,0.576) -- (2.091,0.539) -- (2.135,0.500) -- (2.178,0.458) -- (2.220,0.412) -- (2.258,0.361) -- (2.294,0.304) -- (2.326,0.242) -- (2.353,0.172) -- (2.374,0.095) -- (2.389,0.011) -- (2.396,-0.082) -- (2.393,-0.183);
\draw (2.393,-0.183) -- (2.375,-0.344) -- (2.347,-0.503) -- (2.307,-0.661) -- (2.257,-0.815) -- (2.197,-0.966) -- (2.127,-1.112) -- (2.047,-1.253) -- (1.958,-1.388) -- (1.860,-1.517) -- (1.753,-1.639) -- (1.638,-1.754) -- (1.516,-1.860) -- (1.387,-1.959) -- (1.252,-2.048);
\draw (1.296,-0.099) -- (1.291,-0.173) -- (1.288,-0.241) -- (1.288,-0.306) -- (1.289,-0.367) -- (1.293,-0.427) -- (1.299,-0.484) -- (1.305,-0.541) -- (1.313,-0.597) -- (1.322,-0.653) -- (1.330,-0.709) -- (1.337,-0.765) -- (1.343,-0.822) -- (1.347,-0.880) -- (1.348,-0.940) -- (1.345,-1.000) -- (1.336,-1.063) -- (1.322,-1.126) -- (1.301,-1.191) -- (1.272,-1.257) -- (1.233,-1.324) -- (1.185,-1.390) -- (1.124,-1.456) -- (1.051,-1.519) -- (0.965,-1.578);
\draw[white, line width=4.0pt, line cap=butt] (1.707,-0.548) -- (1.680,-0.584) -- (1.653,-0.617) -- (1.624,-0.646) -- (1.595,-0.674) -- (1.565,-0.699) -- (1.534,-0.721) -- (1.503,-0.742) -- (1.471,-0.761) -- (1.440,-0.778) -- (1.408,-0.794) -- (1.376,-0.809) -- (1.343,-0.822) -- (1.311,-0.835) -- (1.279,-0.847) -- (1.247,-0.858) -- (1.215,-0.869) -- (1.183,-0.880) -- (1.151,-0.891) -- (1.118,-0.901) -- (1.086,-0.913) -- (1.054,-0.924) -- (1.022,-0.936) -- (0.990,-0.948) -- (0.958,-0.962);
\draw (1.845,-0.141) -- (1.831,-0.245) -- (1.808,-0.339) -- (1.777,-0.422) -- (1.740,-0.496) -- (1.698,-0.560) -- (1.653,-0.617) -- (1.605,-0.665) -- (1.554,-0.706) -- (1.503,-0.742) -- (1.450,-0.773) -- (1.397,-0.799) -- (1.343,-0.822) -- (1.290,-0.843) -- (1.236,-0.862) -- (1.183,-0.880) -- (1.129,-0.898) -- (1.076,-0.916) -- (1.022,-0.936) -- (0.968,-0.957) -- (0.914,-0.981) -- (0.858,-1.007) -- (0.801,-1.037) -- (0.741,-1.071) -- (0.678,-1.109);
\draw (0.678,-1.109) -- (0.625,-1.140) -- (0.572,-1.168) -- (0.516,-1.193) -- (0.460,-1.216) -- (0.403,-1.236) -- (0.345,-1.254) -- (0.286,-1.268) -- (0.226,-1.280) -- (0.166,-1.289) -- (0.106,-1.296) -- (0.045,-1.299) -- (-0.016,-1.300) -- (-0.077,-1.298) -- (-0.137,-1.293) -- (-0.197,-1.285);
\draw (1.252,-2.048) -- (1.163,-2.096) -- (1.078,-2.132) -- (0.996,-2.158) -- (0.917,-2.174) -- (0.843,-2.181) -- (0.773,-2.181) -- (0.706,-2.175) -- (0.643,-2.164) -- (0.583,-2.148) -- (0.526,-2.129) -- (0.471,-2.107) -- (0.418,-2.083) -- (0.366,-2.058) -- (0.315,-2.033) -- (0.264,-2.007) -- (0.213,-1.981) -- (0.161,-1.956) -- (0.108,-1.933) -- (0.052,-1.911) -- (-0.006,-1.891) -- (-0.067,-1.872) -- (-0.133,-1.856) -- (-0.204,-1.842) -- (-0.281,-1.829);
\draw[white, line width=4.0pt, line cap=butt] (0.698,-1.775) -- (0.672,-1.800) -- (0.647,-1.825) -- (0.622,-1.850) -- (0.599,-1.876) -- (0.575,-1.901) -- (0.553,-1.927) -- (0.531,-1.953) -- (0.508,-1.979) -- (0.486,-2.005) -- (0.464,-2.032) -- (0.441,-2.058) -- (0.418,-2.083) -- (0.395,-2.109) -- (0.370,-2.134) -- (0.345,-2.159) -- (0.318,-2.183) -- (0.291,-2.207) -- (0.262,-2.230) -- (0.231,-2.252) -- (0.199,-2.273) -- (0.165,-2.292) -- (0.129,-2.310) -- (0.091,-2.327) -- (0.050,-2.342);
\draw (0.965,-1.578) -- (0.899,-1.620) -- (0.839,-1.661) -- (0.785,-1.701) -- (0.735,-1.742) -- (0.689,-1.783) -- (0.647,-1.825) -- (0.606,-1.867) -- (0.568,-1.910) -- (0.531,-1.953) -- (0.494,-1.997) -- (0.456,-2.040) -- (0.418,-2.083) -- (0.378,-2.126) -- (0.336,-2.167) -- (0.291,-2.207) -- (0.242,-2.244) -- (0.188,-2.279) -- (0.129,-2.310) -- (0.064,-2.337) -- (-0.007,-2.359) -- (-0.086,-2.375) -- (-0.171,-2.383) -- (-0.264,-2.383) -- (-0.364,-2.372);
\draw (-0.197,-1.285) -- (-0.266,-1.272) -- (-0.334,-1.256) -- (-0.401,-1.237) -- (-0.467,-1.213) -- (-0.531,-1.187);
\draw (-1.224,-0.439) -- (-1.245,-0.373) -- (-1.264,-0.306) -- (-1.278,-0.237) -- (-1.289,-0.168) -- (-1.296,-0.099);
\fill (-0.743,-1.066) circle (0.035);
\fill (-0.954,-0.884) circle (0.035);
\fill (-1.120,-0.660) circle (0.035);
\draw (-0.281,-1.829) -- (-0.379,-1.811) -- (-0.475,-1.788) -- (-0.570,-1.760) -- (-0.664,-1.727) -- (-0.756,-1.689);
\draw (-1.741,-0.625) -- (-1.772,-0.531) -- (-1.798,-0.435) -- (-1.819,-0.338) -- (-1.834,-0.240) -- (-1.845,-0.141);
\fill (-1.058,-1.518) circle (0.035);
\fill (-1.357,-1.257) circle (0.035);
\fill (-1.594,-0.939) circle (0.035);
\draw (-0.364,-2.372) -- (-0.491,-2.349) -- (-0.617,-2.319) -- (-0.740,-2.283) -- (-0.862,-2.240) -- (-0.980,-2.191);
\draw (-2.259,-0.811) -- (-2.299,-0.688) -- (-2.333,-0.564) -- (-2.360,-0.438) -- (-2.380,-0.311) -- (-2.393,-0.183);
\fill (-1.372,-1.969) circle (0.035);
\fill (-1.760,-1.631) circle (0.035);
\fill (-2.068,-1.219) circle (0.035);
\draw (-2.393,-0.183) -- (-2.400,-0.032) -- (-2.397,0.120) -- (-2.385,0.271) -- (-2.363,0.420) -- (-2.332,0.568) -- (-2.291,0.714) -- (-2.242,0.857) -- (-2.183,0.997) -- (-2.116,1.132) -- (-2.041,1.263) -- (-1.957,1.390) -- (-1.865,1.510) -- (-1.767,1.625) -- (-1.661,1.733) -- (-1.548,1.834);
\draw (-1.296,-0.099) -- (-1.303,-0.025) -- (-1.310,0.043) -- (-1.319,0.107) -- (-1.330,0.167) -- (-1.343,0.225) -- (-1.357,0.282) -- (-1.372,0.336) -- (-1.389,0.391) -- (-1.405,0.444) -- (-1.422,0.498) -- (-1.438,0.553) -- (-1.453,0.609) -- (-1.465,0.666) -- (-1.475,0.724) -- (-1.481,0.785) -- (-1.482,0.847) -- (-1.478,0.912) -- (-1.467,0.980) -- (-1.448,1.049) -- (-1.420,1.121) -- (-1.382,1.194) -- (-1.332,1.268) -- (-1.270,1.342) -- (-1.193,1.414);
\draw[white, line width=4.0pt, line cap=butt] (-1.770,0.283) -- (-1.750,0.322) -- (-1.727,0.358) -- (-1.704,0.392) -- (-1.679,0.424) -- (-1.653,0.453) -- (-1.626,0.480) -- (-1.598,0.505) -- (-1.570,0.529) -- (-1.541,0.551) -- (-1.512,0.571) -- (-1.482,0.590) -- (-1.453,0.609) -- (-1.423,0.626) -- (-1.393,0.643) -- (-1.363,0.659) -- (-1.333,0.675) -- (-1.302,0.690) -- (-1.272,0.706) -- (-1.242,0.721) -- (-1.212,0.737) -- (-1.182,0.753) -- (-1.152,0.770) -- (-1.123,0.787) -- (-1.093,0.805);
\draw (-1.845,-0.141) -- (-1.847,-0.036) -- (-1.838,0.060) -- (-1.820,0.147) -- (-1.795,0.226) -- (-1.764,0.296) -- (-1.727,0.358) -- (-1.687,0.414) -- (-1.644,0.462) -- (-1.598,0.505) -- (-1.551,0.543) -- (-1.502,0.578) -- (-1.453,0.609) -- (-1.403,0.637) -- (-1.353,0.664) -- (-1.302,0.690) -- (-1.252,0.716) -- (-1.202,0.742) -- (-1.152,0.770) -- (-1.103,0.799) -- (-1.052,0.831) -- (-1.001,0.865) -- (-0.949,0.904) -- (-0.895,0.946) -- (-0.839,0.993);
\draw (-0.839,0.993) -- (-0.788,1.034) -- (-0.735,1.072) -- (-0.680,1.108) -- (-0.624,1.140) -- (-0.566,1.170) -- (-0.507,1.197) -- (-0.446,1.221) -- (-0.385,1.242) -- (-0.322,1.259) -- (-0.259,1.274) -- (-0.195,1.285) -- (-0.130,1.293) -- (-0.065,1.298) -- (-0.000,1.300);
\draw (-1.548,1.834) -- (-1.468,1.895) -- (-1.389,1.944) -- (-1.312,1.982) -- (-1.237,2.009) -- (-1.165,2.028) -- (-1.095,2.039) -- (-1.028,2.043) -- (-0.964,2.041) -- (-0.903,2.035) -- (-0.843,2.024) -- (-0.786,2.011) -- (-0.730,1.996) -- (-0.675,1.979) -- (-0.620,1.961) -- (-0.566,1.943) -- (-0.512,1.926) -- (-0.456,1.909) -- (-0.400,1.894) -- (-0.342,1.881) -- (-0.281,1.870) -- (-0.218,1.861) -- (-0.150,1.855) -- (-0.078,1.851) -- (-0.000,1.850);
\draw[white, line width=4.0pt, line cap=butt] (-0.960,1.648) -- (-0.937,1.677) -- (-0.916,1.705) -- (-0.896,1.734) -- (-0.876,1.763) -- (-0.857,1.792) -- (-0.839,1.821) -- (-0.821,1.850) -- (-0.803,1.879) -- (-0.785,1.908) -- (-0.767,1.938) -- (-0.749,1.967) -- (-0.730,1.996) -- (-0.710,2.025) -- (-0.690,2.053) -- (-0.669,2.082) -- (-0.646,2.110) -- (-0.623,2.137) -- (-0.597,2.164) -- (-0.570,2.190) -- (-0.542,2.216) -- (-0.511,2.241) -- (-0.478,2.264) -- (-0.443,2.286) -- (-0.405,2.307);
\draw (-1.193,1.414) -- (-1.135,1.465) -- (-1.082,1.514) -- (-1.034,1.562) -- (-0.991,1.610) -- (-0.952,1.658) -- (-0.916,1.705) -- (-0.883,1.753) -- (-0.851,1.802) -- (-0.821,1.850) -- (-0.791,1.899) -- (-0.761,1.947) -- (-0.730,1.996) -- (-0.697,2.044) -- (-0.661,2.091) -- (-0.623,2.137) -- (-0.580,2.182) -- (-0.532,2.224) -- (-0.478,2.264) -- (-0.418,2.301) -- (-0.351,2.333) -- (-0.276,2.360) -- (-0.193,2.381) -- (-0.101,2.395) -- (-0.000,2.400);
\end{tikzpicture}

%% file: figures/Th-3-4.tex
\begin{tikzpicture}[line width=1.00pt, line cap=round, line join=round, every node/.style={font=\small}]
\draw (0.000,2.200) -- (0.137,2.196) -- (0.273,2.183) -- (0.408,2.162) -- (0.542,2.132) -- (0.673,2.095) -- (0.802,2.049) -- (0.928,1.995) -- (1.050,1.933) -- (1.168,1.864) -- (1.281,1.788) -- (1.390,1.705) -- (1.493,1.616) -- (1.591,1.520) -- (1.682,1.418) -- (1.767,1.311);
\draw (0.000,1.000) -- (0.056,1.001) -- (0.108,1.006) -- (0.158,1.014) -- (0.206,1.024) -- (0.253,1.037) -- (0.299,1.052) -- (0.345,1.069) -- (0.391,1.087) -- (0.438,1.106) -- (0.486,1.125) -- (0.534,1.144) -- (0.584,1.161) -- (0.636,1.177) -- (0.689,1.189) -- (0.744,1.198) -- (0.800,1.202) -- (0.859,1.201) -- (0.918,1.194) -- (0.980,1.179) -- (1.042,1.155) -- (1.105,1.122) -- (1.167,1.077) -- (1.228,1.022) -- (1.285,0.953);
\draw[white, line width=4.0pt, line cap=butt] (0.352,1.497) -- (0.383,1.473) -- (0.412,1.449) -- (0.438,1.423) -- (0.461,1.396) -- (0.482,1.369) -- (0.502,1.340) -- (0.519,1.311) -- (0.535,1.282) -- (0.549,1.252) -- (0.562,1.222) -- (0.574,1.191) -- (0.584,1.161) -- (0.594,1.131) -- (0.603,1.101) -- (0.612,1.071) -- (0.620,1.041) -- (0.628,1.011) -- (0.635,0.981) -- (0.643,0.952) -- (0.651,0.923) -- (0.659,0.895) -- (0.667,0.867) -- (0.676,0.839) -- (0.686,0.812);
\draw (0.000,1.600) -- (0.089,1.594) -- (0.170,1.579) -- (0.243,1.555) -- (0.307,1.525) -- (0.363,1.489) -- (0.412,1.449) -- (0.454,1.405) -- (0.489,1.359) -- (0.519,1.311) -- (0.545,1.262) -- (0.566,1.212) -- (0.584,1.161) -- (0.600,1.111) -- (0.615,1.061) -- (0.628,1.011) -- (0.640,0.962) -- (0.653,0.914) -- (0.667,0.867) -- (0.682,0.821) -- (0.700,0.775) -- (0.720,0.731) -- (0.743,0.686) -- (0.771,0.642) -- (0.803,0.596);
\draw (0.803,0.596) -- (0.828,0.561) -- (0.851,0.525) -- (0.872,0.489) -- (0.892,0.451) -- (0.911,0.413) -- (0.928,0.374) -- (0.943,0.334) -- (0.956,0.293) -- (0.968,0.252) -- (0.977,0.211) -- (0.986,0.169) -- (0.992,0.127) -- (0.996,0.085) -- (0.999,0.043) -- (1.000,0.000);
\draw (1.767,1.311) -- (1.813,1.240) -- (1.849,1.171) -- (1.874,1.103) -- (1.889,1.038) -- (1.897,0.974) -- (1.898,0.914) -- (1.892,0.855) -- (1.882,0.800) -- (1.867,0.746) -- (1.848,0.694) -- (1.827,0.644) -- (1.804,0.596) -- (1.780,0.549) -- (1.755,0.502) -- (1.731,0.457) -- (1.707,0.411) -- (1.684,0.365) -- (1.663,0.319) -- (1.645,0.272) -- (1.629,0.223) -- (1.617,0.172) -- (1.607,0.118) -- (1.602,0.061) -- (1.600,0.000);
\draw[white, line width=4.0pt, line cap=butt] (1.475,0.767) -- (1.500,0.751) -- (1.526,0.735) -- (1.553,0.720) -- (1.579,0.705) -- (1.607,0.691) -- (1.634,0.677) -- (1.662,0.664) -- (1.690,0.651) -- (1.719,0.638) -- (1.747,0.624) -- (1.776,0.610) -- (1.804,0.596) -- (1.833,0.581) -- (1.861,0.566) -- (1.889,0.549) -- (1.916,0.532) -- (1.944,0.513) -- (1.970,0.493) -- (1.996,0.471) -- (2.021,0.448) -- (2.045,0.423) -- (2.069,0.397) -- (2.090,0.368) -- (2.111,0.337);
\draw (1.285,0.953) -- (1.323,0.905) -- (1.362,0.863) -- (1.401,0.825) -- (1.441,0.792) -- (1.483,0.762) -- (1.526,0.735) -- (1.570,0.710) -- (1.616,0.687) -- (1.662,0.664) -- (1.709,0.642) -- (1.757,0.619) -- (1.804,0.596) -- (1.851,0.571) -- (1.898,0.543) -- (1.944,0.513) -- (1.988,0.479) -- (2.029,0.440) -- (2.069,0.397) -- (2.104,0.347) -- (2.136,0.292) -- (2.162,0.230) -- (2.182,0.161) -- (2.195,0.084) -- (2.200,0.000);
\draw (2.200,0.000) -- (2.195,-0.146) -- (2.181,-0.292) -- (2.156,-0.437) -- (2.122,-0.579) -- (2.079,-0.719) -- (2.027,-0.856) -- (1.965,-0.989) -- (1.895,-1.118) -- (1.816,-1.241) -- (1.730,-1.360) -- (1.635,-1.472) -- (1.534,-1.577) -- (1.425,-1.676) -- (1.311,-1.767);
\draw (1.000,0.000) -- (1.001,-0.056) -- (1.006,-0.108) -- (1.014,-0.158) -- (1.024,-0.206) -- (1.037,-0.253) -- (1.052,-0.299) -- (1.069,-0.345) -- (1.087,-0.391) -- (1.106,-0.438) -- (1.125,-0.486) -- (1.144,-0.534) -- (1.161,-0.584) -- (1.177,-0.636) -- (1.189,-0.689) -- (1.198,-0.744) -- (1.202,-0.800) -- (1.201,-0.859) -- (1.194,-0.918) -- (1.179,-0.980) -- (1.155,-1.042) -- (1.122,-1.105) -- (1.077,-1.167) -- (1.022,-1.228) -- (0.953,-1.285);
\draw[white, line width=4.0pt, line cap=butt] (1.497,-0.352) -- (1.473,-0.383) -- (1.449,-0.412) -- (1.423,-0.438) -- (1.396,-0.461) -- (1.369,-0.482) -- (1.340,-0.502) -- (1.311,-0.519) -- (1.282,-0.535) -- (1.252,-0.549) -- (1.222,-0.562) -- (1.191,-0.574) -- (1.161,-0.584) -- (1.131,-0.594) -- (1.101,-0.603) -- (1.071,-0.612) -- (1.041,-0.620) -- (1.011,-0.628) -- (0.981,-0.635) -- (0.952,-0.643) -- (0.923,-0.651) -- (0.895,-0.659) -- (0.867,-0.667) -- (0.839,-0.676) -- (0.812,-0.686);
\draw (1.600,0.000) -- (1.594,-0.089) -- (1.579,-0.170) -- (1.555,-0.243) -- (1.525,-0.307) -- (1.489,-0.363) -- (1.449,-0.412) -- (1.405,-0.454) -- (1.359,-0.489) -- (1.311,-0.519) -- (1.262,-0.545) -- (1.212,-0.566) -- (1.161,-0.584) -- (1.111,-0.600) -- (1.061,-0.615) -- (1.011,-0.628) -- (0.962,-0.640) -- (0.914,-0.653) -- (0.867,-0.667) -- (0.821,-0.682) -- (0.775,-0.700) -- (0.731,-0.720) -- (0.686,-0.743) -- (0.642,-0.771) -- (0.596,-0.803);
\draw (0.596,-0.803) -- (0.561,-0.828) -- (0.525,-0.851) -- (0.489,-0.872) -- (0.451,-0.892) -- (0.413,-0.911) -- (0.374,-0.928) -- (0.334,-0.943) -- (0.293,-0.956) -- (0.252,-0.968) -- (0.211,-0.977) -- (0.169,-0.986) -- (0.127,-0.992) -- (0.085,-0.996) -- (0.043,-0.999) -- (0.000,-1.000);
\draw (1.311,-1.767) -- (1.240,-1.813) -- (1.171,-1.849) -- (1.103,-1.874) -- (1.038,-1.889) -- (0.974,-1.897) -- (0.914,-1.898) -- (0.855,-1.892) -- (0.800,-1.882) -- (0.746,-1.867) -- (0.694,-1.848) -- (0.644,-1.827) -- (0.596,-1.804) -- (0.549,-1.780) -- (0.502,-1.755) -- (0.457,-1.731) -- (0.411,-1.707) -- (0.365,-1.684) -- (0.319,-1.663) -- (0.272,-1.645) -- (0.223,-1.629) -- (0.172,-1.617) -- (0.118,-1.607) -- (0.061,-1.602) -- (0.000,-1.600);
\draw[white, line width=4.0pt, line cap=butt] (0.767,-1.475) -- (0.751,-1.500) -- (0.735,-1.526) -- (0.720,-1.553) -- (0.705,-1.579) -- (0.691,-1.607) -- (0.677,-1.634) -- (0.664,-1.662) -- (0.651,-1.690) -- (0.638,-1.719) -- (0.624,-1.747) -- (0.610,-1.776) -- (0.596,-1.804) -- (0.581,-1.833) -- (0.566,-1.861) -- (0.549,-1.889) -- (0.532,-1.916) -- (0.513,-1.944) -- (0.493,-1.970) -- (0.471,-1.996) -- (0.448,-2.021) -- (0.423,-2.045) -- (0.397,-2.069) -- (0.368,-2.090) -- (0.337,-2.111);
\draw (0.953,-1.285) -- (0.905,-1.323) -- (0.863,-1.362) -- (0.825,-1.401) -- (0.792,-1.441) -- (0.762,-1.483) -- (0.735,-1.526) -- (0.710,-1.570) -- (0.687,-1.616) -- (0.664,-1.662) -- (0.642,-1.709) -- (0.619,-1.757) -- (0.596,-1.804) -- (0.571,-1.851) -- (0.543,-1.898) -- (0.513,-1.944) -- (0.479,-1.988) -- (0.440,-2.029) -- (0.397,-2.069) -- (0.347,-2.104) -- (0.292,-2.136) -- (0.230,-2.162) -- (0.161,-2.182) -- (0.084,-2.195) -- (0.000,-2.200);
\draw (0.000,-2.200) -- (-0.137,-2.196) -- (-0.273,-2.183) -- (-0.408,-2.162) -- (-0.542,-2.132) -- (-0.673,-2.095) -- (-0.802,-2.049) -- (-0.928,-1.995) -- (-1.050,-1.933) -- (-1.168,-1.864) -- (-1.281,-1.788) -- (-1.390,-1.705) -- (-1.493,-1.616) -- (-1.591,-1.520) -- (-1.682,-1.418) -- (-1.767,-1.311);
\draw (0.000,-1.000) -- (-0.056,-1.001) -- (-0.108,-1.006) -- (-0.158,-1.014) -- (-0.206,-1.024) -- (-0.253,-1.037) -- (-0.299,-1.052) -- (-0.345,-1.069) -- (-0.391,-1.087) -- (-0.438,-1.106) -- (-0.486,-1.125) -- (-0.534,-1.144) -- (-0.584,-1.161) -- (-0.636,-1.177) -- (-0.689,-1.189) -- (-0.744,-1.198) -- (-0.800,-1.202) -- (-0.859,-1.201) -- (-0.918,-1.194) -- (-0.980,-1.179) -- (-1.042,-1.155) -- (-1.105,-1.122) -- (-1.167,-1.077) -- (-1.228,-1.022) -- (-1.285,-0.953);
\draw[white, line width=4.0pt, line cap=butt] (-0.352,-1.497) -- (-0.383,-1.473) -- (-0.412,-1.449) -- (-0.438,-1.423) -- (-0.461,-1.396) -- (-0.482,-1.369) -- (-0.502,-1.340) -- (-0.519,-1.311) -- (-0.535,-1.282) -- (-0.549,-1.252) -- (-0.562,-1.222) -- (-0.574,-1.191) -- (-0.584,-1.161) -- (-0.594,-1.131) -- (-0.603,-1.101) -- (-0.612,-1.071) -- (-0.620,-1.041) -- (-0.628,-1.011) -- (-0.635,-0.981) -- (-0.643,-0.952) -- (-0.651,-0.923) -- (-0.659,-0.895) -- (-0.667,-0.867) -- (-0.676,-0.839) -- (-0.686,-0.812);
\draw (0.000,-1.600) -- (-0.089,-1.594) -- (-0.170,-1.579) -- (-0.243,-1.555) -- (-0.307,-1.525) -- (-0.363,-1.489) -- (-0.412,-1.449) -- (-0.454,-1.405) -- (-0.489,-1.359) -- (-0.519,-1.311) -- (-0.545,-1.262) -- (-0.566,-1.212) -- (-0.584,-1.161) -- (-0.600,-1.111) -- (-0.615,-1.061) -- (-0.628,-1.011) -- (-0.640,-0.962) -- (-0.653,-0.914) -- (-0.667,-0.867) -- (-0.682,-0.821) -- (-0.700,-0.775) -- (-0.720,-0.731) -- (-0.743,-0.686) -- (-0.771,-0.642) -- (-0.803,-0.596);
\draw (-0.803,-0.596) -- (-0.828,-0.561) -- (-0.851,-0.525) -- (-0.872,-0.489) -- (-0.892,-0.451) -- (-0.911,-0.413) -- (-0.928,-0.374) -- (-0.943,-0.334) -- (-0.956,-0.293) -- (-0.968,-0.252) -- (-0.977,-0.211) -- (-0.986,-0.169) -- (-0.992,-0.127) -- (-0.996,-0.085) -- (-0.999,-0.043) -- (-1.000,0.000);
\draw (-1.767,-1.311) -- (-1.813,-1.240) -- (-1.849,-1.171) -- (-1.874,-1.103) -- (-1.889,-1.038) -- (-1.897,-0.974) -- (-1.898,-0.914) -- (-1.892,-0.855) -- (-1.882,-0.800) -- (-1.867,-0.746) -- (-1.848,-0.694) -- (-1.827,-0.644) -- (-1.804,-0.596) -- (-1.780,-0.549) -- (-1.755,-0.502) -- (-1.731,-0.457) -- (-1.707,-0.411) -- (-1.684,-0.365) -- (-1.663,-0.319) -- (-1.645,-0.272) -- (-1.629,-0.223) -- (-1.617,-0.172) -- (-1.607,-0.118) -- (-1.602,-0.061) -- (-1.600,0.000);
\draw[white, line width=4.0pt, line cap=butt] (-1.475,-0.767) -- (-1.500,-0.751) -- (-1.526,-0.735) -- (-1.553,-0.720) -- (-1.579,-0.705) -- (-1.607,-0.691) -- (-1.634,-0.677) -- (-1.662,-0.664) -- (-1.690,-0.651) -- (-1.719,-0.638) -- (-1.747,-0.624) -- (-1.776,-0.610) -- (-1.804,-0.596) -- (-1.833,-0.581) -- (-1.861,-0.566) -- (-1.889,-0.549) -- (-1.916,-0.532) -- (-1.944,-0.513) -- (-1.970,-0.493) -- (-1.996,-0.471) -- (-2.021,-0.448) -- (-2.045,-0.423) -- (-2.069,-0.397) -- (-2.090,-0.368) -- (-2.111,-0.337);
\draw (-1.285,-0.953) -- (-1.323,-0.905) -- (-1.362,-0.863) -- (-1.401,-0.825) -- (-1.441,-0.792) -- (-1.483,-0.762) -- (-1.526,-0.735) -- (-1.570,-0.710) -- (-1.616,-0.687) -- (-1.662,-0.664) -- (-1.709,-0.642) -- (-1.757,-0.619) -- (-1.804,-0.596) -- (-1.851,-0.571) -- (-1.898,-0.543) -- (-1.944,-0.513) -- (-1.988,-0.479) -- (-2.029,-0.440) -- (-2.069,-0.397) -- (-2.104,-0.347) -- (-2.136,-0.292) -- (-2.162,-0.230) -- (-2.182,-0.161) -- (-2.195,-0.084) -- (-2.200,0.000);
\draw (-2.200,0.000) -- (-2.196,0.137) -- (-2.183,0.273) -- (-2.162,0.408) -- (-2.132,0.542) -- (-2.095,0.673) -- (-2.049,0.802) -- (-1.995,0.928) -- (-1.933,1.050) -- (-1.864,1.168) -- (-1.788,1.281) -- (-1.705,1.390) -- (-1.616,1.493) -- (-1.520,1.591) -- (-1.418,1.682) -- (-1.311,1.767);
\draw (-1.000,0.000) -- (-1.001,0.056) -- (-1.006,0.108) -- (-1.014,0.158) -- (-1.024,0.206) -- (-1.037,0.253) -- (-1.052,0.299) -- (-1.069,0.345) -- (-1.087,0.391) -- (-1.106,0.438) -- (-1.125,0.486) -- (-1.144,0.534) -- (-1.161,0.584) -- (-1.177,0.636) -- (-1.189,0.689) -- (-1.198,0.744) -- (-1.202,0.800) -- (-1.201,0.859) -- (-1.194,0.918) -- (-1.179,0.980) -- (-1.155,1.042) -- (-1.122,1.105) -- (-1.077,1.167) -- (-1.022,1.228) -- (-0.953,1.285);
\draw[white, line width=4.0pt, line cap=butt] (-1.497,0.352) -- (-1.473,0.383) -- (-1.449,0.412) -- (-1.423,0.438) -- (-1.396,0.461) -- (-1.369,0.482) -- (-1.340,0.502) -- (-1.311,0.519) -- (-1.282,0.535) -- (-1.252,0.549) -- (-1.222,0.562) -- (-1.191,0.574) -- (-1.161,0.584) -- (-1.131,0.594) -- (-1.101,0.603) -- (-1.071,0.612) -- (-1.041,0.620) -- (-1.011,0.628) -- (-0.981,0.635) -- (-0.952,0.643) -- (-0.923,0.651) -- (-0.895,0.659) -- (-0.867,0.667) -- (-0.839,0.676) -- (-0.812,0.686);
\draw (-1.600,0.000) -- (-1.594,0.089) -- (-1.579,0.170) -- (-1.555,0.243) -- (-1.525,0.307) -- (-1.489,0.363) -- (-1.449,0.412) -- (-1.405,0.454) -- (-1.359,0.489) -- (-1.311,0.519) -- (-1.262,0.545) -- (-1.212,0.566) -- (-1.161,0.584) -- (-1.111,0.600) -- (-1.061,0.615) -- (-1.011,0.628) -- (-0.962,0.640) -- (-0.914,0.653) -- (-0.867,0.667) -- (-0.821,0.682) -- (-0.775,0.700) -- (-0.731,0.720) -- (-0.686,0.743) -- (-0.642,0.771) -- (-0.596,0.803);
\draw (-0.596,0.803) -- (-0.558,0.830) -- (-0.520,0.854) -- (-0.481,0.877) -- (-0.440,0.898) -- (-0.399,0.917) -- (-0.357,0.934) -- (-0.314,0.950) -- (-0.270,0.963) -- (-0.226,0.974) -- (-0.181,0.983) -- (-0.136,0.991) -- (-0.091,0.996) -- (-0.046,0.999) -- (-0.000,1.000);
\draw (-1.311,1.767) -- (-1.240,1.813) -- (-1.171,1.849) -- (-1.103,1.874) -- (-1.038,1.889) -- (-0.974,1.897) -- (-0.914,1.898) -- (-0.855,1.892) -- (-0.800,1.882) -- (-0.746,1.867) -- (-0.694,1.848) -- (-0.644,1.827) -- (-0.596,1.804) -- (-0.549,1.780) -- (-0.502,1.755) -- (-0.457,1.731) -- (-0.411,1.707) -- (-0.365,1.684) -- (-0.319,1.663) -- (-0.272,1.645) -- (-0.223,1.629) -- (-0.172,1.617) -- (-0.118,1.607) -- (-0.061,1.602) -- (-0.000,1.600);
\draw[white, line width=4.0pt, line cap=butt] (-0.767,1.475) -- (-0.751,1.500) -- (-0.735,1.526) -- (-0.720,1.553) -- (-0.705,1.579) -- (-0.691,1.607) -- (-0.677,1.634) -- (-0.664,1.662) -- (-0.651,1.690) -- (-0.638,1.719) -- (-0.624,1.747) -- (-0.610,1.776) -- (-0.596,1.804) -- (-0.581,1.833) -- (-0.566,1.861) -- (-0.549,1.889) -- (-0.532,1.916) -- (-0.513,1.944) -- (-0.493,1.970) -- (-0.471,1.996) -- (-0.448,2.021) -- (-0.423,2.045) -- (-0.397,2.069) -- (-0.368,2.090) -- (-0.337,2.111);
\draw (-0.953,1.285) -- (-0.905,1.323) -- (-0.863,1.362) -- (-0.825,1.401) -- (-0.792,1.441) -- (-0.762,1.483) -- (-0.735,1.526) -- (-0.710,1.570) -- (-0.687,1.616) -- (-0.664,1.662) -- (-0.642,1.709) -- (-0.619,1.757) -- (-0.596,1.804) -- (-0.571,1.851) -- (-0.543,1.898) -- (-0.513,1.944) -- (-0.479,1.988) -- (-0.440,2.029) -- (-0.397,2.069) -- (-0.347,2.104) -- (-0.292,2.136) -- (-0.230,2.162) -- (-0.161,2.182) -- (-0.084,2.195) -- (-0.000,2.200);
\end{tikzpicture}

%% file: figures/Th-3-2.tex
\begin{tikzpicture}[line width=1.00pt, line cap=round, line join=round, every node/.style={font=\small}]
\draw (0.000,2.200) -- (0.139,2.196) -- (0.278,2.182) -- (0.415,2.160) -- (0.551,2.130) -- (0.685,2.091) -- (0.815,2.043) -- (0.943,1.988) -- (1.067,1.924) -- (1.186,1.853) -- (1.301,1.774) -- (1.411,1.688) -- (1.515,1.595) -- (1.613,1.496) -- (1.704,1.391) -- (1.789,1.281);
\draw (0.000,1.000) -- (0.057,1.001) -- (0.110,1.006) -- (0.161,1.013) -- (0.210,1.023) -- (0.257,1.036) -- (0.304,1.051) -- (0.351,1.067) -- (0.398,1.085) -- (0.446,1.103) -- (0.494,1.122) -- (0.543,1.140) -- (0.594,1.156) -- (0.646,1.171) -- (0.700,1.183) -- (0.756,1.191) -- (0.813,1.194) -- (0.872,1.192) -- (0.932,1.183) -- (0.994,1.166) -- (1.057,1.141) -- (1.120,1.106) -- (1.183,1.060) -- (1.243,1.002) -- (1.301,0.932);
\draw[white, line width=4.0pt, line cap=butt] (0.359,1.495) -- (0.390,1.472) -- (0.419,1.447) -- (0.445,1.421) -- (0.469,1.394) -- (0.491,1.366) -- (0.510,1.337) -- (0.528,1.308) -- (0.544,1.278) -- (0.558,1.248) -- (0.572,1.217) -- (0.583,1.187) -- (0.594,1.156) -- (0.604,1.126) -- (0.613,1.095) -- (0.622,1.065) -- (0.630,1.035) -- (0.638,1.005) -- (0.645,0.975) -- (0.653,0.945) -- (0.661,0.916) -- (0.669,0.887) -- (0.677,0.859) -- (0.686,0.831) -- (0.696,0.803);
\draw (0.000,1.600) -- (0.091,1.594) -- (0.173,1.579) -- (0.247,1.555) -- (0.312,1.524) -- (0.369,1.488) -- (0.419,1.447) -- (0.461,1.403) -- (0.497,1.356) -- (0.528,1.308) -- (0.554,1.258) -- (0.576,1.207) -- (0.594,1.156) -- (0.610,1.105) -- (0.624,1.055) -- (0.638,1.005) -- (0.650,0.955) -- (0.663,0.907) -- (0.677,0.859) -- (0.692,0.812) -- (0.710,0.766) -- (0.730,0.721) -- (0.753,0.675) -- (0.781,0.629) -- (0.813,0.582);
\draw (0.813,0.582) -- (0.842,0.540) -- (0.869,0.496) -- (0.893,0.450) -- (0.915,0.404) -- (0.934,0.356) -- (0.951,0.308) -- (0.966,0.259) -- (0.978,0.209) -- (0.987,0.158) -- (0.994,0.107) -- (0.998,0.056) -- (1.000,0.004) -- (0.999,-0.047) -- (0.995,-0.098) -- (0.989,-0.149);
\draw (1.301,0.932) -- (1.347,0.863) -- (1.390,0.793) -- (1.429,0.721) -- (1.464,0.646) -- (1.495,0.570) -- (1.522,0.493) -- (1.546,0.414) -- (1.565,0.334) -- (1.580,0.253) -- (1.591,0.171) -- (1.598,0.089) -- (1.600,0.007) -- (1.598,-0.075) -- (1.592,-0.157) -- (1.582,-0.239);
\draw (1.789,1.281) -- (1.852,1.187) -- (1.911,1.090) -- (1.964,0.991) -- (2.013,0.889) -- (2.056,0.784) -- (2.093,0.677) -- (2.125,0.569) -- (2.152,0.459) -- (2.172,0.348) -- (2.187,0.235) -- (2.197,0.123) -- (2.200,0.010) -- (2.198,-0.103) -- (2.189,-0.216) -- (2.175,-0.329);
\draw (0.989,-0.149) -- (0.981,-0.195) -- (0.971,-0.240) -- (0.959,-0.285) -- (0.944,-0.329) -- (0.928,-0.373) -- (0.910,-0.415) -- (0.889,-0.457) -- (0.867,-0.498) -- (0.843,-0.537) -- (0.817,-0.576) -- (0.790,-0.613) -- (0.761,-0.649) -- (0.730,-0.684) -- (0.697,-0.717);
\draw (2.175,-0.329) -- (2.158,-0.412) -- (2.133,-0.488) -- (2.102,-0.554) -- (2.067,-0.613) -- (2.027,-0.664) -- (1.984,-0.708) -- (1.938,-0.746) -- (1.891,-0.778) -- (1.842,-0.806) -- (1.792,-0.830) -- (1.741,-0.850) -- (1.690,-0.868) -- (1.639,-0.885) -- (1.588,-0.900) -- (1.538,-0.915) -- (1.488,-0.931) -- (1.440,-0.947) -- (1.392,-0.966) -- (1.344,-0.986) -- (1.298,-1.010) -- (1.252,-1.037) -- (1.207,-1.068) -- (1.161,-1.105) -- (1.115,-1.147);
\draw[white, line width=4.0pt, line cap=butt] (1.583,-0.509) -- (1.588,-0.539) -- (1.595,-0.569) -- (1.603,-0.599) -- (1.611,-0.629) -- (1.620,-0.658) -- (1.630,-0.688) -- (1.640,-0.718) -- (1.650,-0.747) -- (1.660,-0.777) -- (1.670,-0.807) -- (1.680,-0.838) -- (1.690,-0.868) -- (1.699,-0.899) -- (1.708,-0.931) -- (1.715,-0.963) -- (1.722,-0.995) -- (1.727,-1.028) -- (1.731,-1.061) -- (1.734,-1.095) -- (1.735,-1.130) -- (1.734,-1.165) -- (1.730,-1.201) -- (1.725,-1.237) -- (1.717,-1.274);
\draw (1.582,-0.239) -- (1.575,-0.301) -- (1.571,-0.359) -- (1.572,-0.414) -- (1.577,-0.467) -- (1.584,-0.519) -- (1.595,-0.569) -- (1.608,-0.619) -- (1.623,-0.668) -- (1.640,-0.718) -- (1.657,-0.767) -- (1.674,-0.817) -- (1.690,-0.868) -- (1.705,-0.920) -- (1.718,-0.974) -- (1.727,-1.028) -- (1.733,-1.084) -- (1.735,-1.142) -- (1.730,-1.201) -- (1.720,-1.261) -- (1.701,-1.324) -- (1.675,-1.387) -- (1.638,-1.450) -- (1.592,-1.514) -- (1.534,-1.577);
\draw (0.697,-0.717) -- (0.659,-0.752) -- (0.620,-0.785) -- (0.579,-0.816) -- (0.536,-0.844) -- (0.492,-0.871) -- (0.446,-0.895) -- (0.400,-0.917) -- (0.352,-0.936) -- (0.304,-0.953) -- (0.254,-0.967) -- (0.204,-0.979) -- (0.154,-0.988) -- (0.103,-0.995) -- (0.051,-0.999) -- (0.000,-1.000);
\draw (1.115,-1.147) -- (1.055,-1.203) -- (0.992,-1.256) -- (0.926,-1.305) -- (0.858,-1.351) -- (0.787,-1.393) -- (0.714,-1.432) -- (0.640,-1.467) -- (0.564,-1.497) -- (0.486,-1.524) -- (0.407,-1.547) -- (0.327,-1.566) -- (0.246,-1.581) -- (0.164,-1.592) -- (0.082,-1.598) -- (0.000,-1.600);
\draw (1.534,-1.577) -- (1.451,-1.654) -- (1.364,-1.726) -- (1.273,-1.794) -- (1.179,-1.857) -- (1.082,-1.915) -- (0.982,-1.969) -- (0.880,-2.016) -- (0.775,-2.059) -- (0.668,-2.096) -- (0.559,-2.128) -- (0.449,-2.154) -- (0.338,-2.174) -- (0.226,-2.188) -- (0.113,-2.197) -- (0.000,-2.200);
\draw (0.000,-2.200) -- (-0.139,-2.196) -- (-0.278,-2.182) -- (-0.415,-2.160) -- (-0.551,-2.130) -- (-0.685,-2.091) -- (-0.815,-2.043) -- (-0.943,-1.988) -- (-1.067,-1.924) -- (-1.186,-1.853) -- (-1.301,-1.774) -- (-1.411,-1.688) -- (-1.515,-1.595) -- (-1.613,-1.496) -- (-1.704,-1.391) -- (-1.789,-1.281);
\draw (0.000,-1.000) -- (-0.057,-1.001) -- (-0.110,-1.006) -- (-0.161,-1.013) -- (-0.210,-1.023) -- (-0.257,-1.036) -- (-0.304,-1.051) -- (-0.351,-1.067) -- (-0.398,-1.085) -- (-0.446,-1.103) -- (-0.494,-1.122) -- (-0.543,-1.140) -- (-0.594,-1.156) -- (-0.646,-1.171) -- (-0.700,-1.183) -- (-0.756,-1.191) -- (-0.813,-1.194) -- (-0.872,-1.192) -- (-0.932,-1.183) -- (-0.994,-1.166) -- (-1.057,-1.141) -- (-1.120,-1.106) -- (-1.183,-1.060) -- (-1.243,-1.002) -- (-1.301,-0.932);
\draw[white, line width=4.0pt, line cap=butt] (-0.359,-1.495) -- (-0.390,-1.472) -- (-0.419,-1.447) -- (-0.445,-1.421) -- (-0.469,-1.394) -- (-0.491,-1.366) -- (-0.510,-1.337) -- (-0.528,-1.308) -- (-0.544,-1.278) -- (-0.558,-1.248) -- (-0.572,-1.217) -- (-0.583,-1.187) -- (-0.594,-1.156) -- (-0.604,-1.126) -- (-0.613,-1.095) -- (-0.622,-1.065) -- (-0.630,-1.035) -- (-0.638,-1.005) -- (-0.645,-0.975) -- (-0.653,-0.945) -- (-0.661,-0.916) -- (-0.669,-0.887) -- (-0.677,-0.859) -- (-0.686,-0.831) -- (-0.696,-0.803);
\draw (0.000,-1.600) -- (-0.091,-1.594) -- (-0.173,-1.579) -- (-0.247,-1.555) -- (-0.312,-1.524) -- (-0.369,-1.488) -- (-0.419,-1.447) -- (-0.461,-1.403) -- (-0.497,-1.356) -- (-0.528,-1.308) -- (-0.554,-1.258) -- (-0.576,-1.207) -- (-0.594,-1.156) -- (-0.610,-1.105) -- (-0.624,-1.055) -- (-0.638,-1.005) -- (-0.650,-0.955) -- (-0.663,-0.907) -- (-0.677,-0.859) -- (-0.692,-0.812) -- (-0.710,-0.766) -- (-0.730,-0.721) -- (-0.753,-0.675) -- (-0.781,-0.629) -- (-0.813,-0.582);
\draw (-0.813,-0.582) -- (-0.842,-0.540) -- (-0.869,-0.496) -- (-0.893,-0.450) -- (-0.915,-0.404) -- (-0.934,-0.356) -- (-0.951,-0.308) -- (-0.966,-0.259) -- (-0.978,-0.209) -- (-0.987,-0.158) -- (-0.994,-0.107) -- (-0.998,-0.056) -- (-1.000,-0.004) -- (-0.999,0.047) -- (-0.995,0.098) -- (-0.989,0.149);
\draw (-1.301,-0.932) -- (-1.347,-0.863) -- (-1.390,-0.793) -- (-1.429,-0.721) -- (-1.464,-0.646) -- (-1.495,-0.570) -- (-1.522,-0.493) -- (-1.546,-0.414) -- (-1.565,-0.334) -- (-1.580,-0.253) -- (-1.591,-0.171) -- (-1.598,-0.089) -- (-1.600,-0.007) -- (-1.598,0.075) -- (-1.592,0.157) -- (-1.582,0.239);
\draw (-1.789,-1.281) -- (-1.852,-1.187) -- (-1.911,-1.090) -- (-1.964,-0.991) -- (-2.013,-0.889) -- (-2.056,-0.784) -- (-2.093,-0.677) -- (-2.125,-0.569) -- (-2.152,-0.459) -- (-2.172,-0.348) -- (-2.187,-0.235) -- (-2.197,-0.123) -- (-2.200,-0.010) -- (-2.198,0.103) -- (-2.189,0.216) -- (-2.175,0.329);
\draw (-0.989,0.149) -- (-0.981,0.192) -- (-0.972,0.234) -- (-0.961,0.276) -- (-0.948,0.318) -- (-0.934,0.358) -- (-0.917,0.398) -- (-0.899,0.438) -- (-0.879,0.476) -- (-0.858,0.514) -- (-0.835,0.550) -- (-0.810,0.586) -- (-0.784,0.621) -- (-0.757,0.654) -- (-0.727,0.686) -- (-0.697,0.717);
\draw (-2.175,0.329) -- (-2.158,0.412) -- (-2.133,0.488) -- (-2.102,0.554) -- (-2.067,0.613) -- (-2.027,0.664) -- (-1.984,0.708) -- (-1.938,0.746) -- (-1.891,0.778) -- (-1.842,0.806) -- (-1.792,0.830) -- (-1.741,0.850) -- (-1.690,0.868) -- (-1.639,0.885) -- (-1.588,0.900) -- (-1.538,0.915) -- (-1.488,0.931) -- (-1.440,0.947) -- (-1.392,0.966) -- (-1.344,0.986) -- (-1.298,1.010) -- (-1.252,1.037) -- (-1.207,1.068) -- (-1.161,1.105) -- (-1.115,1.147);
\draw[white, line width=4.0pt, line cap=butt] (-1.583,0.509) -- (-1.588,0.539) -- (-1.595,0.569) -- (-1.603,0.599) -- (-1.611,0.629) -- (-1.620,0.658) -- (-1.630,0.688) -- (-1.640,0.718) -- (-1.650,0.747) -- (-1.660,0.777) -- (-1.670,0.807) -- (-1.680,0.838) -- (-1.690,0.868) -- (-1.699,0.899) -- (-1.708,0.931) -- (-1.715,0.963) -- (-1.722,0.995) -- (-1.727,1.028) -- (-1.731,1.061) -- (-1.734,1.095) -- (-1.735,1.130) -- (-1.734,1.165) -- (-1.730,1.201) -- (-1.725,1.237) -- (-1.717,1.274);
\draw (-1.582,0.239) -- (-1.575,0.301) -- (-1.571,0.359) -- (-1.572,0.414) -- (-1.577,0.467) -- (-1.584,0.519) -- (-1.595,0.569) -- (-1.608,0.619) -- (-1.623,0.668) -- (-1.640,0.718) -- (-1.657,0.767) -- (-1.674,0.817) -- (-1.690,0.868) -- (-1.705,0.920) -- (-1.718,0.974) -- (-1.727,1.028) -- (-1.733,1.084) -- (-1.735,1.142) -- (-1.730,1.201) -- (-1.720,1.261) -- (-1.701,1.324) -- (-1.675,1.387) -- (-1.638,1.450) -- (-1.592,1.514) -- (-1.534,1.577);
\draw (-0.697,0.717) -- (-0.657,0.754) -- (-0.614,0.789) -- (-0.570,0.822) -- (-0.524,0.852) -- (-0.476,0.880) -- (-0.427,0.904) -- (-0.376,0.927) -- (-0.325,0.946) -- (-0.272,0.962) -- (-0.219,0.976) -- (-0.165,0.986) -- (-0.110,0.994) -- (-0.055,0.998) -- (-0.000,1.000);
\draw (-1.115,1.147) -- (-1.050,1.207) -- (-0.982,1.263) -- (-0.911,1.315) -- (-0.838,1.363) -- (-0.761,1.407) -- (-0.683,1.447) -- (-0.602,1.482) -- (-0.519,1.513) -- (-0.435,1.540) -- (-0.350,1.561) -- (-0.263,1.578) -- (-0.176,1.590) -- (-0.088,1.598) -- (-0.000,1.600);
\draw (-1.534,1.577) -- (-1.444,1.659) -- (-1.351,1.736) -- (-1.253,1.808) -- (-1.152,1.874) -- (-1.047,1.935) -- (-0.939,1.990) -- (-0.828,2.038) -- (-0.714,2.081) -- (-0.598,2.117) -- (-0.481,2.147) -- (-0.362,2.170) -- (-0.242,2.187) -- (-0.121,2.197) -- (-0.000,2.200);
\end{tikzpicture}

%% file: figures/tangle-even.tex
\begin{tikzpicture}[line width=1.00pt, line cap=round, every node/.style={font=\small}]
\draw (0.000,0.000) -- (1.500,0.000);
\draw (0.000,0.600) -- (1.500,0.600);
\draw (0.000,1.200) -- (1.500,1.200);
\draw (0.000,1.800) -- (1.500,1.800);
\draw[fill=white] (0.150,0.400) rectangle (1.350,1.400);
\node at (0.750,0.900) {$a_1$};
\draw (1.500,0.000) -- (3.000,0.000);
\draw (1.500,0.600) -- (3.000,0.600);
\draw (1.500,1.200) -- (3.000,1.200);
\draw (1.500,1.800) -- (3.000,1.800);
\draw[fill=white] (1.650,1.000) rectangle (2.850,2.000);
\node at (2.250,1.500) {$a_2$};
\draw (3.000,0.000) -- (4.100,0.000);
\draw (3.000,0.600) -- (3.220,0.600);
\draw (3.880,0.600) -- (4.100,0.600);
\node at (3.550,0.600) {$\cdots$};
\draw (3.000,1.200) -- (3.220,1.200);
\draw (3.880,1.200) -- (4.100,1.200);
\node at (3.550,1.200) {$\cdots$};
\draw (3.000,1.800) -- (3.220,1.800);
\draw (3.880,1.800) -- (4.100,1.800);
\node at (3.550,1.800) {$\cdots$};
\draw (4.100,0.000) -- (5.600,0.000);
\draw (4.100,0.600) -- (5.600,0.600);
\draw (4.100,1.200) -- (5.600,1.200);
\draw (4.100,1.800) -- (5.600,1.800);
\draw[fill=white] (4.250,1.000) rectangle (5.450,2.000);
\node at (4.850,1.500) {$a_\ell$};
\draw (5.600,1.200) .. controls (6.000,1.200) and (6.000,0.600) .. (5.600,0.600);
\draw (5.600,1.800) .. controls (6.800,1.800) and (6.800,0.000) .. (5.600,0.000);
\draw (-0.500,0.000) -- (0.000,0.000);
\draw (-0.500,0.600) -- (0.000,0.600);
\draw (-0.500,1.200) -- (0.000,1.200);
\draw (-0.500,1.800) -- (0.000,1.800);
\draw[densely dashed, thin, gray, rounded corners=8pt] (-0.5,-0.350) rectangle (6.850,2.150);
\end{tikzpicture}

%% file: figures/tangle-odd.tex
\begin{tikzpicture}[line width=1.00pt, line cap=round, every node/.style={font=\small}]
\draw (0.000,0.000) -- (1.500,0.000);
\draw (0.000,0.600) -- (1.500,0.600);
\draw (0.000,1.200) -- (1.500,1.200);
\draw (0.000,1.800) -- (1.500,1.800);
\draw[fill=white] (0.150,0.400) rectangle (1.350,1.400);
\node at (0.750,0.900) {$a_1$};
\draw (1.500,0.000) -- (3.000,0.000);
\draw (1.500,0.600) -- (3.000,0.600);
\draw (1.500,1.200) -- (3.000,1.200);
\draw (1.500,1.800) -- (3.000,1.800);
\draw[fill=white] (1.650,1.000) rectangle (2.850,2.000);
\node at (2.250,1.500) {$a_2$};
\draw (3.000,0.000) -- (4.100,0.000);
\draw (3.000,0.600) -- (3.220,0.600);
\draw (3.880,0.600) -- (4.100,0.600);
\node at (3.550,0.600) {$\cdots$};
\draw (3.000,1.200) -- (3.220,1.200);
\draw (3.880,1.200) -- (4.100,1.200);
\node at (3.550,1.200) {$\cdots$};
\draw (3.000,1.800) -- (3.220,1.800);
\draw (3.880,1.800) -- (4.100,1.800);
\node at (3.550,1.800) {$\cdots$};
\draw (4.100,0.000) -- (5.600,0.000);
\draw (4.100,0.600) -- (5.600,0.600);
\draw (4.100,1.200) -- (5.600,1.200);
\draw (4.100,1.800) -- (5.600,1.800);
\draw[fill=white] (4.250,0.400) rectangle (5.450,1.400);
\node at (4.850,0.900) {$a_\ell$};
\draw (5.600,1.800) .. controls (6.000,1.800) and (6.000,1.200) .. (5.600,1.200);
\draw (5.600,0.600) .. controls (6.000,0.600) and (6.000,0.000) .. (5.600,0.000);
\draw (-0.500,0.000) -- (0.000,0.000);
\draw (-0.500,0.600) -- (0.000,0.600);
\draw (-0.500,1.200) -- (0.000,1.200);
\draw (-0.500,1.800) -- (0.000,1.800);
\draw[densely dashed, thin, gray, rounded corners=8pt] (-0.5,-0.350) rectangle (6.250,2.150);
\end{tikzpicture}

%% file: figures/twobridge-even.tex
\begin{tikzpicture}[line width=1.00pt, line cap=round, every node/.style={font=\small}]
\draw (0.000,0.000) -- (1.500,0.000);
\draw (0.000,0.600) -- (1.500,0.600);
\draw (0.000,1.200) -- (1.500,1.200);
\draw (0.000,1.800) -- (1.500,1.800);
\draw[fill=white] (0.150,0.400) rectangle (1.350,1.400);
\node at (0.750,0.900) {$a_1$};
\draw (1.500,0.000) -- (3.000,0.000);
\draw (1.500,0.600) -- (3.000,0.600);
\draw (1.500,1.200) -- (3.000,1.200);
\draw (1.500,1.800) -- (3.000,1.800);
\draw[fill=white] (1.650,1.000) rectangle (2.850,2.000);
\node at (2.250,1.500) {$a_2$};
\draw (3.000,0.000) -- (4.100,0.000);
\draw (3.000,0.600) -- (3.220,0.600);
\draw (3.880,0.600) -- (4.100,0.600);
\node at (3.550,0.600) {$\cdots$};
\draw (3.000,1.200) -- (3.220,1.200);
\draw (3.880,1.200) -- (4.100,1.200);
\node at (3.550,1.200) {$\cdots$};
\draw (3.000,1.800) -- (3.220,1.800);
\draw (3.880,1.800) -- (4.100,1.800);
\node at (3.550,1.800) {$\cdots$};
\draw (4.100,0.000) -- (5.600,0.000);
\draw (4.100,0.600) -- (5.600,0.600);
\draw (4.100,1.200) -- (5.600,1.200);
\draw (4.100,1.800) -- (5.600,1.800);
\draw[fill=white] (4.250,1.000) rectangle (5.450,2.000);
\node at (4.850,1.500) {$a_\ell$};
\draw (5.600,1.200) .. controls (6.000,1.200) and (6.000,0.600) .. (5.600,0.600);
\draw (5.600,1.800) .. controls (6.800,1.800) and (6.800,0.000) .. (5.600,0.000);
\draw (0.000,1.800) .. controls (-0.400,1.800) and (-0.400,1.200) .. (0.000,1.200);
\draw (0.000,0.600) .. controls (-0.400,0.600) and (-0.400,0.000) .. (0.000,0.000);
\end{tikzpicture}

%% file: figures/twobridge-odd.tex
\begin{tikzpicture}[line width=1.00pt, line cap=round, every node/.style={font=\small}]
\draw (0.000,0.000) -- (1.500,0.000);
\draw (0.000,0.600) -- (1.500,0.600);
\draw (0.000,1.200) -- (1.500,1.200);
\draw (0.000,1.800) -- (1.500,1.800);
\draw[fill=white] (0.150,0.400) rectangle (1.350,1.400);
\node at (0.750,0.900) {$a_1$};
\draw (1.500,0.000) -- (3.000,0.000);
\draw (1.500,0.600) -- (3.000,0.600);
\draw (1.500,1.200) -- (3.000,1.200);
\draw (1.500,1.800) -- (3.000,1.800);
\draw[fill=white] (1.650,1.000) rectangle (2.850,2.000);
\node at (2.250,1.500) {$a_2$};
\draw (3.000,0.000) -- (4.100,0.000);
\draw (3.000,0.600) -- (3.220,0.600);
\draw (3.880,0.600) -- (4.100,0.600);
\node at (3.550,0.600) {$\cdots$};
\draw (3.000,1.200) -- (3.220,1.200);
\draw (3.880,1.200) -- (4.100,1.200);
\node at (3.550,1.200) {$\cdots$};
\draw (3.000,1.800) -- (3.220,1.800);
\draw (3.880,1.800) -- (4.100,1.800);
\node at (3.550,1.800) {$\cdots$};
\draw (4.100,0.000) -- (5.600,0.000);
\draw (4.100,0.600) -- (5.600,0.600);
\draw (4.100,1.200) -- (5.600,1.200);
\draw (4.100,1.800) -- (5.600,1.800);
\draw[fill=white] (4.250,0.400) rectangle (5.450,1.400);
\node at (4.850,0.900) {$a_\ell$};
\draw (5.600,1.800) .. controls (6.000,1.800) and (6.000,1.200) .. (5.600,1.200);
\draw (5.600,0.600) .. controls (6.000,0.600) and (6.000,0.000) .. (5.600,0.000);
\draw (0.000,1.800) .. controls (-0.400,1.800) and (-0.400,1.200) .. (0.000,1.200);
\draw (0.000,0.600) .. controls (-0.400,0.600) and (-0.400,0.000) .. (0.000,0.000);
\end{tikzpicture}

%% file: figures/montesinos-general.tex
\begin{tikzpicture}[line width=1.00pt, line cap=round, every node/.style={font=\small}]
\draw (0.000,-0.960) -- (0.200,-0.960);
\draw (0.000,-0.640) -- (0.200,-0.640);
\draw (0.000,-0.320) -- (0.200,-0.320);
\draw (0.000,0.000) -- (0.200,0.000);
\draw[fill=white] (0.200,-1.140) rectangle (1.500,0.180);
\node at (0.850,-0.480) {$\alpha_1/\beta_1$};
\draw[rounded corners=21.91pt] (0.000,0.000) -- (-1.790,0.000) -- (-1.790,-0.820);
\draw[rounded corners=12.80pt] (0.000,-0.320) -- (-1.470,-0.320) -- (-1.470,-0.820);
\draw[rounded corners=17.64pt] (0.000,-0.640) -- (-0.620,-0.640) -- (-0.620,-1.260);
\draw[rounded corners=8.54pt] (0.000,-0.960) -- (-0.300,-0.960) -- (-0.300,-1.260);
\draw (0.000,-2.520) -- (0.200,-2.520);
\draw (0.000,-2.200) -- (0.200,-2.200);
\draw (0.000,-1.880) -- (0.200,-1.880);
\draw (0.000,-1.560) -- (0.200,-1.560);
\draw[fill=white] (0.200,-2.700) rectangle (1.500,-1.380);
\node at (0.850,-2.040) {$\alpha_2/\beta_2$};
\draw[rounded corners=8.54pt] (0.000,-1.560) -- (-0.300,-1.560) -- (-0.300,-1.260);
\draw[rounded corners=17.64pt] (0.000,-1.880) -- (-0.620,-1.880) -- (-0.620,-1.260);
\draw[rounded corners=17.64pt] (0.000,-2.200) -- (-0.620,-2.200) -- (-0.620,-2.820);
\draw[rounded corners=8.54pt] (0.000,-2.520) -- (-0.300,-2.520) -- (-0.300,-2.820);
\draw (0.000,-5.080) -- (0.200,-5.080);
\draw (0.000,-4.760) -- (0.200,-4.760);
\draw (0.000,-4.440) -- (0.200,-4.440);
\draw (0.000,-4.120) -- (0.200,-4.120);
\draw[fill=white] (0.200,-5.260) rectangle (1.500,-3.940);
\node at (0.850,-4.600) {$\alpha_k/\beta_k$};
\draw[rounded corners=8.54pt] (0.000,-4.120) -- (-0.300,-4.120) -- (-0.300,-3.820);
\draw[rounded corners=17.64pt] (0.000,-4.440) -- (-0.620,-4.440) -- (-0.620,-3.820);
\draw[rounded corners=21.91pt] (0.000,-5.080) -- (-1.790,-5.080) -- (-1.790,-4.260);
\draw[rounded corners=12.80pt] (0.000,-4.760) -- (-1.470,-4.760) -- (-1.470,-4.260);
\draw (-0.620,-2.820) -- (-0.620,-3.070);
\draw (-0.620,-3.570) -- (-0.620,-3.820);
\draw (-0.300,-2.820) -- (-0.300,-3.070);
\draw (-0.300,-3.570) -- (-0.300,-3.820);
\fill (-0.460,-3.460) circle (0.03);
\fill (-0.460,-3.320) circle (0.03);
\fill (-0.460,-3.180) circle (0.03);
\draw (-1.470,-0.820) -- (-1.470,-2.090);
\draw (-1.470,-2.990) -- (-1.470,-4.260);
\draw (-1.790,-0.820) -- (-1.790,-2.090);
\draw (-1.790,-2.990) -- (-1.790,-4.260);
\draw (-1.470,-2.090) -- (-1.470,-2.990);
\draw (-1.790,-2.090) -- (-1.790,-2.990);
\draw[fill=white] (-2.040,-2.990) rectangle (-1.220,-2.090);
\node at (-1.630,-2.540) {$e$};
\end{tikzpicture}

%% file: figures/Jq.tex
\begin{tikzpicture}[line width=1.00pt, line cap=round, every node/.style={font=\small}]
\draw (0.000,-1.800) -- (0.600,-1.800);
\draw (0.000,0.000) -- (0.600,0.000);
\draw (0.000,-1.200) .. controls (0.300,-1.200) and (0.300,-0.600) .. (0.600,-0.600);
\draw[white, line width=4.0pt, line cap=butt] (0.149,-0.662) .. controls (0.271,-0.778) and (0.329,-1.022) .. (0.451,-1.138);
\draw (0.000,-0.600) .. controls (0.300,-0.600) and (0.300,-1.200) .. (0.600,-1.200);
\draw (0.600,-1.800) -- (1.200,-1.800);
\draw (0.600,0.000) -- (1.200,0.000);
\draw (0.600,-1.200) .. controls (0.900,-1.200) and (0.900,-0.600) .. (1.200,-0.600);
\draw[white, line width=4.0pt, line cap=butt] (0.749,-0.662) .. controls (0.871,-0.778) and (0.929,-1.022) .. (1.051,-1.138);
\draw (0.600,-0.600) .. controls (0.900,-0.600) and (0.900,-1.200) .. (1.200,-1.200);
\draw (1.200,-1.800) -- (2.700,-1.800);
\draw (1.200,-1.200) -- (2.700,-1.200);
\draw (1.200,-0.600) -- (2.700,-0.600);
\draw (1.200,0.000) -- (2.700,0.000);
\draw[fill=white] (1.350,-0.800) rectangle (2.550,0.200);
\node at (1.950,-0.300) {$q+2$};
\draw (2.700,-0.600) .. controls (3.100,-0.600) and (3.100,-1.200) .. (2.700,-1.200);
\draw (2.700,0.000) .. controls (3.900,0.000) and (3.900,-1.800) .. (2.700,-1.800);
\draw[rounded corners=29.88pt] (0.000,0.000) -- (-2.150,0.000) -- (-2.150,-1.100);
\draw[rounded corners=12.80pt] (0.000,-0.600) -- (-1.550,-0.600) -- (-1.550,-1.100);
\draw[rounded corners=28.45pt] (0.000,-1.200) -- (-1.000,-1.200) -- (-1.000,-2.200);
\draw[rounded corners=11.38pt] (0.000,-1.800) -- (-0.400,-1.800) -- (-0.400,-2.200);
\draw (0.000,-4.400) -- (0.600,-4.400);
\draw (0.000,-2.600) -- (0.600,-2.600);
\draw (0.000,-3.800) .. controls (0.300,-3.800) and (0.300,-3.200) .. (0.600,-3.200);
\draw[white, line width=4.0pt, line cap=butt] (0.149,-3.262) .. controls (0.271,-3.378) and (0.329,-3.622) .. (0.451,-3.738);
\draw (0.000,-3.200) .. controls (0.300,-3.200) and (0.300,-3.800) .. (0.600,-3.800);
\draw (0.600,-4.400) -- (1.200,-4.400);
\draw (0.600,-2.600) -- (1.200,-2.600);
\draw (0.600,-3.800) .. controls (0.900,-3.800) and (0.900,-3.200) .. (1.200,-3.200);
\draw[white, line width=4.0pt, line cap=butt] (0.749,-3.262) .. controls (0.871,-3.378) and (0.929,-3.622) .. (1.051,-3.738);
\draw (0.600,-3.200) .. controls (0.900,-3.200) and (0.900,-3.800) .. (1.200,-3.800);
\draw (1.200,-4.400) -- (1.800,-4.400);
\draw (1.200,-2.600) -- (1.800,-2.600);
\draw (1.200,-3.800) .. controls (1.500,-3.800) and (1.500,-3.200) .. (1.800,-3.200);
\draw[white, line width=4.0pt, line cap=butt] (1.349,-3.262) .. controls (1.471,-3.378) and (1.529,-3.622) .. (1.651,-3.738);
\draw (1.200,-3.200) .. controls (1.500,-3.200) and (1.500,-3.800) .. (1.800,-3.800);
\draw (1.800,-4.400) -- (2.400,-4.400);
\draw (1.800,-3.800) -- (2.400,-3.800);
\draw (1.800,-3.200) .. controls (2.100,-3.200) and (2.100,-2.600) .. (2.400,-2.600);
\draw[white, line width=4.0pt, line cap=butt] (1.949,-2.662) .. controls (2.071,-2.778) and (2.129,-3.022) .. (2.251,-3.138);
\draw (1.800,-2.600) .. controls (2.100,-2.600) and (2.100,-3.200) .. (2.400,-3.200);
\draw (2.400,-4.400) -- (3.000,-4.400);
\draw (2.400,-3.800) -- (3.000,-3.800);
\draw (2.400,-3.200) .. controls (2.700,-3.200) and (2.700,-2.600) .. (3.000,-2.600);
\draw[white, line width=4.0pt, line cap=butt] (2.549,-2.662) .. controls (2.671,-2.778) and (2.729,-3.022) .. (2.851,-3.138);
\draw (2.400,-2.600) .. controls (2.700,-2.600) and (2.700,-3.200) .. (3.000,-3.200);
\draw (3.000,-4.400) -- (3.600,-4.400);
\draw (3.000,-3.800) -- (3.600,-3.800);
\draw (3.000,-3.200) .. controls (3.300,-3.200) and (3.300,-2.600) .. (3.600,-2.600);
\draw[white, line width=4.0pt, line cap=butt] (3.149,-2.662) .. controls (3.271,-2.778) and (3.329,-3.022) .. (3.451,-3.138);
\draw (3.000,-2.600) .. controls (3.300,-2.600) and (3.300,-3.200) .. (3.600,-3.200);
\draw (3.600,-3.200) .. controls (4.000,-3.200) and (4.000,-3.800) .. (3.600,-3.800);
\draw (3.600,-2.600) .. controls (4.800,-2.600) and (4.800,-4.400) .. (3.600,-4.400);
\draw[rounded corners=11.38pt] (0.000,-2.600) -- (-0.400,-2.600) -- (-0.400,-2.200);
\draw[rounded corners=28.45pt] (0.000,-3.200) -- (-1.000,-3.200) -- (-1.000,-2.200);
\draw[rounded corners=28.45pt] (0.000,-3.800) -- (-1.000,-3.800) -- (-1.000,-4.800);
\draw[rounded corners=11.38pt] (0.000,-4.400) -- (-0.400,-4.400) -- (-0.400,-4.800);
\draw (0.000,-7.000) -- (0.600,-7.000);
\draw (0.000,-5.200) -- (0.600,-5.200);
\draw (0.000,-6.400) .. controls (0.300,-6.400) and (0.300,-5.800) .. (0.600,-5.800);
\draw[white, line width=4.0pt, line cap=butt] (0.149,-5.862) .. controls (0.271,-5.978) and (0.329,-6.222) .. (0.451,-6.338);
\draw (0.000,-5.800) .. controls (0.300,-5.800) and (0.300,-6.400) .. (0.600,-6.400);
\draw (0.600,-7.000) -- (1.200,-7.000);
\draw (0.600,-5.200) -- (1.200,-5.200);
\draw (0.600,-6.400) .. controls (0.900,-6.400) and (0.900,-5.800) .. (1.200,-5.800);
\draw[white, line width=4.0pt, line cap=butt] (0.749,-5.862) .. controls (0.871,-5.978) and (0.929,-6.222) .. (1.051,-6.338);
\draw (0.600,-5.800) .. controls (0.900,-5.800) and (0.900,-6.400) .. (1.200,-6.400);
\draw (1.200,-7.000) -- (1.800,-7.000);
\draw (1.200,-6.400) -- (1.800,-6.400);
\draw (1.200,-5.800) .. controls (1.500,-5.800) and (1.500,-5.200) .. (1.800,-5.200);
\draw[white, line width=4.0pt, line cap=butt] (1.349,-5.262) .. controls (1.471,-5.378) and (1.529,-5.622) .. (1.651,-5.738);
\draw (1.200,-5.200) .. controls (1.500,-5.200) and (1.500,-5.800) .. (1.800,-5.800);
\draw (1.800,-7.000) -- (2.400,-7.000);
\draw (1.800,-6.400) -- (2.400,-6.400);
\draw (1.800,-5.800) .. controls (2.100,-5.800) and (2.100,-5.200) .. (2.400,-5.200);
\draw[white, line width=4.0pt, line cap=butt] (1.949,-5.262) .. controls (2.071,-5.378) and (2.129,-5.622) .. (2.251,-5.738);
\draw (1.800,-5.200) .. controls (2.100,-5.200) and (2.100,-5.800) .. (2.400,-5.800);
\draw (2.400,-5.800) .. controls (2.800,-5.800) and (2.800,-6.400) .. (2.400,-6.400);
\draw (2.400,-5.200) .. controls (3.600,-5.200) and (3.600,-7.000) .. (2.400,-7.000);
\draw[rounded corners=11.38pt] (0.000,-5.200) -- (-0.400,-5.200) -- (-0.400,-4.800);
\draw[rounded corners=28.45pt] (0.000,-5.800) -- (-1.000,-5.800) -- (-1.000,-4.800);
\draw[rounded corners=29.88pt] (0.000,-7.000) -- (-2.150,-7.000) -- (-2.150,-5.900);
\draw[rounded corners=12.80pt] (0.000,-6.400) -- (-1.550,-6.400) -- (-1.550,-5.900);
\draw (-1.550,-1.100) -- (-1.550,-3.200);
\draw (-1.550,-3.800) -- (-1.550,-5.900);
\draw (-2.150,-1.100) -- (-2.150,-3.200);
\draw (-2.150,-3.800) -- (-2.150,-5.900);
\draw (-2.150,-3.200) .. controls (-2.150,-3.500) and (-1.550,-3.500) .. (-1.550,-3.800);
\draw[white, line width=4.0pt, line cap=butt] (-1.612,-3.349) .. controls (-1.728,-3.471) and (-1.972,-3.529) .. (-2.088,-3.651);
\draw (-1.550,-3.200) .. controls (-1.550,-3.500) and (-2.150,-3.500) .. (-2.150,-3.800);
\end{tikzpicture}

%% file: figures/gl-eta-type.tex
\begin{tikzpicture}[strand/.style={line width=1.1pt, line cap=round}, shaded/.style={fill=black!22},
                    rim/.style={black!35, line width=0.4pt}, every node/.style={font=\small}]
\node at (0.00,1.45) {Type I};
\node at (2.60,1.45) {Type II};
\node[anchor=east, overlay] at (-1.25,0.00) {$\eta(c)=+1$};
\begin{scope}[shift={(0.00,0.00)}]
\fill[shaded] (0,0) -- (45:0.95) arc (45:135:0.95) -- cycle;
\fill[shaded] (0,0) -- (225:0.95) arc (225:315:0.95) -- cycle;
\draw[rim] (0,0) circle (0.95);
\draw[strand] (-0.672,0.672) -- (-0.113,0.113);
\draw[strand, -{Stealth[length=5pt]}] (0.113,-0.113) -- (0.672,-0.672);
\draw[strand, -{Stealth[length=5pt]}] (-0.672,-0.672) -- (0.672,0.672);
\end{scope}
\begin{scope}[shift={(2.60,0.00)}]
\fill[shaded] (0,0) -- (45:0.95) arc (45:135:0.95) -- cycle;
\fill[shaded] (0,0) -- (225:0.95) arc (225:315:0.95) -- cycle;
\draw[rim] (0,0) circle (0.95);
\draw[strand] (0.672,-0.672) -- (0.113,-0.113);
\draw[strand, -{Stealth[length=5pt]}] (-0.113,0.113) -- (-0.672,0.672);
\draw[strand, -{Stealth[length=5pt]}] (-0.672,-0.672) -- (0.672,0.672);
\end{scope}
\node[anchor=east, overlay] at (-1.25,-2.50) {$\eta(c)=-1$};
\begin{scope}[shift={(0.00,-2.50)}]
\fill[shaded] (0,0) -- (45:0.95) arc (45:135:0.95) -- cycle;
\fill[shaded] (0,0) -- (225:0.95) arc (225:315:0.95) -- cycle;
\draw[rim] (0,0) circle (0.95);
\draw[strand] (-0.672,-0.672) -- (-0.113,-0.113);
\draw[strand, -{Stealth[length=5pt]}] (0.113,0.113) -- (0.672,0.672);
\draw[strand, -{Stealth[length=5pt]}] (-0.672,0.672) -- (0.672,-0.672);
\end{scope}
\begin{scope}[shift={(2.60,-2.50)}]
\fill[shaded] (0,0) -- (45:0.95) arc (45:135:0.95) -- cycle;
\fill[shaded] (0,0) -- (225:0.95) arc (225:315:0.95) -- cycle;
\draw[rim] (0,0) circle (0.95);
\draw[strand] (-0.672,-0.672) -- (-0.113,-0.113);
\draw[strand, -{Stealth[length=5pt]}] (0.113,0.113) -- (0.672,0.672);
\draw[strand, -{Stealth[length=5pt]}] (0.672,-0.672) -- (-0.672,0.672);
\end{scope}
\end{tikzpicture}

%% file: figures/Th-3q-checkerboard.tex
\begin{tikzpicture}[knot/.style={line width=1.0pt, line cap=round, line join=round}, knotdot/.style={},
                    shaded/.style={fill=black!22}, whitereg/.style={fill=white}, reglabel/.style={font=\small, inner sep=1pt}]
\fill[shaded, rounded corners=8pt] (-2.750,-2.734) rectangle (2.746,2.750);
\fill[whitereg] (0.000,2.400) -- (0.151,2.395) -- (0.302,2.381) -- (0.451,2.357) -- (0.599,2.324) -- (0.744,2.282) -- (0.887,2.230) -- (1.025,2.170) -- (1.160,2.101) -- (1.290,2.024) -- (1.415,1.938) -- (1.534,1.845) -- (1.648,1.745) -- (1.754,1.638) -- (1.854,1.524) -- (1.946,1.404) -- (1.981,1.353) -- (2.011,1.303) -- (2.037,1.253) -- (2.058,1.205) -- (2.076,1.157) -- (2.090,1.111) -- (2.101,1.065) -- (2.109,1.021) -- (2.114,0.978) -- (2.116,0.937) -- (2.116,0.896) -- (2.114,0.857) -- (2.110,0.819) -- (2.105,0.782) -- (2.097,0.745) -- (2.089,0.710) -- (2.079,0.675) -- (2.069,0.642) -- (2.058,0.608) -- (2.046,0.576) -- (2.018,0.597) -- (1.990,0.617) -- (1.963,0.638) -- (1.935,0.658) -- (1.907,0.678) -- (1.880,0.698) -- (1.852,0.719) -- (1.825,0.740) -- (1.797,0.761) -- (1.770,0.784) -- (1.743,0.807) -- (1.716,0.831) -- (1.690,0.857) -- (1.663,0.884) -- (1.637,0.912) -- (1.610,0.942) -- (1.583,0.974) -- (1.556,1.008) -- (1.529,1.044) -- (1.500,1.082) -- (1.461,1.133) -- (1.421,1.178) -- (1.380,1.219) -- (1.338,1.255) -- (1.296,1.287) -- (1.254,1.314) -- (1.212,1.338) -- (1.170,1.358) -- (1.129,1.375) -- (1.089,1.388) -- (1.049,1.399) -- (1.009,1.407) -- (0.971,1.412) -- (0.933,1.415) -- (0.895,1.417) -- (0.858,1.417) -- (0.822,1.415) -- (0.787,1.412) -- (0.752,1.407) -- (0.717,1.402) -- (0.702,1.433) -- (0.685,1.464) -- (0.666,1.495) -- (0.647,1.525) -- (0.625,1.555) -- (0.602,1.584) -- (0.577,1.613) -- (0.550,1.641) -- (0.521,1.669) -- (0.489,1.695) -- (0.454,1.720) -- (0.417,1.744) -- (0.376,1.766) -- (0.333,1.786) -- (0.286,1.804) -- (0.235,1.819) -- (0.182,1.832) -- (0.125,1.842) -- (0.064,1.848) -- (0.000,1.850) -- (-0.047,1.850) -- (-0.093,1.852) -- (-0.136,1.854) -- (-0.178,1.857) -- (-0.218,1.861) -- (-0.256,1.866) -- (-0.294,1.872) -- (-0.330,1.878) -- (-0.365,1.886) -- (-0.400,1.894) -- (-0.434,1.903) -- (-0.468,1.912) -- (-0.501,1.922) -- (-0.533,1.933) -- (-0.566,1.943) -- (-0.599,1.954) -- (-0.631,1.965) -- (-0.664,1.975) -- (-0.697,1.986) -- (-0.730,1.996) -- (-0.710,2.025) -- (-0.690,2.053) -- (-0.669,2.082) -- (-0.646,2.110) -- (-0.623,2.137) -- (-0.597,2.164) -- (-0.570,2.190) -- (-0.542,2.216) -- (-0.511,2.241) -- (-0.478,2.264) -- (-0.443,2.286) -- (-0.405,2.307) -- (-0.365,2.327) -- (-0.322,2.345) -- (-0.276,2.360) -- (-0.227,2.374) -- (-0.175,2.385) -- (-0.120,2.393) -- (-0.062,2.398) -- cycle;
\fill[whitereg] (0.000,1.300) -- (0.045,1.300) -- (0.088,1.301) -- (0.129,1.302) -- (0.169,1.305) -- (0.207,1.307) -- (0.244,1.311) -- (0.280,1.315) -- (0.315,1.320) -- (0.350,1.326) -- (0.384,1.332) -- (0.418,1.338) -- (0.451,1.345) -- (0.484,1.352) -- (0.517,1.360) -- (0.550,1.368) -- (0.583,1.375) -- (0.616,1.382) -- (0.650,1.390) -- (0.683,1.396) -- (0.717,1.402) -- (0.733,1.371) -- (0.747,1.340) -- (0.761,1.309) -- (0.774,1.277) -- (0.787,1.246) -- (0.800,1.215) -- (0.814,1.184) -- (0.827,1.153) -- (0.841,1.122) -- (0.855,1.091) -- (0.870,1.059) -- (0.886,1.028) -- (0.903,0.997) -- (0.920,0.965) -- (0.939,0.933) -- (0.959,0.900) -- (0.981,0.867) -- (1.004,0.832) -- (1.028,0.797) -- (1.054,0.761) -- (1.089,0.710) -- (1.121,0.659) -- (1.150,0.606) -- (1.177,0.551) -- (1.202,0.496) -- (1.224,0.439) -- (1.243,0.381) -- (1.259,0.323) -- (1.273,0.264) -- (1.284,0.204) -- (1.292,0.144) -- (1.297,0.083) -- (1.300,0.022) -- (1.299,-0.038) -- (1.296,-0.099) -- (1.293,-0.144) -- (1.291,-0.187) -- (1.289,-0.228) -- (1.288,-0.268) -- (1.288,-0.306) -- (1.288,-0.343) -- (1.290,-0.380) -- (1.292,-0.415) -- (1.295,-0.450) -- (1.299,-0.484) -- (1.303,-0.518) -- (1.307,-0.552) -- (1.312,-0.586) -- (1.317,-0.619) -- (1.322,-0.653) -- (1.327,-0.686) -- (1.332,-0.720) -- (1.336,-0.754) -- (1.340,-0.788) -- (1.343,-0.822) -- (1.311,-0.835) -- (1.279,-0.847) -- (1.247,-0.858) -- (1.215,-0.869) -- (1.183,-0.880) -- (1.151,-0.891) -- (1.118,-0.901) -- (1.086,-0.913) -- (1.054,-0.924) -- (1.022,-0.936) -- (0.990,-0.948) -- (0.958,-0.962) -- (0.925,-0.976) -- (0.892,-0.991) -- (0.858,-1.007) -- (0.824,-1.025) -- (0.789,-1.044) -- (0.754,-1.064) -- (0.717,-1.086) -- (0.678,-1.109) -- (0.625,-1.140) -- (0.572,-1.168) -- (0.516,-1.193) -- (0.460,-1.216) -- (0.403,-1.236) -- (0.345,-1.254) -- (0.286,-1.268) -- (0.226,-1.280) -- (0.166,-1.289) -- (0.106,-1.296) -- (0.045,-1.299) -- (-0.016,-1.300) -- (-0.077,-1.298) -- (-0.137,-1.293) -- (-0.197,-1.285) -- (-0.266,-1.272) -- (-0.334,-1.256) -- (-0.401,-1.237) -- (-0.467,-1.213) -- (-0.531,-1.187) -- (-0.594,-1.156) -- (-0.655,-1.123) -- (-0.715,-1.086) -- (-0.772,-1.046) -- (-0.827,-1.003) -- (-0.879,-0.957) -- (-0.930,-0.909) -- (-0.977,-0.858) -- (-1.022,-0.804) -- (-1.063,-0.748) -- (-1.102,-0.690) -- (-1.137,-0.630) -- (-1.169,-0.568) -- (-1.198,-0.504) -- (-1.224,-0.439) -- (-1.245,-0.373) -- (-1.264,-0.306) -- (-1.278,-0.237) -- (-1.289,-0.168) -- (-1.296,-0.099) -- (-1.300,-0.054) -- (-1.304,-0.011) -- (-1.308,0.030) -- (-1.314,0.069) -- (-1.319,0.107) -- (-1.326,0.144) -- (-1.333,0.179) -- (-1.340,0.214) -- (-1.348,0.248) -- (-1.357,0.282) -- (-1.366,0.315) -- (-1.376,0.347) -- (-1.385,0.380) -- (-1.395,0.412) -- (-1.405,0.444) -- (-1.415,0.477) -- (-1.425,0.509) -- (-1.435,0.542) -- (-1.444,0.575) -- (-1.453,0.609) -- (-1.423,0.626) -- (-1.393,0.643) -- (-1.363,0.659) -- (-1.333,0.675) -- (-1.302,0.690) -- (-1.272,0.706) -- (-1.242,0.721) -- (-1.212,0.737) -- (-1.182,0.753) -- (-1.152,0.770) -- (-1.123,0.787) -- (-1.093,0.805) -- (-1.062,0.824) -- (-1.032,0.844) -- (-1.001,0.865) -- (-0.970,0.888) -- (-0.939,0.912) -- (-0.906,0.937) -- (-0.873,0.965) -- (-0.839,0.993) -- (-0.788,1.034) -- (-0.735,1.072) -- (-0.680,1.108) -- (-0.624,1.140) -- (-0.566,1.170) -- (-0.507,1.197) -- (-0.446,1.221) -- (-0.385,1.242) -- (-0.322,1.259) -- (-0.259,1.274) -- (-0.195,1.285) -- (-0.130,1.293) -- (-0.065,1.298) -- cycle;
\fill[whitereg] (1.845,-0.141) -- (1.849,-0.094) -- (1.853,-0.049) -- (1.859,-0.005) -- (1.865,0.036) -- (1.872,0.075) -- (1.880,0.113) -- (1.889,0.150) -- (1.898,0.186) -- (1.908,0.221) -- (1.919,0.255) -- (1.931,0.288) -- (1.943,0.321) -- (1.955,0.353) -- (1.968,0.385) -- (1.981,0.416) -- (1.994,0.448) -- (2.007,0.480) -- (2.020,0.511) -- (2.033,0.543) -- (2.046,0.576) -- (2.073,0.554) -- (2.100,0.532) -- (2.127,0.508) -- (2.153,0.484) -- (2.178,0.458) -- (2.203,0.431) -- (2.227,0.402) -- (2.251,0.371) -- (2.273,0.339) -- (2.294,0.304) -- (2.314,0.268) -- (2.332,0.228) -- (2.348,0.187) -- (2.362,0.142) -- (2.374,0.095) -- (2.384,0.046) -- (2.391,-0.007) -- (2.395,-0.063) -- (2.396,-0.121) -- (2.393,-0.183) -- (2.375,-0.344) -- (2.347,-0.503) -- (2.307,-0.661) -- (2.257,-0.815) -- (2.197,-0.966) -- (2.127,-1.112) -- (2.047,-1.253) -- (1.958,-1.388) -- (1.860,-1.517) -- (1.753,-1.639) -- (1.638,-1.754) -- (1.516,-1.860) -- (1.387,-1.959) -- (1.252,-2.048) -- (1.198,-2.078) -- (1.146,-2.104) -- (1.094,-2.126) -- (1.044,-2.144) -- (0.996,-2.158) -- (0.948,-2.168) -- (0.902,-2.176) -- (0.858,-2.180) -- (0.815,-2.182) -- (0.773,-2.181) -- (0.733,-2.178) -- (0.694,-2.173) -- (0.656,-2.166) -- (0.619,-2.158) -- (0.583,-2.148) -- (0.549,-2.137) -- (0.515,-2.125) -- (0.482,-2.112) -- (0.450,-2.098) -- (0.418,-2.083) -- (0.441,-2.058) -- (0.464,-2.032) -- (0.486,-2.005) -- (0.508,-1.979) -- (0.531,-1.953) -- (0.553,-1.927) -- (0.575,-1.901) -- (0.599,-1.876) -- (0.622,-1.850) -- (0.647,-1.825) -- (0.672,-1.800) -- (0.698,-1.775) -- (0.726,-1.750) -- (0.755,-1.726) -- (0.785,-1.701) -- (0.817,-1.677) -- (0.851,-1.653) -- (0.887,-1.628) -- (0.925,-1.604) -- (0.965,-1.578) -- (1.018,-1.543) -- (1.067,-1.507) -- (1.111,-1.469) -- (1.150,-1.430) -- (1.185,-1.390) -- (1.215,-1.350) -- (1.242,-1.310) -- (1.265,-1.270) -- (1.285,-1.231) -- (1.301,-1.191) -- (1.315,-1.152) -- (1.326,-1.113) -- (1.334,-1.075) -- (1.340,-1.038) -- (1.345,-1.000) -- (1.347,-0.964) -- (1.348,-0.928) -- (1.348,-0.892) -- (1.346,-0.857) -- (1.343,-0.822) -- (1.376,-0.809) -- (1.408,-0.794) -- (1.440,-0.778) -- (1.471,-0.761) -- (1.503,-0.742) -- (1.534,-0.721) -- (1.565,-0.699) -- (1.595,-0.674) -- (1.624,-0.646) -- (1.653,-0.617) -- (1.680,-0.584) -- (1.707,-0.548) -- (1.732,-0.510) -- (1.755,-0.468) -- (1.777,-0.422) -- (1.796,-0.373) -- (1.813,-0.321) -- (1.827,-0.264) -- (1.838,-0.204) -- cycle;
\fill[whitereg] (-0.281,-1.829) -- (-0.234,-1.836) -- (-0.189,-1.844) -- (-0.147,-1.853) -- (-0.106,-1.862) -- (-0.067,-1.872) -- (-0.030,-1.883) -- (0.006,-1.895) -- (0.041,-1.907) -- (0.075,-1.920) -- (0.108,-1.933) -- (0.140,-1.947) -- (0.172,-1.961) -- (0.203,-1.976) -- (0.234,-1.991) -- (0.264,-2.007) -- (0.295,-2.022) -- (0.326,-2.038) -- (0.356,-2.053) -- (0.387,-2.068) -- (0.418,-2.083) -- (0.395,-2.109) -- (0.370,-2.134) -- (0.345,-2.159) -- (0.318,-2.183) -- (0.291,-2.207) -- (0.262,-2.230) -- (0.231,-2.252) -- (0.199,-2.273) -- (0.165,-2.292) -- (0.129,-2.310) -- (0.091,-2.327) -- (0.050,-2.342) -- (0.007,-2.355) -- (-0.038,-2.366) -- (-0.086,-2.375) -- (-0.136,-2.381) -- (-0.189,-2.384) -- (-0.245,-2.383) -- (-0.303,-2.380) -- (-0.364,-2.372) -- (-0.491,-2.349) -- (-0.617,-2.319) -- (-0.740,-2.283) -- (-0.862,-2.240) -- (-0.980,-2.191) -- (-1.097,-2.135) -- (-1.210,-2.073) -- (-1.319,-2.005) -- (-1.425,-1.931) -- (-1.526,-1.852) -- (-1.624,-1.768) -- (-1.716,-1.678) -- (-1.760,-1.631) -- (-1.357,-1.257) -- (-1.323,-1.293) -- (-1.251,-1.362) -- (-1.177,-1.428) -- (-1.098,-1.489) -- (-1.017,-1.546) -- (-0.932,-1.598) -- (-0.845,-1.646) -- (-0.756,-1.689) -- (-0.664,-1.727) -- (-0.570,-1.760) -- (-0.475,-1.788) -- (-0.379,-1.811) -- cycle;
\fill[whitereg] (-1.772,-0.531) -- (-1.741,-0.625) -- (-1.705,-0.718) -- (-1.664,-0.808) -- (-1.618,-0.896) -- (-1.568,-0.982) -- (-1.513,-1.064) -- (-1.454,-1.144) -- (-1.390,-1.220) -- (-1.357,-1.257) -- (-1.760,-1.631) -- (-1.804,-1.583) -- (-1.886,-1.484) -- (-1.963,-1.381) -- (-2.034,-1.274) -- (-2.100,-1.163) -- (-2.159,-1.048) -- (-2.212,-0.931) -- (-2.259,-0.811) -- (-2.299,-0.688) -- (-2.333,-0.564) -- (-2.360,-0.438) -- (-2.380,-0.311) -- (-2.393,-0.183) -- (-2.400,-0.032) -- (-2.397,0.120) -- (-2.385,0.271) -- (-2.363,0.420) -- (-2.332,0.568) -- (-2.291,0.714) -- (-2.242,0.857) -- (-2.183,0.997) -- (-2.116,1.132) -- (-2.041,1.263) -- (-1.957,1.390) -- (-1.865,1.510) -- (-1.767,1.625) -- (-1.661,1.733) -- (-1.548,1.834) -- (-1.500,1.872) -- (-1.452,1.906) -- (-1.405,1.935) -- (-1.358,1.960) -- (-1.312,1.982) -- (-1.266,1.999) -- (-1.222,2.014) -- (-1.179,2.025) -- (-1.136,2.033) -- (-1.095,2.039) -- (-1.055,2.042) -- (-1.015,2.043) -- (-0.977,2.042) -- (-0.940,2.039) -- (-0.903,2.035) -- (-0.867,2.029) -- (-0.832,2.022) -- (-0.797,2.014) -- (-0.763,2.005) -- (-0.730,1.996) -- (-0.749,1.967) -- (-0.767,1.938) -- (-0.785,1.908) -- (-0.803,1.879) -- (-0.821,1.850) -- (-0.839,1.821) -- (-0.857,1.792) -- (-0.876,1.763) -- (-0.896,1.734) -- (-0.916,1.705) -- (-0.937,1.677) -- (-0.960,1.648) -- (-0.983,1.620) -- (-1.008,1.591) -- (-1.034,1.562) -- (-1.062,1.534) -- (-1.092,1.504) -- (-1.124,1.475) -- (-1.157,1.445) -- (-1.193,1.414) -- (-1.241,1.371) -- (-1.283,1.327) -- (-1.321,1.283) -- (-1.353,1.238) -- (-1.382,1.194) -- (-1.406,1.150) -- (-1.426,1.107) -- (-1.443,1.064) -- (-1.457,1.021) -- (-1.467,0.980) -- (-1.474,0.939) -- (-1.479,0.899) -- (-1.482,0.860) -- (-1.482,0.822) -- (-1.481,0.785) -- (-1.478,0.748) -- (-1.473,0.712) -- (-1.467,0.677) -- (-1.460,0.643) -- (-1.453,0.609) -- (-1.482,0.590) -- (-1.512,0.571) -- (-1.541,0.551) -- (-1.570,0.529) -- (-1.598,0.505) -- (-1.626,0.480) -- (-1.653,0.453) -- (-1.679,0.424) -- (-1.704,0.392) -- (-1.727,0.358) -- (-1.750,0.322) -- (-1.770,0.283) -- (-1.789,0.241) -- (-1.806,0.196) -- (-1.820,0.147) -- (-1.832,0.096) -- (-1.841,0.042) -- (-1.846,-0.016) -- (-1.847,-0.077) -- (-1.845,-0.141) -- (-1.834,-0.240) -- (-1.819,-0.338) -- (-1.798,-0.435) -- cycle;
\draw[knot] (0.000,2.400) -- (0.151,2.395) -- (0.302,2.381) -- (0.451,2.357) -- (0.599,2.324) -- (0.744,2.282) -- (0.887,2.230) -- (1.025,2.170) -- (1.160,2.101) -- (1.290,2.024) -- (1.415,1.938) -- (1.534,1.845) -- (1.648,1.745) -- (1.754,1.638) -- (1.854,1.524) -- (1.946,1.404);
\draw[knot] (0.000,1.850) -- (0.064,1.848) -- (0.125,1.842) -- (0.182,1.832) -- (0.235,1.819) -- (0.286,1.804) -- (0.333,1.786) -- (0.376,1.766) -- (0.417,1.744) -- (0.454,1.720) -- (0.489,1.695) -- (0.521,1.669) -- (0.550,1.641) -- (0.577,1.613) -- (0.602,1.584) -- (0.625,1.555) -- (0.647,1.525) -- (0.666,1.495) -- (0.685,1.464) -- (0.702,1.433) -- (0.717,1.402) -- (0.733,1.371) -- (0.747,1.340) -- (0.761,1.309) -- (0.774,1.277) -- (0.787,1.246) -- (0.800,1.215) -- (0.814,1.184) -- (0.827,1.153) -- (0.841,1.122) -- (0.855,1.091) -- (0.870,1.059) -- (0.886,1.028) -- (0.903,0.997) -- (0.920,0.965) -- (0.939,0.933) -- (0.959,0.900) -- (0.981,0.867) -- (1.004,0.832) -- (1.028,0.797) -- (1.054,0.761);
\draw[knot] (0.000,1.300) -- (0.045,1.300) -- (0.088,1.301) -- (0.129,1.302) -- (0.169,1.305) -- (0.207,1.307) -- (0.244,1.311) -- (0.280,1.315) -- (0.315,1.320) -- (0.350,1.326) -- (0.384,1.332) -- (0.418,1.338) -- (0.451,1.345) -- (0.484,1.352) -- (0.517,1.360) -- (0.550,1.368) -- (0.583,1.375);
\draw[knot] (0.858,1.417) -- (0.895,1.417) -- (0.933,1.415) -- (0.971,1.412) -- (1.009,1.407) -- (1.049,1.399) -- (1.089,1.388) -- (1.129,1.375) -- (1.170,1.358) -- (1.212,1.338) -- (1.254,1.314) -- (1.296,1.287) -- (1.338,1.255) -- (1.380,1.219) -- (1.421,1.178) -- (1.461,1.133) -- (1.500,1.082);
\draw[knot] (1.054,0.761) -- (1.089,0.710) -- (1.121,0.659) -- (1.150,0.606) -- (1.177,0.551) -- (1.202,0.496) -- (1.224,0.439) -- (1.243,0.381) -- (1.259,0.323) -- (1.273,0.264) -- (1.284,0.204) -- (1.292,0.144) -- (1.297,0.083) -- (1.300,0.022) -- (1.299,-0.038) -- (1.296,-0.099);
\draw[knot] (1.500,1.082) -- (1.529,1.044) -- (1.556,1.008) -- (1.583,0.974) -- (1.610,0.942) -- (1.637,0.912) -- (1.663,0.884) -- (1.690,0.857) -- (1.716,0.831) -- (1.743,0.807) -- (1.770,0.784) -- (1.797,0.761) -- (1.825,0.740) -- (1.852,0.719) -- (1.880,0.698) -- (1.907,0.678) -- (1.935,0.658) -- (1.963,0.638) -- (1.990,0.617) -- (2.018,0.597) -- (2.046,0.576) -- (2.073,0.554) -- (2.100,0.532) -- (2.127,0.508) -- (2.153,0.484) -- (2.178,0.458) -- (2.203,0.431) -- (2.227,0.402) -- (2.251,0.371) -- (2.273,0.339) -- (2.294,0.304) -- (2.314,0.268) -- (2.332,0.228) -- (2.348,0.187) -- (2.362,0.142) -- (2.374,0.095) -- (2.384,0.046) -- (2.391,-0.007) -- (2.395,-0.063) -- (2.396,-0.121) -- (2.393,-0.183);
\draw[knot] (1.946,1.404) -- (1.981,1.353) -- (2.011,1.303) -- (2.037,1.253) -- (2.058,1.205) -- (2.076,1.157) -- (2.090,1.111) -- (2.101,1.065) -- (2.109,1.021) -- (2.114,0.978) -- (2.116,0.937) -- (2.116,0.896) -- (2.114,0.857) -- (2.110,0.819) -- (2.105,0.782) -- (2.097,0.745) -- (2.089,0.710);
\draw[knot] (1.994,0.448) -- (1.981,0.416) -- (1.968,0.385) -- (1.955,0.353) -- (1.943,0.321) -- (1.931,0.288) -- (1.919,0.255) -- (1.908,0.221) -- (1.898,0.186) -- (1.889,0.150) -- (1.880,0.113) -- (1.872,0.075) -- (1.865,0.036) -- (1.859,-0.005) -- (1.853,-0.049) -- (1.849,-0.094) -- (1.845,-0.141);
\draw[knot] (2.393,-0.183) -- (2.375,-0.344) -- (2.347,-0.503) -- (2.307,-0.661) -- (2.257,-0.815) -- (2.197,-0.966) -- (2.127,-1.112) -- (2.047,-1.253) -- (1.958,-1.388) -- (1.860,-1.517) -- (1.753,-1.639) -- (1.638,-1.754) -- (1.516,-1.860) -- (1.387,-1.959) -- (1.252,-2.048);
\draw[knot] (1.845,-0.141) -- (1.838,-0.204) -- (1.827,-0.264) -- (1.813,-0.321) -- (1.796,-0.373) -- (1.777,-0.422) -- (1.755,-0.468) -- (1.732,-0.510) -- (1.707,-0.548) -- (1.680,-0.584) -- (1.653,-0.617) -- (1.624,-0.646) -- (1.595,-0.674) -- (1.565,-0.699) -- (1.534,-0.721) -- (1.503,-0.742) -- (1.471,-0.761) -- (1.440,-0.778) -- (1.408,-0.794) -- (1.376,-0.809) -- (1.343,-0.822) -- (1.311,-0.835) -- (1.279,-0.847) -- (1.247,-0.858) -- (1.215,-0.869) -- (1.183,-0.880) -- (1.151,-0.891) -- (1.118,-0.901) -- (1.086,-0.913) -- (1.054,-0.924) -- (1.022,-0.936) -- (0.990,-0.948) -- (0.958,-0.962) -- (0.925,-0.976) -- (0.892,-0.991) -- (0.858,-1.007) -- (0.824,-1.025) -- (0.789,-1.044) -- (0.754,-1.064) -- (0.717,-1.086) -- (0.678,-1.109);
\draw[knot] (1.296,-0.099) -- (1.293,-0.144) -- (1.291,-0.187) -- (1.289,-0.228) -- (1.288,-0.268) -- (1.288,-0.306) -- (1.288,-0.343) -- (1.290,-0.380) -- (1.292,-0.415) -- (1.295,-0.450) -- (1.299,-0.484) -- (1.303,-0.518) -- (1.307,-0.552) -- (1.312,-0.586) -- (1.317,-0.619) -- (1.322,-0.653) -- (1.327,-0.686);
\draw[knot] (1.347,-0.964) -- (1.345,-1.000) -- (1.340,-1.038) -- (1.334,-1.075) -- (1.326,-1.113) -- (1.315,-1.152) -- (1.301,-1.191) -- (1.285,-1.231) -- (1.265,-1.270) -- (1.242,-1.310) -- (1.215,-1.350) -- (1.185,-1.390) -- (1.150,-1.430) -- (1.111,-1.469) -- (1.067,-1.507) -- (1.018,-1.543) -- (0.965,-1.578);
\draw[knot] (0.678,-1.109) -- (0.625,-1.140) -- (0.572,-1.168) -- (0.516,-1.193) -- (0.460,-1.216) -- (0.403,-1.236) -- (0.345,-1.254) -- (0.286,-1.268) -- (0.226,-1.280) -- (0.166,-1.289) -- (0.106,-1.296) -- (0.045,-1.299) -- (-0.016,-1.300) -- (-0.077,-1.298) -- (-0.137,-1.293) -- (-0.197,-1.285);
\draw[knot] (0.965,-1.578) -- (0.925,-1.604) -- (0.887,-1.628) -- (0.851,-1.653) -- (0.817,-1.677) -- (0.785,-1.701) -- (0.755,-1.726) -- (0.726,-1.750) -- (0.698,-1.775) -- (0.672,-1.800) -- (0.647,-1.825) -- (0.622,-1.850) -- (0.599,-1.876) -- (0.575,-1.901) -- (0.553,-1.927) -- (0.531,-1.953) -- (0.508,-1.979) -- (0.486,-2.005) -- (0.464,-2.032) -- (0.441,-2.058) -- (0.418,-2.083) -- (0.395,-2.109) -- (0.370,-2.134) -- (0.345,-2.159) -- (0.318,-2.183) -- (0.291,-2.207) -- (0.262,-2.230) -- (0.231,-2.252) -- (0.199,-2.273) -- (0.165,-2.292) -- (0.129,-2.310) -- (0.091,-2.327) -- (0.050,-2.342) -- (0.007,-2.355) -- (-0.038,-2.366) -- (-0.086,-2.375) -- (-0.136,-2.381) -- (-0.189,-2.384) -- (-0.245,-2.383) -- (-0.303,-2.380) -- (-0.364,-2.372);
\draw[knot] (1.252,-2.048) -- (1.198,-2.078) -- (1.146,-2.104) -- (1.094,-2.126) -- (1.044,-2.144) -- (0.996,-2.158) -- (0.948,-2.168) -- (0.902,-2.176) -- (0.858,-2.180) -- (0.815,-2.182) -- (0.773,-2.181) -- (0.733,-2.178) -- (0.694,-2.173) -- (0.656,-2.166) -- (0.619,-2.158) -- (0.583,-2.148) -- (0.549,-2.137);
\draw[knot] (0.295,-2.022) -- (0.264,-2.007) -- (0.234,-1.991) -- (0.203,-1.976) -- (0.172,-1.961) -- (0.140,-1.947) -- (0.108,-1.933) -- (0.075,-1.920) -- (0.041,-1.907) -- (0.006,-1.895) -- (-0.030,-1.883) -- (-0.067,-1.872) -- (-0.106,-1.862) -- (-0.147,-1.853) -- (-0.189,-1.844) -- (-0.234,-1.836) -- (-0.281,-1.829);
\draw[knot] (-0.197,-1.285) -- (-0.266,-1.272) -- (-0.334,-1.256) -- (-0.401,-1.237) -- (-0.467,-1.213) -- (-0.531,-1.187);
\draw[knot] (-1.224,-0.439) -- (-1.245,-0.373) -- (-1.264,-0.306) -- (-1.278,-0.237) -- (-1.289,-0.168) -- (-1.296,-0.099);
\draw[knot] (-0.281,-1.829) -- (-0.379,-1.811) -- (-0.475,-1.788) -- (-0.570,-1.760) -- (-0.664,-1.727) -- (-0.756,-1.689);
\draw[knot] (-1.741,-0.625) -- (-1.772,-0.531) -- (-1.798,-0.435) -- (-1.819,-0.338) -- (-1.834,-0.240) -- (-1.845,-0.141);
\draw[knot] (-0.364,-2.372) -- (-0.491,-2.349) -- (-0.617,-2.319) -- (-0.740,-2.283) -- (-0.862,-2.240) -- (-0.980,-2.191);
\draw[knot] (-2.259,-0.811) -- (-2.299,-0.688) -- (-2.333,-0.564) -- (-2.360,-0.438) -- (-2.380,-0.311) -- (-2.393,-0.183);
\draw[knot] (-2.393,-0.183) -- (-2.400,-0.032) -- (-2.397,0.120) -- (-2.385,0.271) -- (-2.363,0.420) -- (-2.332,0.568) -- (-2.291,0.714) -- (-2.242,0.857) -- (-2.183,0.997) -- (-2.116,1.132) -- (-2.041,1.263) -- (-1.957,1.390) -- (-1.865,1.510) -- (-1.767,1.625) -- (-1.661,1.733) -- (-1.548,1.834);
\draw[knot] (-1.845,-0.141) -- (-1.847,-0.077) -- (-1.846,-0.016) -- (-1.841,0.042) -- (-1.832,0.096) -- (-1.820,0.147) -- (-1.806,0.196) -- (-1.789,0.241) -- (-1.770,0.283) -- (-1.750,0.322) -- (-1.727,0.358) -- (-1.704,0.392) -- (-1.679,0.424) -- (-1.653,0.453) -- (-1.626,0.480) -- (-1.598,0.505) -- (-1.570,0.529) -- (-1.541,0.551) -- (-1.512,0.571) -- (-1.482,0.590) -- (-1.453,0.609) -- (-1.423,0.626) -- (-1.393,0.643) -- (-1.363,0.659) -- (-1.333,0.675) -- (-1.302,0.690) -- (-1.272,0.706) -- (-1.242,0.721) -- (-1.212,0.737) -- (-1.182,0.753) -- (-1.152,0.770) -- (-1.123,0.787) -- (-1.093,0.805) -- (-1.062,0.824) -- (-1.032,0.844) -- (-1.001,0.865) -- (-0.970,0.888) -- (-0.939,0.912) -- (-0.906,0.937) -- (-0.873,0.965) -- (-0.839,0.993);
\draw[knot] (-1.296,-0.099) -- (-1.300,-0.054) -- (-1.304,-0.011) -- (-1.308,0.030) -- (-1.314,0.069) -- (-1.319,0.107) -- (-1.326,0.144) -- (-1.333,0.179) -- (-1.340,0.214) -- (-1.348,0.248) -- (-1.357,0.282) -- (-1.366,0.315) -- (-1.376,0.347) -- (-1.385,0.380) -- (-1.395,0.412) -- (-1.405,0.444) -- (-1.415,0.477);
\draw[knot] (-1.478,0.748) -- (-1.481,0.785) -- (-1.482,0.822) -- (-1.482,0.860) -- (-1.479,0.899) -- (-1.474,0.939) -- (-1.467,0.980) -- (-1.457,1.021) -- (-1.443,1.064) -- (-1.426,1.107) -- (-1.406,1.150) -- (-1.382,1.194) -- (-1.353,1.238) -- (-1.321,1.283) -- (-1.283,1.327) -- (-1.241,1.371) -- (-1.193,1.414);
\draw[knot] (-0.839,0.993) -- (-0.788,1.034) -- (-0.735,1.072) -- (-0.680,1.108) -- (-0.624,1.140) -- (-0.566,1.170) -- (-0.507,1.197) -- (-0.446,1.221) -- (-0.385,1.242) -- (-0.322,1.259) -- (-0.259,1.274) -- (-0.195,1.285) -- (-0.130,1.293) -- (-0.065,1.298) -- (-0.000,1.300);
\draw[knot] (-1.193,1.414) -- (-1.157,1.445) -- (-1.124,1.475) -- (-1.092,1.504) -- (-1.062,1.534) -- (-1.034,1.562) -- (-1.008,1.591) -- (-0.983,1.620) -- (-0.960,1.648) -- (-0.937,1.677) -- (-0.916,1.705) -- (-0.896,1.734) -- (-0.876,1.763) -- (-0.857,1.792) -- (-0.839,1.821) -- (-0.821,1.850) -- (-0.803,1.879) -- (-0.785,1.908) -- (-0.767,1.938) -- (-0.749,1.967) -- (-0.730,1.996) -- (-0.710,2.025) -- (-0.690,2.053) -- (-0.669,2.082) -- (-0.646,2.110) -- (-0.623,2.137) -- (-0.597,2.164) -- (-0.570,2.190) -- (-0.542,2.216) -- (-0.511,2.241) -- (-0.478,2.264) -- (-0.443,2.286) -- (-0.405,2.307) -- (-0.365,2.327) -- (-0.322,2.345) -- (-0.276,2.360) -- (-0.227,2.374) -- (-0.175,2.385) -- (-0.120,2.393) -- (-0.062,2.398) -- (-0.000,2.400);
\draw[knot] (-1.548,1.834) -- (-1.500,1.872) -- (-1.452,1.906) -- (-1.405,1.935) -- (-1.358,1.960) -- (-1.312,1.982) -- (-1.266,1.999) -- (-1.222,2.014) -- (-1.179,2.025) -- (-1.136,2.033) -- (-1.095,2.039) -- (-1.055,2.042) -- (-1.015,2.043) -- (-0.977,2.042) -- (-0.940,2.039) -- (-0.903,2.035) -- (-0.867,2.029);
\draw[knot] (-0.599,1.954) -- (-0.566,1.943) -- (-0.533,1.933) -- (-0.501,1.922) -- (-0.468,1.912) -- (-0.434,1.903) -- (-0.400,1.894) -- (-0.365,1.886) -- (-0.330,1.878) -- (-0.294,1.872) -- (-0.256,1.866) -- (-0.218,1.861) -- (-0.178,1.857) -- (-0.136,1.854) -- (-0.093,1.852) -- (-0.047,1.850) -- (-0.000,1.850);
\fill[knotdot] (-0.743,-1.066) circle (0.035);
\fill[knotdot] (-0.954,-0.884) circle (0.035);
\fill[knotdot] (-1.120,-0.660) circle (0.035);
\fill[knotdot] (-1.058,-1.518) circle (0.035);
\fill[knotdot] (-1.357,-1.257) circle (0.035);
\fill[knotdot] (-1.594,-0.939) circle (0.035);
\fill[knotdot] (-1.372,-1.969) circle (0.035);
\fill[knotdot] (-1.760,-1.631) circle (0.035);
\fill[knotdot] (-2.068,-1.219) circle (0.035);
\node[reglabel] (w0) at (0.000,0.000) {$5$};
\node[reglabel] (w1) at (-1.960,0.821) {$1$};
\node[reglabel] (w2) at (0.968,1.892) {$2$};
\node[reglabel] (w3) at (1.813,-1.109) {$3$};
\node[reglabel] (w4) at (-1.046,-1.850) {$4$};
\end{tikzpicture}

%% file: figures/Th-3q-tait-white-overlay.tex
\begin{tikzpicture}[tait/.style={line width=1.3pt, black, line cap=round},
                    vert/.style={circle, fill=white, draw=black, line width=1.0pt, inner sep=0pt, minimum size=0.48cm, font=\small},
                    knot/.style={line width=0.8pt, black!25, line cap=round, line join=round}, knotdot/.style={black!35}]
\draw[knot] (0.000,2.400) -- (0.151,2.395) -- (0.302,2.381) -- (0.451,2.357) -- (0.599,2.324) -- (0.744,2.282) -- (0.887,2.230) -- (1.025,2.170) -- (1.160,2.101) -- (1.290,2.024) -- (1.415,1.938) -- (1.534,1.845) -- (1.648,1.745) -- (1.754,1.638) -- (1.854,1.524) -- (1.946,1.404);
\draw[knot] (0.000,1.850) -- (0.064,1.848) -- (0.125,1.842) -- (0.182,1.832) -- (0.235,1.819) -- (0.286,1.804) -- (0.333,1.786) -- (0.376,1.766) -- (0.417,1.744) -- (0.454,1.720) -- (0.489,1.695) -- (0.521,1.669) -- (0.550,1.641) -- (0.577,1.613) -- (0.602,1.584) -- (0.625,1.555) -- (0.647,1.525) -- (0.666,1.495) -- (0.685,1.464) -- (0.702,1.433) -- (0.717,1.402) -- (0.733,1.371) -- (0.747,1.340) -- (0.761,1.309) -- (0.774,1.277) -- (0.787,1.246) -- (0.800,1.215) -- (0.814,1.184) -- (0.827,1.153) -- (0.841,1.122) -- (0.855,1.091) -- (0.870,1.059) -- (0.886,1.028) -- (0.903,0.997) -- (0.920,0.965) -- (0.939,0.933) -- (0.959,0.900) -- (0.981,0.867) -- (1.004,0.832) -- (1.028,0.797) -- (1.054,0.761);
\draw[knot] (0.000,1.300) -- (0.045,1.300) -- (0.088,1.301) -- (0.129,1.302) -- (0.169,1.305) -- (0.207,1.307) -- (0.244,1.311) -- (0.280,1.315) -- (0.315,1.320) -- (0.350,1.326) -- (0.384,1.332) -- (0.418,1.338) -- (0.451,1.345) -- (0.484,1.352) -- (0.517,1.360) -- (0.550,1.368) -- (0.583,1.375);
\draw[knot] (0.858,1.417) -- (0.895,1.417) -- (0.933,1.415) -- (0.971,1.412) -- (1.009,1.407) -- (1.049,1.399) -- (1.089,1.388) -- (1.129,1.375) -- (1.170,1.358) -- (1.212,1.338) -- (1.254,1.314) -- (1.296,1.287) -- (1.338,1.255) -- (1.380,1.219) -- (1.421,1.178) -- (1.461,1.133) -- (1.500,1.082);
\draw[knot] (1.054,0.761) -- (1.089,0.710) -- (1.121,0.659) -- (1.150,0.606) -- (1.177,0.551) -- (1.202,0.496) -- (1.224,0.439) -- (1.243,0.381) -- (1.259,0.323) -- (1.273,0.264) -- (1.284,0.204) -- (1.292,0.144) -- (1.297,0.083) -- (1.300,0.022) -- (1.299,-0.038) -- (1.296,-0.099);
\draw[knot] (1.500,1.082) -- (1.529,1.044) -- (1.556,1.008) -- (1.583,0.974) -- (1.610,0.942) -- (1.637,0.912) -- (1.663,0.884) -- (1.690,0.857) -- (1.716,0.831) -- (1.743,0.807) -- (1.770,0.784) -- (1.797,0.761) -- (1.825,0.740) -- (1.852,0.719) -- (1.880,0.698) -- (1.907,0.678) -- (1.935,0.658) -- (1.963,0.638) -- (1.990,0.617) -- (2.018,0.597) -- (2.046,0.576) -- (2.073,0.554) -- (2.100,0.532) -- (2.127,0.508) -- (2.153,0.484) -- (2.178,0.458) -- (2.203,0.431) -- (2.227,0.402) -- (2.251,0.371) -- (2.273,0.339) -- (2.294,0.304) -- (2.314,0.268) -- (2.332,0.228) -- (2.348,0.187) -- (2.362,0.142) -- (2.374,0.095) -- (2.384,0.046) -- (2.391,-0.007) -- (2.395,-0.063) -- (2.396,-0.121) -- (2.393,-0.183);
\draw[knot] (1.946,1.404) -- (1.981,1.353) -- (2.011,1.303) -- (2.037,1.253) -- (2.058,1.205) -- (2.076,1.157) -- (2.090,1.111) -- (2.101,1.065) -- (2.109,1.021) -- (2.114,0.978) -- (2.116,0.937) -- (2.116,0.896) -- (2.114,0.857) -- (2.110,0.819) -- (2.105,0.782) -- (2.097,0.745) -- (2.089,0.710);
\draw[knot] (1.994,0.448) -- (1.981,0.416) -- (1.968,0.385) -- (1.955,0.353) -- (1.943,0.321) -- (1.931,0.288) -- (1.919,0.255) -- (1.908,0.221) -- (1.898,0.186) -- (1.889,0.150) -- (1.880,0.113) -- (1.872,0.075) -- (1.865,0.036) -- (1.859,-0.005) -- (1.853,-0.049) -- (1.849,-0.094) -- (1.845,-0.141);
\draw[knot] (2.393,-0.183) -- (2.375,-0.344) -- (2.347,-0.503) -- (2.307,-0.661) -- (2.257,-0.815) -- (2.197,-0.966) -- (2.127,-1.112) -- (2.047,-1.253) -- (1.958,-1.388) -- (1.860,-1.517) -- (1.753,-1.639) -- (1.638,-1.754) -- (1.516,-1.860) -- (1.387,-1.959) -- (1.252,-2.048);
\draw[knot] (1.845,-0.141) -- (1.838,-0.204) -- (1.827,-0.264) -- (1.813,-0.321) -- (1.796,-0.373) -- (1.777,-0.422) -- (1.755,-0.468) -- (1.732,-0.510) -- (1.707,-0.548) -- (1.680,-0.584) -- (1.653,-0.617) -- (1.624,-0.646) -- (1.595,-0.674) -- (1.565,-0.699) -- (1.534,-0.721) -- (1.503,-0.742) -- (1.471,-0.761) -- (1.440,-0.778) -- (1.408,-0.794) -- (1.376,-0.809) -- (1.343,-0.822) -- (1.311,-0.835) -- (1.279,-0.847) -- (1.247,-0.858) -- (1.215,-0.869) -- (1.183,-0.880) -- (1.151,-0.891) -- (1.118,-0.901) -- (1.086,-0.913) -- (1.054,-0.924) -- (1.022,-0.936) -- (0.990,-0.948) -- (0.958,-0.962) -- (0.925,-0.976) -- (0.892,-0.991) -- (0.858,-1.007) -- (0.824,-1.025) -- (0.789,-1.044) -- (0.754,-1.064) -- (0.717,-1.086) -- (0.678,-1.109);
\draw[knot] (1.296,-0.099) -- (1.293,-0.144) -- (1.291,-0.187) -- (1.289,-0.228) -- (1.288,-0.268) -- (1.288,-0.306) -- (1.288,-0.343) -- (1.290,-0.380) -- (1.292,-0.415) -- (1.295,-0.450) -- (1.299,-0.484) -- (1.303,-0.518) -- (1.307,-0.552) -- (1.312,-0.586) -- (1.317,-0.619) -- (1.322,-0.653) -- (1.327,-0.686);
\draw[knot] (1.347,-0.964) -- (1.345,-1.000) -- (1.340,-1.038) -- (1.334,-1.075) -- (1.326,-1.113) -- (1.315,-1.152) -- (1.301,-1.191) -- (1.285,-1.231) -- (1.265,-1.270) -- (1.242,-1.310) -- (1.215,-1.350) -- (1.185,-1.390) -- (1.150,-1.430) -- (1.111,-1.469) -- (1.067,-1.507) -- (1.018,-1.543) -- (0.965,-1.578);
\draw[knot] (0.678,-1.109) -- (0.625,-1.140) -- (0.572,-1.168) -- (0.516,-1.193) -- (0.460,-1.216) -- (0.403,-1.236) -- (0.345,-1.254) -- (0.286,-1.268) -- (0.226,-1.280) -- (0.166,-1.289) -- (0.106,-1.296) -- (0.045,-1.299) -- (-0.016,-1.300) -- (-0.077,-1.298) -- (-0.137,-1.293) -- (-0.197,-1.285);
\draw[knot] (0.965,-1.578) -- (0.925,-1.604) -- (0.887,-1.628) -- (0.851,-1.653) -- (0.817,-1.677) -- (0.785,-1.701) -- (0.755,-1.726) -- (0.726,-1.750) -- (0.698,-1.775) -- (0.672,-1.800) -- (0.647,-1.825) -- (0.622,-1.850) -- (0.599,-1.876) -- (0.575,-1.901) -- (0.553,-1.927) -- (0.531,-1.953) -- (0.508,-1.979) -- (0.486,-2.005) -- (0.464,-2.032) -- (0.441,-2.058) -- (0.418,-2.083) -- (0.395,-2.109) -- (0.370,-2.134) -- (0.345,-2.159) -- (0.318,-2.183) -- (0.291,-2.207) -- (0.262,-2.230) -- (0.231,-2.252) -- (0.199,-2.273) -- (0.165,-2.292) -- (0.129,-2.310) -- (0.091,-2.327) -- (0.050,-2.342) -- (0.007,-2.355) -- (-0.038,-2.366) -- (-0.086,-2.375) -- (-0.136,-2.381) -- (-0.189,-2.384) -- (-0.245,-2.383) -- (-0.303,-2.380) -- (-0.364,-2.372);
\draw[knot] (1.252,-2.048) -- (1.198,-2.078) -- (1.146,-2.104) -- (1.094,-2.126) -- (1.044,-2.144) -- (0.996,-2.158) -- (0.948,-2.168) -- (0.902,-2.176) -- (0.858,-2.180) -- (0.815,-2.182) -- (0.773,-2.181) -- (0.733,-2.178) -- (0.694,-2.173) -- (0.656,-2.166) -- (0.619,-2.158) -- (0.583,-2.148) -- (0.549,-2.137);
\draw[knot] (0.295,-2.022) -- (0.264,-2.007) -- (0.234,-1.991) -- (0.203,-1.976) -- (0.172,-1.961) -- (0.140,-1.947) -- (0.108,-1.933) -- (0.075,-1.920) -- (0.041,-1.907) -- (0.006,-1.895) -- (-0.030,-1.883) -- (-0.067,-1.872) -- (-0.106,-1.862) -- (-0.147,-1.853) -- (-0.189,-1.844) -- (-0.234,-1.836) -- (-0.281,-1.829);
\draw[knot] (-0.197,-1.285) -- (-0.266,-1.272) -- (-0.334,-1.256) -- (-0.401,-1.237) -- (-0.467,-1.213) -- (-0.531,-1.187);
\draw[knot] (-1.224,-0.439) -- (-1.245,-0.373) -- (-1.264,-0.306) -- (-1.278,-0.237) -- (-1.289,-0.168) -- (-1.296,-0.099);
\draw[knot] (-0.281,-1.829) -- (-0.379,-1.811) -- (-0.475,-1.788) -- (-0.570,-1.760) -- (-0.664,-1.727) -- (-0.756,-1.689);
\draw[knot] (-1.741,-0.625) -- (-1.772,-0.531) -- (-1.798,-0.435) -- (-1.819,-0.338) -- (-1.834,-0.240) -- (-1.845,-0.141);
\draw[knot] (-0.364,-2.372) -- (-0.491,-2.349) -- (-0.617,-2.319) -- (-0.740,-2.283) -- (-0.862,-2.240) -- (-0.980,-2.191);
\draw[knot] (-2.259,-0.811) -- (-2.299,-0.688) -- (-2.333,-0.564) -- (-2.360,-0.438) -- (-2.380,-0.311) -- (-2.393,-0.183);
\draw[knot] (-2.393,-0.183) -- (-2.400,-0.032) -- (-2.397,0.120) -- (-2.385,0.271) -- (-2.363,0.420) -- (-2.332,0.568) -- (-2.291,0.714) -- (-2.242,0.857) -- (-2.183,0.997) -- (-2.116,1.132) -- (-2.041,1.263) -- (-1.957,1.390) -- (-1.865,1.510) -- (-1.767,1.625) -- (-1.661,1.733) -- (-1.548,1.834);
\draw[knot] (-1.845,-0.141) -- (-1.847,-0.077) -- (-1.846,-0.016) -- (-1.841,0.042) -- (-1.832,0.096) -- (-1.820,0.147) -- (-1.806,0.196) -- (-1.789,0.241) -- (-1.770,0.283) -- (-1.750,0.322) -- (-1.727,0.358) -- (-1.704,0.392) -- (-1.679,0.424) -- (-1.653,0.453) -- (-1.626,0.480) -- (-1.598,0.505) -- (-1.570,0.529) -- (-1.541,0.551) -- (-1.512,0.571) -- (-1.482,0.590) -- (-1.453,0.609) -- (-1.423,0.626) -- (-1.393,0.643) -- (-1.363,0.659) -- (-1.333,0.675) -- (-1.302,0.690) -- (-1.272,0.706) -- (-1.242,0.721) -- (-1.212,0.737) -- (-1.182,0.753) -- (-1.152,0.770) -- (-1.123,0.787) -- (-1.093,0.805) -- (-1.062,0.824) -- (-1.032,0.844) -- (-1.001,0.865) -- (-0.970,0.888) -- (-0.939,0.912) -- (-0.906,0.937) -- (-0.873,0.965) -- (-0.839,0.993);
\draw[knot] (-1.296,-0.099) -- (-1.300,-0.054) -- (-1.304,-0.011) -- (-1.308,0.030) -- (-1.314,0.069) -- (-1.319,0.107) -- (-1.326,0.144) -- (-1.333,0.179) -- (-1.340,0.214) -- (-1.348,0.248) -- (-1.357,0.282) -- (-1.366,0.315) -- (-1.376,0.347) -- (-1.385,0.380) -- (-1.395,0.412) -- (-1.405,0.444) -- (-1.415,0.477);
\draw[knot] (-1.478,0.748) -- (-1.481,0.785) -- (-1.482,0.822) -- (-1.482,0.860) -- (-1.479,0.899) -- (-1.474,0.939) -- (-1.467,0.980) -- (-1.457,1.021) -- (-1.443,1.064) -- (-1.426,1.107) -- (-1.406,1.150) -- (-1.382,1.194) -- (-1.353,1.238) -- (-1.321,1.283) -- (-1.283,1.327) -- (-1.241,1.371) -- (-1.193,1.414);
\draw[knot] (-0.839,0.993) -- (-0.788,1.034) -- (-0.735,1.072) -- (-0.680,1.108) -- (-0.624,1.140) -- (-0.566,1.170) -- (-0.507,1.197) -- (-0.446,1.221) -- (-0.385,1.242) -- (-0.322,1.259) -- (-0.259,1.274) -- (-0.195,1.285) -- (-0.130,1.293) -- (-0.065,1.298) -- (-0.000,1.300);
\draw[knot] (-1.193,1.414) -- (-1.157,1.445) -- (-1.124,1.475) -- (-1.092,1.504) -- (-1.062,1.534) -- (-1.034,1.562) -- (-1.008,1.591) -- (-0.983,1.620) -- (-0.960,1.648) -- (-0.937,1.677) -- (-0.916,1.705) -- (-0.896,1.734) -- (-0.876,1.763) -- (-0.857,1.792) -- (-0.839,1.821) -- (-0.821,1.850) -- (-0.803,1.879) -- (-0.785,1.908) -- (-0.767,1.938) -- (-0.749,1.967) -- (-0.730,1.996) -- (-0.710,2.025) -- (-0.690,2.053) -- (-0.669,2.082) -- (-0.646,2.110) -- (-0.623,2.137) -- (-0.597,2.164) -- (-0.570,2.190) -- (-0.542,2.216) -- (-0.511,2.241) -- (-0.478,2.264) -- (-0.443,2.286) -- (-0.405,2.307) -- (-0.365,2.327) -- (-0.322,2.345) -- (-0.276,2.360) -- (-0.227,2.374) -- (-0.175,2.385) -- (-0.120,2.393) -- (-0.062,2.398) -- (-0.000,2.400);
\draw[knot] (-1.548,1.834) -- (-1.500,1.872) -- (-1.452,1.906) -- (-1.405,1.935) -- (-1.358,1.960) -- (-1.312,1.982) -- (-1.266,1.999) -- (-1.222,2.014) -- (-1.179,2.025) -- (-1.136,2.033) -- (-1.095,2.039) -- (-1.055,2.042) -- (-1.015,2.043) -- (-0.977,2.042) -- (-0.940,2.039) -- (-0.903,2.035) -- (-0.867,2.029);
\draw[knot] (-0.599,1.954) -- (-0.566,1.943) -- (-0.533,1.933) -- (-0.501,1.922) -- (-0.468,1.912) -- (-0.434,1.903) -- (-0.400,1.894) -- (-0.365,1.886) -- (-0.330,1.878) -- (-0.294,1.872) -- (-0.256,1.866) -- (-0.218,1.861) -- (-0.178,1.857) -- (-0.136,1.854) -- (-0.093,1.852) -- (-0.047,1.850) -- (-0.000,1.850);
\fill[knotdot] (-0.743,-1.066) circle (0.035);
\fill[knotdot] (-0.954,-0.884) circle (0.035);
\fill[knotdot] (-1.120,-0.660) circle (0.035);
\fill[knotdot] (-1.058,-1.518) circle (0.035);
\fill[knotdot] (-1.357,-1.257) circle (0.035);
\fill[knotdot] (-1.594,-0.939) circle (0.035);
\fill[knotdot] (-1.372,-1.969) circle (0.035);
\fill[knotdot] (-1.760,-1.631) circle (0.035);
\fill[knotdot] (-2.068,-1.219) circle (0.035);
\draw[tait] (0.000,0.000) -- (0.008,0.015) -- (0.015,0.030) -- (0.023,0.044) -- (0.030,0.059) -- (0.038,0.074) -- (0.045,0.089) -- (0.053,0.103) -- (0.060,0.118) -- (0.068,0.133) -- (0.076,0.148) -- (0.083,0.162) -- (0.091,0.177) -- (0.098,0.192) -- (0.106,0.207) -- (0.113,0.221) -- (0.121,0.236) -- (0.128,0.251) -- (0.136,0.266) -- (0.143,0.280) -- (0.151,0.295) -- (0.159,0.310) -- (0.166,0.325) -- (0.174,0.339) -- (0.181,0.354) -- (0.189,0.369) -- (0.196,0.384) -- (0.204,0.398) -- (0.211,0.413) -- (0.219,0.428) -- (0.227,0.443) -- (0.234,0.458) -- (0.242,0.472) -- (0.249,0.487) -- (0.257,0.502) -- (0.264,0.517) -- (0.272,0.531) -- (0.279,0.546) -- (0.287,0.561) -- (0.295,0.576) -- (0.302,0.590) -- (0.310,0.605) -- (0.317,0.620) -- (0.325,0.635) -- (0.332,0.649) -- (0.340,0.664) -- (0.347,0.679) -- (0.355,0.694) -- (0.362,0.708) -- (0.370,0.723) -- (0.378,0.738) -- (0.385,0.753) -- (0.393,0.767) -- (0.400,0.782) -- (0.408,0.797) -- (0.415,0.812) -- (0.423,0.827) -- (0.430,0.841) -- (0.438,0.856) -- (0.446,0.871) -- (0.453,0.886) -- (0.461,0.900) -- (0.468,0.915) -- (0.476,0.930) -- (0.483,0.945) -- (0.491,0.959) -- (0.498,0.974) -- (0.506,0.989) -- (0.514,1.004) -- (0.521,1.018) -- (0.529,1.033) -- (0.536,1.048) -- (0.544,1.063) -- (0.551,1.077) -- (0.559,1.092) -- (0.566,1.107) -- (0.574,1.122) -- (0.582,1.136) -- (0.589,1.151) -- (0.597,1.166) -- (0.604,1.181) -- (0.612,1.195) -- (0.619,1.210) -- (0.627,1.225) -- (0.634,1.240) -- (0.642,1.255) -- (0.649,1.269) -- (0.657,1.284) -- (0.665,1.299) -- (0.672,1.314) -- (0.680,1.328) -- (0.687,1.343) -- (0.695,1.358) -- (0.702,1.373) -- (0.710,1.387) -- (0.717,1.402) -- (0.725,1.417) -- (0.733,1.432) -- (0.740,1.447) -- (0.748,1.461) -- (0.755,1.476) -- (0.763,1.491) -- (0.771,1.506) -- (0.778,1.521) -- (0.786,1.536) -- (0.793,1.550) -- (0.801,1.565) -- (0.809,1.580) -- (0.816,1.595) -- (0.824,1.610) -- (0.831,1.625) -- (0.839,1.639) -- (0.847,1.654) -- (0.854,1.669) -- (0.862,1.684) -- (0.869,1.699) -- (0.877,1.714) -- (0.884,1.729) -- (0.892,1.743) -- (0.900,1.758) -- (0.907,1.773) -- (0.915,1.788) -- (0.922,1.803) -- (0.930,1.818) -- (0.938,1.832) -- (0.945,1.847) -- (0.953,1.862) -- (0.960,1.877) -- (0.968,1.892);  
\draw[tait] (0.968,1.892) -- (1.061,1.841) -- (1.151,1.786) -- (1.239,1.727) -- (1.323,1.663) -- (1.404,1.595) -- (1.482,1.523) -- (1.556,1.447) -- (1.626,1.368) -- (1.692,1.286) -- (1.746,1.211) -- (1.790,1.146) -- (1.831,1.079) -- (1.869,1.011) -- (1.905,0.941) -- (1.939,0.870) -- (1.969,0.798) -- (1.998,0.725) -- (2.023,0.651) -- (2.046,0.576) -- (2.064,0.504) -- (2.081,0.431) -- (2.095,0.358) -- (2.106,0.284) -- (2.115,0.210) -- (2.121,0.136) -- (2.124,0.062) -- (2.125,-0.013) -- (2.123,-0.087) -- (2.119,-0.162) -- (2.109,-0.262) -- (2.094,-0.361) -- (2.075,-0.460) -- (2.051,-0.557) -- (2.022,-0.654) -- (1.989,-0.749) -- (1.951,-0.842) -- (1.909,-0.933) -- (1.863,-1.022) -- (1.813,-1.109);  
\draw[tait] (0.000,0.000) -- (0.014,-0.009) -- (0.028,-0.017) -- (0.042,-0.026) -- (0.057,-0.035) -- (0.071,-0.043) -- (0.085,-0.052) -- (0.099,-0.061) -- (0.113,-0.069) -- (0.127,-0.078) -- (0.141,-0.087) -- (0.156,-0.095) -- (0.170,-0.104) -- (0.184,-0.112) -- (0.198,-0.121) -- (0.212,-0.130) -- (0.226,-0.138) -- (0.240,-0.147) -- (0.255,-0.156) -- (0.269,-0.164) -- (0.283,-0.173) -- (0.297,-0.182) -- (0.311,-0.190) -- (0.325,-0.199) -- (0.339,-0.208) -- (0.354,-0.216) -- (0.368,-0.225) -- (0.382,-0.234) -- (0.396,-0.242) -- (0.410,-0.251) -- (0.424,-0.260) -- (0.438,-0.268) -- (0.453,-0.277) -- (0.467,-0.286) -- (0.481,-0.294) -- (0.495,-0.303) -- (0.509,-0.312) -- (0.523,-0.320) -- (0.537,-0.329) -- (0.552,-0.337) -- (0.566,-0.346) -- (0.580,-0.355) -- (0.594,-0.363) -- (0.608,-0.372) -- (0.622,-0.381) -- (0.636,-0.389) -- (0.650,-0.398) -- (0.665,-0.407) -- (0.679,-0.415) -- (0.693,-0.424) -- (0.707,-0.433) -- (0.721,-0.441) -- (0.735,-0.450) -- (0.749,-0.459) -- (0.764,-0.467) -- (0.778,-0.476) -- (0.792,-0.485) -- (0.806,-0.493) -- (0.820,-0.502) -- (0.834,-0.511) -- (0.848,-0.519) -- (0.863,-0.528) -- (0.877,-0.537) -- (0.891,-0.545) -- (0.905,-0.554) -- (0.919,-0.562) -- (0.933,-0.571) -- (0.947,-0.580) -- (0.962,-0.588) -- (0.976,-0.597) -- (0.990,-0.606) -- (1.004,-0.614) -- (1.018,-0.623) -- (1.032,-0.632) -- (1.046,-0.640) -- (1.061,-0.649) -- (1.075,-0.658) -- (1.089,-0.666) -- (1.103,-0.675) -- (1.117,-0.684) -- (1.131,-0.692) -- (1.145,-0.701) -- (1.160,-0.710) -- (1.174,-0.718) -- (1.188,-0.727) -- (1.202,-0.736) -- (1.216,-0.744) -- (1.230,-0.753) -- (1.244,-0.762) -- (1.259,-0.770) -- (1.273,-0.779) -- (1.287,-0.787) -- (1.301,-0.796) -- (1.315,-0.805) -- (1.329,-0.813) -- (1.343,-0.822) -- (1.358,-0.831) -- (1.372,-0.840) -- (1.386,-0.848) -- (1.400,-0.857) -- (1.414,-0.866) -- (1.429,-0.874) -- (1.443,-0.883) -- (1.457,-0.892) -- (1.471,-0.900) -- (1.486,-0.909) -- (1.500,-0.918) -- (1.514,-0.927) -- (1.528,-0.935) -- (1.542,-0.944) -- (1.557,-0.953) -- (1.571,-0.961) -- (1.585,-0.970) -- (1.599,-0.979) -- (1.614,-0.987) -- (1.628,-0.996) -- (1.642,-1.005) -- (1.656,-1.014) -- (1.670,-1.022) -- (1.685,-1.031) -- (1.699,-1.040) -- (1.713,-1.048) -- (1.727,-1.057) -- (1.741,-1.066) -- (1.756,-1.074) -- (1.770,-1.083) -- (1.784,-1.092) -- (1.798,-1.100) -- (1.813,-1.109);  
\draw[tait] (1.813,-1.109) -- (1.755,-1.198) -- (1.693,-1.284) -- (1.627,-1.367) -- (1.557,-1.446) -- (1.483,-1.522) -- (1.406,-1.593) -- (1.325,-1.661) -- (1.240,-1.725) -- (1.153,-1.785) -- (1.075,-1.833) -- (1.006,-1.872) -- (0.937,-1.907) -- (0.866,-1.941) -- (0.793,-1.971) -- (0.720,-1.999) -- (0.646,-2.024) -- (0.571,-2.047) -- (0.495,-2.067) -- (0.418,-2.083) -- (0.345,-2.097) -- (0.271,-2.108) -- (0.197,-2.116) -- (0.123,-2.121) -- (0.049,-2.124) -- (-0.026,-2.125) -- (-0.100,-2.123) -- (-0.175,-2.118) -- (-0.249,-2.110) -- (-0.323,-2.100) -- (-0.407,-2.086) -- (-0.491,-2.068) -- (-0.573,-2.046) -- (-0.655,-2.021) -- (-0.736,-1.993) -- (-0.816,-1.962) -- (-0.894,-1.928) -- (-0.971,-1.890) -- (-1.046,-1.850);  
\draw[tait] (0.000,0.000) -- (-0.015,0.006) -- (-0.031,0.013) -- (-0.046,0.019) -- (-0.061,0.026) -- (-0.076,0.032) -- (-0.092,0.038) -- (-0.107,0.045) -- (-0.122,0.051) -- (-0.138,0.058) -- (-0.153,0.064) -- (-0.168,0.070) -- (-0.183,0.077) -- (-0.199,0.083) -- (-0.214,0.090) -- (-0.229,0.096) -- (-0.245,0.103) -- (-0.260,0.109) -- (-0.275,0.115) -- (-0.291,0.122) -- (-0.306,0.128) -- (-0.321,0.135) -- (-0.336,0.141) -- (-0.352,0.147) -- (-0.367,0.154) -- (-0.382,0.160) -- (-0.398,0.167) -- (-0.413,0.173) -- (-0.428,0.179) -- (-0.443,0.186) -- (-0.459,0.192) -- (-0.474,0.199) -- (-0.489,0.205) -- (-0.505,0.211) -- (-0.520,0.218) -- (-0.535,0.224) -- (-0.550,0.231) -- (-0.566,0.237) -- (-0.581,0.243) -- (-0.596,0.250) -- (-0.612,0.256) -- (-0.627,0.263) -- (-0.642,0.269) -- (-0.658,0.275) -- (-0.673,0.282) -- (-0.688,0.288) -- (-0.703,0.295) -- (-0.719,0.301) -- (-0.734,0.308) -- (-0.749,0.314) -- (-0.765,0.320) -- (-0.780,0.327) -- (-0.795,0.333) -- (-0.810,0.340) -- (-0.826,0.346) -- (-0.841,0.352) -- (-0.856,0.359) -- (-0.872,0.365) -- (-0.887,0.372) -- (-0.902,0.378) -- (-0.917,0.384) -- (-0.933,0.391) -- (-0.948,0.397) -- (-0.963,0.404) -- (-0.979,0.410) -- (-0.994,0.416) -- (-1.009,0.423) -- (-1.025,0.429) -- (-1.040,0.436) -- (-1.055,0.442) -- (-1.070,0.448) -- (-1.086,0.455) -- (-1.101,0.461) -- (-1.116,0.468) -- (-1.132,0.474) -- (-1.147,0.480) -- (-1.162,0.487) -- (-1.177,0.493) -- (-1.193,0.500) -- (-1.208,0.506) -- (-1.223,0.513) -- (-1.239,0.519) -- (-1.254,0.525) -- (-1.269,0.532) -- (-1.284,0.538) -- (-1.300,0.545) -- (-1.315,0.551) -- (-1.330,0.557) -- (-1.346,0.564) -- (-1.361,0.570) -- (-1.376,0.577) -- (-1.391,0.583) -- (-1.407,0.589) -- (-1.422,0.596) -- (-1.437,0.602) -- (-1.453,0.609) -- (-1.468,0.615) -- (-1.483,0.621) -- (-1.499,0.628) -- (-1.514,0.634) -- (-1.530,0.641) -- (-1.545,0.647) -- (-1.560,0.654) -- (-1.576,0.660) -- (-1.591,0.667) -- (-1.606,0.673) -- (-1.622,0.679) -- (-1.637,0.686) -- (-1.652,0.692) -- (-1.668,0.699) -- (-1.683,0.705) -- (-1.699,0.712) -- (-1.714,0.718) -- (-1.729,0.725) -- (-1.745,0.731) -- (-1.760,0.737) -- (-1.775,0.744) -- (-1.791,0.750) -- (-1.806,0.757) -- (-1.822,0.763) -- (-1.837,0.770) -- (-1.852,0.776) -- (-1.868,0.783) -- (-1.883,0.789) -- (-1.898,0.795) -- (-1.914,0.802) -- (-1.929,0.808) -- (-1.945,0.815) -- (-1.960,0.821);  
\draw[tait] (-1.960,0.821) -- (-1.919,0.913) -- (-1.874,1.003) -- (-1.824,1.090) -- (-1.771,1.175) -- (-1.713,1.257) -- (-1.652,1.337) -- (-1.587,1.414) -- (-1.518,1.487) -- (-1.446,1.557) -- (-1.371,1.624) -- (-1.313,1.671) -- (-1.254,1.716) -- (-1.193,1.759) -- (-1.130,1.799) -- (-1.067,1.838) -- (-1.002,1.874) -- (-0.935,1.908) -- (-0.868,1.940) -- (-0.799,1.969) -- (-0.730,1.996) -- (-0.656,2.021) -- (-0.581,2.044) -- (-0.505,2.064) -- (-0.428,2.081) -- (-0.351,2.096) -- (-0.274,2.107) -- (-0.196,2.116) -- (-0.118,2.122) -- (-0.039,2.125) -- (0.053,2.124) -- (0.159,2.119) -- (0.264,2.109) -- (0.368,2.093) -- (0.472,2.072) -- (0.575,2.046) -- (0.676,2.015) -- (0.775,1.979) -- (0.873,1.938) -- (0.968,1.892);  
\draw[tait] (0.000,0.000) -- (-1.046,-1.850);  
\draw[tait] (-1.046,-1.850) -- (-1.110,-1.812) -- (-1.173,-1.772) -- (-1.234,-1.730) -- (-1.293,-1.686) -- (-1.351,-1.640);
\draw[tait] (-2.000,-0.718) -- (-2.016,-0.671) -- (-2.031,-0.625) -- (-2.045,-0.577) -- (-2.058,-0.530) -- (-2.070,-0.482) -- (-2.080,-0.434) -- (-2.090,-0.386) -- (-2.098,-0.338) -- (-2.105,-0.289) -- (-2.111,-0.240) -- (-2.116,-0.191) -- (-2.120,-0.139) -- (-2.123,-0.081) -- (-2.125,-0.023) -- (-2.125,0.034) -- (-2.123,0.092) -- (-2.120,0.150) -- (-2.115,0.207) -- (-2.108,0.265) -- (-2.100,0.322) -- (-2.091,0.379) -- (-2.080,0.436) -- (-2.067,0.492) -- (-2.053,0.548) -- (-2.037,0.604) -- (-2.020,0.659) -- (-2.002,0.713) -- (-1.982,0.768) -- (-1.960,0.821);
\fill (-1.597,-1.402) circle (0.035);
\fill (-1.738,-1.223) circle (0.035);
\fill (-1.859,-1.029) circle (0.035);
\node[vert] (v0) at (0.000,0.000) {$v_{5}$};
\node[vert] (v1) at (-1.960,0.821) {$v_{1}$};
\node[vert] (v2) at (0.968,1.892) {$v_{2}$};
\node[vert] (v3) at (1.813,-1.109) {$v_{3}$};
\node[vert] (v4) at (-1.046,-1.850) {$v_{4}$};
\end{tikzpicture}

%% file: figures/Lh.tex
\begin{tikzpicture}[line width=1.00pt, line cap=round, every node/.style={font=\small}]
\draw (0.000,0.000) -- (0.600,0.000);
\draw (0.000,1.800) -- (0.600,1.800);
\draw (0.000,0.600) .. controls (0.300,0.600) and (0.300,1.200) .. (0.600,1.200);
\draw[white, line width=4.0pt, line cap=butt] (0.149,1.138) .. controls (0.271,1.022) and (0.329,0.778) .. (0.451,0.662);
\draw (0.000,1.200) .. controls (0.300,1.200) and (0.300,0.600) .. (0.600,0.600);
\draw (0.600,0.000) -- (1.200,0.000);
\draw (0.600,1.800) -- (1.200,1.800);
\draw (0.600,0.600) .. controls (0.900,0.600) and (0.900,1.200) .. (1.200,1.200);
\draw[white, line width=4.0pt, line cap=butt] (0.749,1.138) .. controls (0.871,1.022) and (0.929,0.778) .. (1.051,0.662);
\draw (0.600,1.200) .. controls (0.900,1.200) and (0.900,0.600) .. (1.200,0.600);
\draw (1.200,0.000) -- (1.800,0.000);
\draw (1.200,0.600) -- (1.800,0.600);
\draw (1.200,1.200) .. controls (1.500,1.200) and (1.500,1.800) .. (1.800,1.800);
\draw[white, line width=4.0pt, line cap=butt] (1.349,1.738) .. controls (1.471,1.622) and (1.529,1.378) .. (1.651,1.262);
\draw (1.200,1.800) .. controls (1.500,1.800) and (1.500,1.200) .. (1.800,1.200);
\draw (1.800,0.000) -- (2.400,0.000);
\draw (1.800,0.600) -- (2.400,0.600);
\draw (1.800,1.200) .. controls (2.100,1.200) and (2.100,1.800) .. (2.400,1.800);
\draw[white, line width=4.0pt, line cap=butt] (1.949,1.738) .. controls (2.071,1.622) and (2.129,1.378) .. (2.251,1.262);
\draw (1.800,1.800) .. controls (2.100,1.800) and (2.100,1.200) .. (2.400,1.200);
\draw (2.400,0.000) -- (3.500,0.000);
\draw (2.400,0.600) -- (2.620,0.600);
\draw (3.280,0.600) -- (3.500,0.600);
\node at (2.950,0.600) {$\cdots$};
\draw (2.400,1.200) -- (2.620,1.200);
\draw (3.280,1.200) -- (3.500,1.200);
\node at (2.950,1.200) {$\cdots$};
\draw (2.400,1.800) -- (2.620,1.800);
\draw (3.280,1.800) -- (3.500,1.800);
\node at (2.950,1.800) {$\cdots$};
\draw (3.500,0.000) -- (4.100,0.000);
\draw (3.500,0.600) -- (4.100,0.600);
\draw (3.500,1.200) .. controls (3.800,1.200) and (3.800,1.800) .. (4.100,1.800);
\draw[white, line width=4.0pt, line cap=butt] (3.649,1.738) .. controls (3.771,1.622) and (3.829,1.378) .. (3.951,1.262);
\draw (3.500,1.800) .. controls (3.800,1.800) and (3.800,1.200) .. (4.100,1.200);
\draw (4.100,0.000) -- (4.700,0.000);
\draw (4.100,0.600) -- (4.700,0.600);
\draw (4.100,1.200) .. controls (4.400,1.200) and (4.400,1.800) .. (4.700,1.800);
\draw[white, line width=4.0pt, line cap=butt] (4.249,1.738) .. controls (4.371,1.622) and (4.429,1.378) .. (4.551,1.262);
\draw (4.100,1.800) .. controls (4.400,1.800) and (4.400,1.200) .. (4.700,1.200);
\draw[thin, decorate, decoration={brace, mirror, amplitude=4pt}] (0.050,-0.200) -- (4.650,-0.200) node[midway, below=5pt] {$2^{[h-1]}$};
\draw (4.700,0.000) -- (5.300,0.000);
\draw (4.700,1.800) -- (5.300,1.800);
\draw (4.700,0.600) .. controls (5.000,0.600) and (5.000,1.200) .. (5.300,1.200);
\draw[white, line width=4.0pt, line cap=butt] (4.849,1.138) .. controls (4.971,1.022) and (5.029,0.778) .. (5.151,0.662);
\draw (4.700,1.200) .. controls (5.000,1.200) and (5.000,0.600) .. (5.300,0.600);
\draw (5.300,0.000) -- (5.900,0.000);
\draw (5.300,1.800) -- (5.900,1.800);
\draw (5.300,0.600) .. controls (5.600,0.600) and (5.600,1.200) .. (5.900,1.200);
\draw[white, line width=4.0pt, line cap=butt] (5.449,1.138) .. controls (5.571,1.022) and (5.629,0.778) .. (5.751,0.662);
\draw (5.300,1.200) .. controls (5.600,1.200) and (5.600,0.600) .. (5.900,0.600);
\draw (5.900,0.000) -- (6.500,0.000);
\draw (5.900,1.800) -- (6.500,1.800);
\draw (5.900,0.600) .. controls (6.200,0.600) and (6.200,1.200) .. (6.500,1.200);
\draw[white, line width=4.0pt, line cap=butt] (6.049,1.138) .. controls (6.171,1.022) and (6.229,0.778) .. (6.351,0.662);
\draw (5.900,1.200) .. controls (6.200,1.200) and (6.200,0.600) .. (6.500,0.600);
\draw (6.500,0.000) -- (8.000,0.000);
\draw (6.500,0.600) -- (8.000,0.600);
\draw (6.500,1.200) -- (8.000,1.200);
\draw (6.500,1.800) -- (8.000,1.800);
\draw[fill=white] (6.650,1.000) rectangle (7.850,2.000);
\node at (7.250,1.500) {$2h+1$};
\draw (8.000,0.000) -- (8.600,0.000);
\draw (8.000,1.800) -- (8.600,1.800);
\draw (8.000,0.600) .. controls (8.300,0.600) and (8.300,1.200) .. (8.600,1.200);
\draw[white, line width=4.0pt, line cap=butt] (8.149,1.138) .. controls (8.271,1.022) and (8.329,0.778) .. (8.451,0.662);
\draw (8.000,1.200) .. controls (8.300,1.200) and (8.300,0.600) .. (8.600,0.600);
\draw (8.600,0.000) -- (9.200,0.000);
\draw (8.600,1.800) -- (9.200,1.800);
\draw (8.600,0.600) .. controls (8.900,0.600) and (8.900,1.200) .. (9.200,1.200);
\draw[white, line width=4.0pt, line cap=butt] (8.749,1.138) .. controls (8.871,1.022) and (8.929,0.778) .. (9.051,0.662);
\draw (8.600,1.200) .. controls (8.900,1.200) and (8.900,0.600) .. (9.200,0.600);
\draw (9.200,1.800) .. controls (9.600,1.800) and (9.600,1.200) .. (9.200,1.200);
\draw (9.200,0.600) .. controls (9.600,0.600) and (9.600,0.000) .. (9.200,0.000);
\draw (0.000,1.800) .. controls (-0.400,1.800) and (-0.400,1.200) .. (0.000,1.200);
\draw (0.000,0.600) .. controls (-0.400,0.600) and (-0.400,0.000) .. (0.000,0.000);
\end{tikzpicture}

%% file: figures/Kmn.tex
\begin{tikzpicture}[line width=1.00pt, line cap=round, every node/.style={font=\small}]
\draw (0.000,0.000) -- (1.500,0.000);
\draw (0.000,0.600) -- (1.500,0.600);
\draw (0.000,1.200) -- (1.500,1.200);
\draw (0.000,1.800) -- (1.500,1.800);
\draw[fill=white] (0.150,0.400) rectangle (1.350,1.400);
\node at (0.750,0.900) {$2n+1$};
\draw (1.500,0.000) -- (2.100,0.000);
\draw (1.500,0.600) -- (2.100,0.600);
\draw (1.500,1.200) .. controls (1.800,1.200) and (1.800,1.800) .. (2.100,1.800);
\draw[white, line width=4.0pt, line cap=butt] (1.649,1.738) .. controls (1.771,1.622) and (1.829,1.378) .. (1.951,1.262);
\draw (1.500,1.800) .. controls (1.800,1.800) and (1.800,1.200) .. (2.100,1.200);
\draw (2.100,0.000) -- (2.700,0.000);
\draw (2.100,0.600) -- (2.700,0.600);
\draw (2.100,1.200) .. controls (2.400,1.200) and (2.400,1.800) .. (2.700,1.800);
\draw[white, line width=4.0pt, line cap=butt] (2.249,1.738) .. controls (2.371,1.622) and (2.429,1.378) .. (2.551,1.262);
\draw (2.100,1.800) .. controls (2.400,1.800) and (2.400,1.200) .. (2.700,1.200);
\draw (2.700,0.000) -- (3.300,0.000);
\draw (2.700,1.800) -- (3.300,1.800);
\draw (2.700,0.600) .. controls (3.000,0.600) and (3.000,1.200) .. (3.300,1.200);
\draw[white, line width=4.0pt, line cap=butt] (2.849,1.138) .. controls (2.971,1.022) and (3.029,0.778) .. (3.151,0.662);
\draw (2.700,1.200) .. controls (3.000,1.200) and (3.000,0.600) .. (3.300,0.600);
\draw (3.300,0.000) -- (3.900,0.000);
\draw (3.300,1.800) -- (3.900,1.800);
\draw (3.300,0.600) .. controls (3.600,0.600) and (3.600,1.200) .. (3.900,1.200);
\draw[white, line width=4.0pt, line cap=butt] (3.449,1.138) .. controls (3.571,1.022) and (3.629,0.778) .. (3.751,0.662);
\draw (3.300,1.200) .. controls (3.600,1.200) and (3.600,0.600) .. (3.900,0.600);
\draw (3.900,0.000) -- (5.000,0.000);
\draw (3.900,0.600) -- (4.120,0.600);
\draw (4.780,0.600) -- (5.000,0.600);
\node at (4.450,0.600) {$\cdots$};
\draw (3.900,1.200) -- (4.120,1.200);
\draw (4.780,1.200) -- (5.000,1.200);
\node at (4.450,1.200) {$\cdots$};
\draw (3.900,1.800) -- (4.120,1.800);
\draw (4.780,1.800) -- (5.000,1.800);
\node at (4.450,1.800) {$\cdots$};
\draw (5.000,0.000) -- (5.600,0.000);
\draw (5.000,1.800) -- (5.600,1.800);
\draw (5.000,0.600) .. controls (5.300,0.600) and (5.300,1.200) .. (5.600,1.200);
\draw[white, line width=4.0pt, line cap=butt] (5.149,1.138) .. controls (5.271,1.022) and (5.329,0.778) .. (5.451,0.662);
\draw (5.000,1.200) .. controls (5.300,1.200) and (5.300,0.600) .. (5.600,0.600);
\draw (5.600,0.000) -- (6.200,0.000);
\draw (5.600,1.800) -- (6.200,1.800);
\draw (5.600,0.600) .. controls (5.900,0.600) and (5.900,1.200) .. (6.200,1.200);
\draw[white, line width=4.0pt, line cap=butt] (5.749,1.138) .. controls (5.871,1.022) and (5.929,0.778) .. (6.051,0.662);
\draw (5.600,1.200) .. controls (5.900,1.200) and (5.900,0.600) .. (6.200,0.600);
\draw[thin, decorate, decoration={brace, mirror, amplitude=4pt}] (1.550,-0.200) -- (6.150,-0.200) node[midway, below=5pt] {$2^{[2m-2]}$};
\draw (6.200,0.000) -- (6.800,0.000);
\draw (6.200,0.600) -- (6.800,0.600);
\draw (6.200,1.200) .. controls (6.500,1.200) and (6.500,1.800) .. (6.800,1.800);
\draw[white, line width=4.0pt, line cap=butt] (6.349,1.738) .. controls (6.471,1.622) and (6.529,1.378) .. (6.651,1.262);
\draw (6.200,1.800) .. controls (6.500,1.800) and (6.500,1.200) .. (6.800,1.200);
\draw (6.800,0.000) -- (7.400,0.000);
\draw (6.800,0.600) -- (7.400,0.600);
\draw (6.800,1.200) .. controls (7.100,1.200) and (7.100,1.800) .. (7.400,1.800);
\draw[white, line width=4.0pt, line cap=butt] (6.949,1.738) .. controls (7.071,1.622) and (7.129,1.378) .. (7.251,1.262);
\draw (6.800,1.800) .. controls (7.100,1.800) and (7.100,1.200) .. (7.400,1.200);
\draw (7.400,0.000) -- (8.000,0.000);
\draw (7.400,0.600) -- (8.000,0.600);
\draw (7.400,1.200) .. controls (7.700,1.200) and (7.700,1.800) .. (8.000,1.800);
\draw[white, line width=4.0pt, line cap=butt] (7.549,1.738) .. controls (7.671,1.622) and (7.729,1.378) .. (7.851,1.262);
\draw (7.400,1.800) .. controls (7.700,1.800) and (7.700,1.200) .. (8.000,1.200);
\draw (8.000,0.000) -- (8.600,0.000);
\draw (8.000,0.600) -- (8.600,0.600);
\draw (8.000,1.200) .. controls (8.300,1.200) and (8.300,1.800) .. (8.600,1.800);
\draw[white, line width=4.0pt, line cap=butt] (8.149,1.738) .. controls (8.271,1.622) and (8.329,1.378) .. (8.451,1.262);
\draw (8.000,1.800) .. controls (8.300,1.800) and (8.300,1.200) .. (8.600,1.200);
\draw (8.600,0.000) -- (9.200,0.000);
\draw (8.600,1.800) -- (9.200,1.800);
\draw (8.600,0.600) .. controls (8.900,0.600) and (8.900,1.200) .. (9.200,1.200);
\draw[white, line width=4.0pt, line cap=butt] (8.749,1.138) .. controls (8.871,1.022) and (8.929,0.778) .. (9.051,0.662);
\draw (8.600,1.200) .. controls (8.900,1.200) and (8.900,0.600) .. (9.200,0.600);
\draw (9.200,0.000) -- (9.800,0.000);
\draw (9.200,1.800) -- (9.800,1.800);
\draw (9.200,0.600) .. controls (9.500,0.600) and (9.500,1.200) .. (9.800,1.200);
\draw[white, line width=4.0pt, line cap=butt] (9.349,1.138) .. controls (9.471,1.022) and (9.529,0.778) .. (9.651,0.662);
\draw (9.200,1.200) .. controls (9.500,1.200) and (9.500,0.600) .. (9.800,0.600);
\draw (9.800,0.000) -- (10.400,0.000);
\draw (9.800,1.800) -- (10.400,1.800);
\draw (9.800,0.600) .. controls (10.100,0.600) and (10.100,1.200) .. (10.400,1.200);
\draw[white, line width=4.0pt, line cap=butt] (9.949,1.138) .. controls (10.071,1.022) and (10.129,0.778) .. (10.251,0.662);
\draw (9.800,1.200) .. controls (10.100,1.200) and (10.100,0.600) .. (10.400,0.600);
\draw (10.400,0.000) -- (11.000,0.000);
\draw (10.400,1.800) -- (11.000,1.800);
\draw (10.400,0.600) .. controls (10.700,0.600) and (10.700,1.200) .. (11.000,1.200);
\draw[white, line width=4.0pt, line cap=butt] (10.549,1.138) .. controls (10.671,1.022) and (10.729,0.778) .. (10.851,0.662);
\draw (10.400,1.200) .. controls (10.700,1.200) and (10.700,0.600) .. (11.000,0.600);
\draw (11.000,0.000) -- (11.600,0.000);
\draw (11.000,1.800) -- (11.600,1.800);
\draw (11.000,0.600) .. controls (11.300,0.600) and (11.300,1.200) .. (11.600,1.200);
\draw[white, line width=4.0pt, line cap=butt] (11.149,1.138) .. controls (11.271,1.022) and (11.329,0.778) .. (11.451,0.662);
\draw (11.000,1.200) .. controls (11.300,1.200) and (11.300,0.600) .. (11.600,0.600);
\draw (11.600,0.000) -- (12.200,0.000);
\draw (11.600,0.600) -- (12.200,0.600);
\draw (11.600,1.200) .. controls (11.900,1.200) and (11.900,1.800) .. (12.200,1.800);
\draw[white, line width=4.0pt, line cap=butt] (11.749,1.738) .. controls (11.871,1.622) and (11.929,1.378) .. (12.051,1.262);
\draw (11.600,1.800) .. controls (11.900,1.800) and (11.900,1.200) .. (12.200,1.200);
\draw (12.200,0.000) -- (12.800,0.000);
\draw (12.200,0.600) -- (12.800,0.600);
\draw (12.200,1.200) .. controls (12.500,1.200) and (12.500,1.800) .. (12.800,1.800);
\draw[white, line width=4.0pt, line cap=butt] (12.349,1.738) .. controls (12.471,1.622) and (12.529,1.378) .. (12.651,1.262);
\draw (12.200,1.800) .. controls (12.500,1.800) and (12.500,1.200) .. (12.800,1.200);
\draw (12.800,0.000) -- (13.400,0.000);
\draw (12.800,0.600) -- (13.400,0.600);
\draw (12.800,1.200) .. controls (13.100,1.200) and (13.100,1.800) .. (13.400,1.800);
\draw[white, line width=4.0pt, line cap=butt] (12.949,1.738) .. controls (13.071,1.622) and (13.129,1.378) .. (13.251,1.262);
\draw (12.800,1.800) .. controls (13.100,1.800) and (13.100,1.200) .. (13.400,1.200);
\draw (13.400,0.000) -- (14.000,0.000);
\draw (13.400,0.600) -- (14.000,0.600);
\draw (13.400,1.200) .. controls (13.700,1.200) and (13.700,1.800) .. (14.000,1.800);
\draw[white, line width=4.0pt, line cap=butt] (13.549,1.738) .. controls (13.671,1.622) and (13.729,1.378) .. (13.851,1.262);
\draw (13.400,1.800) .. controls (13.700,1.800) and (13.700,1.200) .. (14.000,1.200);
\draw (14.000,1.200) .. controls (14.400,1.200) and (14.400,0.600) .. (14.000,0.600);
\draw (14.000,1.800) .. controls (15.200,1.800) and (15.200,0.000) .. (14.000,0.000);
\draw (0.000,1.800) .. controls (-0.400,1.800) and (-0.400,1.200) .. (0.000,1.200);
\draw (0.000,0.600) .. controls (-0.400,0.600) and (-0.400,0.000) .. (0.000,0.000);
\end{tikzpicture}

%% file: figures/K354.tex
\begin{tikzpicture}[line width=1.00pt, line cap=round, every node/.style={font=\small}]
\draw (0.000,0.000) -- (0.600,0.000);
\draw (0.000,1.800) -- (0.600,1.800);
\draw (0.000,0.600) .. controls (0.300,0.600) and (0.300,1.200) .. (0.600,1.200);
\draw[white, line width=4.0pt, line cap=butt] (0.149,1.138) .. controls (0.271,1.022) and (0.329,0.778) .. (0.451,0.662);
\draw (0.000,1.200) .. controls (0.300,1.200) and (0.300,0.600) .. (0.600,0.600);
\draw (0.600,0.000) -- (1.200,0.000);
\draw (0.600,1.800) -- (1.200,1.800);
\draw (0.600,0.600) .. controls (0.900,0.600) and (0.900,1.200) .. (1.200,1.200);
\draw[white, line width=4.0pt, line cap=butt] (0.749,1.138) .. controls (0.871,1.022) and (0.929,0.778) .. (1.051,0.662);
\draw (0.600,1.200) .. controls (0.900,1.200) and (0.900,0.600) .. (1.200,0.600);
\draw (1.200,0.000) -- (1.800,0.000);
\draw (1.200,1.800) -- (1.800,1.800);
\draw (1.200,0.600) .. controls (1.500,0.600) and (1.500,1.200) .. (1.800,1.200);
\draw[white, line width=4.0pt, line cap=butt] (1.349,1.138) .. controls (1.471,1.022) and (1.529,0.778) .. (1.651,0.662);
\draw (1.200,1.200) .. controls (1.500,1.200) and (1.500,0.600) .. (1.800,0.600);
\draw (1.800,0.000) -- (2.400,0.000);
\draw (1.800,0.600) -- (2.400,0.600);
\draw (1.800,1.200) .. controls (2.100,1.200) and (2.100,1.800) .. (2.400,1.800);
\draw[white, line width=4.0pt, line cap=butt] (1.949,1.738) .. controls (2.071,1.622) and (2.129,1.378) .. (2.251,1.262);
\draw (1.800,1.800) .. controls (2.100,1.800) and (2.100,1.200) .. (2.400,1.200);
\draw (2.400,0.000) -- (3.000,0.000);
\draw (2.400,0.600) -- (3.000,0.600);
\draw (2.400,1.200) .. controls (2.700,1.200) and (2.700,1.800) .. (3.000,1.800);
\draw[white, line width=4.0pt, line cap=butt] (2.549,1.738) .. controls (2.671,1.622) and (2.729,1.378) .. (2.851,1.262);
\draw (2.400,1.800) .. controls (2.700,1.800) and (2.700,1.200) .. (3.000,1.200);
\draw (3.000,0.000) -- (3.600,0.000);
\draw (3.000,0.600) -- (3.600,0.600);
\draw (3.000,1.200) .. controls (3.300,1.200) and (3.300,1.800) .. (3.600,1.800);
\draw[white, line width=4.0pt, line cap=butt] (3.149,1.738) .. controls (3.271,1.622) and (3.329,1.378) .. (3.451,1.262);
\draw (3.000,1.800) .. controls (3.300,1.800) and (3.300,1.200) .. (3.600,1.200);
\draw (3.600,0.000) -- (4.200,0.000);
\draw (3.600,0.600) -- (4.200,0.600);
\draw (3.600,1.200) .. controls (3.900,1.200) and (3.900,1.800) .. (4.200,1.800);
\draw[white, line width=4.0pt, line cap=butt] (3.749,1.738) .. controls (3.871,1.622) and (3.929,1.378) .. (4.051,1.262);
\draw (3.600,1.800) .. controls (3.900,1.800) and (3.900,1.200) .. (4.200,1.200);
\draw (4.200,0.000) -- (4.800,0.000);
\draw (4.200,0.600) -- (4.800,0.600);
\draw (4.200,1.200) .. controls (4.500,1.200) and (4.500,1.800) .. (4.800,1.800);
\draw[white, line width=4.0pt, line cap=butt] (4.349,1.738) .. controls (4.471,1.622) and (4.529,1.378) .. (4.651,1.262);
\draw (4.200,1.800) .. controls (4.500,1.800) and (4.500,1.200) .. (4.800,1.200);
\draw (4.800,0.000) -- (5.400,0.000);
\draw (4.800,1.800) -- (5.400,1.800);
\draw (4.800,0.600) .. controls (5.100,0.600) and (5.100,1.200) .. (5.400,1.200);
\draw[white, line width=4.0pt, line cap=butt] (4.949,1.138) .. controls (5.071,1.022) and (5.129,0.778) .. (5.251,0.662);
\draw (4.800,1.200) .. controls (5.100,1.200) and (5.100,0.600) .. (5.400,0.600);
\draw (5.400,0.000) -- (6.000,0.000);
\draw (5.400,1.800) -- (6.000,1.800);
\draw (5.400,0.600) .. controls (5.700,0.600) and (5.700,1.200) .. (6.000,1.200);
\draw[white, line width=4.0pt, line cap=butt] (5.549,1.138) .. controls (5.671,1.022) and (5.729,0.778) .. (5.851,0.662);
\draw (5.400,1.200) .. controls (5.700,1.200) and (5.700,0.600) .. (6.000,0.600);
\draw (6.000,0.000) -- (6.600,0.000);
\draw (6.000,1.800) -- (6.600,1.800);
\draw (6.000,0.600) .. controls (6.300,0.600) and (6.300,1.200) .. (6.600,1.200);
\draw[white, line width=4.0pt, line cap=butt] (6.149,1.138) .. controls (6.271,1.022) and (6.329,0.778) .. (6.451,0.662);
\draw (6.000,1.200) .. controls (6.300,1.200) and (6.300,0.600) .. (6.600,0.600);
\draw (6.600,0.000) -- (7.200,0.000);
\draw (6.600,1.800) -- (7.200,1.800);
\draw (6.600,0.600) .. controls (6.900,0.600) and (6.900,1.200) .. (7.200,1.200);
\draw[white, line width=4.0pt, line cap=butt] (6.749,1.138) .. controls (6.871,1.022) and (6.929,0.778) .. (7.051,0.662);
\draw (6.600,1.200) .. controls (6.900,1.200) and (6.900,0.600) .. (7.200,0.600);
\draw (7.200,1.800) .. controls (7.600,1.800) and (7.600,1.200) .. (7.200,1.200);
\draw (7.200,0.600) .. controls (7.600,0.600) and (7.600,0.000) .. (7.200,0.000);
\draw (0.000,1.800) .. controls (-0.400,1.800) and (-0.400,1.200) .. (0.000,1.200);
\draw (0.000,0.600) .. controls (-0.400,0.600) and (-0.400,0.000) .. (0.000,0.000);
\end{tikzpicture}

%% file: figures/Jq-nonalt-highlight.tex
\begin{tikzpicture}[line width=1.00pt, line cap=round, every node/.style={font=\small}]
\draw (0.000,-1.800) -- (0.600,-1.800);
\draw (0.000,0.000) -- (0.600,0.000);
\draw (0.000,-1.200) .. controls (0.300,-1.200) and (0.300,-0.600) .. (0.600,-0.600);
\draw[white, line width=4.0pt, line cap=butt] (0.149,-0.662) .. controls (0.271,-0.778) and (0.329,-1.022) .. (0.451,-1.138);
\draw (0.000,-0.600) .. controls (0.300,-0.600) and (0.300,-1.200) .. (0.600,-1.200);
\draw (0.600,-1.800) -- (1.200,-1.800);
\draw (0.600,-1.200) -- (1.200,-1.200);
\draw (0.600,0.000) .. controls (0.900,0.000) and (0.900,-0.600) .. (1.200,-0.600);
\draw[white, line width=4.0pt, line cap=butt] (0.749,-0.538) .. controls (0.871,-0.422) and (0.929,-0.178) .. (1.051,-0.062);
\draw (0.600,-0.600) .. controls (0.900,-0.600) and (0.900,0.000) .. (1.200,0.000);
\draw (1.200,-1.800) -- (2.700,-1.800);
\draw (1.200,-1.200) -- (2.700,-1.200);
\draw (1.200,-0.600) -- (2.700,-0.600);
\draw (1.200,0.000) -- (2.700,0.000);
\draw[fill=white] (1.350,-1.400) rectangle (2.550,-0.400);
\node at (1.950,-0.900) {$q+1$};
\draw (2.700,0.000) .. controls (3.100,0.000) and (3.100,-0.600) .. (2.700,-0.600);
\draw (2.700,-1.200) .. controls (3.100,-1.200) and (3.100,-1.800) .. (2.700,-1.800);
\draw[rounded corners=29.88pt] (0.000,0.000) -- (-2.150,0.000) -- (-2.150,-1.100);
\draw[rounded corners=12.80pt] (0.000,-0.600) -- (-1.550,-0.600) -- (-1.550,-1.100);
\draw[rounded corners=28.45pt] (0.000,-1.200) -- (-1.000,-1.200) -- (-1.000,-2.200);
\draw[rounded corners=11.38pt] (0.000,-1.800) -- (-0.400,-1.800) -- (-0.400,-2.200);
\draw (0.000,-4.400) -- (0.600,-4.400);
\draw (0.000,-2.600) -- (0.600,-2.600);
\draw (0.000,-3.800) .. controls (0.300,-3.800) and (0.300,-3.200) .. (0.600,-3.200);
\draw[white, line width=4.0pt, line cap=butt] (0.149,-3.262) .. controls (0.271,-3.378) and (0.329,-3.622) .. (0.451,-3.738);
\draw (0.000,-3.200) .. controls (0.300,-3.200) and (0.300,-3.800) .. (0.600,-3.800);
\draw (0.600,-4.400) -- (1.200,-4.400);
\draw (0.600,-2.600) -- (1.200,-2.600);
\draw (0.600,-3.800) .. controls (0.900,-3.800) and (0.900,-3.200) .. (1.200,-3.200);
\draw[white, line width=4.0pt, line cap=butt] (0.749,-3.262) .. controls (0.871,-3.378) and (0.929,-3.622) .. (1.051,-3.738);
\draw (0.600,-3.200) .. controls (0.900,-3.200) and (0.900,-3.800) .. (1.200,-3.800);
\draw (1.200,-4.400) -- (1.800,-4.400);
\draw (1.200,-3.800) -- (1.800,-3.800);
\draw (1.200,-2.600) .. controls (1.500,-2.600) and (1.500,-3.200) .. (1.800,-3.200);
\draw[white, line width=4.0pt, line cap=butt] (1.349,-3.138) .. controls (1.471,-3.022) and (1.529,-2.778) .. (1.651,-2.662);
\draw (1.200,-3.200) .. controls (1.500,-3.200) and (1.500,-2.600) .. (1.800,-2.600);
\draw (1.800,-4.400) -- (2.400,-4.400);
\draw (1.800,-2.600) -- (2.400,-2.600);
\draw (1.800,-3.800) .. controls (2.100,-3.800) and (2.100,-3.200) .. (2.400,-3.200);
\draw[white, line width=4.0pt, line cap=butt] (1.949,-3.262) .. controls (2.071,-3.378) and (2.129,-3.622) .. (2.251,-3.738);
\draw (1.800,-3.200) .. controls (2.100,-3.200) and (2.100,-3.800) .. (2.400,-3.800);
\draw (2.400,-4.400) -- (3.000,-4.400);
\draw (2.400,-2.600) -- (3.000,-2.600);
\draw (2.400,-3.800) .. controls (2.700,-3.800) and (2.700,-3.200) .. (3.000,-3.200);
\draw[white, line width=4.0pt, line cap=butt] (2.549,-3.262) .. controls (2.671,-3.378) and (2.729,-3.622) .. (2.851,-3.738);
\draw (2.400,-3.200) .. controls (2.700,-3.200) and (2.700,-3.800) .. (3.000,-3.800);
\draw (3.000,-2.600) .. controls (3.400,-2.600) and (3.400,-3.200) .. (3.000,-3.200);
\draw (3.000,-3.800) .. controls (3.400,-3.800) and (3.400,-4.400) .. (3.000,-4.400);
\draw[rounded corners=11.38pt] (0.000,-2.600) -- (-0.400,-2.600) -- (-0.400,-2.200);
\draw[rounded corners=28.45pt] (0.000,-3.200) -- (-1.000,-3.200) -- (-1.000,-2.200);
\draw[rounded corners=28.45pt] (0.000,-3.800) -- (-1.000,-3.800) -- (-1.000,-4.800);
\draw[rounded corners=11.38pt] (0.000,-4.400) -- (-0.400,-4.400) -- (-0.400,-4.800);
\draw (0.000,-7.000) -- (0.600,-7.000);
\draw (0.000,-5.200) -- (0.600,-5.200);
\draw (0.000,-5.800) .. controls (0.300,-5.800) and (0.300,-6.400) .. (0.600,-6.400);
\draw[white, line width=4.0pt, line cap=butt] (0.149,-6.338) .. controls (0.271,-6.222) and (0.329,-5.978) .. (0.451,-5.862);
\draw (0.000,-6.400) .. controls (0.300,-6.400) and (0.300,-5.800) .. (0.600,-5.800);
\draw (0.600,-7.000) -- (1.200,-7.000);
\draw (0.600,-5.200) -- (1.200,-5.200);
\draw (0.600,-5.800) .. controls (0.900,-5.800) and (0.900,-6.400) .. (1.200,-6.400);
\draw[white, line width=4.0pt, line cap=butt] (0.749,-6.338) .. controls (0.871,-6.222) and (0.929,-5.978) .. (1.051,-5.862);
\draw (0.600,-6.400) .. controls (0.900,-6.400) and (0.900,-5.800) .. (1.200,-5.800);
\draw (1.200,-7.000) -- (1.800,-7.000);
\draw (1.200,-5.200) -- (1.800,-5.200);
\draw (1.200,-5.800) .. controls (1.500,-5.800) and (1.500,-6.400) .. (1.800,-6.400);
\draw[white, line width=4.0pt, line cap=butt] (1.349,-6.338) .. controls (1.471,-6.222) and (1.529,-5.978) .. (1.651,-5.862);
\draw (1.200,-6.400) .. controls (1.500,-6.400) and (1.500,-5.800) .. (1.800,-5.800);
\draw (1.800,-5.200) .. controls (2.200,-5.200) and (2.200,-5.800) .. (1.800,-5.800);
\draw (1.800,-6.400) .. controls (2.200,-6.400) and (2.200,-7.000) .. (1.800,-7.000);
\draw[rounded corners=11.38pt] (0.000,-5.200) -- (-0.400,-5.200) -- (-0.400,-4.800);
\draw[rounded corners=28.45pt] (0.000,-5.800) -- (-1.000,-5.800) -- (-1.000,-4.800);
\draw[rounded corners=29.88pt] (0.000,-7.000) -- (-2.150,-7.000) -- (-2.150,-5.900);
\draw[rounded corners=12.80pt] (0.000,-6.400) -- (-1.550,-6.400) -- (-1.550,-5.900);
\draw (-1.550,-1.100) -- (-1.550,-3.500);
\draw (-1.550,-3.500) -- (-1.550,-5.900);
\draw (-2.150,-1.100) -- (-2.150,-3.500);
\draw (-2.150,-3.500) -- (-2.150,-5.900);
\draw[red!75!black, line width=1.6pt] (0.000,-0.600) .. controls (0.300,-0.600) and (0.300,-1.200) .. (0.600,-1.200);
\draw[red!75!black, line width=1.6pt, rounded corners=12.80pt] (0.000,-0.600) -- (-1.550,-0.600) -- (-1.550,-1.100);
\draw[red!75!black, line width=1.6pt] (-1.550,-1.100) -- (-1.550,-3.500);
\draw[red!75!black, line width=1.6pt] (-1.550,-3.500) -- (-1.550,-5.900);
\draw[red!75!black, line width=1.6pt, rounded corners=12.80pt] (0.000,-6.400) -- (-1.550,-6.400) -- (-1.550,-5.900);
\draw[red!75!black, line width=1.6pt] (0.000,-6.400) .. controls (0.300,-6.400) and (0.300,-5.800) .. (0.600,-5.800);
\draw[red!75!black, line width=0.8pt] (0.300,-0.900) circle (0.27);
\draw[red!75!black, line width=0.8pt] (0.300,-6.100) circle (0.27);
\end{tikzpicture}

%% file: figures/sq-plumbing.tex
\begin{tikzpicture}[
  vtx/.style={circle,fill,inner sep=1.8pt},
  wt/.style={above=2pt,font=\small},
  lab/.style={below=3pt,font=\small}]
  \node[vtx,label={[wt]$2$},label={[lab]$v_1$}]       (v1) at (0,0) {};
  \node                                              (dots) at (1.1,0) {$\cdots$};
  \node[vtx,label={[wt]$2$},label={[lab]$v_q$}]       (vq) at (2.2,0) {};
  \node[vtx,label={[wt]$3$},label={[lab]$v_{q+1}$}]   (a1) at (3.6,0) {};
  \node[vtx,label={[wt]$2$},label={[lab,xshift=-12pt,yshift=0pt]$v_{q+2}$}] (c)  at (5,0) {};
  \node[vtx,label={[wt]$2$},label={[lab]$v_{q+3}$}]   (b1) at (6.4,0) {};
  \node[vtx,label={[wt]$3$},label={[lab]$v_{q+4}$}]   (b2) at (7.8,0) {};
  \node[vtx,label={[wt]$2$},label={[lab]$v_{q+5}$}]   (b3) at (9.2,0) {};
  \node[vtx,label={[left=0pt,xshift=2pt,yshift=2pt,font=\small]$v_{q+6}$},label={[lab,yshift=-2pt]$3$}] (d) at (5,-1.2) {};
  \draw (v1) -- (dots) -- (vq) -- (a1) -- (c) -- (b1) -- (b2) -- (b3);
  \draw (c) -- (d);
\end{tikzpicture}

%% file: figures/cc3-1-13a4237.tex
\begin{tikzpicture}[line width=1.0pt, line cap=round, line join=round]
\draw (4.014,-5.640) -- (3.976,-5.590) -- (3.947,-5.537) -- (3.927,-5.483) -- (3.918,-5.428) -- (3.918,-5.372) -- (3.927,-5.317) -- (3.947,-5.263) -- (3.976,-5.210) -- (4.014,-5.160) -- (4.061,-5.113) -- (4.116,-5.069) -- (4.180,-5.029) -- (4.251,-4.993) -- (4.329,-4.961) -- (4.413,-4.933) -- (4.502,-4.908) -- (4.597,-4.887) -- (4.696,-4.869) -- (4.799,-4.854) -- (4.906,-4.842) -- (5.015,-4.832) -- (5.126,-4.825) -- (5.239,-4.818) -- (5.354,-4.814) -- (5.470,-4.810) -- (5.587,-4.807) -- (5.705,-4.805) -- (5.824,-4.804) -- (5.942,-4.802) -- (6.062,-4.802) -- (6.181,-4.801) -- (6.301,-4.801) -- (6.420,-4.800) -- (6.540,-4.800) -- (6.660,-4.800) -- (6.780,-4.800) -- (6.899,-4.799) -- (7.019,-4.799) -- (7.138,-4.798) -- (7.258,-4.798) -- (7.376,-4.796) -- (7.495,-4.795) -- (7.613,-4.793) -- (7.730,-4.790) -- (7.846,-4.786) -- (7.961,-4.781) -- (8.075,-4.775) -- (8.186,-4.767) -- (8.296,-4.756) -- (8.403,-4.743) -- (8.507,-4.727) -- (8.608,-4.708) -- (8.705,-4.685) -- (8.797,-4.658) -- (8.885,-4.625) -- (8.967,-4.588) -- (9.045,-4.546) -- (9.116,-4.498) -- (9.181,-4.445) -- (9.240,-4.385) -- (9.293,-4.321) -- (9.340,-4.250) -- (9.381,-4.175) -- (9.415,-4.095) -- (9.443,-4.011) -- (9.466,-3.924) -- (9.483,-3.833) -- (9.494,-3.741) -- (9.499,-3.647) -- (9.499,-3.553) -- (9.494,-3.459) -- (9.483,-3.367) -- (9.466,-3.276) -- (9.443,-3.189) -- (9.415,-3.105) -- (9.381,-3.025) -- (9.340,-2.950) -- (9.293,-2.879) -- (9.240,-2.815) -- (9.181,-2.755) -- (9.116,-2.702) -- (9.045,-2.654) -- (8.967,-2.612) -- (8.885,-2.575) -- (8.797,-2.542) -- (8.705,-2.515) -- (8.608,-2.492) -- (8.507,-2.473) -- (8.403,-2.457) -- (8.296,-2.444) -- (8.186,-2.433) -- (8.075,-2.425) -- (7.961,-2.419) -- (7.846,-2.414) -- (7.730,-2.410) -- (7.613,-2.407) -- (7.495,-2.405) -- (7.376,-2.404) -- (7.258,-2.402) -- (7.138,-2.402) -- (7.019,-2.401) -- (6.899,-2.401) -- (6.780,-2.400) -- (6.660,-2.400) -- (6.540,-2.400) -- (6.420,-2.400) -- (6.300,-2.400) -- (6.180,-2.400) -- (6.060,-2.400) -- (5.940,-2.400) -- (5.820,-2.400) -- (5.700,-2.400) -- (5.580,-2.400) -- (5.460,-2.400) -- (5.340,-2.400) -- (5.220,-2.400) -- (5.101,-2.401) -- (4.981,-2.401) -- (4.862,-2.402) -- (4.742,-2.402) -- (4.624,-2.404) -- (4.505,-2.405) -- (4.387,-2.407) -- (4.270,-2.410) -- (4.154,-2.414) -- (4.039,-2.419) -- (3.925,-2.425) -- (3.814,-2.434) -- (3.704,-2.444) -- (3.597,-2.457) -- (3.493,-2.473) -- (3.392,-2.492) -- (3.295,-2.515) -- (3.203,-2.543) -- (3.115,-2.575) -- (3.032,-2.612) -- (2.955,-2.655) -- (2.883,-2.703) -- (2.817,-2.757) -- (2.757,-2.817) -- (2.703,-2.883) -- (2.655,-2.955) -- (2.612,-3.032) -- (2.575,-3.115) -- (2.543,-3.203) -- (2.515,-3.295) -- (2.492,-3.392) -- (2.473,-3.493) -- (2.457,-3.597) -- (2.444,-3.704) -- (2.434,-3.814) -- (2.425,-3.925) -- (2.419,-4.039) -- (2.414,-4.154) -- (2.410,-4.270) -- (2.407,-4.387) -- (2.405,-4.505) -- (2.404,-4.624) -- (2.402,-4.742) -- (2.402,-4.862) -- (2.401,-4.981) -- (2.401,-5.101) -- (2.401,-5.220) -- (2.401,-5.340) -- (2.401,-5.460) -- (2.401,-5.580) -- (2.401,-5.699) -- (2.401,-5.819) -- (2.402,-5.938) -- (2.402,-6.058) -- (2.404,-6.176) -- (2.405,-6.295) -- (2.407,-6.413) -- (2.410,-6.530) -- (2.414,-6.646) -- (2.419,-6.761) -- (2.425,-6.875) -- (2.433,-6.986) -- (2.444,-7.096) -- (2.457,-7.203) -- (2.473,-7.307) -- (2.492,-7.408) -- (2.515,-7.505) -- (2.542,-7.597) -- (2.575,-7.685) -- (2.612,-7.767) -- (2.654,-7.845) -- (2.702,-7.916) -- (2.755,-7.981) -- (2.815,-8.040) -- (2.879,-8.093) -- (2.950,-8.140) -- (3.025,-8.181) -- (3.105,-8.215) -- (3.189,-8.243) -- (3.276,-8.266) -- (3.367,-8.283) -- (3.459,-8.294) -- (3.553,-8.299) -- (3.647,-8.299) -- (3.741,-8.294) -- (3.833,-8.283) -- (3.924,-8.266) -- (4.011,-8.243) -- (4.095,-8.215) -- (4.175,-8.181) -- (4.250,-8.140) -- (4.321,-8.093) -- (4.385,-8.040) -- (4.445,-7.981) -- (4.498,-7.916) -- (4.546,-7.845) -- (4.588,-7.767) -- (4.625,-7.685) -- (4.658,-7.597) -- (4.685,-7.505) -- (4.708,-7.408) -- (4.727,-7.307) -- (4.743,-7.203) -- (4.756,-7.096) -- (4.767,-6.986) -- (4.775,-6.875) -- (4.781,-6.761) -- (4.786,-6.646) -- (4.790,-6.530) -- (4.793,-6.413) -- (4.795,-6.295) -- (4.796,-6.176) -- (4.798,-6.058) -- (4.798,-5.938) -- (4.799,-5.819) -- (4.799,-5.699) -- (4.800,-5.580) -- (4.800,-5.460) -- (4.800,-5.340) -- (4.800,-5.220) -- (4.800,-5.100) -- (4.800,-4.980) -- (4.800,-4.860) -- (4.800,-4.740) -- (4.800,-4.620) -- (4.800,-4.500) -- (4.800,-4.380) -- (4.800,-4.260) -- (4.800,-4.140) -- (4.800,-4.020) -- (4.801,-3.901) -- (4.801,-3.781) -- (4.802,-3.662) -- (4.802,-3.542) -- (4.804,-3.424) -- (4.805,-3.305) -- (4.807,-3.187) -- (4.810,-3.070) -- (4.814,-2.954) -- (4.819,-2.839) -- (4.825,-2.725) -- (4.834,-2.614) -- (4.844,-2.504) -- (4.857,-2.397) -- (4.873,-2.293) -- (4.892,-2.192) -- (4.915,-2.095) -- (4.943,-2.003) -- (4.975,-1.915) -- (5.012,-1.832) -- (5.055,-1.755) -- (5.103,-1.683) -- (5.157,-1.617) -- (5.217,-1.557) -- (5.283,-1.503) -- (5.355,-1.455) -- (5.432,-1.412) -- (5.515,-1.375) -- (5.603,-1.343) -- (5.695,-1.315) -- (5.792,-1.292) -- (5.893,-1.273) -- (5.997,-1.257) -- (6.104,-1.244) -- (6.214,-1.234) -- (6.325,-1.225) -- (6.439,-1.219) -- (6.554,-1.214) -- (6.670,-1.210) -- (6.787,-1.207) -- (6.905,-1.205) -- (7.024,-1.204) -- (7.142,-1.202) -- (7.262,-1.202) -- (7.381,-1.201) -- (7.501,-1.201) -- (7.620,-1.201) -- (7.740,-1.201) -- (7.860,-1.201) -- (7.980,-1.201) -- (8.099,-1.201) -- (8.219,-1.201) -- (8.338,-1.202) -- (8.458,-1.202) -- (8.576,-1.204) -- (8.695,-1.205) -- (8.813,-1.207) -- (8.930,-1.210) -- (9.046,-1.214) -- (9.161,-1.219) -- (9.275,-1.225) -- (9.386,-1.234) -- (9.496,-1.244) -- (9.603,-1.257) -- (9.707,-1.273) -- (9.808,-1.292) -- (9.905,-1.315) -- (9.997,-1.343) -- (10.085,-1.375) -- (10.168,-1.412) -- (10.245,-1.455) -- (10.317,-1.503) -- (10.383,-1.557) -- (10.443,-1.617) -- (10.497,-1.683) -- (10.545,-1.755) -- (10.588,-1.832) -- (10.625,-1.915) -- (10.657,-2.003) -- (10.685,-2.095) -- (10.708,-2.192) -- (10.727,-2.293) -- (10.743,-2.397) -- (10.756,-2.504) -- (10.766,-2.614) -- (10.775,-2.725) -- (10.781,-2.839) -- (10.786,-2.954) -- (10.790,-3.070) -- (10.793,-3.187) -- (10.795,-3.305) -- (10.796,-3.424) -- (10.798,-3.542) -- (10.798,-3.662) -- (10.799,-3.781) -- (10.799,-3.901) -- (10.799,-4.020) -- (10.799,-4.140) -- (10.799,-4.260) -- (10.799,-4.380) -- (10.799,-4.499) -- (10.799,-4.619) -- (10.798,-4.738) -- (10.798,-4.858) -- (10.796,-4.976) -- (10.795,-5.095) -- (10.793,-5.213) -- (10.790,-5.330) -- (10.786,-5.446) -- (10.781,-5.561) -- (10.775,-5.675) -- (10.766,-5.786) -- (10.756,-5.896) -- (10.743,-6.003) -- (10.727,-6.107) -- (10.708,-6.208) -- (10.685,-6.305) -- (10.657,-6.397) -- (10.625,-6.485) -- (10.588,-6.568) -- (10.545,-6.645) -- (10.497,-6.717) -- (10.443,-6.783) -- (10.383,-6.843) -- (10.317,-6.897) -- (10.245,-6.945) -- (10.168,-6.988) -- (10.085,-7.025) -- (9.997,-7.057) -- (9.905,-7.085) -- (9.808,-7.108) -- (9.707,-7.127) -- (9.603,-7.143) -- (9.496,-7.156) -- (9.386,-7.166) -- (9.275,-7.175) -- (9.161,-7.181) -- (9.046,-7.186) -- (8.930,-7.190) -- (8.813,-7.193) -- (8.695,-7.195) -- (8.576,-7.196) -- (8.458,-7.198) -- (8.338,-7.198) -- (8.219,-7.199) -- (8.099,-7.199) -- (7.980,-7.200) -- (7.860,-7.200) -- (7.740,-7.200) -- (7.620,-7.200) -- (7.500,-7.200) -- (7.380,-7.200) -- (7.260,-7.200) -- (7.140,-7.200) -- (7.020,-7.200) -- (6.900,-7.200) -- (6.780,-7.200) -- (6.660,-7.200) -- (6.540,-7.200) -- (6.420,-7.200) -- (6.300,-7.200) -- (6.180,-7.200) -- (6.060,-7.200) -- (5.940,-7.200) -- (5.820,-7.200) -- (5.700,-7.200) -- (5.580,-7.200) -- (5.460,-7.200) -- (5.340,-7.200) -- (5.220,-7.200) -- (5.100,-7.200) -- (4.980,-7.200) -- (4.860,-7.200) -- (4.740,-7.200) -- (4.620,-7.200) -- (4.500,-7.200) -- (4.380,-7.200) -- (4.260,-7.200) -- (4.140,-7.200) -- (4.020,-7.200) -- (3.901,-7.199) -- (3.781,-7.199) -- (3.662,-7.198) -- (3.542,-7.198) -- (3.424,-7.196) -- (3.305,-7.195) -- (3.187,-7.193) -- (3.070,-7.190) -- (2.954,-7.186) -- (2.839,-7.181) -- (2.725,-7.175) -- (2.614,-7.166) -- (2.504,-7.156) -- (2.397,-7.143) -- (2.293,-7.127) -- (2.192,-7.108) -- (2.095,-7.085) -- (2.003,-7.057) -- (1.915,-7.025) -- (1.832,-6.988) -- (1.755,-6.945) -- (1.683,-6.897) -- (1.617,-6.843) -- (1.557,-6.783) -- (1.503,-6.717) -- (1.455,-6.645) -- (1.412,-6.568) -- (1.375,-6.485) -- (1.343,-6.398) -- (1.316,-6.305) -- (1.293,-6.209) -- (1.274,-6.108) -- (1.259,-6.005) -- (1.246,-5.899) -- (1.237,-5.790) -- (1.230,-5.680) -- (1.226,-5.568) -- (1.224,-5.456) -- (1.224,-5.344) -- (1.226,-5.232) -- (1.230,-5.120) -- (1.237,-5.010) -- (1.246,-4.901) -- (1.259,-4.795) -- (1.274,-4.692) -- (1.293,-4.591) -- (1.316,-4.495) -- (1.343,-4.402) -- (1.375,-4.315) -- (1.412,-4.232) -- (1.455,-4.155) -- (1.503,-4.083) -- (1.557,-4.017) -- (1.617,-3.957) -- (1.683,-3.903) -- (1.755,-3.855) -- (1.832,-3.812) -- (1.915,-3.775) -- (2.003,-3.743) -- (2.095,-3.715) -- (2.192,-3.692) -- (2.293,-3.673) -- (2.397,-3.657) -- (2.504,-3.644) -- (2.613,-3.634) -- (2.725,-3.625) -- (2.839,-3.619) -- (2.954,-3.614) -- (3.070,-3.610) -- (3.187,-3.608) -- (3.304,-3.606) -- (3.422,-3.605) -- (3.541,-3.604) -- (3.659,-3.604) -- (3.778,-3.605) -- (3.896,-3.606) -- (4.013,-3.608) -- (4.130,-3.610) -- (4.246,-3.614) -- (4.361,-3.619) -- (4.475,-3.625) -- (4.587,-3.634) -- (4.696,-3.644) -- (4.803,-3.657) -- (4.907,-3.673) -- (5.008,-3.692) -- (5.105,-3.715) -- (5.197,-3.743) -- (5.285,-3.775) -- (5.368,-3.812) -- (5.445,-3.855) -- (5.517,-3.903) -- (5.583,-3.957) -- (5.643,-4.017) -- (5.697,-4.083) -- (5.745,-4.155) -- (5.788,-4.232) -- (5.825,-4.315) -- (5.857,-4.403) -- (5.885,-4.495) -- (5.908,-4.592) -- (5.927,-4.693) -- (5.943,-4.797) -- (5.956,-4.904) -- (5.967,-5.013) -- (5.975,-5.125) -- (5.981,-5.239) -- (5.986,-5.354) -- (5.990,-5.470) -- (5.993,-5.587) -- (5.996,-5.704) -- (5.998,-5.822) -- (5.999,-5.941) -- (6.001,-6.059) -- (6.002,-6.178) -- (6.004,-6.296) -- (6.007,-6.413) -- (6.010,-6.530) -- (6.014,-6.646) -- (6.019,-6.761) -- (6.025,-6.875) -- (6.033,-6.987) -- (6.044,-7.096) -- (6.057,-7.203) -- (6.073,-7.307) -- (6.092,-7.408) -- (6.115,-7.505) -- (6.142,-7.597) -- (6.175,-7.685) -- (6.212,-7.767) -- (6.254,-7.845) -- (6.302,-7.916) -- (6.355,-7.981) -- (6.415,-8.040) -- (6.479,-8.093) -- (6.550,-8.140) -- (6.625,-8.181) -- (6.705,-8.215) -- (6.789,-8.243) -- (6.876,-8.266) -- (6.967,-8.283) -- (7.059,-8.294) -- (7.153,-8.299) -- (7.247,-8.299) -- (7.341,-8.294) -- (7.433,-8.283) -- (7.524,-8.266) -- (7.611,-8.243) -- (7.695,-8.215) -- (7.775,-8.181) -- (7.850,-8.140) -- (7.921,-8.093) -- (7.985,-8.040) -- (8.045,-7.981) -- (8.098,-7.916) -- (8.146,-7.845) -- (8.188,-7.767) -- (8.225,-7.685) -- (8.258,-7.597) -- (8.285,-7.505) -- (8.308,-7.408) -- (8.327,-7.307) -- (8.343,-7.203) -- (8.356,-7.096) -- (8.366,-6.987) -- (8.375,-6.875) -- (8.381,-6.761) -- (8.386,-6.646) -- (8.390,-6.530) -- (8.392,-6.413) -- (8.394,-6.296) -- (8.395,-6.178) -- (8.396,-6.059) -- (8.396,-5.941) -- (8.395,-5.822) -- (8.394,-5.704) -- (8.392,-5.587) -- (8.390,-5.470) -- (8.386,-5.354) -- (8.382,-5.239) -- (8.375,-5.126) -- (8.368,-5.015) -- (8.358,-4.906) -- (8.346,-4.799) -- (8.331,-4.696) -- (8.313,-4.597) -- (8.292,-4.502) -- (8.267,-4.413) -- (8.239,-4.329) -- (8.207,-4.251) -- (8.171,-4.180) -- (8.131,-4.116) -- (8.087,-4.061) -- (8.040,-4.014) -- (7.990,-3.976) -- (7.937,-3.947) -- (7.883,-3.927) -- (7.828,-3.917) -- (7.772,-3.917) -- (7.716,-3.927) -- (7.662,-3.946) -- (7.609,-3.975) -- (7.559,-4.012) -- (7.511,-4.059) -- (7.466,-4.113) -- (7.424,-4.175) -- (7.386,-4.244) -- (7.351,-4.319) -- (7.319,-4.399) -- (7.289,-4.484) -- (7.262,-4.572) -- (7.236,-4.663) -- (7.210,-4.755) -- (7.185,-4.849) -- (7.160,-4.942) -- (7.132,-5.034) -- (7.103,-5.124) -- (7.071,-5.211) -- (7.035,-5.295) -- (6.995,-5.375) -- (6.950,-5.450) -- (6.901,-5.521) -- (6.845,-5.585) -- (6.785,-5.645) -- (6.718,-5.698) -- (6.646,-5.746) -- (6.569,-5.788) -- (6.486,-5.825) -- (6.398,-5.857) -- (6.305,-5.884) -- (6.209,-5.907) -- (6.108,-5.926) -- (6.005,-5.942) -- (5.899,-5.954) -- (5.790,-5.963) -- (5.680,-5.970) -- (5.569,-5.974) -- (5.456,-5.976) -- (5.344,-5.976) -- (5.232,-5.974) -- (5.121,-5.970) -- (5.011,-5.964) -- (4.903,-5.955) -- (4.798,-5.944) -- (4.695,-5.930) -- (4.596,-5.912) -- (4.502,-5.892) -- (4.412,-5.867) -- (4.329,-5.839) -- (4.251,-5.807) -- (4.180,-5.770) -- (4.116,-5.730) -- (4.061,-5.687) -- cycle;
\draw[white, line width=4.5pt, line cap=butt] (4.800,-4.980) -- (4.800,-4.860) -- (4.800,-4.854) -- (4.800,-4.740) -- (4.800,-4.620);
\draw[line cap=butt] (4.800,-4.980) -- (4.800,-4.860) -- (4.800,-4.854) -- (4.800,-4.740) -- (4.800,-4.620);
\draw[white, line width=4.5pt, line cap=butt] (5.705,-4.805) -- (5.824,-4.804) -- (5.942,-4.802) -- (5.944,-4.802) -- (6.062,-4.802) -- (6.181,-4.801);
\draw[line cap=butt] (5.705,-4.805) -- (5.824,-4.804) -- (5.942,-4.802) -- (5.944,-4.802) -- (6.062,-4.802) -- (6.181,-4.801);
\draw[white, line width=4.5pt, line cap=butt] (7.262,-4.572) -- (7.236,-4.663) -- (7.210,-4.755) -- (7.199,-4.798) -- (7.185,-4.849) -- (7.160,-4.942);
\draw[line cap=butt] (7.262,-4.572) -- (7.236,-4.663) -- (7.210,-4.755) -- (7.199,-4.798) -- (7.185,-4.849) -- (7.160,-4.942);
\draw[white, line width=4.5pt, line cap=butt] (8.186,-4.767) -- (8.296,-4.756) -- (8.339,-4.751) -- (8.403,-4.743) -- (8.507,-4.727);
\draw[line cap=butt] (8.186,-4.767) -- (8.296,-4.756) -- (8.339,-4.751) -- (8.403,-4.743) -- (8.507,-4.727);
\draw[white, line width=4.5pt, line cap=butt] (4.834,-2.614) -- (4.844,-2.504) -- (4.856,-2.402) -- (4.857,-2.397) -- (4.873,-2.293) -- (4.892,-2.192);
\draw[line cap=butt] (4.834,-2.614) -- (4.844,-2.504) -- (4.856,-2.402) -- (4.857,-2.397) -- (4.873,-2.293) -- (4.892,-2.192);
\draw[white, line width=4.5pt, line cap=butt] (2.473,-3.493) -- (2.457,-3.597) -- (2.450,-3.650) -- (2.444,-3.704) -- (2.434,-3.814);
\draw[line cap=butt] (2.473,-3.493) -- (2.457,-3.597) -- (2.450,-3.650) -- (2.444,-3.704) -- (2.434,-3.814);
\draw[white, line width=4.5pt, line cap=butt] (2.614,-7.166) -- (2.504,-7.156) -- (2.450,-7.150) -- (2.397,-7.143) -- (2.293,-7.127);
\draw[line cap=butt] (2.614,-7.166) -- (2.504,-7.156) -- (2.450,-7.150) -- (2.397,-7.143) -- (2.293,-7.127);
\draw[white, line width=4.5pt, line cap=butt] (4.708,-7.408) -- (4.727,-7.307) -- (4.743,-7.203) -- (4.744,-7.200) -- (4.756,-7.096) -- (4.767,-6.986);
\draw[line cap=butt] (4.708,-7.408) -- (4.727,-7.307) -- (4.743,-7.203) -- (4.744,-7.200) -- (4.756,-7.096) -- (4.767,-6.986);
\draw[white, line width=4.5pt, line cap=butt] (5.011,-5.964) -- (4.903,-5.955) -- (4.798,-5.944) -- (4.798,-5.944) -- (4.695,-5.930) -- (4.596,-5.912);
\draw[line cap=butt] (5.011,-5.964) -- (4.903,-5.955) -- (4.798,-5.944) -- (4.798,-5.944) -- (4.695,-5.930) -- (4.596,-5.912);
\draw[white, line width=4.5pt, line cap=butt] (4.587,-3.634) -- (4.696,-3.644) -- (4.802,-3.657) -- (4.803,-3.657) -- (4.907,-3.673) -- (5.008,-3.692);
\draw[line cap=butt] (4.587,-3.634) -- (4.696,-3.644) -- (4.802,-3.657) -- (4.803,-3.657) -- (4.907,-3.673) -- (5.008,-3.692);
\draw[white, line width=4.5pt, line cap=butt] (8.308,-7.408) -- (8.327,-7.307) -- (8.343,-7.203) -- (8.344,-7.198) -- (8.356,-7.096) -- (8.366,-6.987);
\draw[line cap=butt] (8.308,-7.408) -- (8.327,-7.307) -- (8.343,-7.203) -- (8.344,-7.198) -- (8.356,-7.096) -- (8.366,-6.987);
\draw[white, line width=4.5pt, line cap=butt] (6.180,-7.200) -- (6.060,-7.200) -- (6.056,-7.200) -- (5.940,-7.200) -- (5.820,-7.200);
\draw[line cap=butt] (6.180,-7.200) -- (6.060,-7.200) -- (6.056,-7.200) -- (5.940,-7.200) -- (5.820,-7.200);
\draw[white, line width=4.5pt, line cap=butt] (5.996,-5.704) -- (5.998,-5.822) -- (5.999,-5.941) -- (5.999,-5.942) -- (6.001,-6.059) -- (6.002,-6.178);
\draw[line cap=butt] (5.996,-5.704) -- (5.998,-5.822) -- (5.999,-5.941) -- (5.999,-5.942) -- (6.001,-6.059) -- (6.002,-6.178);
\draw[red!75!black, line width=0.8pt] (7.199,-4.798) circle (0.512);
\end{tikzpicture}

%% file: figures/cc3-2-switched.tex
\begin{tikzpicture}[line width=1.0pt, line cap=round, line join=round]
\draw (4.014,-5.640) -- (3.976,-5.590) -- (3.947,-5.537) -- (3.927,-5.483) -- (3.918,-5.428) -- (3.918,-5.372) -- (3.927,-5.317) -- (3.947,-5.263) -- (3.976,-5.210) -- (4.014,-5.160) -- (4.061,-5.113) -- (4.116,-5.069) -- (4.180,-5.029) -- (4.251,-4.993) -- (4.329,-4.961) -- (4.413,-4.933) -- (4.502,-4.908) -- (4.597,-4.887) -- (4.696,-4.869) -- (4.799,-4.854) -- (4.906,-4.842) -- (5.015,-4.832) -- (5.126,-4.825) -- (5.239,-4.818) -- (5.354,-4.814) -- (5.470,-4.810) -- (5.587,-4.807) -- (5.705,-4.805) -- (5.824,-4.804) -- (5.942,-4.802) -- (6.062,-4.802) -- (6.181,-4.801) -- (6.301,-4.801) -- (6.420,-4.800) -- (6.540,-4.800) -- (6.660,-4.800) -- (6.780,-4.800) -- (6.899,-4.799) -- (7.019,-4.799) -- (7.138,-4.798) -- (7.258,-4.798) -- (7.376,-4.796) -- (7.495,-4.795) -- (7.613,-4.793) -- (7.730,-4.790) -- (7.846,-4.786) -- (7.961,-4.781) -- (8.075,-4.775) -- (8.186,-4.767) -- (8.296,-4.756) -- (8.403,-4.743) -- (8.507,-4.727) -- (8.608,-4.708) -- (8.705,-4.685) -- (8.797,-4.658) -- (8.885,-4.625) -- (8.967,-4.588) -- (9.045,-4.546) -- (9.116,-4.498) -- (9.181,-4.445) -- (9.240,-4.385) -- (9.293,-4.321) -- (9.340,-4.250) -- (9.381,-4.175) -- (9.415,-4.095) -- (9.443,-4.011) -- (9.466,-3.924) -- (9.483,-3.833) -- (9.494,-3.741) -- (9.499,-3.647) -- (9.499,-3.553) -- (9.494,-3.459) -- (9.483,-3.367) -- (9.466,-3.276) -- (9.443,-3.189) -- (9.415,-3.105) -- (9.381,-3.025) -- (9.340,-2.950) -- (9.293,-2.879) -- (9.240,-2.815) -- (9.181,-2.755) -- (9.116,-2.702) -- (9.045,-2.654) -- (8.967,-2.612) -- (8.885,-2.575) -- (8.797,-2.542) -- (8.705,-2.515) -- (8.608,-2.492) -- (8.507,-2.473) -- (8.403,-2.457) -- (8.296,-2.444) -- (8.186,-2.433) -- (8.075,-2.425) -- (7.961,-2.419) -- (7.846,-2.414) -- (7.730,-2.410) -- (7.613,-2.407) -- (7.495,-2.405) -- (7.376,-2.404) -- (7.258,-2.402) -- (7.138,-2.402) -- (7.019,-2.401) -- (6.899,-2.401) -- (6.780,-2.400) -- (6.660,-2.400) -- (6.540,-2.400) -- (6.420,-2.400) -- (6.300,-2.400) -- (6.180,-2.400) -- (6.060,-2.400) -- (5.940,-2.400) -- (5.820,-2.400) -- (5.700,-2.400) -- (5.580,-2.400) -- (5.460,-2.400) -- (5.340,-2.400) -- (5.220,-2.400) -- (5.101,-2.401) -- (4.981,-2.401) -- (4.862,-2.402) -- (4.742,-2.402) -- (4.624,-2.404) -- (4.505,-2.405) -- (4.387,-2.407) -- (4.270,-2.410) -- (4.154,-2.414) -- (4.039,-2.419) -- (3.925,-2.425) -- (3.814,-2.434) -- (3.704,-2.444) -- (3.597,-2.457) -- (3.493,-2.473) -- (3.392,-2.492) -- (3.295,-2.515) -- (3.203,-2.543) -- (3.115,-2.575) -- (3.032,-2.612) -- (2.955,-2.655) -- (2.883,-2.703) -- (2.817,-2.757) -- (2.757,-2.817) -- (2.703,-2.883) -- (2.655,-2.955) -- (2.612,-3.032) -- (2.575,-3.115) -- (2.543,-3.203) -- (2.515,-3.295) -- (2.492,-3.392) -- (2.473,-3.493) -- (2.457,-3.597) -- (2.444,-3.704) -- (2.434,-3.814) -- (2.425,-3.925) -- (2.419,-4.039) -- (2.414,-4.154) -- (2.410,-4.270) -- (2.407,-4.387) -- (2.405,-4.505) -- (2.404,-4.624) -- (2.402,-4.742) -- (2.402,-4.862) -- (2.401,-4.981) -- (2.401,-5.101) -- (2.401,-5.220) -- (2.401,-5.340) -- (2.401,-5.460) -- (2.401,-5.580) -- (2.401,-5.699) -- (2.401,-5.819) -- (2.402,-5.938) -- (2.402,-6.058) -- (2.404,-6.176) -- (2.405,-6.295) -- (2.407,-6.413) -- (2.410,-6.530) -- (2.414,-6.646) -- (2.419,-6.761) -- (2.425,-6.875) -- (2.433,-6.986) -- (2.444,-7.096) -- (2.457,-7.203) -- (2.473,-7.307) -- (2.492,-7.408) -- (2.515,-7.505) -- (2.542,-7.597) -- (2.575,-7.685) -- (2.612,-7.767) -- (2.654,-7.845) -- (2.702,-7.916) -- (2.755,-7.981) -- (2.815,-8.040) -- (2.879,-8.093) -- (2.950,-8.140) -- (3.025,-8.181) -- (3.105,-8.215) -- (3.189,-8.243) -- (3.276,-8.266) -- (3.367,-8.283) -- (3.459,-8.294) -- (3.553,-8.299) -- (3.647,-8.299) -- (3.741,-8.294) -- (3.833,-8.283) -- (3.924,-8.266) -- (4.011,-8.243) -- (4.095,-8.215) -- (4.175,-8.181) -- (4.250,-8.140) -- (4.321,-8.093) -- (4.385,-8.040) -- (4.445,-7.981) -- (4.498,-7.916) -- (4.546,-7.845) -- (4.588,-7.767) -- (4.625,-7.685) -- (4.658,-7.597) -- (4.685,-7.505) -- (4.708,-7.408) -- (4.727,-7.307) -- (4.743,-7.203) -- (4.756,-7.096) -- (4.767,-6.986) -- (4.775,-6.875) -- (4.781,-6.761) -- (4.786,-6.646) -- (4.790,-6.530) -- (4.793,-6.413) -- (4.795,-6.295) -- (4.796,-6.176) -- (4.798,-6.058) -- (4.798,-5.938) -- (4.799,-5.819) -- (4.799,-5.699) -- (4.800,-5.580) -- (4.800,-5.460) -- (4.800,-5.340) -- (4.800,-5.220) -- (4.800,-5.100) -- (4.800,-4.980) -- (4.800,-4.860) -- (4.800,-4.740) -- (4.800,-4.620) -- (4.800,-4.500) -- (4.800,-4.380) -- (4.800,-4.260) -- (4.800,-4.140) -- (4.800,-4.020) -- (4.801,-3.901) -- (4.801,-3.781) -- (4.802,-3.662) -- (4.802,-3.542) -- (4.804,-3.424) -- (4.805,-3.305) -- (4.807,-3.187) -- (4.810,-3.070) -- (4.814,-2.954) -- (4.819,-2.839) -- (4.825,-2.725) -- (4.834,-2.614) -- (4.844,-2.504) -- (4.857,-2.397) -- (4.873,-2.293) -- (4.892,-2.192) -- (4.915,-2.095) -- (4.943,-2.003) -- (4.975,-1.915) -- (5.012,-1.832) -- (5.055,-1.755) -- (5.103,-1.683) -- (5.157,-1.617) -- (5.217,-1.557) -- (5.283,-1.503) -- (5.355,-1.455) -- (5.432,-1.412) -- (5.515,-1.375) -- (5.603,-1.343) -- (5.695,-1.315) -- (5.792,-1.292) -- (5.893,-1.273) -- (5.997,-1.257) -- (6.104,-1.244) -- (6.214,-1.234) -- (6.325,-1.225) -- (6.439,-1.219) -- (6.554,-1.214) -- (6.670,-1.210) -- (6.787,-1.207) -- (6.905,-1.205) -- (7.024,-1.204) -- (7.142,-1.202) -- (7.262,-1.202) -- (7.381,-1.201) -- (7.501,-1.201) -- (7.620,-1.201) -- (7.740,-1.201) -- (7.860,-1.201) -- (7.980,-1.201) -- (8.099,-1.201) -- (8.219,-1.201) -- (8.338,-1.202) -- (8.458,-1.202) -- (8.576,-1.204) -- (8.695,-1.205) -- (8.813,-1.207) -- (8.930,-1.210) -- (9.046,-1.214) -- (9.161,-1.219) -- (9.275,-1.225) -- (9.386,-1.234) -- (9.496,-1.244) -- (9.603,-1.257) -- (9.707,-1.273) -- (9.808,-1.292) -- (9.905,-1.315) -- (9.997,-1.343) -- (10.085,-1.375) -- (10.168,-1.412) -- (10.245,-1.455) -- (10.317,-1.503) -- (10.383,-1.557) -- (10.443,-1.617) -- (10.497,-1.683) -- (10.545,-1.755) -- (10.588,-1.832) -- (10.625,-1.915) -- (10.657,-2.003) -- (10.685,-2.095) -- (10.708,-2.192) -- (10.727,-2.293) -- (10.743,-2.397) -- (10.756,-2.504) -- (10.766,-2.614) -- (10.775,-2.725) -- (10.781,-2.839) -- (10.786,-2.954) -- (10.790,-3.070) -- (10.793,-3.187) -- (10.795,-3.305) -- (10.796,-3.424) -- (10.798,-3.542) -- (10.798,-3.662) -- (10.799,-3.781) -- (10.799,-3.901) -- (10.799,-4.020) -- (10.799,-4.140) -- (10.799,-4.260) -- (10.799,-4.380) -- (10.799,-4.499) -- (10.799,-4.619) -- (10.798,-4.738) -- (10.798,-4.858) -- (10.796,-4.976) -- (10.795,-5.095) -- (10.793,-5.213) -- (10.790,-5.330) -- (10.786,-5.446) -- (10.781,-5.561) -- (10.775,-5.675) -- (10.766,-5.786) -- (10.756,-5.896) -- (10.743,-6.003) -- (10.727,-6.107) -- (10.708,-6.208) -- (10.685,-6.305) -- (10.657,-6.397) -- (10.625,-6.485) -- (10.588,-6.568) -- (10.545,-6.645) -- (10.497,-6.717) -- (10.443,-6.783) -- (10.383,-6.843) -- (10.317,-6.897) -- (10.245,-6.945) -- (10.168,-6.988) -- (10.085,-7.025) -- (9.997,-7.057) -- (9.905,-7.085) -- (9.808,-7.108) -- (9.707,-7.127) -- (9.603,-7.143) -- (9.496,-7.156) -- (9.386,-7.166) -- (9.275,-7.175) -- (9.161,-7.181) -- (9.046,-7.186) -- (8.930,-7.190) -- (8.813,-7.193) -- (8.695,-7.195) -- (8.576,-7.196) -- (8.458,-7.198) -- (8.338,-7.198) -- (8.219,-7.199) -- (8.099,-7.199) -- (7.980,-7.200) -- (7.860,-7.200) -- (7.740,-7.200) -- (7.620,-7.200) -- (7.500,-7.200) -- (7.380,-7.200) -- (7.260,-7.200) -- (7.140,-7.200) -- (7.020,-7.200) -- (6.900,-7.200) -- (6.780,-7.200) -- (6.660,-7.200) -- (6.540,-7.200) -- (6.420,-7.200) -- (6.300,-7.200) -- (6.180,-7.200) -- (6.060,-7.200) -- (5.940,-7.200) -- (5.820,-7.200) -- (5.700,-7.200) -- (5.580,-7.200) -- (5.460,-7.200) -- (5.340,-7.200) -- (5.220,-7.200) -- (5.100,-7.200) -- (4.980,-7.200) -- (4.860,-7.200) -- (4.740,-7.200) -- (4.620,-7.200) -- (4.500,-7.200) -- (4.380,-7.200) -- (4.260,-7.200) -- (4.140,-7.200) -- (4.020,-7.200) -- (3.901,-7.199) -- (3.781,-7.199) -- (3.662,-7.198) -- (3.542,-7.198) -- (3.424,-7.196) -- (3.305,-7.195) -- (3.187,-7.193) -- (3.070,-7.190) -- (2.954,-7.186) -- (2.839,-7.181) -- (2.725,-7.175) -- (2.614,-7.166) -- (2.504,-7.156) -- (2.397,-7.143) -- (2.293,-7.127) -- (2.192,-7.108) -- (2.095,-7.085) -- (2.003,-7.057) -- (1.915,-7.025) -- (1.832,-6.988) -- (1.755,-6.945) -- (1.683,-6.897) -- (1.617,-6.843) -- (1.557,-6.783) -- (1.503,-6.717) -- (1.455,-6.645) -- (1.412,-6.568) -- (1.375,-6.485) -- (1.343,-6.398) -- (1.316,-6.305) -- (1.293,-6.209) -- (1.274,-6.108) -- (1.259,-6.005) -- (1.246,-5.899) -- (1.237,-5.790) -- (1.230,-5.680) -- (1.226,-5.568) -- (1.224,-5.456) -- (1.224,-5.344) -- (1.226,-5.232) -- (1.230,-5.120) -- (1.237,-5.010) -- (1.246,-4.901) -- (1.259,-4.795) -- (1.274,-4.692) -- (1.293,-4.591) -- (1.316,-4.495) -- (1.343,-4.402) -- (1.375,-4.315) -- (1.412,-4.232) -- (1.455,-4.155) -- (1.503,-4.083) -- (1.557,-4.017) -- (1.617,-3.957) -- (1.683,-3.903) -- (1.755,-3.855) -- (1.832,-3.812) -- (1.915,-3.775) -- (2.003,-3.743) -- (2.095,-3.715) -- (2.192,-3.692) -- (2.293,-3.673) -- (2.397,-3.657) -- (2.504,-3.644) -- (2.613,-3.634) -- (2.725,-3.625) -- (2.839,-3.619) -- (2.954,-3.614) -- (3.070,-3.610) -- (3.187,-3.608) -- (3.304,-3.606) -- (3.422,-3.605) -- (3.541,-3.604) -- (3.659,-3.604) -- (3.778,-3.605) -- (3.896,-3.606) -- (4.013,-3.608) -- (4.130,-3.610) -- (4.246,-3.614) -- (4.361,-3.619) -- (4.475,-3.625) -- (4.587,-3.634) -- (4.696,-3.644) -- (4.803,-3.657) -- (4.907,-3.673) -- (5.008,-3.692) -- (5.105,-3.715) -- (5.197,-3.743) -- (5.285,-3.775) -- (5.368,-3.812) -- (5.445,-3.855) -- (5.517,-3.903) -- (5.583,-3.957) -- (5.643,-4.017) -- (5.697,-4.083) -- (5.745,-4.155) -- (5.788,-4.232) -- (5.825,-4.315) -- (5.857,-4.403) -- (5.885,-4.495) -- (5.908,-4.592) -- (5.927,-4.693) -- (5.943,-4.797) -- (5.956,-4.904) -- (5.967,-5.013) -- (5.975,-5.125) -- (5.981,-5.239) -- (5.986,-5.354) -- (5.990,-5.470) -- (5.993,-5.587) -- (5.996,-5.704) -- (5.998,-5.822) -- (5.999,-5.941) -- (6.001,-6.059) -- (6.002,-6.178) -- (6.004,-6.296) -- (6.007,-6.413) -- (6.010,-6.530) -- (6.014,-6.646) -- (6.019,-6.761) -- (6.025,-6.875) -- (6.033,-6.987) -- (6.044,-7.096) -- (6.057,-7.203) -- (6.073,-7.307) -- (6.092,-7.408) -- (6.115,-7.505) -- (6.142,-7.597) -- (6.175,-7.685) -- (6.212,-7.767) -- (6.254,-7.845) -- (6.302,-7.916) -- (6.355,-7.981) -- (6.415,-8.040) -- (6.479,-8.093) -- (6.550,-8.140) -- (6.625,-8.181) -- (6.705,-8.215) -- (6.789,-8.243) -- (6.876,-8.266) -- (6.967,-8.283) -- (7.059,-8.294) -- (7.153,-8.299) -- (7.247,-8.299) -- (7.341,-8.294) -- (7.433,-8.283) -- (7.524,-8.266) -- (7.611,-8.243) -- (7.695,-8.215) -- (7.775,-8.181) -- (7.850,-8.140) -- (7.921,-8.093) -- (7.985,-8.040) -- (8.045,-7.981) -- (8.098,-7.916) -- (8.146,-7.845) -- (8.188,-7.767) -- (8.225,-7.685) -- (8.258,-7.597) -- (8.285,-7.505) -- (8.308,-7.408) -- (8.327,-7.307) -- (8.343,-7.203) -- (8.356,-7.096) -- (8.366,-6.987) -- (8.375,-6.875) -- (8.381,-6.761) -- (8.386,-6.646) -- (8.390,-6.530) -- (8.392,-6.413) -- (8.394,-6.296) -- (8.395,-6.178) -- (8.396,-6.059) -- (8.396,-5.941) -- (8.395,-5.822) -- (8.394,-5.704) -- (8.392,-5.587) -- (8.390,-5.470) -- (8.386,-5.354) -- (8.382,-5.239) -- (8.375,-5.126) -- (8.368,-5.015) -- (8.358,-4.906) -- (8.346,-4.799) -- (8.331,-4.696) -- (8.313,-4.597) -- (8.292,-4.502) -- (8.267,-4.413) -- (8.239,-4.329) -- (8.207,-4.251) -- (8.171,-4.180) -- (8.131,-4.116) -- (8.087,-4.061) -- (8.040,-4.014) -- (7.990,-3.976) -- (7.937,-3.947) -- (7.883,-3.927) -- (7.828,-3.917) -- (7.772,-3.917) -- (7.716,-3.927) -- (7.662,-3.946) -- (7.609,-3.975) -- (7.559,-4.012) -- (7.511,-4.059) -- (7.466,-4.113) -- (7.424,-4.175) -- (7.386,-4.244) -- (7.351,-4.319) -- (7.319,-4.399) -- (7.289,-4.484) -- (7.262,-4.572) -- (7.236,-4.663) -- (7.210,-4.755) -- (7.185,-4.849) -- (7.160,-4.942) -- (7.132,-5.034) -- (7.103,-5.124) -- (7.071,-5.211) -- (7.035,-5.295) -- (6.995,-5.375) -- (6.950,-5.450) -- (6.901,-5.521) -- (6.845,-5.585) -- (6.785,-5.645) -- (6.718,-5.698) -- (6.646,-5.746) -- (6.569,-5.788) -- (6.486,-5.825) -- (6.398,-5.857) -- (6.305,-5.884) -- (6.209,-5.907) -- (6.108,-5.926) -- (6.005,-5.942) -- (5.899,-5.954) -- (5.790,-5.963) -- (5.680,-5.970) -- (5.569,-5.974) -- (5.456,-5.976) -- (5.344,-5.976) -- (5.232,-5.974) -- (5.121,-5.970) -- (5.011,-5.964) -- (4.903,-5.955) -- (4.798,-5.944) -- (4.695,-5.930) -- (4.596,-5.912) -- (4.502,-5.892) -- (4.412,-5.867) -- (4.329,-5.839) -- (4.251,-5.807) -- (4.180,-5.770) -- (4.116,-5.730) -- (4.061,-5.687) -- cycle;
\draw[white, line width=4.5pt, line cap=butt] (4.800,-4.980) -- (4.800,-4.860) -- (4.800,-4.854) -- (4.800,-4.740) -- (4.800,-4.620);
\draw[line cap=butt] (4.800,-4.980) -- (4.800,-4.860) -- (4.800,-4.854) -- (4.800,-4.740) -- (4.800,-4.620);
\draw[white, line width=4.5pt, line cap=butt] (5.705,-4.805) -- (5.824,-4.804) -- (5.942,-4.802) -- (5.944,-4.802) -- (6.062,-4.802) -- (6.181,-4.801);
\draw[line cap=butt] (5.705,-4.805) -- (5.824,-4.804) -- (5.942,-4.802) -- (5.944,-4.802) -- (6.062,-4.802) -- (6.181,-4.801);
\draw[white, line width=4.5pt, line cap=butt] (7.019,-4.799) -- (7.138,-4.798) -- (7.199,-4.798) -- (7.258,-4.798) -- (7.376,-4.796);
\draw[line cap=butt] (7.019,-4.799) -- (7.138,-4.798) -- (7.199,-4.798) -- (7.258,-4.798) -- (7.376,-4.796);
\draw[white, line width=4.5pt, line cap=butt] (8.186,-4.767) -- (8.296,-4.756) -- (8.339,-4.751) -- (8.403,-4.743) -- (8.507,-4.727);
\draw[line cap=butt] (8.186,-4.767) -- (8.296,-4.756) -- (8.339,-4.751) -- (8.403,-4.743) -- (8.507,-4.727);
\draw[white, line width=4.5pt, line cap=butt] (4.834,-2.614) -- (4.844,-2.504) -- (4.856,-2.402) -- (4.857,-2.397) -- (4.873,-2.293) -- (4.892,-2.192);
\draw[line cap=butt] (4.834,-2.614) -- (4.844,-2.504) -- (4.856,-2.402) -- (4.857,-2.397) -- (4.873,-2.293) -- (4.892,-2.192);
\draw[white, line width=4.5pt, line cap=butt] (2.473,-3.493) -- (2.457,-3.597) -- (2.450,-3.650) -- (2.444,-3.704) -- (2.434,-3.814);
\draw[line cap=butt] (2.473,-3.493) -- (2.457,-3.597) -- (2.450,-3.650) -- (2.444,-3.704) -- (2.434,-3.814);
\draw[white, line width=4.5pt, line cap=butt] (2.614,-7.166) -- (2.504,-7.156) -- (2.450,-7.150) -- (2.397,-7.143) -- (2.293,-7.127);
\draw[line cap=butt] (2.614,-7.166) -- (2.504,-7.156) -- (2.450,-7.150) -- (2.397,-7.143) -- (2.293,-7.127);
\draw[white, line width=4.5pt, line cap=butt] (4.708,-7.408) -- (4.727,-7.307) -- (4.743,-7.203) -- (4.744,-7.200) -- (4.756,-7.096) -- (4.767,-6.986);
\draw[line cap=butt] (4.708,-7.408) -- (4.727,-7.307) -- (4.743,-7.203) -- (4.744,-7.200) -- (4.756,-7.096) -- (4.767,-6.986);
\draw[white, line width=4.5pt, line cap=butt] (5.011,-5.964) -- (4.903,-5.955) -- (4.798,-5.944) -- (4.798,-5.944) -- (4.695,-5.930) -- (4.596,-5.912);
\draw[line cap=butt] (5.011,-5.964) -- (4.903,-5.955) -- (4.798,-5.944) -- (4.798,-5.944) -- (4.695,-5.930) -- (4.596,-5.912);
\draw[white, line width=4.5pt, line cap=butt] (4.587,-3.634) -- (4.696,-3.644) -- (4.802,-3.657) -- (4.803,-3.657) -- (4.907,-3.673) -- (5.008,-3.692);
\draw[line cap=butt] (4.587,-3.634) -- (4.696,-3.644) -- (4.802,-3.657) -- (4.803,-3.657) -- (4.907,-3.673) -- (5.008,-3.692);
\draw[white, line width=4.5pt, line cap=butt] (8.308,-7.408) -- (8.327,-7.307) -- (8.343,-7.203) -- (8.344,-7.198) -- (8.356,-7.096) -- (8.366,-6.987);
\draw[line cap=butt] (8.308,-7.408) -- (8.327,-7.307) -- (8.343,-7.203) -- (8.344,-7.198) -- (8.356,-7.096) -- (8.366,-6.987);
\draw[white, line width=4.5pt, line cap=butt] (6.180,-7.200) -- (6.060,-7.200) -- (6.056,-7.200) -- (5.940,-7.200) -- (5.820,-7.200);
\draw[line cap=butt] (6.180,-7.200) -- (6.060,-7.200) -- (6.056,-7.200) -- (5.940,-7.200) -- (5.820,-7.200);
\draw[white, line width=4.5pt, line cap=butt] (5.996,-5.704) -- (5.998,-5.822) -- (5.999,-5.941) -- (5.999,-5.942) -- (6.001,-6.059) -- (6.002,-6.178);
\draw[line cap=butt] (5.996,-5.704) -- (5.998,-5.822) -- (5.999,-5.941) -- (5.999,-5.942) -- (6.001,-6.059) -- (6.002,-6.178);
\draw[red!75!black, line width=0.8pt] (7.199,-4.798) circle (0.512);
\draw[blue!70!black, line width=0.8pt, densely dashed] (8.339,-4.751) circle (0.512);
\end{tikzpicture}

%% file: figures/cc3-3-snappy-mirror-11a169.tex
\begin{tikzpicture}[line width=1.0pt, line cap=round, line join=round]
\draw (6.786,-3.960) -- (6.824,-4.010) -- (6.853,-4.063) -- (6.873,-4.117) -- (6.883,-4.172) -- (6.883,-4.228) -- (6.873,-4.283) -- (6.854,-4.337) -- (6.825,-4.389) -- (6.788,-4.438) -- (6.741,-4.484) -- (6.687,-4.527) -- (6.625,-4.565) -- (6.556,-4.600) -- (6.481,-4.629) -- (6.401,-4.654) -- (6.316,-4.673) -- (6.228,-4.688) -- (6.137,-4.697) -- (6.045,-4.702) -- (5.951,-4.701) -- (5.858,-4.695) -- (5.766,-4.683) -- (5.676,-4.666) -- (5.589,-4.644) -- (5.505,-4.615) -- (5.425,-4.581) -- (5.350,-4.540) -- (5.279,-4.493) -- (5.215,-4.440) -- (5.155,-4.381) -- (5.102,-4.316) -- (5.054,-4.245) -- (5.011,-4.168) -- (4.974,-4.085) -- (4.942,-3.997) -- (4.915,-3.905) -- (4.891,-3.809) -- (4.872,-3.708) -- (4.855,-3.605) -- (4.841,-3.499) -- (4.830,-3.390) -- (4.820,-3.280) -- (4.812,-3.168) -- (4.804,-3.056) -- (4.796,-2.944) -- (4.788,-2.832) -- (4.780,-2.720) -- (4.770,-2.610) -- (4.759,-2.501) -- (4.745,-2.395) -- (4.728,-2.292) -- (4.709,-2.191) -- (4.685,-2.095) -- (4.658,-2.003) -- (4.626,-1.915) -- (4.589,-1.832) -- (4.546,-1.755) -- (4.498,-1.684) -- (4.445,-1.619) -- (4.385,-1.560) -- (4.321,-1.507) -- (4.250,-1.460) -- (4.175,-1.419) -- (4.095,-1.385) -- (4.011,-1.357) -- (3.924,-1.334) -- (3.833,-1.317) -- (3.741,-1.306) -- (3.647,-1.301) -- (3.553,-1.301) -- (3.459,-1.306) -- (3.367,-1.317) -- (3.276,-1.334) -- (3.189,-1.357) -- (3.105,-1.385) -- (3.025,-1.419) -- (2.950,-1.460) -- (2.879,-1.507) -- (2.815,-1.560) -- (2.755,-1.619) -- (2.702,-1.684) -- (2.654,-1.755) -- (2.612,-1.833) -- (2.575,-1.915) -- (2.542,-2.003) -- (2.515,-2.095) -- (2.492,-2.192) -- (2.473,-2.293) -- (2.457,-2.397) -- (2.444,-2.504) -- (2.433,-2.614) -- (2.425,-2.725) -- (2.419,-2.839) -- (2.414,-2.954) -- (2.410,-3.070) -- (2.407,-3.187) -- (2.405,-3.305) -- (2.404,-3.424) -- (2.402,-3.542) -- (2.402,-3.662) -- (2.401,-3.781) -- (2.401,-3.901) -- (2.400,-4.020) -- (2.400,-4.140) -- (2.400,-4.260) -- (2.400,-4.380) -- (2.400,-4.500) -- (2.400,-4.620) -- (2.400,-4.740) -- (2.400,-4.860) -- (2.400,-4.980) -- (2.400,-5.100) -- (2.400,-5.220) -- (2.400,-5.340) -- (2.400,-5.460) -- (2.400,-5.580) -- (2.401,-5.699) -- (2.401,-5.819) -- (2.402,-5.938) -- (2.402,-6.058) -- (2.404,-6.176) -- (2.405,-6.295) -- (2.407,-6.413) -- (2.410,-6.530) -- (2.414,-6.646) -- (2.419,-6.761) -- (2.425,-6.875) -- (2.434,-6.986) -- (2.444,-7.096) -- (2.457,-7.203) -- (2.473,-7.307) -- (2.492,-7.408) -- (2.515,-7.505) -- (2.543,-7.597) -- (2.575,-7.685) -- (2.612,-7.768) -- (2.655,-7.845) -- (2.703,-7.917) -- (2.757,-7.983) -- (2.817,-8.043) -- (2.883,-8.097) -- (2.955,-8.145) -- (3.032,-8.188) -- (3.115,-8.225) -- (3.202,-8.257) -- (3.295,-8.284) -- (3.391,-8.307) -- (3.492,-8.326) -- (3.595,-8.341) -- (3.701,-8.354) -- (3.810,-8.363) -- (3.920,-8.370) -- (4.032,-8.374) -- (4.144,-8.376) -- (4.256,-8.376) -- (4.368,-8.374) -- (4.480,-8.370) -- (4.590,-8.363) -- (4.699,-8.354) -- (4.805,-8.341) -- (4.908,-8.326) -- (5.009,-8.307) -- (5.105,-8.284) -- (5.198,-8.257) -- (5.285,-8.225) -- (5.368,-8.188) -- (5.445,-8.145) -- (5.517,-8.097) -- (5.583,-8.043) -- (5.643,-7.983) -- (5.697,-7.917) -- (5.745,-7.845) -- (5.788,-7.768) -- (5.825,-7.685) -- (5.857,-7.597) -- (5.885,-7.505) -- (5.908,-7.408) -- (5.927,-7.307) -- (5.943,-7.203) -- (5.956,-7.096) -- (5.966,-6.986) -- (5.975,-6.875) -- (5.981,-6.761) -- (5.986,-6.646) -- (5.990,-6.530) -- (5.993,-6.413) -- (5.995,-6.295) -- (5.996,-6.176) -- (5.998,-6.058) -- (5.998,-5.938) -- (5.999,-5.819) -- (5.999,-5.699) -- (6.000,-5.580) -- (6.000,-5.460) -- (6.000,-5.340) -- (6.000,-5.220) -- (6.000,-5.100) -- (6.000,-4.980) -- (6.000,-4.860) -- (6.000,-4.740) -- (6.000,-4.620) -- (6.000,-4.500) -- (6.000,-4.380) -- (6.000,-4.260) -- (6.000,-4.140) -- (6.000,-4.020) -- (6.001,-3.901) -- (6.001,-3.781) -- (6.002,-3.662) -- (6.002,-3.542) -- (6.004,-3.424) -- (6.005,-3.305) -- (6.007,-3.187) -- (6.010,-3.070) -- (6.014,-2.954) -- (6.019,-2.839) -- (6.025,-2.725) -- (6.033,-2.614) -- (6.044,-2.504) -- (6.057,-2.397) -- (6.073,-2.293) -- (6.092,-2.192) -- (6.115,-2.095) -- (6.142,-2.003) -- (6.175,-1.915) -- (6.212,-1.833) -- (6.254,-1.755) -- (6.302,-1.684) -- (6.355,-1.619) -- (6.415,-1.560) -- (6.479,-1.507) -- (6.550,-1.460) -- (6.625,-1.419) -- (6.705,-1.385) -- (6.789,-1.357) -- (6.876,-1.334) -- (6.967,-1.317) -- (7.059,-1.306) -- (7.153,-1.301) -- (7.247,-1.301) -- (7.341,-1.306) -- (7.433,-1.317) -- (7.524,-1.334) -- (7.611,-1.357) -- (7.695,-1.385) -- (7.775,-1.419) -- (7.850,-1.460) -- (7.921,-1.507) -- (7.985,-1.560) -- (8.045,-1.619) -- (8.098,-1.684) -- (8.146,-1.755) -- (8.188,-1.833) -- (8.225,-1.915) -- (8.258,-2.003) -- (8.285,-2.095) -- (8.308,-2.192) -- (8.327,-2.293) -- (8.343,-2.397) -- (8.356,-2.504) -- (8.366,-2.613) -- (8.375,-2.725) -- (8.381,-2.839) -- (8.386,-2.954) -- (8.390,-3.070) -- (8.392,-3.187) -- (8.394,-3.304) -- (8.395,-3.422) -- (8.396,-3.541) -- (8.396,-3.659) -- (8.395,-3.778) -- (8.394,-3.896) -- (8.392,-4.013) -- (8.390,-4.130) -- (8.386,-4.246) -- (8.381,-4.361) -- (8.375,-4.475) -- (8.366,-4.587) -- (8.356,-4.696) -- (8.343,-4.803) -- (8.327,-4.907) -- (8.308,-5.008) -- (8.285,-5.105) -- (8.257,-5.197) -- (8.225,-5.285) -- (8.188,-5.368) -- (8.145,-5.445) -- (8.097,-5.517) -- (8.043,-5.583) -- (7.983,-5.643) -- (7.917,-5.697) -- (7.845,-5.745) -- (7.768,-5.788) -- (7.685,-5.825) -- (7.597,-5.857) -- (7.505,-5.885) -- (7.408,-5.908) -- (7.307,-5.927) -- (7.203,-5.943) -- (7.096,-5.956) -- (6.986,-5.966) -- (6.875,-5.975) -- (6.761,-5.981) -- (6.646,-5.986) -- (6.530,-5.990) -- (6.413,-5.993) -- (6.295,-5.995) -- (6.176,-5.996) -- (6.058,-5.998) -- (5.938,-5.998) -- (5.819,-5.999) -- (5.699,-5.999) -- (5.580,-6.000) -- (5.460,-6.000) -- (5.340,-6.000) -- (5.220,-6.000) -- (5.100,-6.000) -- (4.980,-6.000) -- (4.860,-6.000) -- (4.740,-6.000) -- (4.620,-6.000) -- (4.500,-6.000) -- (4.380,-6.000) -- (4.260,-6.000) -- (4.140,-6.000) -- (4.020,-6.000) -- (3.901,-5.999) -- (3.781,-5.999) -- (3.662,-5.998) -- (3.542,-5.998) -- (3.424,-5.996) -- (3.305,-5.995) -- (3.187,-5.993) -- (3.070,-5.990) -- (2.954,-5.986) -- (2.839,-5.981) -- (2.725,-5.975) -- (2.614,-5.966) -- (2.504,-5.956) -- (2.397,-5.943) -- (2.293,-5.927) -- (2.192,-5.908) -- (2.095,-5.885) -- (2.003,-5.857) -- (1.915,-5.825) -- (1.832,-5.788) -- (1.755,-5.745) -- (1.683,-5.697) -- (1.617,-5.643) -- (1.557,-5.583) -- (1.503,-5.517) -- (1.455,-5.445) -- (1.412,-5.368) -- (1.375,-5.285) -- (1.343,-5.198) -- (1.316,-5.105) -- (1.293,-5.009) -- (1.274,-4.908) -- (1.259,-4.805) -- (1.246,-4.699) -- (1.237,-4.590) -- (1.230,-4.480) -- (1.226,-4.368) -- (1.224,-4.256) -- (1.224,-4.144) -- (1.226,-4.032) -- (1.230,-3.920) -- (1.237,-3.810) -- (1.246,-3.701) -- (1.259,-3.595) -- (1.274,-3.492) -- (1.293,-3.391) -- (1.316,-3.295) -- (1.343,-3.202) -- (1.375,-3.115) -- (1.412,-3.032) -- (1.455,-2.955) -- (1.503,-2.883) -- (1.557,-2.817) -- (1.617,-2.757) -- (1.683,-2.703) -- (1.755,-2.655) -- (1.832,-2.612) -- (1.915,-2.575) -- (2.003,-2.543) -- (2.095,-2.515) -- (2.192,-2.492) -- (2.293,-2.473) -- (2.397,-2.457) -- (2.504,-2.444) -- (2.614,-2.434) -- (2.725,-2.425) -- (2.839,-2.419) -- (2.954,-2.414) -- (3.070,-2.410) -- (3.187,-2.407) -- (3.305,-2.405) -- (3.424,-2.404) -- (3.542,-2.402) -- (3.662,-2.402) -- (3.781,-2.401) -- (3.901,-2.401) -- (4.020,-2.400) -- (4.140,-2.400) -- (4.260,-2.400) -- (4.380,-2.400) -- (4.500,-2.400) -- (4.620,-2.400) -- (4.740,-2.400) -- (4.860,-2.400) -- (4.980,-2.400) -- (5.100,-2.400) -- (5.220,-2.400) -- (5.340,-2.400) -- (5.460,-2.400) -- (5.580,-2.400) -- (5.700,-2.400) -- (5.820,-2.400) -- (5.940,-2.400) -- (6.060,-2.400) -- (6.180,-2.400) -- (6.300,-2.400) -- (6.420,-2.400) -- (6.540,-2.400) -- (6.660,-2.400) -- (6.780,-2.400) -- (6.899,-2.401) -- (7.019,-2.401) -- (7.138,-2.402) -- (7.258,-2.402) -- (7.376,-2.404) -- (7.495,-2.405) -- (7.613,-2.407) -- (7.730,-2.410) -- (7.846,-2.414) -- (7.961,-2.419) -- (8.075,-2.425) -- (8.186,-2.434) -- (8.296,-2.444) -- (8.403,-2.457) -- (8.507,-2.473) -- (8.608,-2.492) -- (8.705,-2.515) -- (8.797,-2.543) -- (8.885,-2.575) -- (8.968,-2.612) -- (9.045,-2.655) -- (9.117,-2.703) -- (9.183,-2.757) -- (9.243,-2.817) -- (9.297,-2.883) -- (9.345,-2.955) -- (9.388,-3.032) -- (9.425,-3.115) -- (9.457,-3.203) -- (9.485,-3.295) -- (9.508,-3.392) -- (9.527,-3.493) -- (9.543,-3.597) -- (9.556,-3.704) -- (9.566,-3.813) -- (9.575,-3.925) -- (9.581,-4.039) -- (9.586,-4.154) -- (9.590,-4.270) -- (9.592,-4.387) -- (9.594,-4.504) -- (9.595,-4.622) -- (9.596,-4.741) -- (9.596,-4.859) -- (9.595,-4.978) -- (9.594,-5.096) -- (9.592,-5.213) -- (9.590,-5.330) -- (9.586,-5.446) -- (9.581,-5.561) -- (9.575,-5.675) -- (9.566,-5.787) -- (9.556,-5.896) -- (9.543,-6.003) -- (9.527,-6.107) -- (9.508,-6.208) -- (9.485,-6.305) -- (9.457,-6.397) -- (9.425,-6.485) -- (9.388,-6.568) -- (9.345,-6.645) -- (9.297,-6.717) -- (9.243,-6.783) -- (9.183,-6.843) -- (9.117,-6.897) -- (9.045,-6.945) -- (8.968,-6.988) -- (8.885,-7.025) -- (8.797,-7.057) -- (8.705,-7.085) -- (8.608,-7.108) -- (8.507,-7.127) -- (8.403,-7.143) -- (8.296,-7.156) -- (8.186,-7.166) -- (8.075,-7.175) -- (7.961,-7.181) -- (7.846,-7.186) -- (7.730,-7.190) -- (7.613,-7.193) -- (7.495,-7.195) -- (7.376,-7.196) -- (7.258,-7.198) -- (7.138,-7.198) -- (7.019,-7.199) -- (6.899,-7.199) -- (6.780,-7.199) -- (6.660,-7.199) -- (6.540,-7.199) -- (6.420,-7.199) -- (6.301,-7.199) -- (6.181,-7.199) -- (6.062,-7.198) -- (5.942,-7.198) -- (5.824,-7.196) -- (5.705,-7.195) -- (5.587,-7.193) -- (5.470,-7.190) -- (5.354,-7.186) -- (5.239,-7.181) -- (5.125,-7.175) -- (5.014,-7.166) -- (4.904,-7.156) -- (4.797,-7.143) -- (4.693,-7.127) -- (4.592,-7.108) -- (4.495,-7.085) -- (4.403,-7.057) -- (4.315,-7.025) -- (4.232,-6.988) -- (4.155,-6.945) -- (4.083,-6.897) -- (4.017,-6.843) -- (3.957,-6.783) -- (3.903,-6.717) -- (3.855,-6.645) -- (3.812,-6.568) -- (3.775,-6.485) -- (3.743,-6.398) -- (3.716,-6.305) -- (3.693,-6.209) -- (3.674,-6.108) -- (3.659,-6.005) -- (3.646,-5.899) -- (3.637,-5.790) -- (3.630,-5.680) -- (3.626,-5.568) -- (3.624,-5.456) -- (3.624,-5.344) -- (3.626,-5.232) -- (3.630,-5.120) -- (3.637,-5.010) -- (3.646,-4.901) -- (3.659,-4.795) -- (3.674,-4.692) -- (3.693,-4.591) -- (3.716,-4.495) -- (3.743,-4.402) -- (3.775,-4.315) -- (3.812,-4.232) -- (3.855,-4.155) -- (3.903,-4.083) -- (3.957,-4.017) -- (4.017,-3.957) -- (4.083,-3.903) -- (4.155,-3.855) -- (4.232,-3.812) -- (4.315,-3.775) -- (4.402,-3.743) -- (4.495,-3.716) -- (4.591,-3.693) -- (4.692,-3.674) -- (4.795,-3.658) -- (4.901,-3.646) -- (5.010,-3.637) -- (5.120,-3.630) -- (5.231,-3.626) -- (5.344,-3.624) -- (5.456,-3.624) -- (5.568,-3.626) -- (5.679,-3.630) -- (5.789,-3.636) -- (5.897,-3.645) -- (6.002,-3.656) -- (6.105,-3.670) -- (6.204,-3.688) -- (6.298,-3.708) -- (6.388,-3.733) -- (6.471,-3.761) -- (6.549,-3.793) -- (6.620,-3.830) -- (6.684,-3.870) -- (6.739,-3.913) -- cycle;
\draw[white, line width=4.5pt, line cap=butt] (6.000,-4.860) -- (6.000,-4.740) -- (6.000,-4.701) -- (6.000,-4.620) -- (6.000,-4.500);
\draw[line cap=butt] (6.000,-4.860) -- (6.000,-4.740) -- (6.000,-4.701) -- (6.000,-4.620) -- (6.000,-4.500);
\draw[white, line width=4.5pt, line cap=butt] (4.891,-3.809) -- (4.872,-3.708) -- (4.862,-3.651) -- (4.855,-3.605) -- (4.841,-3.499);
\draw[line cap=butt] (4.891,-3.809) -- (4.872,-3.708) -- (4.862,-3.651) -- (4.855,-3.605) -- (4.841,-3.499);
\draw[white, line width=4.5pt, line cap=butt] (4.620,-2.400) -- (4.740,-2.400) -- (4.745,-2.400) -- (4.860,-2.400) -- (4.980,-2.400);
\draw[line cap=butt] (4.620,-2.400) -- (4.740,-2.400) -- (4.745,-2.400) -- (4.860,-2.400) -- (4.980,-2.400);
\draw[white, line width=4.5pt, line cap=butt] (2.473,-2.293) -- (2.457,-2.397) -- (2.450,-2.450) -- (2.444,-2.504) -- (2.433,-2.614);
\draw[line cap=butt] (2.473,-2.293) -- (2.457,-2.397) -- (2.450,-2.450) -- (2.444,-2.504) -- (2.433,-2.614);
\draw[white, line width=4.5pt, line cap=butt] (2.614,-5.966) -- (2.504,-5.956) -- (2.402,-5.944) -- (2.397,-5.943) -- (2.293,-5.927) -- (2.192,-5.908);
\draw[line cap=butt] (2.614,-5.966) -- (2.504,-5.956) -- (2.402,-5.944) -- (2.397,-5.943) -- (2.293,-5.927) -- (2.192,-5.908);
\draw[white, line width=4.5pt, line cap=butt] (5.908,-7.408) -- (5.927,-7.307) -- (5.943,-7.203) -- (5.944,-7.198) -- (5.956,-7.096) -- (5.966,-6.986);
\draw[line cap=butt] (5.908,-7.408) -- (5.927,-7.307) -- (5.943,-7.203) -- (5.944,-7.198) -- (5.956,-7.096) -- (5.966,-6.986);
\draw[white, line width=4.5pt, line cap=butt] (6.176,-5.996) -- (6.058,-5.998) -- (5.998,-5.998) -- (5.938,-5.998) -- (5.819,-5.999);
\draw[line cap=butt] (6.176,-5.996) -- (6.058,-5.998) -- (5.998,-5.998) -- (5.938,-5.998) -- (5.819,-5.999);
\draw[white, line width=4.5pt, line cap=butt] (5.789,-3.636) -- (5.897,-3.645) -- (6.002,-3.656) -- (6.002,-3.656) -- (6.105,-3.670) -- (6.204,-3.688);
\draw[line cap=butt] (5.789,-3.636) -- (5.897,-3.645) -- (6.002,-3.656) -- (6.002,-3.656) -- (6.105,-3.670) -- (6.204,-3.688);
\draw[white, line width=4.5pt, line cap=butt] (6.033,-2.614) -- (6.044,-2.504) -- (6.056,-2.400) -- (6.057,-2.397) -- (6.073,-2.293) -- (6.092,-2.192);
\draw[line cap=butt] (6.033,-2.614) -- (6.044,-2.504) -- (6.056,-2.400) -- (6.057,-2.397) -- (6.073,-2.293) -- (6.092,-2.192);
\draw[white, line width=4.5pt, line cap=butt] (8.186,-2.434) -- (8.296,-2.444) -- (8.350,-2.450) -- (8.403,-2.457) -- (8.507,-2.473);
\draw[line cap=butt] (8.186,-2.434) -- (8.296,-2.444) -- (8.350,-2.450) -- (8.403,-2.457) -- (8.507,-2.473);
\draw[white, line width=4.5pt, line cap=butt] (3.693,-6.209) -- (3.674,-6.108) -- (3.659,-6.005) -- (3.658,-5.998) -- (3.646,-5.899) -- (3.637,-5.790);
\draw[line cap=butt] (3.693,-6.209) -- (3.674,-6.108) -- (3.659,-6.005) -- (3.658,-5.998) -- (3.646,-5.899) -- (3.637,-5.790);
\end{tikzpicture}

%% file: figures/13a4237-checkerboard.tex
\begin{tikzpicture}[knot/.style={line width=1.0pt, line cap=round, line join=round},
                    shaded/.style={fill=black!22}, whitereg/.style={fill=white},
                    reglabel/.style={font=\small, inner sep=1pt}]
\fill[shaded, rounded corners=8pt] (0.774,-8.749) rectangle (11.249,-0.751);
\fill[whitereg] (2.450,-3.650) -- (2.457,-3.597) -- (2.473,-3.493) -- (2.492,-3.392) -- (2.515,-3.295) -- (2.543,-3.203) -- (2.575,-3.115) -- (2.612,-3.032) -- (2.655,-2.955) -- (2.703,-2.883) -- (2.757,-2.817) -- (2.817,-2.757) -- (2.883,-2.703) -- (2.955,-2.655) -- (3.032,-2.612) -- (3.115,-2.575) -- (3.203,-2.543) -- (3.295,-2.515) -- (3.392,-2.492) -- (3.493,-2.473) -- (3.597,-2.457) -- (3.704,-2.444) -- (3.814,-2.434) -- (3.925,-2.425) -- (4.039,-2.419) -- (4.154,-2.414) -- (4.270,-2.410) -- (4.387,-2.407) -- (4.505,-2.405) -- (4.624,-2.404) -- (4.742,-2.402) -- (4.856,-2.402) -- (4.844,-2.504) -- (4.834,-2.614) -- (4.825,-2.725) -- (4.819,-2.839) -- (4.814,-2.954) -- (4.810,-3.070) -- (4.807,-3.187) -- (4.805,-3.305) -- (4.804,-3.424) -- (4.802,-3.542) -- (4.802,-3.657) -- (4.696,-3.644) -- (4.587,-3.634) -- (4.475,-3.625) -- (4.361,-3.619) -- (4.246,-3.614) -- (4.130,-3.610) -- (4.013,-3.608) -- (3.896,-3.606) -- (3.778,-3.605) -- (3.659,-3.604) -- (3.541,-3.604) -- (3.422,-3.605) -- (3.304,-3.606) -- (3.187,-3.608) -- (3.070,-3.610) -- (2.954,-3.614) -- (2.839,-3.619) -- (2.725,-3.625) -- (2.613,-3.634) -- (2.504,-3.644) -- cycle;
\fill[whitereg] (5.944,-4.802) -- (5.942,-4.802) -- (5.824,-4.804) -- (5.705,-4.805) -- (5.587,-4.807) -- (5.470,-4.810) -- (5.354,-4.814) -- (5.239,-4.818) -- (5.126,-4.825) -- (5.015,-4.832) -- (4.906,-4.842) -- (4.800,-4.854) -- (4.800,-4.740) -- (4.800,-4.620) -- (4.800,-4.500) -- (4.800,-4.380) -- (4.800,-4.260) -- (4.800,-4.140) -- (4.800,-4.020) -- (4.801,-3.901) -- (4.801,-3.781) -- (4.802,-3.662) -- (4.802,-3.657) -- (4.803,-3.657) -- (4.907,-3.673) -- (5.008,-3.692) -- (5.105,-3.715) -- (5.197,-3.743) -- (5.285,-3.775) -- (5.368,-3.812) -- (5.445,-3.855) -- (5.517,-3.903) -- (5.583,-3.957) -- (5.643,-4.017) -- (5.697,-4.083) -- (5.745,-4.155) -- (5.788,-4.232) -- (5.825,-4.315) -- (5.857,-4.403) -- (5.885,-4.495) -- (5.908,-4.592) -- (5.927,-4.693) -- (5.943,-4.797) -- cycle;
\fill[whitereg] (8.339,-4.751) -- (8.296,-4.756) -- (8.186,-4.767) -- (8.075,-4.775) -- (7.961,-4.781) -- (7.846,-4.786) -- (7.730,-4.790) -- (7.613,-4.793) -- (7.495,-4.795) -- (7.376,-4.796) -- (7.258,-4.798) -- (7.199,-4.798) -- (7.210,-4.755) -- (7.236,-4.663) -- (7.262,-4.572) -- (7.289,-4.484) -- (7.319,-4.399) -- (7.351,-4.319) -- (7.386,-4.244) -- (7.424,-4.175) -- (7.466,-4.113) -- (7.511,-4.059) -- (7.559,-4.012) -- (7.609,-3.975) -- (7.662,-3.946) -- (7.716,-3.927) -- (7.772,-3.917) -- (7.828,-3.917) -- (7.883,-3.927) -- (7.937,-3.947) -- (7.990,-3.976) -- (8.040,-4.014) -- (8.087,-4.061) -- (8.131,-4.116) -- (8.171,-4.180) -- (8.207,-4.251) -- (8.239,-4.329) -- (8.267,-4.413) -- (8.292,-4.502) -- (8.313,-4.597) -- (8.331,-4.696) -- cycle;
\fill[whitereg] (5.944,-4.802) -- (6.062,-4.802) -- (6.181,-4.801) -- (6.301,-4.801) -- (6.420,-4.800) -- (6.540,-4.800) -- (6.660,-4.800) -- (6.780,-4.800) -- (6.899,-4.799) -- (7.019,-4.799) -- (7.138,-4.798) -- (7.199,-4.798) -- (7.185,-4.849) -- (7.160,-4.942) -- (7.132,-5.034) -- (7.103,-5.124) -- (7.071,-5.211) -- (7.035,-5.295) -- (6.995,-5.375) -- (6.950,-5.450) -- (6.901,-5.521) -- (6.845,-5.585) -- (6.785,-5.645) -- (6.718,-5.698) -- (6.646,-5.746) -- (6.569,-5.788) -- (6.486,-5.825) -- (6.398,-5.857) -- (6.305,-5.884) -- (6.209,-5.907) -- (6.108,-5.926) -- (6.005,-5.942) -- (5.999,-5.942) -- (5.999,-5.941) -- (5.998,-5.822) -- (5.996,-5.704) -- (5.993,-5.587) -- (5.990,-5.470) -- (5.986,-5.354) -- (5.981,-5.239) -- (5.975,-5.125) -- (5.967,-5.013) -- (5.956,-4.904) -- cycle;
\fill[whitereg] (2.450,-3.650) -- (2.444,-3.704) -- (2.434,-3.814) -- (2.425,-3.925) -- (2.419,-4.039) -- (2.414,-4.154) -- (2.410,-4.270) -- (2.407,-4.387) -- (2.405,-4.505) -- (2.404,-4.624) -- (2.402,-4.742) -- (2.402,-4.862) -- (2.401,-4.981) -- (2.401,-5.101) -- (2.401,-5.220) -- (2.401,-5.340) -- (2.401,-5.460) -- (2.401,-5.580) -- (2.401,-5.699) -- (2.401,-5.819) -- (2.402,-5.938) -- (2.402,-6.058) -- (2.404,-6.176) -- (2.405,-6.295) -- (2.407,-6.413) -- (2.410,-6.530) -- (2.414,-6.646) -- (2.419,-6.761) -- (2.425,-6.875) -- (2.433,-6.986) -- (2.444,-7.096) -- (2.450,-7.150) -- (2.397,-7.143) -- (2.293,-7.127) -- (2.192,-7.108) -- (2.095,-7.085) -- (2.003,-7.057) -- (1.915,-7.025) -- (1.832,-6.988) -- (1.755,-6.945) -- (1.683,-6.897) -- (1.617,-6.843) -- (1.557,-6.783) -- (1.503,-6.717) -- (1.455,-6.645) -- (1.412,-6.568) -- (1.375,-6.485) -- (1.343,-6.398) -- (1.316,-6.305) -- (1.293,-6.209) -- (1.274,-6.108) -- (1.259,-6.005) -- (1.246,-5.899) -- (1.237,-5.790) -- (1.230,-5.680) -- (1.226,-5.568) -- (1.224,-5.456) -- (1.224,-5.344) -- (1.226,-5.232) -- (1.230,-5.120) -- (1.237,-5.010) -- (1.246,-4.901) -- (1.259,-4.795) -- (1.274,-4.692) -- (1.293,-4.591) -- (1.316,-4.495) -- (1.343,-4.402) -- (1.375,-4.315) -- (1.412,-4.232) -- (1.455,-4.155) -- (1.503,-4.083) -- (1.557,-4.017) -- (1.617,-3.957) -- (1.683,-3.903) -- (1.755,-3.855) -- (1.832,-3.812) -- (1.915,-3.775) -- (2.003,-3.743) -- (2.095,-3.715) -- (2.192,-3.692) -- (2.293,-3.673) -- (2.397,-3.657) -- cycle;
\fill[whitereg] (4.014,-5.640) -- (3.976,-5.590) -- (3.947,-5.537) -- (3.927,-5.483) -- (3.918,-5.428) -- (3.918,-5.372) -- (3.927,-5.317) -- (3.947,-5.263) -- (3.976,-5.210) -- (4.014,-5.160) -- (4.061,-5.113) -- (4.116,-5.069) -- (4.180,-5.029) -- (4.251,-4.993) -- (4.329,-4.961) -- (4.413,-4.933) -- (4.502,-4.908) -- (4.597,-4.887) -- (4.696,-4.869) -- (4.799,-4.854) -- (4.800,-4.854) -- (4.800,-4.860) -- (4.800,-4.980) -- (4.800,-5.100) -- (4.800,-5.220) -- (4.800,-5.340) -- (4.800,-5.460) -- (4.800,-5.580) -- (4.799,-5.699) -- (4.799,-5.819) -- (4.798,-5.938) -- (4.798,-5.944) -- (4.798,-5.944) -- (4.695,-5.930) -- (4.596,-5.912) -- (4.502,-5.892) -- (4.412,-5.867) -- (4.329,-5.839) -- (4.251,-5.807) -- (4.180,-5.770) -- (4.116,-5.730) -- (4.061,-5.687) -- cycle;
\fill[whitereg] (8.339,-4.751) -- (8.403,-4.743) -- (8.507,-4.727) -- (8.608,-4.708) -- (8.705,-4.685) -- (8.797,-4.658) -- (8.885,-4.625) -- (8.967,-4.588) -- (9.045,-4.546) -- (9.116,-4.498) -- (9.181,-4.445) -- (9.240,-4.385) -- (9.293,-4.321) -- (9.340,-4.250) -- (9.381,-4.175) -- (9.415,-4.095) -- (9.443,-4.011) -- (9.466,-3.924) -- (9.483,-3.833) -- (9.494,-3.741) -- (9.499,-3.647) -- (9.499,-3.553) -- (9.494,-3.459) -- (9.483,-3.367) -- (9.466,-3.276) -- (9.443,-3.189) -- (9.415,-3.105) -- (9.381,-3.025) -- (9.340,-2.950) -- (9.293,-2.879) -- (9.240,-2.815) -- (9.181,-2.755) -- (9.116,-2.702) -- (9.045,-2.654) -- (8.967,-2.612) -- (8.885,-2.575) -- (8.797,-2.542) -- (8.705,-2.515) -- (8.608,-2.492) -- (8.507,-2.473) -- (8.403,-2.457) -- (8.296,-2.444) -- (8.186,-2.433) -- (8.075,-2.425) -- (7.961,-2.419) -- (7.846,-2.414) -- (7.730,-2.410) -- (7.613,-2.407) -- (7.495,-2.405) -- (7.376,-2.404) -- (7.258,-2.402) -- (7.138,-2.402) -- (7.019,-2.401) -- (6.899,-2.401) -- (6.780,-2.400) -- (6.660,-2.400) -- (6.540,-2.400) -- (6.420,-2.400) -- (6.300,-2.400) -- (6.180,-2.400) -- (6.060,-2.400) -- (5.940,-2.400) -- (5.820,-2.400) -- (5.700,-2.400) -- (5.580,-2.400) -- (5.460,-2.400) -- (5.340,-2.400) -- (5.220,-2.400) -- (5.101,-2.401) -- (4.981,-2.401) -- (4.862,-2.402) -- (4.856,-2.402) -- (4.857,-2.397) -- (4.873,-2.293) -- (4.892,-2.192) -- (4.915,-2.095) -- (4.943,-2.003) -- (4.975,-1.915) -- (5.012,-1.832) -- (5.055,-1.755) -- (5.103,-1.683) -- (5.157,-1.617) -- (5.217,-1.557) -- (5.283,-1.503) -- (5.355,-1.455) -- (5.432,-1.412) -- (5.515,-1.375) -- (5.603,-1.343) -- (5.695,-1.315) -- (5.792,-1.292) -- (5.893,-1.273) -- (5.997,-1.257) -- (6.104,-1.244) -- (6.214,-1.234) -- (6.325,-1.225) -- (6.439,-1.219) -- (6.554,-1.214) -- (6.670,-1.210) -- (6.787,-1.207) -- (6.905,-1.205) -- (7.024,-1.204) -- (7.142,-1.202) -- (7.262,-1.202) -- (7.381,-1.201) -- (7.501,-1.201) -- (7.620,-1.201) -- (7.740,-1.201) -- (7.860,-1.201) -- (7.980,-1.201) -- (8.099,-1.201) -- (8.219,-1.201) -- (8.338,-1.202) -- (8.458,-1.202) -- (8.576,-1.204) -- (8.695,-1.205) -- (8.813,-1.207) -- (8.930,-1.210) -- (9.046,-1.214) -- (9.161,-1.219) -- (9.275,-1.225) -- (9.386,-1.234) -- (9.496,-1.244) -- (9.603,-1.257) -- (9.707,-1.273) -- (9.808,-1.292) -- (9.905,-1.315) -- (9.997,-1.343) -- (10.085,-1.375) -- (10.168,-1.412) -- (10.245,-1.455) -- (10.317,-1.503) -- (10.383,-1.557) -- (10.443,-1.617) -- (10.497,-1.683) -- (10.545,-1.755) -- (10.588,-1.832) -- (10.625,-1.915) -- (10.657,-2.003) -- (10.685,-2.095) -- (10.708,-2.192) -- (10.727,-2.293) -- (10.743,-2.397) -- (10.756,-2.504) -- (10.766,-2.614) -- (10.775,-2.725) -- (10.781,-2.839) -- (10.786,-2.954) -- (10.790,-3.070) -- (10.793,-3.187) -- (10.795,-3.305) -- (10.796,-3.424) -- (10.798,-3.542) -- (10.798,-3.662) -- (10.799,-3.781) -- (10.799,-3.901) -- (10.799,-4.020) -- (10.799,-4.140) -- (10.799,-4.260) -- (10.799,-4.380) -- (10.799,-4.499) -- (10.799,-4.619) -- (10.798,-4.738) -- (10.798,-4.858) -- (10.796,-4.976) -- (10.795,-5.095) -- (10.793,-5.213) -- (10.790,-5.330) -- (10.786,-5.446) -- (10.781,-5.561) -- (10.775,-5.675) -- (10.766,-5.786) -- (10.756,-5.896) -- (10.743,-6.003) -- (10.727,-6.107) -- (10.708,-6.208) -- (10.685,-6.305) -- (10.657,-6.397) -- (10.625,-6.485) -- (10.588,-6.568) -- (10.545,-6.645) -- (10.497,-6.717) -- (10.443,-6.783) -- (10.383,-6.843) -- (10.317,-6.897) -- (10.245,-6.945) -- (10.168,-6.988) -- (10.085,-7.025) -- (9.997,-7.057) -- (9.905,-7.085) -- (9.808,-7.108) -- (9.707,-7.127) -- (9.603,-7.143) -- (9.496,-7.156) -- (9.386,-7.166) -- (9.275,-7.175) -- (9.161,-7.181) -- (9.046,-7.186) -- (8.930,-7.190) -- (8.813,-7.193) -- (8.695,-7.195) -- (8.576,-7.196) -- (8.458,-7.198) -- (8.344,-7.198) -- (8.356,-7.096) -- (8.366,-6.987) -- (8.375,-6.875) -- (8.381,-6.761) -- (8.386,-6.646) -- (8.390,-6.530) -- (8.392,-6.413) -- (8.394,-6.296) -- (8.395,-6.178) -- (8.396,-6.059) -- (8.396,-5.941) -- (8.395,-5.822) -- (8.394,-5.704) -- (8.392,-5.587) -- (8.390,-5.470) -- (8.386,-5.354) -- (8.382,-5.239) -- (8.375,-5.126) -- (8.368,-5.015) -- (8.358,-4.906) -- (8.346,-4.799) -- cycle;
\fill[whitereg] (4.744,-7.200) -- (4.756,-7.096) -- (4.767,-6.986) -- (4.775,-6.875) -- (4.781,-6.761) -- (4.786,-6.646) -- (4.790,-6.530) -- (4.793,-6.413) -- (4.795,-6.295) -- (4.796,-6.176) -- (4.798,-6.058) -- (4.798,-5.944) -- (4.903,-5.955) -- (5.011,-5.964) -- (5.121,-5.970) -- (5.232,-5.974) -- (5.344,-5.976) -- (5.456,-5.976) -- (5.569,-5.974) -- (5.680,-5.970) -- (5.790,-5.963) -- (5.899,-5.954) -- (5.999,-5.942) -- (6.001,-6.059) -- (6.002,-6.178) -- (6.004,-6.296) -- (6.007,-6.413) -- (6.010,-6.530) -- (6.014,-6.646) -- (6.019,-6.761) -- (6.025,-6.875) -- (6.033,-6.987) -- (6.044,-7.096) -- (6.056,-7.200) -- (5.940,-7.200) -- (5.820,-7.200) -- (5.700,-7.200) -- (5.580,-7.200) -- (5.460,-7.200) -- (5.340,-7.200) -- (5.220,-7.200) -- (5.100,-7.200) -- (4.980,-7.200) -- (4.860,-7.200) -- cycle;
\fill[whitereg] (4.744,-7.200) -- (4.743,-7.203) -- (4.727,-7.307) -- (4.708,-7.408) -- (4.685,-7.505) -- (4.658,-7.597) -- (4.625,-7.685) -- (4.588,-7.767) -- (4.546,-7.845) -- (4.498,-7.916) -- (4.445,-7.981) -- (4.385,-8.040) -- (4.321,-8.093) -- (4.250,-8.140) -- (4.175,-8.181) -- (4.095,-8.215) -- (4.011,-8.243) -- (3.924,-8.266) -- (3.833,-8.283) -- (3.741,-8.294) -- (3.647,-8.299) -- (3.553,-8.299) -- (3.459,-8.294) -- (3.367,-8.283) -- (3.276,-8.266) -- (3.189,-8.243) -- (3.105,-8.215) -- (3.025,-8.181) -- (2.950,-8.140) -- (2.879,-8.093) -- (2.815,-8.040) -- (2.755,-7.981) -- (2.702,-7.916) -- (2.654,-7.845) -- (2.612,-7.767) -- (2.575,-7.685) -- (2.542,-7.597) -- (2.515,-7.505) -- (2.492,-7.408) -- (2.473,-7.307) -- (2.457,-7.203) -- (2.450,-7.150) -- (2.504,-7.156) -- (2.614,-7.166) -- (2.725,-7.175) -- (2.839,-7.181) -- (2.954,-7.186) -- (3.070,-7.190) -- (3.187,-7.193) -- (3.305,-7.195) -- (3.424,-7.196) -- (3.542,-7.198) -- (3.662,-7.198) -- (3.781,-7.199) -- (3.901,-7.199) -- (4.020,-7.200) -- (4.140,-7.200) -- (4.260,-7.200) -- (4.380,-7.200) -- (4.500,-7.200) -- (4.620,-7.200) -- (4.740,-7.200) -- cycle;
\fill[whitereg] (6.056,-7.200) -- (6.060,-7.200) -- (6.180,-7.200) -- (6.300,-7.200) -- (6.420,-7.200) -- (6.540,-7.200) -- (6.660,-7.200) -- (6.780,-7.200) -- (6.900,-7.200) -- (7.020,-7.200) -- (7.140,-7.200) -- (7.260,-7.200) -- (7.380,-7.200) -- (7.500,-7.200) -- (7.620,-7.200) -- (7.740,-7.200) -- (7.860,-7.200) -- (7.980,-7.200) -- (8.099,-7.199) -- (8.219,-7.199) -- (8.338,-7.198) -- (8.344,-7.198) -- (8.343,-7.203) -- (8.327,-7.307) -- (8.308,-7.408) -- (8.285,-7.505) -- (8.258,-7.597) -- (8.225,-7.685) -- (8.188,-7.767) -- (8.146,-7.845) -- (8.098,-7.916) -- (8.045,-7.981) -- (7.985,-8.040) -- (7.921,-8.093) -- (7.850,-8.140) -- (7.775,-8.181) -- (7.695,-8.215) -- (7.611,-8.243) -- (7.524,-8.266) -- (7.433,-8.283) -- (7.341,-8.294) -- (7.247,-8.299) -- (7.153,-8.299) -- (7.059,-8.294) -- (6.967,-8.283) -- (6.876,-8.266) -- (6.789,-8.243) -- (6.705,-8.215) -- (6.625,-8.181) -- (6.550,-8.140) -- (6.479,-8.093) -- (6.415,-8.040) -- (6.355,-7.981) -- (6.302,-7.916) -- (6.254,-7.845) -- (6.212,-7.767) -- (6.175,-7.685) -- (6.142,-7.597) -- (6.115,-7.505) -- (6.092,-7.408) -- (6.073,-7.307) -- (6.057,-7.203) -- cycle;
\draw[knot] (4.929,-4.840) -- (5.015,-4.832) -- (5.126,-4.825) -- (5.239,-4.818) -- (5.354,-4.814) -- (5.470,-4.810) -- (5.587,-4.807) -- (5.705,-4.805) -- (5.824,-4.804) -- (5.942,-4.802) -- (6.062,-4.802) -- (6.181,-4.801) -- (6.301,-4.801) -- (6.420,-4.800) -- (6.540,-4.800) -- (6.660,-4.800) -- (6.780,-4.800) -- (6.899,-4.799) -- (7.019,-4.799) -- (7.069,-4.799);
\draw[knot] (7.329,-4.797) -- (7.376,-4.796) -- (7.495,-4.795) -- (7.613,-4.793) -- (7.730,-4.790) -- (7.846,-4.786) -- (7.961,-4.781) -- (8.075,-4.775) -- (8.186,-4.767) -- (8.296,-4.756) -- (8.403,-4.743) -- (8.507,-4.727) -- (8.608,-4.708) -- (8.705,-4.685) -- (8.797,-4.658) -- (8.885,-4.625) -- (8.967,-4.588) -- (9.045,-4.546) -- (9.116,-4.498) -- (9.181,-4.445) -- (9.240,-4.385) -- (9.293,-4.321) -- (9.340,-4.250) -- (9.381,-4.175) -- (9.415,-4.095) -- (9.443,-4.011) -- (9.466,-3.924) -- (9.483,-3.833) -- (9.494,-3.741) -- (9.499,-3.647) -- (9.499,-3.553) -- (9.494,-3.459) -- (9.483,-3.367) -- (9.466,-3.276) -- (9.443,-3.189) -- (9.415,-3.105) -- (9.381,-3.025) -- (9.340,-2.950) -- (9.293,-2.879) -- (9.240,-2.815) -- (9.181,-2.755) -- (9.116,-2.702) -- (9.045,-2.654) -- (8.967,-2.612) -- (8.885,-2.575) -- (8.797,-2.542) -- (8.705,-2.515) -- (8.608,-2.492) -- (8.507,-2.473) -- (8.403,-2.457) -- (8.296,-2.444) -- (8.186,-2.433) -- (8.075,-2.425) -- (7.961,-2.419) -- (7.846,-2.414) -- (7.730,-2.410) -- (7.613,-2.407) -- (7.495,-2.405) -- (7.376,-2.404) -- (7.258,-2.402) -- (7.138,-2.402) -- (7.019,-2.401) -- (6.899,-2.401) -- (6.780,-2.400) -- (6.660,-2.400) -- (6.540,-2.400) -- (6.420,-2.400) -- (6.300,-2.400) -- (6.180,-2.400) -- (6.060,-2.400) -- (5.940,-2.400) -- (5.820,-2.400) -- (5.700,-2.400) -- (5.580,-2.400) -- (5.460,-2.400) -- (5.340,-2.400) -- (5.220,-2.400) -- (5.101,-2.401) -- (4.986,-2.401);
\draw[knot] (4.726,-2.403) -- (4.624,-2.404) -- (4.505,-2.405) -- (4.387,-2.407) -- (4.270,-2.410) -- (4.154,-2.414) -- (4.039,-2.419) -- (3.925,-2.425) -- (3.814,-2.434) -- (3.704,-2.444) -- (3.597,-2.457) -- (3.493,-2.473) -- (3.392,-2.492) -- (3.295,-2.515) -- (3.203,-2.543) -- (3.115,-2.575) -- (3.032,-2.612) -- (2.955,-2.655) -- (2.883,-2.703) -- (2.817,-2.757) -- (2.757,-2.817) -- (2.703,-2.883) -- (2.655,-2.955) -- (2.612,-3.032) -- (2.575,-3.115) -- (2.543,-3.203) -- (2.515,-3.295) -- (2.492,-3.392) -- (2.473,-3.493) -- (2.457,-3.597) -- (2.444,-3.704) -- (2.434,-3.814) -- (2.425,-3.925) -- (2.419,-4.039) -- (2.414,-4.154) -- (2.410,-4.270) -- (2.407,-4.387) -- (2.405,-4.505) -- (2.404,-4.624) -- (2.402,-4.742) -- (2.402,-4.862) -- (2.401,-4.981) -- (2.401,-5.101) -- (2.401,-5.220) -- (2.401,-5.340) -- (2.401,-5.460) -- (2.401,-5.580) -- (2.401,-5.699) -- (2.401,-5.819) -- (2.402,-5.938) -- (2.402,-6.058) -- (2.404,-6.176) -- (2.405,-6.295) -- (2.407,-6.413) -- (2.410,-6.530) -- (2.414,-6.646) -- (2.419,-6.761) -- (2.425,-6.875) -- (2.433,-6.986) -- (2.437,-7.020);
\draw[knot] (2.468,-7.278) -- (2.473,-7.307) -- (2.492,-7.408) -- (2.515,-7.505) -- (2.542,-7.597) -- (2.575,-7.685) -- (2.612,-7.767) -- (2.654,-7.845) -- (2.702,-7.916) -- (2.755,-7.981) -- (2.815,-8.040) -- (2.879,-8.093) -- (2.950,-8.140) -- (3.025,-8.181) -- (3.105,-8.215) -- (3.189,-8.243) -- (3.276,-8.266) -- (3.367,-8.283) -- (3.459,-8.294) -- (3.553,-8.299) -- (3.647,-8.299) -- (3.741,-8.294) -- (3.833,-8.283) -- (3.924,-8.266) -- (4.011,-8.243) -- (4.095,-8.215) -- (4.175,-8.181) -- (4.250,-8.140) -- (4.321,-8.093) -- (4.385,-8.040) -- (4.445,-7.981) -- (4.498,-7.916) -- (4.546,-7.845) -- (4.588,-7.767) -- (4.625,-7.685) -- (4.658,-7.597) -- (4.685,-7.505) -- (4.708,-7.408) -- (4.727,-7.307) -- (4.743,-7.203) -- (4.756,-7.096) -- (4.767,-6.986) -- (4.775,-6.875) -- (4.781,-6.761) -- (4.786,-6.646) -- (4.790,-6.530) -- (4.793,-6.413) -- (4.795,-6.295) -- (4.796,-6.176) -- (4.797,-6.074);
\draw[knot] (4.799,-5.814) -- (4.799,-5.699) -- (4.800,-5.580) -- (4.800,-5.460) -- (4.800,-5.340) -- (4.800,-5.220) -- (4.800,-5.100) -- (4.800,-4.980) -- (4.800,-4.860) -- (4.800,-4.740) -- (4.800,-4.620) -- (4.800,-4.500) -- (4.800,-4.380) -- (4.800,-4.260) -- (4.800,-4.140) -- (4.800,-4.020) -- (4.801,-3.901) -- (4.801,-3.787);
\draw[knot] (4.803,-3.527) -- (4.804,-3.424) -- (4.805,-3.305) -- (4.807,-3.187) -- (4.810,-3.070) -- (4.814,-2.954) -- (4.819,-2.839) -- (4.825,-2.725) -- (4.834,-2.614) -- (4.844,-2.504) -- (4.857,-2.397) -- (4.873,-2.293) -- (4.892,-2.192) -- (4.915,-2.095) -- (4.943,-2.003) -- (4.975,-1.915) -- (5.012,-1.832) -- (5.055,-1.755) -- (5.103,-1.683) -- (5.157,-1.617) -- (5.217,-1.557) -- (5.283,-1.503) -- (5.355,-1.455) -- (5.432,-1.412) -- (5.515,-1.375) -- (5.603,-1.343) -- (5.695,-1.315) -- (5.792,-1.292) -- (5.893,-1.273) -- (5.997,-1.257) -- (6.104,-1.244) -- (6.214,-1.234) -- (6.325,-1.225) -- (6.439,-1.219) -- (6.554,-1.214) -- (6.670,-1.210) -- (6.787,-1.207) -- (6.905,-1.205) -- (7.024,-1.204) -- (7.142,-1.202) -- (7.262,-1.202) -- (7.381,-1.201) -- (7.501,-1.201) -- (7.620,-1.201) -- (7.740,-1.201) -- (7.860,-1.201) -- (7.980,-1.201) -- (8.099,-1.201) -- (8.219,-1.201) -- (8.338,-1.202) -- (8.458,-1.202) -- (8.576,-1.204) -- (8.695,-1.205) -- (8.813,-1.207) -- (8.930,-1.210) -- (9.046,-1.214) -- (9.161,-1.219) -- (9.275,-1.225) -- (9.386,-1.234) -- (9.496,-1.244) -- (9.603,-1.257) -- (9.707,-1.273) -- (9.808,-1.292) -- (9.905,-1.315) -- (9.997,-1.343) -- (10.085,-1.375) -- (10.168,-1.412) -- (10.245,-1.455) -- (10.317,-1.503) -- (10.383,-1.557) -- (10.443,-1.617) -- (10.497,-1.683) -- (10.545,-1.755) -- (10.588,-1.832) -- (10.625,-1.915) -- (10.657,-2.003) -- (10.685,-2.095) -- (10.708,-2.192) -- (10.727,-2.293) -- (10.743,-2.397) -- (10.756,-2.504) -- (10.766,-2.614) -- (10.775,-2.725) -- (10.781,-2.839) -- (10.786,-2.954) -- (10.790,-3.070) -- (10.793,-3.187) -- (10.795,-3.305) -- (10.796,-3.424) -- (10.798,-3.542) -- (10.798,-3.662) -- (10.799,-3.781) -- (10.799,-3.901) -- (10.799,-4.020) -- (10.799,-4.140) -- (10.799,-4.260) -- (10.799,-4.380) -- (10.799,-4.499) -- (10.799,-4.619) -- (10.798,-4.738) -- (10.798,-4.858) -- (10.796,-4.976) -- (10.795,-5.095) -- (10.793,-5.213) -- (10.790,-5.330) -- (10.786,-5.446) -- (10.781,-5.561) -- (10.775,-5.675) -- (10.766,-5.786) -- (10.756,-5.896) -- (10.743,-6.003) -- (10.727,-6.107) -- (10.708,-6.208) -- (10.685,-6.305) -- (10.657,-6.397) -- (10.625,-6.485) -- (10.588,-6.568) -- (10.545,-6.645) -- (10.497,-6.717) -- (10.443,-6.783) -- (10.383,-6.843) -- (10.317,-6.897) -- (10.245,-6.945) -- (10.168,-6.988) -- (10.085,-7.025) -- (9.997,-7.057) -- (9.905,-7.085) -- (9.808,-7.108) -- (9.707,-7.127) -- (9.603,-7.143) -- (9.496,-7.156) -- (9.386,-7.166) -- (9.275,-7.175) -- (9.161,-7.181) -- (9.046,-7.186) -- (8.930,-7.190) -- (8.813,-7.193) -- (8.695,-7.195) -- (8.576,-7.196) -- (8.474,-7.197);
\draw[knot] (8.214,-7.199) -- (8.099,-7.199) -- (7.980,-7.200) -- (7.860,-7.200) -- (7.740,-7.200) -- (7.620,-7.200) -- (7.500,-7.200) -- (7.380,-7.200) -- (7.260,-7.200) -- (7.140,-7.200) -- (7.020,-7.200) -- (6.900,-7.200) -- (6.780,-7.200) -- (6.660,-7.200) -- (6.540,-7.200) -- (6.420,-7.200) -- (6.300,-7.200) -- (6.180,-7.200) -- (6.060,-7.200) -- (5.940,-7.200) -- (5.820,-7.200) -- (5.700,-7.200) -- (5.580,-7.200) -- (5.460,-7.200) -- (5.340,-7.200) -- (5.220,-7.200) -- (5.100,-7.200) -- (4.980,-7.200) -- (4.874,-7.200);
\draw[knot] (4.614,-7.200) -- (4.500,-7.200) -- (4.380,-7.200) -- (4.260,-7.200) -- (4.140,-7.200) -- (4.020,-7.200) -- (3.901,-7.199) -- (3.781,-7.199) -- (3.662,-7.198) -- (3.542,-7.198) -- (3.424,-7.196) -- (3.305,-7.195) -- (3.187,-7.193) -- (3.070,-7.190) -- (2.954,-7.186) -- (2.839,-7.181) -- (2.725,-7.175) -- (2.614,-7.166) -- (2.504,-7.156) -- (2.397,-7.143) -- (2.293,-7.127) -- (2.192,-7.108) -- (2.095,-7.085) -- (2.003,-7.057) -- (1.915,-7.025) -- (1.832,-6.988) -- (1.755,-6.945) -- (1.683,-6.897) -- (1.617,-6.843) -- (1.557,-6.783) -- (1.503,-6.717) -- (1.455,-6.645) -- (1.412,-6.568) -- (1.375,-6.485) -- (1.343,-6.398) -- (1.316,-6.305) -- (1.293,-6.209) -- (1.274,-6.108) -- (1.259,-6.005) -- (1.246,-5.899) -- (1.237,-5.790) -- (1.230,-5.680) -- (1.226,-5.568) -- (1.224,-5.456) -- (1.224,-5.344) -- (1.226,-5.232) -- (1.230,-5.120) -- (1.237,-5.010) -- (1.246,-4.901) -- (1.259,-4.795) -- (1.274,-4.692) -- (1.293,-4.591) -- (1.316,-4.495) -- (1.343,-4.402) -- (1.375,-4.315) -- (1.412,-4.232) -- (1.455,-4.155) -- (1.503,-4.083) -- (1.557,-4.017) -- (1.617,-3.957) -- (1.683,-3.903) -- (1.755,-3.855) -- (1.832,-3.812) -- (1.915,-3.775) -- (2.003,-3.743) -- (2.095,-3.715) -- (2.192,-3.692) -- (2.293,-3.673) -- (2.322,-3.668);
\draw[knot] (2.580,-3.637) -- (2.613,-3.634) -- (2.725,-3.625) -- (2.839,-3.619) -- (2.954,-3.614) -- (3.070,-3.610) -- (3.187,-3.608) -- (3.304,-3.606) -- (3.422,-3.605) -- (3.541,-3.604) -- (3.659,-3.604) -- (3.778,-3.605) -- (3.896,-3.606) -- (4.013,-3.608) -- (4.130,-3.610) -- (4.246,-3.614) -- (4.361,-3.619) -- (4.475,-3.625) -- (4.587,-3.634) -- (4.696,-3.644) -- (4.803,-3.657) -- (4.907,-3.673) -- (5.008,-3.692) -- (5.105,-3.715) -- (5.197,-3.743) -- (5.285,-3.775) -- (5.368,-3.812) -- (5.445,-3.855) -- (5.517,-3.903) -- (5.583,-3.957) -- (5.643,-4.017) -- (5.697,-4.083) -- (5.745,-4.155) -- (5.788,-4.232) -- (5.825,-4.315) -- (5.857,-4.403) -- (5.885,-4.495) -- (5.908,-4.592) -- (5.924,-4.674);
\draw[knot] (5.959,-4.932) -- (5.967,-5.013) -- (5.975,-5.125) -- (5.981,-5.239) -- (5.986,-5.354) -- (5.990,-5.470) -- (5.993,-5.587) -- (5.996,-5.704) -- (5.998,-5.822) -- (5.999,-5.941) -- (6.001,-6.059) -- (6.002,-6.178) -- (6.004,-6.296) -- (6.007,-6.413) -- (6.010,-6.530) -- (6.014,-6.646) -- (6.019,-6.761) -- (6.025,-6.875) -- (6.033,-6.987) -- (6.041,-7.071);
\draw[knot] (6.077,-7.328) -- (6.092,-7.408) -- (6.115,-7.505) -- (6.142,-7.597) -- (6.175,-7.685) -- (6.212,-7.767) -- (6.254,-7.845) -- (6.302,-7.916) -- (6.355,-7.981) -- (6.415,-8.040) -- (6.479,-8.093) -- (6.550,-8.140) -- (6.625,-8.181) -- (6.705,-8.215) -- (6.789,-8.243) -- (6.876,-8.266) -- (6.967,-8.283) -- (7.059,-8.294) -- (7.153,-8.299) -- (7.247,-8.299) -- (7.341,-8.294) -- (7.433,-8.283) -- (7.524,-8.266) -- (7.611,-8.243) -- (7.695,-8.215) -- (7.775,-8.181) -- (7.850,-8.140) -- (7.921,-8.093) -- (7.985,-8.040) -- (8.045,-7.981) -- (8.098,-7.916) -- (8.146,-7.845) -- (8.188,-7.767) -- (8.225,-7.685) -- (8.258,-7.597) -- (8.285,-7.505) -- (8.308,-7.408) -- (8.327,-7.307) -- (8.343,-7.203) -- (8.356,-7.096) -- (8.366,-6.987) -- (8.375,-6.875) -- (8.381,-6.761) -- (8.386,-6.646) -- (8.390,-6.530) -- (8.392,-6.413) -- (8.394,-6.296) -- (8.395,-6.178) -- (8.396,-6.059) -- (8.396,-5.941) -- (8.395,-5.822) -- (8.394,-5.704) -- (8.392,-5.587) -- (8.390,-5.470) -- (8.386,-5.354) -- (8.382,-5.239) -- (8.375,-5.126) -- (8.368,-5.015) -- (8.358,-4.906) -- (8.355,-4.880);
\draw[knot] (8.318,-4.623) -- (8.313,-4.597) -- (8.292,-4.502) -- (8.267,-4.413) -- (8.239,-4.329) -- (8.207,-4.251) -- (8.171,-4.180) -- (8.131,-4.116) -- (8.087,-4.061) -- (8.040,-4.014) -- (7.990,-3.976) -- (7.937,-3.947) -- (7.883,-3.927) -- (7.828,-3.917) -- (7.772,-3.917) -- (7.716,-3.927) -- (7.662,-3.946) -- (7.609,-3.975) -- (7.559,-4.012) -- (7.511,-4.059) -- (7.466,-4.113) -- (7.424,-4.175) -- (7.386,-4.244) -- (7.351,-4.319) -- (7.319,-4.399) -- (7.289,-4.484) -- (7.262,-4.572) -- (7.236,-4.663) -- (7.210,-4.755) -- (7.185,-4.849) -- (7.160,-4.942) -- (7.132,-5.034) -- (7.103,-5.124) -- (7.071,-5.211) -- (7.035,-5.295) -- (6.995,-5.375) -- (6.950,-5.450) -- (6.901,-5.521) -- (6.845,-5.585) -- (6.785,-5.645) -- (6.718,-5.698) -- (6.646,-5.746) -- (6.569,-5.788) -- (6.486,-5.825) -- (6.398,-5.857) -- (6.305,-5.884) -- (6.209,-5.907) -- (6.128,-5.923);
\draw[knot] (5.870,-5.956) -- (5.790,-5.963) -- (5.680,-5.970) -- (5.569,-5.974) -- (5.456,-5.976) -- (5.344,-5.976) -- (5.232,-5.974) -- (5.121,-5.970) -- (5.011,-5.964) -- (4.903,-5.955) -- (4.798,-5.944) -- (4.695,-5.930) -- (4.596,-5.912) -- (4.502,-5.892) -- (4.412,-5.867) -- (4.329,-5.839) -- (4.251,-5.807) -- (4.180,-5.770) -- (4.116,-5.730) -- (4.061,-5.687) -- (4.014,-5.640) -- (3.976,-5.590) -- (3.947,-5.537) -- (3.927,-5.483) -- (3.918,-5.428) -- (3.918,-5.372) -- (3.927,-5.317) -- (3.947,-5.263) -- (3.976,-5.210) -- (4.014,-5.160) -- (4.061,-5.113) -- (4.116,-5.069) -- (4.180,-5.029) -- (4.251,-4.993) -- (4.329,-4.961) -- (4.413,-4.933) -- (4.502,-4.908) -- (4.597,-4.887) -- (4.671,-4.874);
\node[reglabel] (w1) at (4.210,-3.013) {$2$};
\node[reglabel] (w2) at (5.282,-4.336) {$4$};
\node[reglabel] (w3) at (7.797,-4.370) {$5$};
\node[reglabel] (w4) at (6.465,-5.283) {$7$};
\node[reglabel] (w5) at (1.812,-5.400) {$1$};
\node[reglabel] (w6) at (4.379,-5.400) {$6$};
\node[reglabel] (w7) at (9.584,-5.610) {$3$};
\node[reglabel] (w8) at (5.400,-6.588) {$10$};
\node[reglabel] (w9) at (3.553,-7.748) {$8$};
\node[reglabel] (w10) at (7.228,-7.750) {$9$};
\end{tikzpicture}

%% file: figures/13a4237-tait-white-overlay.tex
\begin{tikzpicture}[tait/.style={line width=1.3pt, black},
                    vert/.style={circle, fill=white, draw=black, line width=1.0pt, inner sep=0pt, minimum size=0.6cm, font=\small},
                    knot/.style={line width=0.8pt, black!25, line cap=round, line join=round}]
\draw[knot] (4.929,-4.840) -- (5.015,-4.832) -- (5.126,-4.825) -- (5.239,-4.818) -- (5.354,-4.814) -- (5.470,-4.810) -- (5.587,-4.807) -- (5.705,-4.805) -- (5.824,-4.804) -- (5.942,-4.802) -- (6.062,-4.802) -- (6.181,-4.801) -- (6.301,-4.801) -- (6.420,-4.800) -- (6.540,-4.800) -- (6.660,-4.800) -- (6.780,-4.800) -- (6.899,-4.799) -- (7.019,-4.799) -- (7.069,-4.799);
\draw[knot] (7.329,-4.797) -- (7.376,-4.796) -- (7.495,-4.795) -- (7.613,-4.793) -- (7.730,-4.790) -- (7.846,-4.786) -- (7.961,-4.781) -- (8.075,-4.775) -- (8.186,-4.767) -- (8.296,-4.756) -- (8.403,-4.743) -- (8.507,-4.727) -- (8.608,-4.708) -- (8.705,-4.685) -- (8.797,-4.658) -- (8.885,-4.625) -- (8.967,-4.588) -- (9.045,-4.546) -- (9.116,-4.498) -- (9.181,-4.445) -- (9.240,-4.385) -- (9.293,-4.321) -- (9.340,-4.250) -- (9.381,-4.175) -- (9.415,-4.095) -- (9.443,-4.011) -- (9.466,-3.924) -- (9.483,-3.833) -- (9.494,-3.741) -- (9.499,-3.647) -- (9.499,-3.553) -- (9.494,-3.459) -- (9.483,-3.367) -- (9.466,-3.276) -- (9.443,-3.189) -- (9.415,-3.105) -- (9.381,-3.025) -- (9.340,-2.950) -- (9.293,-2.879) -- (9.240,-2.815) -- (9.181,-2.755) -- (9.116,-2.702) -- (9.045,-2.654) -- (8.967,-2.612) -- (8.885,-2.575) -- (8.797,-2.542) -- (8.705,-2.515) -- (8.608,-2.492) -- (8.507,-2.473) -- (8.403,-2.457) -- (8.296,-2.444) -- (8.186,-2.433) -- (8.075,-2.425) -- (7.961,-2.419) -- (7.846,-2.414) -- (7.730,-2.410) -- (7.613,-2.407) -- (7.495,-2.405) -- (7.376,-2.404) -- (7.258,-2.402) -- (7.138,-2.402) -- (7.019,-2.401) -- (6.899,-2.401) -- (6.780,-2.400) -- (6.660,-2.400) -- (6.540,-2.400) -- (6.420,-2.400) -- (6.300,-2.400) -- (6.180,-2.400) -- (6.060,-2.400) -- (5.940,-2.400) -- (5.820,-2.400) -- (5.700,-2.400) -- (5.580,-2.400) -- (5.460,-2.400) -- (5.340,-2.400) -- (5.220,-2.400) -- (5.101,-2.401) -- (4.986,-2.401);
\draw[knot] (4.726,-2.403) -- (4.624,-2.404) -- (4.505,-2.405) -- (4.387,-2.407) -- (4.270,-2.410) -- (4.154,-2.414) -- (4.039,-2.419) -- (3.925,-2.425) -- (3.814,-2.434) -- (3.704,-2.444) -- (3.597,-2.457) -- (3.493,-2.473) -- (3.392,-2.492) -- (3.295,-2.515) -- (3.203,-2.543) -- (3.115,-2.575) -- (3.032,-2.612) -- (2.955,-2.655) -- (2.883,-2.703) -- (2.817,-2.757) -- (2.757,-2.817) -- (2.703,-2.883) -- (2.655,-2.955) -- (2.612,-3.032) -- (2.575,-3.115) -- (2.543,-3.203) -- (2.515,-3.295) -- (2.492,-3.392) -- (2.473,-3.493) -- (2.457,-3.597) -- (2.444,-3.704) -- (2.434,-3.814) -- (2.425,-3.925) -- (2.419,-4.039) -- (2.414,-4.154) -- (2.410,-4.270) -- (2.407,-4.387) -- (2.405,-4.505) -- (2.404,-4.624) -- (2.402,-4.742) -- (2.402,-4.862) -- (2.401,-4.981) -- (2.401,-5.101) -- (2.401,-5.220) -- (2.401,-5.340) -- (2.401,-5.460) -- (2.401,-5.580) -- (2.401,-5.699) -- (2.401,-5.819) -- (2.402,-5.938) -- (2.402,-6.058) -- (2.404,-6.176) -- (2.405,-6.295) -- (2.407,-6.413) -- (2.410,-6.530) -- (2.414,-6.646) -- (2.419,-6.761) -- (2.425,-6.875) -- (2.433,-6.986) -- (2.437,-7.020);
\draw[knot] (2.468,-7.278) -- (2.473,-7.307) -- (2.492,-7.408) -- (2.515,-7.505) -- (2.542,-7.597) -- (2.575,-7.685) -- (2.612,-7.767) -- (2.654,-7.845) -- (2.702,-7.916) -- (2.755,-7.981) -- (2.815,-8.040) -- (2.879,-8.093) -- (2.950,-8.140) -- (3.025,-8.181) -- (3.105,-8.215) -- (3.189,-8.243) -- (3.276,-8.266) -- (3.367,-8.283) -- (3.459,-8.294) -- (3.553,-8.299) -- (3.647,-8.299) -- (3.741,-8.294) -- (3.833,-8.283) -- (3.924,-8.266) -- (4.011,-8.243) -- (4.095,-8.215) -- (4.175,-8.181) -- (4.250,-8.140) -- (4.321,-8.093) -- (4.385,-8.040) -- (4.445,-7.981) -- (4.498,-7.916) -- (4.546,-7.845) -- (4.588,-7.767) -- (4.625,-7.685) -- (4.658,-7.597) -- (4.685,-7.505) -- (4.708,-7.408) -- (4.727,-7.307) -- (4.743,-7.203) -- (4.756,-7.096) -- (4.767,-6.986) -- (4.775,-6.875) -- (4.781,-6.761) -- (4.786,-6.646) -- (4.790,-6.530) -- (4.793,-6.413) -- (4.795,-6.295) -- (4.796,-6.176) -- (4.797,-6.074);
\draw[knot] (4.799,-5.814) -- (4.799,-5.699) -- (4.800,-5.580) -- (4.800,-5.460) -- (4.800,-5.340) -- (4.800,-5.220) -- (4.800,-5.100) -- (4.800,-4.980) -- (4.800,-4.860) -- (4.800,-4.740) -- (4.800,-4.620) -- (4.800,-4.500) -- (4.800,-4.380) -- (4.800,-4.260) -- (4.800,-4.140) -- (4.800,-4.020) -- (4.801,-3.901) -- (4.801,-3.787);
\draw[knot] (4.803,-3.527) -- (4.804,-3.424) -- (4.805,-3.305) -- (4.807,-3.187) -- (4.810,-3.070) -- (4.814,-2.954) -- (4.819,-2.839) -- (4.825,-2.725) -- (4.834,-2.614) -- (4.844,-2.504) -- (4.857,-2.397) -- (4.873,-2.293) -- (4.892,-2.192) -- (4.915,-2.095) -- (4.943,-2.003) -- (4.975,-1.915) -- (5.012,-1.832) -- (5.055,-1.755) -- (5.103,-1.683) -- (5.157,-1.617) -- (5.217,-1.557) -- (5.283,-1.503) -- (5.355,-1.455) -- (5.432,-1.412) -- (5.515,-1.375) -- (5.603,-1.343) -- (5.695,-1.315) -- (5.792,-1.292) -- (5.893,-1.273) -- (5.997,-1.257) -- (6.104,-1.244) -- (6.214,-1.234) -- (6.325,-1.225) -- (6.439,-1.219) -- (6.554,-1.214) -- (6.670,-1.210) -- (6.787,-1.207) -- (6.905,-1.205) -- (7.024,-1.204) -- (7.142,-1.202) -- (7.262,-1.202) -- (7.381,-1.201) -- (7.501,-1.201) -- (7.620,-1.201) -- (7.740,-1.201) -- (7.860,-1.201) -- (7.980,-1.201) -- (8.099,-1.201) -- (8.219,-1.201) -- (8.338,-1.202) -- (8.458,-1.202) -- (8.576,-1.204) -- (8.695,-1.205) -- (8.813,-1.207) -- (8.930,-1.210) -- (9.046,-1.214) -- (9.161,-1.219) -- (9.275,-1.225) -- (9.386,-1.234) -- (9.496,-1.244) -- (9.603,-1.257) -- (9.707,-1.273) -- (9.808,-1.292) -- (9.905,-1.315) -- (9.997,-1.343) -- (10.085,-1.375) -- (10.168,-1.412) -- (10.245,-1.455) -- (10.317,-1.503) -- (10.383,-1.557) -- (10.443,-1.617) -- (10.497,-1.683) -- (10.545,-1.755) -- (10.588,-1.832) -- (10.625,-1.915) -- (10.657,-2.003) -- (10.685,-2.095) -- (10.708,-2.192) -- (10.727,-2.293) -- (10.743,-2.397) -- (10.756,-2.504) -- (10.766,-2.614) -- (10.775,-2.725) -- (10.781,-2.839) -- (10.786,-2.954) -- (10.790,-3.070) -- (10.793,-3.187) -- (10.795,-3.305) -- (10.796,-3.424) -- (10.798,-3.542) -- (10.798,-3.662) -- (10.799,-3.781) -- (10.799,-3.901) -- (10.799,-4.020) -- (10.799,-4.140) -- (10.799,-4.260) -- (10.799,-4.380) -- (10.799,-4.499) -- (10.799,-4.619) -- (10.798,-4.738) -- (10.798,-4.858) -- (10.796,-4.976) -- (10.795,-5.095) -- (10.793,-5.213) -- (10.790,-5.330) -- (10.786,-5.446) -- (10.781,-5.561) -- (10.775,-5.675) -- (10.766,-5.786) -- (10.756,-5.896) -- (10.743,-6.003) -- (10.727,-6.107) -- (10.708,-6.208) -- (10.685,-6.305) -- (10.657,-6.397) -- (10.625,-6.485) -- (10.588,-6.568) -- (10.545,-6.645) -- (10.497,-6.717) -- (10.443,-6.783) -- (10.383,-6.843) -- (10.317,-6.897) -- (10.245,-6.945) -- (10.168,-6.988) -- (10.085,-7.025) -- (9.997,-7.057) -- (9.905,-7.085) -- (9.808,-7.108) -- (9.707,-7.127) -- (9.603,-7.143) -- (9.496,-7.156) -- (9.386,-7.166) -- (9.275,-7.175) -- (9.161,-7.181) -- (9.046,-7.186) -- (8.930,-7.190) -- (8.813,-7.193) -- (8.695,-7.195) -- (8.576,-7.196) -- (8.474,-7.197);
\draw[knot] (8.214,-7.199) -- (8.099,-7.199) -- (7.980,-7.200) -- (7.860,-7.200) -- (7.740,-7.200) -- (7.620,-7.200) -- (7.500,-7.200) -- (7.380,-7.200) -- (7.260,-7.200) -- (7.140,-7.200) -- (7.020,-7.200) -- (6.900,-7.200) -- (6.780,-7.200) -- (6.660,-7.200) -- (6.540,-7.200) -- (6.420,-7.200) -- (6.300,-7.200) -- (6.180,-7.200) -- (6.060,-7.200) -- (5.940,-7.200) -- (5.820,-7.200) -- (5.700,-7.200) -- (5.580,-7.200) -- (5.460,-7.200) -- (5.340,-7.200) -- (5.220,-7.200) -- (5.100,-7.200) -- (4.980,-7.200) -- (4.874,-7.200);
\draw[knot] (4.614,-7.200) -- (4.500,-7.200) -- (4.380,-7.200) -- (4.260,-7.200) -- (4.140,-7.200) -- (4.020,-7.200) -- (3.901,-7.199) -- (3.781,-7.199) -- (3.662,-7.198) -- (3.542,-7.198) -- (3.424,-7.196) -- (3.305,-7.195) -- (3.187,-7.193) -- (3.070,-7.190) -- (2.954,-7.186) -- (2.839,-7.181) -- (2.725,-7.175) -- (2.614,-7.166) -- (2.504,-7.156) -- (2.397,-7.143) -- (2.293,-7.127) -- (2.192,-7.108) -- (2.095,-7.085) -- (2.003,-7.057) -- (1.915,-7.025) -- (1.832,-6.988) -- (1.755,-6.945) -- (1.683,-6.897) -- (1.617,-6.843) -- (1.557,-6.783) -- (1.503,-6.717) -- (1.455,-6.645) -- (1.412,-6.568) -- (1.375,-6.485) -- (1.343,-6.398) -- (1.316,-6.305) -- (1.293,-6.209) -- (1.274,-6.108) -- (1.259,-6.005) -- (1.246,-5.899) -- (1.237,-5.790) -- (1.230,-5.680) -- (1.226,-5.568) -- (1.224,-5.456) -- (1.224,-5.344) -- (1.226,-5.232) -- (1.230,-5.120) -- (1.237,-5.010) -- (1.246,-4.901) -- (1.259,-4.795) -- (1.274,-4.692) -- (1.293,-4.591) -- (1.316,-4.495) -- (1.343,-4.402) -- (1.375,-4.315) -- (1.412,-4.232) -- (1.455,-4.155) -- (1.503,-4.083) -- (1.557,-4.017) -- (1.617,-3.957) -- (1.683,-3.903) -- (1.755,-3.855) -- (1.832,-3.812) -- (1.915,-3.775) -- (2.003,-3.743) -- (2.095,-3.715) -- (2.192,-3.692) -- (2.293,-3.673) -- (2.322,-3.668);
\draw[knot] (2.580,-3.637) -- (2.613,-3.634) -- (2.725,-3.625) -- (2.839,-3.619) -- (2.954,-3.614) -- (3.070,-3.610) -- (3.187,-3.608) -- (3.304,-3.606) -- (3.422,-3.605) -- (3.541,-3.604) -- (3.659,-3.604) -- (3.778,-3.605) -- (3.896,-3.606) -- (4.013,-3.608) -- (4.130,-3.610) -- (4.246,-3.614) -- (4.361,-3.619) -- (4.475,-3.625) -- (4.587,-3.634) -- (4.696,-3.644) -- (4.803,-3.657) -- (4.907,-3.673) -- (5.008,-3.692) -- (5.105,-3.715) -- (5.197,-3.743) -- (5.285,-3.775) -- (5.368,-3.812) -- (5.445,-3.855) -- (5.517,-3.903) -- (5.583,-3.957) -- (5.643,-4.017) -- (5.697,-4.083) -- (5.745,-4.155) -- (5.788,-4.232) -- (5.825,-4.315) -- (5.857,-4.403) -- (5.885,-4.495) -- (5.908,-4.592) -- (5.924,-4.674);
\draw[knot] (5.959,-4.932) -- (5.967,-5.013) -- (5.975,-5.125) -- (5.981,-5.239) -- (5.986,-5.354) -- (5.990,-5.470) -- (5.993,-5.587) -- (5.996,-5.704) -- (5.998,-5.822) -- (5.999,-5.941) -- (6.001,-6.059) -- (6.002,-6.178) -- (6.004,-6.296) -- (6.007,-6.413) -- (6.010,-6.530) -- (6.014,-6.646) -- (6.019,-6.761) -- (6.025,-6.875) -- (6.033,-6.987) -- (6.041,-7.071);
\draw[knot] (6.077,-7.328) -- (6.092,-7.408) -- (6.115,-7.505) -- (6.142,-7.597) -- (6.175,-7.685) -- (6.212,-7.767) -- (6.254,-7.845) -- (6.302,-7.916) -- (6.355,-7.981) -- (6.415,-8.040) -- (6.479,-8.093) -- (6.550,-8.140) -- (6.625,-8.181) -- (6.705,-8.215) -- (6.789,-8.243) -- (6.876,-8.266) -- (6.967,-8.283) -- (7.059,-8.294) -- (7.153,-8.299) -- (7.247,-8.299) -- (7.341,-8.294) -- (7.433,-8.283) -- (7.524,-8.266) -- (7.611,-8.243) -- (7.695,-8.215) -- (7.775,-8.181) -- (7.850,-8.140) -- (7.921,-8.093) -- (7.985,-8.040) -- (8.045,-7.981) -- (8.098,-7.916) -- (8.146,-7.845) -- (8.188,-7.767) -- (8.225,-7.685) -- (8.258,-7.597) -- (8.285,-7.505) -- (8.308,-7.408) -- (8.327,-7.307) -- (8.343,-7.203) -- (8.356,-7.096) -- (8.366,-6.987) -- (8.375,-6.875) -- (8.381,-6.761) -- (8.386,-6.646) -- (8.390,-6.530) -- (8.392,-6.413) -- (8.394,-6.296) -- (8.395,-6.178) -- (8.396,-6.059) -- (8.396,-5.941) -- (8.395,-5.822) -- (8.394,-5.704) -- (8.392,-5.587) -- (8.390,-5.470) -- (8.386,-5.354) -- (8.382,-5.239) -- (8.375,-5.126) -- (8.368,-5.015) -- (8.358,-4.906) -- (8.355,-4.880);
\draw[knot] (8.318,-4.623) -- (8.313,-4.597) -- (8.292,-4.502) -- (8.267,-4.413) -- (8.239,-4.329) -- (8.207,-4.251) -- (8.171,-4.180) -- (8.131,-4.116) -- (8.087,-4.061) -- (8.040,-4.014) -- (7.990,-3.976) -- (7.937,-3.947) -- (7.883,-3.927) -- (7.828,-3.917) -- (7.772,-3.917) -- (7.716,-3.927) -- (7.662,-3.946) -- (7.609,-3.975) -- (7.559,-4.012) -- (7.511,-4.059) -- (7.466,-4.113) -- (7.424,-4.175) -- (7.386,-4.244) -- (7.351,-4.319) -- (7.319,-4.399) -- (7.289,-4.484) -- (7.262,-4.572) -- (7.236,-4.663) -- (7.210,-4.755) -- (7.185,-4.849) -- (7.160,-4.942) -- (7.132,-5.034) -- (7.103,-5.124) -- (7.071,-5.211) -- (7.035,-5.295) -- (6.995,-5.375) -- (6.950,-5.450) -- (6.901,-5.521) -- (6.845,-5.585) -- (6.785,-5.645) -- (6.718,-5.698) -- (6.646,-5.746) -- (6.569,-5.788) -- (6.486,-5.825) -- (6.398,-5.857) -- (6.305,-5.884) -- (6.209,-5.907) -- (6.128,-5.923);
\draw[knot] (5.870,-5.956) -- (5.790,-5.963) -- (5.680,-5.970) -- (5.569,-5.974) -- (5.456,-5.976) -- (5.344,-5.976) -- (5.232,-5.974) -- (5.121,-5.970) -- (5.011,-5.964) -- (4.903,-5.955) -- (4.798,-5.944) -- (4.695,-5.930) -- (4.596,-5.912) -- (4.502,-5.892) -- (4.412,-5.867) -- (4.329,-5.839) -- (4.251,-5.807) -- (4.180,-5.770) -- (4.116,-5.730) -- (4.061,-5.687) -- (4.014,-5.640) -- (3.976,-5.590) -- (3.947,-5.537) -- (3.927,-5.483) -- (3.918,-5.428) -- (3.918,-5.372) -- (3.927,-5.317) -- (3.947,-5.263) -- (3.976,-5.210) -- (4.014,-5.160) -- (4.061,-5.113) -- (4.116,-5.069) -- (4.180,-5.029) -- (4.251,-4.993) -- (4.329,-4.961) -- (4.413,-4.933) -- (4.502,-4.908) -- (4.597,-4.887) -- (4.671,-4.874);
\draw[tait] (5.282,-4.336) -- (5.262,-4.358) -- (5.242,-4.379) -- (5.221,-4.401) -- (5.201,-4.422) -- (5.180,-4.444) -- (5.160,-4.465) -- (5.139,-4.486) -- (5.119,-4.508) -- (5.098,-4.529) -- (5.078,-4.551) -- (5.058,-4.572) -- (5.037,-4.593) -- (5.017,-4.615) -- (4.996,-4.636) -- (4.976,-4.658) -- (4.956,-4.680) -- (4.936,-4.701) -- (4.916,-4.723) -- (4.896,-4.745) -- (4.877,-4.766) -- (4.857,-4.788) -- (4.838,-4.810) -- (4.819,-4.832) -- (4.800,-4.854) -- (4.781,-4.877) -- (4.763,-4.899) -- (4.744,-4.921) -- (4.726,-4.944) -- (4.708,-4.966) -- (4.690,-4.989) -- (4.673,-5.011) -- (4.655,-5.034) -- (4.638,-5.057) -- (4.620,-5.080) -- (4.603,-5.103) -- (4.586,-5.126) -- (4.568,-5.148) -- (4.551,-5.171) -- (4.534,-5.194) -- (4.517,-5.217) -- (4.500,-5.240) -- (4.483,-5.263) -- (4.466,-5.286) -- (4.448,-5.309) -- (4.431,-5.332) -- (4.414,-5.355) -- (4.396,-5.377) -- (4.379,-5.400);  
\draw[tait] (5.282,-4.336) -- (5.310,-4.356) -- (5.337,-4.375) -- (5.366,-4.394) -- (5.394,-4.414) -- (5.422,-4.433) -- (5.451,-4.452) -- (5.479,-4.472) -- (5.508,-4.491) -- (5.536,-4.510) -- (5.565,-4.530) -- (5.593,-4.549) -- (5.622,-4.568) -- (5.650,-4.588) -- (5.678,-4.607) -- (5.706,-4.627) -- (5.734,-4.646) -- (5.761,-4.665) -- (5.788,-4.685) -- (5.815,-4.704) -- (5.842,-4.724) -- (5.868,-4.743) -- (5.894,-4.763) -- (5.919,-4.783) -- (5.944,-4.802) -- (5.968,-4.822) -- (5.992,-4.842) -- (6.016,-4.862) -- (6.039,-4.882) -- (6.062,-4.902) -- (6.084,-4.922) -- (6.106,-4.942) -- (6.128,-4.962) -- (6.150,-4.982) -- (6.171,-5.002) -- (6.192,-5.022) -- (6.213,-5.042) -- (6.234,-5.062) -- (6.255,-5.082) -- (6.276,-5.102) -- (6.296,-5.122) -- (6.317,-5.142) -- (6.338,-5.163) -- (6.359,-5.183) -- (6.380,-5.203) -- (6.401,-5.223) -- (6.422,-5.243) -- (6.443,-5.263) -- (6.465,-5.283);  
\draw[tait] (7.797,-4.370) -- (7.772,-4.388) -- (7.747,-4.405) -- (7.723,-4.423) -- (7.699,-4.441) -- (7.675,-4.458) -- (7.650,-4.476) -- (7.626,-4.493) -- (7.603,-4.510) -- (7.579,-4.528) -- (7.555,-4.545) -- (7.531,-4.563) -- (7.506,-4.580) -- (7.482,-4.598) -- (7.458,-4.616) -- (7.433,-4.633) -- (7.408,-4.651) -- (7.383,-4.669) -- (7.358,-4.687) -- (7.332,-4.705) -- (7.307,-4.723) -- (7.280,-4.742) -- (7.254,-4.760) -- (7.227,-4.779) -- (7.199,-4.798) -- (7.171,-4.817) -- (7.143,-4.836) -- (7.114,-4.856) -- (7.085,-4.875) -- (7.055,-4.895) -- (7.025,-4.915) -- (6.995,-4.935) -- (6.964,-4.955) -- (6.934,-4.976) -- (6.903,-4.996) -- (6.872,-5.016) -- (6.840,-5.037) -- (6.809,-5.057) -- (6.778,-5.078) -- (6.746,-5.099) -- (6.715,-5.119) -- (6.683,-5.140) -- (6.652,-5.160) -- (6.620,-5.181) -- (6.589,-5.202) -- (6.558,-5.222) -- (6.526,-5.242) -- (6.496,-5.263) -- (6.465,-5.283);  
\draw[tait] (7.797,-4.370) -- (7.819,-4.385) -- (7.840,-4.400) -- (7.860,-4.414) -- (7.879,-4.428) -- (7.898,-4.441) -- (7.916,-4.454) -- (7.934,-4.467) -- (7.951,-4.479) -- (7.969,-4.492) -- (7.987,-4.504) -- (8.005,-4.517) -- (8.024,-4.531) -- (8.043,-4.544) -- (8.063,-4.558) -- (8.084,-4.573) -- (8.106,-4.589) -- (8.129,-4.605) -- (8.154,-4.622) -- (8.180,-4.640) -- (8.208,-4.660) -- (8.237,-4.680) -- (8.269,-4.702) -- (8.303,-4.726) -- (8.339,-4.751) -- (8.377,-4.778) -- (8.418,-4.806) -- (8.461,-4.836) -- (8.506,-4.867) -- (8.552,-4.899) -- (8.601,-4.932) -- (8.650,-4.967) -- (8.702,-5.002) -- (8.754,-5.038) -- (8.808,-5.075) -- (8.862,-5.113) -- (8.918,-5.151) -- (8.973,-5.189) -- (9.030,-5.228) -- (9.086,-5.267) -- (9.143,-5.306) -- (9.200,-5.345) -- (9.256,-5.384) -- (9.313,-5.423) -- (9.369,-5.462) -- (9.424,-5.500) -- (9.478,-5.537) -- (9.532,-5.574) -- (9.584,-5.610);  
\draw[tait] (4.210,-3.013) -- (4.234,-2.988) -- (4.251,-2.963) -- (4.264,-2.939) -- (4.272,-2.915) -- (4.276,-2.892) -- (4.278,-2.869) -- (4.278,-2.846) -- (4.278,-2.823) -- (4.277,-2.800) -- (4.277,-2.777) -- (4.279,-2.754) -- (4.284,-2.731) -- (4.292,-2.707) -- (4.304,-2.683) -- (4.321,-2.658) -- (4.345,-2.633) -- (4.376,-2.607) -- (4.414,-2.581) -- (4.461,-2.553) -- (4.517,-2.525) -- (4.584,-2.496) -- (4.663,-2.466) -- (4.753,-2.434) -- (4.856,-2.402) -- (4.977,-2.364) -- (5.118,-2.316) -- (5.276,-2.261) -- (5.451,-2.199) -- (5.640,-2.132) -- (5.842,-2.061) -- (6.054,-1.986) -- (6.276,-1.910) -- (6.504,-1.833) -- (6.737,-1.756) -- (6.973,-1.682) -- (7.211,-1.610) -- (7.448,-1.542) -- (7.683,-1.480) -- (7.914,-1.424) -- (8.139,-1.377) -- (8.356,-1.338) -- (8.563,-1.309) -- (8.758,-1.292) -- (8.940,-1.287) -- (9.107,-1.296) -- (9.256,-1.320) -- (9.387,-1.361) -- (9.496,-1.418) -- (9.587,-1.494) -- (9.663,-1.586) -- (9.725,-1.695) -- (9.774,-1.817) -- (9.812,-1.954) -- (9.838,-2.102) -- (9.855,-2.262) -- (9.863,-2.432) -- (9.863,-2.611) -- (9.856,-2.798) -- (9.843,-2.992) -- (9.825,-3.191) -- (9.803,-3.394) -- (9.778,-3.602) -- (9.751,-3.811) -- (9.724,-4.021) -- (9.696,-4.232) -- (9.669,-4.441) -- (9.644,-4.648) -- (9.622,-4.852) -- (9.604,-5.051) -- (9.591,-5.244) -- (9.584,-5.431) -- (9.584,-5.610);  
\draw[tait] (4.210,-3.013) -- (4.136,-3.038) -- (4.060,-3.062) -- (3.982,-3.085) -- (3.904,-3.106) -- (3.824,-3.126) -- (3.744,-3.146) -- (3.663,-3.165) -- (3.582,-3.184) -- (3.501,-3.203) -- (3.420,-3.222) -- (3.340,-3.242) -- (3.260,-3.262) -- (3.182,-3.283) -- (3.104,-3.306) -- (3.028,-3.330) -- (2.954,-3.355) -- (2.882,-3.383) -- (2.811,-3.413) -- (2.744,-3.445) -- (2.679,-3.480) -- (2.617,-3.517) -- (2.558,-3.558) -- (2.502,-3.603) -- (2.450,-3.650) -- (2.402,-3.702) -- (2.358,-3.757) -- (2.317,-3.816) -- (2.279,-3.878) -- (2.244,-3.942) -- (2.212,-4.010) -- (2.182,-4.079) -- (2.155,-4.151) -- (2.129,-4.225) -- (2.105,-4.301) -- (2.082,-4.378) -- (2.061,-4.456) -- (2.041,-4.535) -- (2.021,-4.615) -- (2.002,-4.695) -- (1.983,-4.776) -- (1.965,-4.856) -- (1.945,-4.937) -- (1.926,-5.016) -- (1.906,-5.096) -- (1.884,-5.174) -- (1.862,-5.251) -- (1.838,-5.326) -- (1.812,-5.400);  
\draw[tait] (1.812,-5.400) -- (1.838,-5.474) -- (1.864,-5.549) -- (1.889,-5.627) -- (1.913,-5.705) -- (1.937,-5.784) -- (1.961,-5.864) -- (1.984,-5.945) -- (2.008,-6.026) -- (2.031,-6.107) -- (2.055,-6.187) -- (2.078,-6.267) -- (2.102,-6.347) -- (2.127,-6.425) -- (2.152,-6.502) -- (2.177,-6.578) -- (2.203,-6.652) -- (2.230,-6.724) -- (2.258,-6.793) -- (2.287,-6.860) -- (2.317,-6.925) -- (2.348,-6.986) -- (2.381,-7.044) -- (2.415,-7.099) -- (2.450,-7.150) -- (2.487,-7.197) -- (2.526,-7.240) -- (2.566,-7.280) -- (2.607,-7.316) -- (2.650,-7.349) -- (2.693,-7.380) -- (2.738,-7.408) -- (2.784,-7.434) -- (2.830,-7.458) -- (2.877,-7.481) -- (2.925,-7.501) -- (2.973,-7.521) -- (3.021,-7.539) -- (3.070,-7.557) -- (3.119,-7.574) -- (3.168,-7.591) -- (3.218,-7.608) -- (3.267,-7.625) -- (3.316,-7.643) -- (3.364,-7.662) -- (3.412,-7.681) -- (3.460,-7.702) -- (3.507,-7.724) -- (3.553,-7.748);  
\draw[tait] (5.400,-6.588) -- (5.373,-6.614) -- (5.347,-6.639) -- (5.321,-6.665) -- (5.297,-6.691) -- (5.272,-6.717) -- (5.248,-6.743) -- (5.225,-6.769) -- (5.201,-6.794) -- (5.177,-6.820) -- (5.153,-6.846) -- (5.129,-6.872) -- (5.105,-6.898) -- (5.080,-6.924) -- (5.055,-6.950) -- (5.029,-6.975) -- (5.002,-7.001) -- (4.974,-7.026) -- (4.945,-7.052) -- (4.915,-7.077) -- (4.884,-7.102) -- (4.851,-7.127) -- (4.817,-7.151) -- (4.781,-7.176) -- (4.744,-7.200) -- (4.704,-7.224) -- (4.663,-7.248) -- (4.620,-7.272) -- (4.576,-7.295) -- (4.530,-7.318) -- (4.484,-7.341) -- (4.435,-7.364) -- (4.386,-7.387) -- (4.336,-7.410) -- (4.285,-7.433) -- (4.234,-7.455) -- (4.182,-7.478) -- (4.129,-7.500) -- (4.076,-7.523) -- (4.023,-7.545) -- (3.970,-7.568) -- (3.917,-7.590) -- (3.863,-7.612) -- (3.810,-7.635) -- (3.758,-7.657) -- (3.706,-7.680) -- (3.654,-7.703) -- (3.603,-7.725) -- (3.553,-7.748);  
\draw[tait] (4.379,-5.400) -- (4.396,-5.423) -- (4.413,-5.445) -- (4.430,-5.468) -- (4.447,-5.490) -- (4.463,-5.512) -- (4.480,-5.534) -- (4.496,-5.556) -- (4.512,-5.578) -- (4.528,-5.600) -- (4.544,-5.622) -- (4.561,-5.644) -- (4.577,-5.666) -- (4.594,-5.688) -- (4.611,-5.710) -- (4.628,-5.733) -- (4.645,-5.755) -- (4.663,-5.778) -- (4.681,-5.801) -- (4.699,-5.824) -- (4.718,-5.848) -- (4.737,-5.871) -- (4.757,-5.895) -- (4.777,-5.919) -- (4.798,-5.944) -- (4.820,-5.969) -- (4.842,-5.994) -- (4.865,-6.020) -- (4.888,-6.046) -- (4.912,-6.072) -- (4.936,-6.098) -- (4.960,-6.125) -- (4.985,-6.151) -- (5.011,-6.178) -- (5.036,-6.205) -- (5.062,-6.233) -- (5.088,-6.260) -- (5.114,-6.287) -- (5.140,-6.315) -- (5.166,-6.342) -- (5.193,-6.370) -- (5.219,-6.397) -- (5.245,-6.425) -- (5.271,-6.452) -- (5.297,-6.480) -- (5.323,-6.507) -- (5.349,-6.534) -- (5.375,-6.561) -- (5.400,-6.588);  
\draw[tait] (4.210,-3.013) -- (4.235,-3.040) -- (4.260,-3.066) -- (4.285,-3.093) -- (4.310,-3.120) -- (4.335,-3.146) -- (4.361,-3.173) -- (4.386,-3.199) -- (4.411,-3.226) -- (4.437,-3.253) -- (4.462,-3.279) -- (4.488,-3.306) -- (4.513,-3.332) -- (4.538,-3.359) -- (4.563,-3.386) -- (4.588,-3.413) -- (4.613,-3.439) -- (4.637,-3.466) -- (4.662,-3.493) -- (4.686,-3.520) -- (4.710,-3.547) -- (4.733,-3.574) -- (4.756,-3.602) -- (4.779,-3.629) -- (4.802,-3.657) -- (4.824,-3.684) -- (4.846,-3.712) -- (4.867,-3.740) -- (4.888,-3.768) -- (4.909,-3.796) -- (4.930,-3.824) -- (4.950,-3.852) -- (4.970,-3.881) -- (4.990,-3.909) -- (5.010,-3.937) -- (5.029,-3.966) -- (5.049,-3.994) -- (5.068,-4.023) -- (5.087,-4.051) -- (5.107,-4.080) -- (5.126,-4.108) -- (5.145,-4.137) -- (5.164,-4.165) -- (5.184,-4.194) -- (5.203,-4.223) -- (5.223,-4.251) -- (5.242,-4.279) -- (5.262,-4.308) -- (5.282,-4.336);  
\draw[tait] (9.584,-5.610) -- (9.532,-5.677) -- (9.481,-5.746) -- (9.428,-5.816) -- (9.376,-5.887) -- (9.324,-5.959) -- (9.271,-6.032) -- (9.219,-6.105) -- (9.166,-6.178) -- (9.114,-6.251) -- (9.061,-6.324) -- (9.009,-6.397) -- (8.956,-6.469) -- (8.904,-6.540) -- (8.852,-6.610) -- (8.800,-6.679) -- (8.748,-6.746) -- (8.696,-6.811) -- (8.645,-6.874) -- (8.594,-6.935) -- (8.543,-6.994) -- (8.493,-7.049) -- (8.443,-7.102) -- (8.393,-7.152) -- (8.344,-7.198) -- (8.295,-7.241) -- (8.246,-7.281) -- (8.198,-7.317) -- (8.151,-7.350) -- (8.103,-7.381) -- (8.056,-7.409) -- (8.009,-7.435) -- (7.963,-7.459) -- (7.916,-7.481) -- (7.870,-7.502) -- (7.824,-7.521) -- (7.778,-7.539) -- (7.732,-7.556) -- (7.687,-7.573) -- (7.641,-7.589) -- (7.595,-7.604) -- (7.550,-7.620) -- (7.504,-7.636) -- (7.458,-7.653) -- (7.413,-7.670) -- (7.367,-7.688) -- (7.321,-7.707) -- (7.274,-7.728) -- (7.228,-7.750);  
\draw[tait] (5.400,-6.588) -- (5.427,-6.614) -- (5.453,-6.639) -- (5.478,-6.665) -- (5.503,-6.691) -- (5.528,-6.717) -- (5.552,-6.743) -- (5.576,-6.769) -- (5.600,-6.794) -- (5.623,-6.820) -- (5.647,-6.846) -- (5.671,-6.872) -- (5.696,-6.898) -- (5.721,-6.924) -- (5.746,-6.949) -- (5.772,-6.975) -- (5.799,-7.001) -- (5.827,-7.026) -- (5.856,-7.051) -- (5.886,-7.077) -- (5.917,-7.102) -- (5.950,-7.126) -- (5.984,-7.151) -- (6.019,-7.176) -- (6.056,-7.200) -- (6.095,-7.224) -- (6.136,-7.248) -- (6.178,-7.272) -- (6.222,-7.295) -- (6.267,-7.319) -- (6.313,-7.342) -- (6.360,-7.365) -- (6.409,-7.388) -- (6.458,-7.411) -- (6.508,-7.434) -- (6.559,-7.456) -- (6.610,-7.479) -- (6.662,-7.501) -- (6.714,-7.524) -- (6.766,-7.546) -- (6.818,-7.569) -- (6.871,-7.591) -- (6.923,-7.614) -- (6.975,-7.636) -- (7.027,-7.659) -- (7.078,-7.682) -- (7.129,-7.704) -- (7.179,-7.727) -- (7.228,-7.750);  
\draw[tait] (6.465,-5.283) -- (6.446,-5.311) -- (6.426,-5.338) -- (6.408,-5.366) -- (6.389,-5.393) -- (6.370,-5.421) -- (6.352,-5.448) -- (6.333,-5.476) -- (6.315,-5.503) -- (6.296,-5.531) -- (6.278,-5.558) -- (6.259,-5.586) -- (6.240,-5.613) -- (6.222,-5.641) -- (6.203,-5.668) -- (6.184,-5.696) -- (6.164,-5.723) -- (6.145,-5.751) -- (6.125,-5.778) -- (6.105,-5.806) -- (6.085,-5.833) -- (6.064,-5.860) -- (6.043,-5.888) -- (6.021,-5.915) -- (5.999,-5.942) -- (5.977,-5.969) -- (5.954,-5.996) -- (5.931,-6.024) -- (5.907,-6.051) -- (5.883,-6.078) -- (5.859,-6.105) -- (5.834,-6.132) -- (5.809,-6.158) -- (5.784,-6.185) -- (5.759,-6.212) -- (5.734,-6.239) -- (5.708,-6.266) -- (5.682,-6.293) -- (5.656,-6.320) -- (5.631,-6.347) -- (5.605,-6.373) -- (5.579,-6.400) -- (5.553,-6.427) -- (5.527,-6.454) -- (5.501,-6.481) -- (5.476,-6.508) -- (5.450,-6.534) -- (5.425,-6.561) -- (5.400,-6.588);  
\node[vert] (v1) at (4.210,-3.013) {$v_{2}$};
\node[vert] (v2) at (5.282,-4.336) {$v_{4}$};
\node[vert] (v3) at (7.797,-4.370) {$v_{5}$};
\node[vert] (v4) at (6.465,-5.283) {$v_{7}$};
\node[vert] (v5) at (1.812,-5.400) {$v_{1}$};
\node[vert] (v6) at (4.379,-5.400) {$v_{6}$};
\node[vert] (v7) at (9.584,-5.610) {$v_{3}$};
\node[vert] (v8) at (5.400,-6.588) {$v_{10}$};
\node[vert] (v9) at (3.553,-7.748) {$v_{8}$};
\node[vert] (v10) at (7.228,-7.750) {$v_{9}$};
\end{tikzpicture}

%% file: figures/10_109-checkerboard.tex
\begin{tikzpicture}[knot/.style={line width=1.0pt, line cap=round, line join=round},
                    shaded/.style={fill=black!22}, reglabel/.style={font=\small, inner sep=1pt}]
\fill[shaded] (7.026,-5.801) -- (7.025,-5.912) -- (7.024,-6.028) -- (7.022,-6.144) -- (7.020,-6.259) -- (7.018,-6.373) -- (7.014,-6.487) -- (7.009,-6.599) -- (7.003,-6.710) -- (6.995,-6.819) -- (6.985,-6.926) -- (6.972,-7.030) -- (6.956,-7.132) -- (6.938,-7.230) -- (6.915,-7.325) -- (6.888,-7.415) -- (6.857,-7.501) -- (6.820,-7.582) -- (6.779,-7.657) -- (6.732,-7.727) -- (6.679,-7.792) -- (6.620,-7.850) -- (6.556,-7.903) -- (6.486,-7.950) -- (6.411,-7.992) -- (6.330,-8.028) -- (6.244,-8.059) -- (6.154,-8.086) -- (6.060,-8.108) -- (5.962,-8.127) -- (5.861,-8.142) -- (5.757,-8.153) -- (5.651,-8.162) -- (5.544,-8.169) -- (5.435,-8.173) -- (5.325,-8.175) -- (5.216,-8.175) -- (5.106,-8.173) -- (4.997,-8.169) -- (4.890,-8.162) -- (4.784,-8.153) -- (4.680,-8.142) -- (4.579,-8.127) -- (4.481,-8.108) -- (4.387,-8.086) -- (4.297,-8.059) -- (4.211,-8.028) -- (4.131,-7.992) -- (4.055,-7.950) -- (3.985,-7.903) -- (3.921,-7.850) -- (3.862,-7.792) -- (3.809,-7.727) -- (3.762,-7.657) -- (3.721,-7.582) -- (3.684,-7.501) -- (3.653,-7.415) -- (3.626,-7.325) -- (3.604,-7.230) -- (3.585,-7.132) -- (3.569,-7.030) -- (3.557,-6.926) -- (3.546,-6.819) -- (3.538,-6.710) -- (3.532,-6.599) -- (3.527,-6.487) -- (3.524,-6.373) -- (3.521,-6.259) -- (3.519,-6.144) -- (3.517,-6.028) -- (3.516,-5.912) -- (3.515,-5.801) -- (3.615,-5.813) -- (3.722,-5.823) -- (3.831,-5.832) -- (3.942,-5.838) -- (4.054,-5.843) -- (4.168,-5.846) -- (4.282,-5.849) -- (4.397,-5.851) -- (4.513,-5.853) -- (4.629,-5.854) -- (4.683,-5.854) -- (4.695,-5.900) -- (4.720,-5.990) -- (4.745,-6.079) -- (4.772,-6.165) -- (4.801,-6.248) -- (4.832,-6.326) -- (4.867,-6.399) -- (4.904,-6.466) -- (4.944,-6.527) -- (4.988,-6.580) -- (5.035,-6.625) -- (5.084,-6.662) -- (5.136,-6.690) -- (5.189,-6.708) -- (5.243,-6.718) -- (5.298,-6.718) -- (5.352,-6.708) -- (5.405,-6.689) -- (5.456,-6.661) -- (5.505,-6.623) -- (5.550,-6.577) -- (5.593,-6.523) -- (5.632,-6.461) -- (5.668,-6.392) -- (5.699,-6.316) -- (5.727,-6.234) -- (5.751,-6.147) -- (5.771,-6.054) -- (5.789,-5.957) -- (5.803,-5.857) -- (5.803,-5.854) -- (5.912,-5.854) -- (6.028,-5.853) -- (6.144,-5.851) -- (6.259,-5.849) -- (6.373,-5.846) -- (6.487,-5.843) -- (6.599,-5.838) -- (6.710,-5.832) -- (6.819,-5.823) -- (6.926,-5.813) -- cycle;
\fill[shaded] (7.026,-5.801) -- (7.026,-5.796) -- (7.026,-5.679) -- (7.027,-5.563) -- (7.027,-5.446) -- (7.027,-5.329) -- (7.027,-5.212) -- (7.027,-5.095) -- (7.027,-4.978) -- (7.027,-4.861) -- (7.027,-4.743) -- (7.027,-4.626) -- (7.027,-4.509) -- (7.027,-4.392) -- (7.027,-4.275) -- (7.027,-4.158) -- (7.027,-4.041) -- (7.027,-3.924) -- (7.027,-3.807) -- (7.026,-3.690) -- (7.026,-3.574) -- (7.026,-3.569) -- (7.030,-3.569) -- (7.132,-3.585) -- (7.230,-3.603) -- (7.325,-3.626) -- (7.415,-3.653) -- (7.501,-3.684) -- (7.581,-3.720) -- (7.656,-3.762) -- (7.726,-3.808) -- (7.790,-3.861) -- (7.848,-3.918) -- (7.899,-3.982) -- (7.945,-4.050) -- (7.985,-4.124) -- (8.018,-4.202) -- (8.046,-4.284) -- (8.068,-4.369) -- (8.084,-4.457) -- (8.095,-4.548) -- (8.100,-4.639) -- (8.100,-4.731) -- (8.095,-4.822) -- (8.084,-4.913) -- (8.068,-5.001) -- (8.046,-5.086) -- (8.018,-5.168) -- (7.985,-5.246) -- (7.945,-5.320) -- (7.899,-5.388) -- (7.848,-5.451) -- (7.790,-5.509) -- (7.726,-5.562) -- (7.656,-5.608) -- (7.581,-5.650) -- (7.501,-5.686) -- (7.415,-5.717) -- (7.325,-5.744) -- (7.230,-5.766) -- (7.132,-5.785) -- (7.030,-5.801) -- cycle;
\fill[shaded] (2.347,-3.566) -- (2.347,-3.456) -- (2.347,-3.340) -- (2.348,-3.225) -- (2.350,-3.110) -- (2.353,-2.996) -- (2.356,-2.883) -- (2.361,-2.771) -- (2.367,-2.660) -- (2.375,-2.551) -- (2.385,-2.444) -- (2.398,-2.339) -- (2.413,-2.238) -- (2.432,-2.140) -- (2.455,-2.045) -- (2.482,-1.955) -- (2.513,-1.869) -- (2.549,-1.788) -- (2.591,-1.713) -- (2.638,-1.643) -- (2.691,-1.578) -- (2.750,-1.520) -- (2.814,-1.467) -- (2.884,-1.420) -- (2.959,-1.378) -- (3.040,-1.342) -- (3.126,-1.311) -- (3.216,-1.284) -- (3.311,-1.261) -- (3.409,-1.242) -- (3.511,-1.227) -- (3.615,-1.214) -- (3.722,-1.204) -- (3.831,-1.196) -- (3.942,-1.190) -- (4.054,-1.185) -- (4.167,-1.181) -- (4.282,-1.179) -- (4.396,-1.177) -- (4.512,-1.176) -- (4.627,-1.175) -- (4.743,-1.175) -- (4.858,-1.176) -- (4.973,-1.177) -- (5.088,-1.179) -- (5.202,-1.181) -- (5.316,-1.185) -- (5.428,-1.190) -- (5.539,-1.196) -- (5.648,-1.204) -- (5.755,-1.214) -- (5.859,-1.227) -- (5.961,-1.242) -- (6.059,-1.261) -- (6.154,-1.284) -- (6.244,-1.311) -- (6.330,-1.342) -- (6.410,-1.378) -- (6.486,-1.420) -- (6.556,-1.467) -- (6.620,-1.520) -- (6.679,-1.578) -- (6.732,-1.643) -- (6.779,-1.713) -- (6.820,-1.788) -- (6.857,-1.869) -- (6.888,-1.955) -- (6.915,-2.045) -- (6.938,-2.140) -- (6.956,-2.238) -- (6.972,-2.339) -- (6.985,-2.444) -- (6.995,-2.551) -- (7.003,-2.660) -- (7.009,-2.771) -- (7.014,-2.883) -- (7.018,-2.996) -- (7.020,-3.111) -- (7.022,-3.226) -- (7.024,-3.342) -- (7.025,-3.458) -- (7.026,-3.569) -- (6.926,-3.557) -- (6.819,-3.546) -- (6.710,-3.538) -- (6.599,-3.532) -- (6.487,-3.527) -- (6.373,-3.524) -- (6.259,-3.521) -- (6.144,-3.519) -- (6.028,-3.517) -- (5.912,-3.516) -- (5.801,-3.515) -- (5.801,-3.511) -- (5.785,-3.409) -- (5.766,-3.311) -- (5.744,-3.216) -- (5.717,-3.126) -- (5.686,-3.040) -- (5.650,-2.960) -- (5.608,-2.885) -- (5.562,-2.815) -- (5.509,-2.751) -- (5.451,-2.693) -- (5.388,-2.642) -- (5.320,-2.596) -- (5.246,-2.557) -- (5.168,-2.523) -- (5.086,-2.495) -- (5.001,-2.473) -- (4.913,-2.457) -- (4.822,-2.446) -- (4.731,-2.441) -- (4.639,-2.441) -- (4.548,-2.446) -- (4.457,-2.457) -- (4.369,-2.473) -- (4.284,-2.495) -- (4.202,-2.523) -- (4.124,-2.557) -- (4.050,-2.596) -- (3.982,-2.642) -- (3.918,-2.693) -- (3.861,-2.751) -- (3.808,-2.815) -- (3.762,-2.885) -- (3.720,-2.960) -- (3.684,-3.040) -- (3.653,-3.126) -- (3.626,-3.216) -- (3.603,-3.311) -- (3.585,-3.409) -- (3.569,-3.511) -- (3.569,-3.515) -- (3.458,-3.516) -- (3.342,-3.517) -- (3.226,-3.519) -- (3.111,-3.521) -- (2.997,-3.523) -- (2.883,-3.527) -- (2.771,-3.532) -- (2.661,-3.538) -- (2.552,-3.545) -- (2.446,-3.555) -- cycle;
\fill[shaded] (2.347,-3.566) -- (2.347,-3.571) -- (2.347,-3.687) -- (2.348,-3.802) -- (2.350,-3.917) -- (2.353,-4.031) -- (2.356,-4.145) -- (2.361,-4.257) -- (2.367,-4.368) -- (2.375,-4.477) -- (2.385,-4.584) -- (2.392,-4.640) -- (2.340,-4.634) -- (2.240,-4.619) -- (2.144,-4.601) -- (2.052,-4.580) -- (1.964,-4.556) -- (1.882,-4.528) -- (1.806,-4.496) -- (1.737,-4.461) -- (1.675,-4.422) -- (1.621,-4.379) -- (1.575,-4.333) -- (1.538,-4.285) -- (1.510,-4.233) -- (1.491,-4.180) -- (1.481,-4.126) -- (1.481,-4.072) -- (1.491,-4.018) -- (1.510,-3.965) -- (1.538,-3.914) -- (1.575,-3.865) -- (1.621,-3.819) -- (1.675,-3.777) -- (1.737,-3.738) -- (1.807,-3.702) -- (1.882,-3.671) -- (1.965,-3.643) -- (2.052,-3.619) -- (2.145,-3.599) -- (2.241,-3.581) -- (2.342,-3.567) -- cycle;
\fill[shaded] (3.515,-5.801) -- (3.511,-5.801) -- (3.409,-5.785) -- (3.311,-5.766) -- (3.216,-5.744) -- (3.126,-5.717) -- (3.040,-5.686) -- (2.959,-5.649) -- (2.884,-5.608) -- (2.814,-5.560) -- (2.750,-5.508) -- (2.691,-5.449) -- (2.638,-5.385) -- (2.591,-5.315) -- (2.549,-5.239) -- (2.513,-5.158) -- (2.482,-5.073) -- (2.455,-4.982) -- (2.432,-4.888) -- (2.413,-4.790) -- (2.398,-4.688) -- (2.392,-4.640) -- (2.443,-4.646) -- (2.548,-4.657) -- (2.656,-4.666) -- (2.764,-4.674) -- (2.874,-4.682) -- (2.983,-4.689) -- (3.093,-4.696) -- (3.201,-4.705) -- (3.309,-4.714) -- (3.415,-4.725) -- (3.515,-4.738) -- (3.515,-4.745) -- (3.515,-4.862) -- (3.515,-4.978) -- (3.514,-5.095) -- (3.514,-5.212) -- (3.514,-5.329) -- (3.514,-5.446) -- (3.515,-5.563) -- (3.515,-5.679) -- (3.515,-5.796) -- cycle;
\fill[shaded] (5.803,-5.854) -- (5.796,-5.855) -- (5.679,-5.855) -- (5.563,-5.855) -- (5.446,-5.856) -- (5.329,-5.856) -- (5.212,-5.856) -- (5.095,-5.856) -- (4.978,-5.855) -- (4.862,-5.855) -- (4.745,-5.855) -- (4.683,-5.854) -- (4.671,-5.809) -- (4.646,-5.718) -- (4.619,-5.628) -- (4.590,-5.540) -- (4.559,-5.455) -- (4.524,-5.373) -- (4.485,-5.295) -- (4.441,-5.221) -- (4.393,-5.153) -- (4.339,-5.090) -- (4.280,-5.032) -- (4.215,-4.980) -- (4.144,-4.933) -- (4.069,-4.891) -- (3.988,-4.855) -- (3.902,-4.824) -- (3.812,-4.797) -- (3.717,-4.774) -- (3.620,-4.755) -- (3.518,-4.739) -- (3.515,-4.738) -- (3.516,-4.629) -- (3.517,-4.513) -- (3.519,-4.397) -- (3.521,-4.282) -- (3.524,-4.168) -- (3.527,-4.054) -- (3.532,-3.942) -- (3.538,-3.831) -- (3.546,-3.722) -- (3.557,-3.615) -- (3.569,-3.515) -- (3.574,-3.515) -- (3.690,-3.515) -- (3.807,-3.514) -- (3.924,-3.514) -- (4.041,-3.514) -- (4.158,-3.514) -- (4.275,-3.514) -- (4.392,-3.514) -- (4.509,-3.514) -- (4.626,-3.514) -- (4.743,-3.514) -- (4.861,-3.514) -- (4.978,-3.514) -- (5.095,-3.514) -- (5.212,-3.514) -- (5.329,-3.514) -- (5.446,-3.514) -- (5.563,-3.514) -- (5.679,-3.515) -- (5.796,-3.515) -- (5.801,-3.515) -- (5.813,-3.615) -- (5.823,-3.722) -- (5.831,-3.831) -- (5.838,-3.942) -- (5.842,-4.054) -- (5.846,-4.167) -- (5.849,-4.282) -- (5.850,-4.396) -- (5.852,-4.512) -- (5.852,-4.627) -- (5.852,-4.743) -- (5.852,-4.858) -- (5.851,-4.973) -- (5.849,-5.088) -- (5.846,-5.202) -- (5.843,-5.315) -- (5.838,-5.428) -- (5.832,-5.538) -- (5.825,-5.647) -- (5.815,-5.753) -- cycle;
\draw[knot] (7.026,-5.681) -- (7.026,-5.679) -- (7.027,-5.563) -- (7.027,-5.446) -- (7.027,-5.329) -- (7.027,-5.212) -- (7.027,-5.095) -- (7.027,-4.978) -- (7.027,-4.861) -- (7.027,-4.743) -- (7.027,-4.626) -- (7.027,-4.509) -- (7.027,-4.392) -- (7.027,-4.275) -- (7.027,-4.158) -- (7.027,-4.041) -- (7.027,-3.924) -- (7.027,-3.807) -- (7.026,-3.690) -- (7.026,-3.574) -- (7.025,-3.458) -- (7.024,-3.342) -- (7.022,-3.226) -- (7.020,-3.111) -- (7.018,-2.996) -- (7.014,-2.883) -- (7.009,-2.771) -- (7.003,-2.660) -- (6.995,-2.551) -- (6.985,-2.444) -- (6.972,-2.339) -- (6.956,-2.238) -- (6.938,-2.140) -- (6.915,-2.045) -- (6.888,-1.955) -- (6.857,-1.869) -- (6.820,-1.788) -- (6.779,-1.713) -- (6.732,-1.643) -- (6.679,-1.578) -- (6.620,-1.520) -- (6.556,-1.467) -- (6.486,-1.420) -- (6.410,-1.378) -- (6.330,-1.342) -- (6.244,-1.311) -- (6.154,-1.284) -- (6.059,-1.261) -- (5.961,-1.242) -- (5.859,-1.227) -- (5.755,-1.214) -- (5.648,-1.204) -- (5.539,-1.196) -- (5.428,-1.190) -- (5.316,-1.185) -- (5.202,-1.181) -- (5.088,-1.179) -- (4.973,-1.177) -- (4.858,-1.176) -- (4.743,-1.175) -- (4.627,-1.175) -- (4.512,-1.176) -- (4.396,-1.177) -- (4.282,-1.179) -- (4.167,-1.181) -- (4.054,-1.185) -- (3.942,-1.190) -- (3.831,-1.196) -- (3.722,-1.204) -- (3.615,-1.214) -- (3.511,-1.227) -- (3.409,-1.242) -- (3.311,-1.261) -- (3.216,-1.284) -- (3.126,-1.311) -- (3.040,-1.342) -- (2.959,-1.378) -- (2.884,-1.420) -- (2.814,-1.467) -- (2.750,-1.520) -- (2.691,-1.578) -- (2.638,-1.643) -- (2.591,-1.713) -- (2.549,-1.788) -- (2.513,-1.869) -- (2.482,-1.955) -- (2.455,-2.045) -- (2.432,-2.140) -- (2.413,-2.238) -- (2.398,-2.339) -- (2.385,-2.444) -- (2.375,-2.551) -- (2.367,-2.660) -- (2.361,-2.771) -- (2.356,-2.883) -- (2.353,-2.996) -- (2.350,-3.110) -- (2.348,-3.225) -- (2.347,-3.340) -- (2.347,-3.446);
\draw[knot] (2.347,-3.686) -- (2.347,-3.687) -- (2.348,-3.802) -- (2.350,-3.917) -- (2.353,-4.031) -- (2.356,-4.145) -- (2.361,-4.257) -- (2.367,-4.368) -- (2.375,-4.477) -- (2.385,-4.584) -- (2.398,-4.688) -- (2.413,-4.790) -- (2.432,-4.888) -- (2.455,-4.982) -- (2.482,-5.073) -- (2.513,-5.158) -- (2.549,-5.239) -- (2.591,-5.315) -- (2.638,-5.385) -- (2.691,-5.449) -- (2.750,-5.508) -- (2.814,-5.560) -- (2.884,-5.608) -- (2.959,-5.649) -- (3.040,-5.686) -- (3.126,-5.717) -- (3.216,-5.744) -- (3.311,-5.766) -- (3.397,-5.783);
\draw[knot] (3.635,-5.815) -- (3.722,-5.823) -- (3.831,-5.832) -- (3.942,-5.838) -- (4.054,-5.843) -- (4.168,-5.846) -- (4.282,-5.849) -- (4.397,-5.851) -- (4.513,-5.853) -- (4.629,-5.854) -- (4.745,-5.855) -- (4.862,-5.855) -- (4.978,-5.855) -- (5.095,-5.856) -- (5.212,-5.856) -- (5.329,-5.856) -- (5.446,-5.856) -- (5.563,-5.855) -- (5.679,-5.855) -- (5.683,-5.855);
\draw[knot] (5.923,-5.854) -- (6.028,-5.853) -- (6.144,-5.851) -- (6.259,-5.849) -- (6.373,-5.846) -- (6.487,-5.843) -- (6.599,-5.838) -- (6.710,-5.832) -- (6.819,-5.823) -- (6.926,-5.813) -- (7.030,-5.801) -- (7.132,-5.785) -- (7.230,-5.766) -- (7.325,-5.744) -- (7.415,-5.717) -- (7.501,-5.686) -- (7.581,-5.650) -- (7.656,-5.608) -- (7.726,-5.562) -- (7.790,-5.509) -- (7.848,-5.451) -- (7.899,-5.388) -- (7.945,-5.320) -- (7.985,-5.246) -- (8.018,-5.168) -- (8.046,-5.086) -- (8.068,-5.001) -- (8.084,-4.913) -- (8.095,-4.822) -- (8.100,-4.731) -- (8.100,-4.639) -- (8.095,-4.548) -- (8.084,-4.457) -- (8.068,-4.369) -- (8.046,-4.284) -- (8.018,-4.202) -- (7.985,-4.124) -- (7.945,-4.050) -- (7.899,-3.982) -- (7.848,-3.918) -- (7.790,-3.861) -- (7.726,-3.808) -- (7.656,-3.762) -- (7.581,-3.720) -- (7.501,-3.684) -- (7.415,-3.653) -- (7.325,-3.626) -- (7.230,-3.603) -- (7.144,-3.587);
\draw[knot] (6.907,-3.555) -- (6.819,-3.546) -- (6.710,-3.538) -- (6.599,-3.532) -- (6.487,-3.527) -- (6.373,-3.524) -- (6.259,-3.521) -- (6.144,-3.519) -- (6.028,-3.517) -- (5.912,-3.516) -- (5.796,-3.515) -- (5.679,-3.515) -- (5.563,-3.514) -- (5.446,-3.514) -- (5.329,-3.514) -- (5.212,-3.514) -- (5.095,-3.514) -- (4.978,-3.514) -- (4.861,-3.514) -- (4.743,-3.514) -- (4.626,-3.514) -- (4.509,-3.514) -- (4.392,-3.514) -- (4.275,-3.514) -- (4.158,-3.514) -- (4.041,-3.514) -- (3.924,-3.514) -- (3.807,-3.514) -- (3.690,-3.515) -- (3.689,-3.515);
\draw[knot] (3.449,-3.516) -- (3.342,-3.517) -- (3.226,-3.519) -- (3.111,-3.521) -- (2.997,-3.523) -- (2.883,-3.527) -- (2.771,-3.532) -- (2.661,-3.538) -- (2.552,-3.545) -- (2.446,-3.555) -- (2.342,-3.567) -- (2.241,-3.581) -- (2.145,-3.599) -- (2.052,-3.619) -- (1.965,-3.643) -- (1.882,-3.671) -- (1.807,-3.702) -- (1.737,-3.738) -- (1.675,-3.777) -- (1.621,-3.819) -- (1.575,-3.865) -- (1.538,-3.914) -- (1.510,-3.965) -- (1.491,-4.018) -- (1.481,-4.072) -- (1.481,-4.126) -- (1.491,-4.180) -- (1.510,-4.233) -- (1.538,-4.285) -- (1.575,-4.333) -- (1.621,-4.379) -- (1.675,-4.422) -- (1.737,-4.461) -- (1.806,-4.496) -- (1.882,-4.528) -- (1.964,-4.556) -- (2.052,-4.580) -- (2.144,-4.601) -- (2.240,-4.619) -- (2.273,-4.623);
\draw[knot] (2.511,-4.653) -- (2.548,-4.657) -- (2.656,-4.666) -- (2.764,-4.674) -- (2.874,-4.682) -- (2.983,-4.689) -- (3.093,-4.696) -- (3.201,-4.705) -- (3.309,-4.714) -- (3.415,-4.725) -- (3.518,-4.739) -- (3.620,-4.755) -- (3.717,-4.774) -- (3.812,-4.797) -- (3.902,-4.824) -- (3.988,-4.855) -- (4.069,-4.891) -- (4.144,-4.933) -- (4.215,-4.980) -- (4.280,-5.032) -- (4.339,-5.090) -- (4.393,-5.153) -- (4.441,-5.221) -- (4.485,-5.295) -- (4.524,-5.373) -- (4.559,-5.455) -- (4.590,-5.540) -- (4.619,-5.628) -- (4.646,-5.718) -- (4.651,-5.738);
\draw[knot] (4.714,-5.970) -- (4.720,-5.990) -- (4.745,-6.079) -- (4.772,-6.165) -- (4.801,-6.248) -- (4.832,-6.326) -- (4.867,-6.399) -- (4.904,-6.466) -- (4.944,-6.527) -- (4.988,-6.580) -- (5.035,-6.625) -- (5.084,-6.662) -- (5.136,-6.690) -- (5.189,-6.708) -- (5.243,-6.718) -- (5.298,-6.718) -- (5.352,-6.708) -- (5.405,-6.689) -- (5.456,-6.661) -- (5.505,-6.623) -- (5.550,-6.577) -- (5.593,-6.523) -- (5.632,-6.461) -- (5.668,-6.392) -- (5.699,-6.316) -- (5.727,-6.234) -- (5.751,-6.147) -- (5.771,-6.054) -- (5.789,-5.957) -- (5.803,-5.857) -- (5.815,-5.753) -- (5.825,-5.647) -- (5.832,-5.538) -- (5.838,-5.428) -- (5.843,-5.315) -- (5.846,-5.202) -- (5.849,-5.088) -- (5.851,-4.973) -- (5.852,-4.858) -- (5.852,-4.743) -- (5.852,-4.627) -- (5.852,-4.512) -- (5.850,-4.396) -- (5.849,-4.282) -- (5.846,-4.167) -- (5.842,-4.054) -- (5.838,-3.942) -- (5.831,-3.831) -- (5.823,-3.722) -- (5.815,-3.635);
\draw[knot] (5.783,-3.397) -- (5.766,-3.311) -- (5.744,-3.216) -- (5.717,-3.126) -- (5.686,-3.040) -- (5.650,-2.960) -- (5.608,-2.885) -- (5.562,-2.815) -- (5.509,-2.751) -- (5.451,-2.693) -- (5.388,-2.642) -- (5.320,-2.596) -- (5.246,-2.557) -- (5.168,-2.523) -- (5.086,-2.495) -- (5.001,-2.473) -- (4.913,-2.457) -- (4.822,-2.446) -- (4.731,-2.441) -- (4.639,-2.441) -- (4.548,-2.446) -- (4.457,-2.457) -- (4.369,-2.473) -- (4.284,-2.495) -- (4.202,-2.523) -- (4.124,-2.557) -- (4.050,-2.596) -- (3.982,-2.642) -- (3.918,-2.693) -- (3.861,-2.751) -- (3.808,-2.815) -- (3.762,-2.885) -- (3.720,-2.960) -- (3.684,-3.040) -- (3.653,-3.126) -- (3.626,-3.216) -- (3.603,-3.311) -- (3.585,-3.409) -- (3.569,-3.511) -- (3.557,-3.615) -- (3.546,-3.722) -- (3.538,-3.831) -- (3.532,-3.942) -- (3.527,-4.054) -- (3.524,-4.168) -- (3.521,-4.282) -- (3.519,-4.397) -- (3.517,-4.513) -- (3.516,-4.618);
\draw[knot] (3.515,-4.858) -- (3.515,-4.862) -- (3.515,-4.978) -- (3.514,-5.095) -- (3.514,-5.212) -- (3.514,-5.329) -- (3.514,-5.446) -- (3.515,-5.563) -- (3.515,-5.679) -- (3.515,-5.796) -- (3.516,-5.912) -- (3.517,-6.028) -- (3.519,-6.144) -- (3.521,-6.259) -- (3.524,-6.373) -- (3.527,-6.487) -- (3.532,-6.599) -- (3.538,-6.710) -- (3.546,-6.819) -- (3.557,-6.926) -- (3.569,-7.030) -- (3.585,-7.132) -- (3.604,-7.230) -- (3.626,-7.325) -- (3.653,-7.415) -- (3.684,-7.501) -- (3.721,-7.582) -- (3.762,-7.657) -- (3.809,-7.727) -- (3.862,-7.792) -- (3.921,-7.850) -- (3.985,-7.903) -- (4.055,-7.950) -- (4.131,-7.992) -- (4.211,-8.028) -- (4.297,-8.059) -- (4.387,-8.086) -- (4.481,-8.108) -- (4.579,-8.127) -- (4.680,-8.142) -- (4.784,-8.153) -- (4.890,-8.162) -- (4.997,-8.169) -- (5.106,-8.173) -- (5.216,-8.175) -- (5.325,-8.175) -- (5.435,-8.173) -- (5.544,-8.169) -- (5.651,-8.162) -- (5.757,-8.153) -- (5.861,-8.142) -- (5.962,-8.127) -- (6.060,-8.108) -- (6.154,-8.086) -- (6.244,-8.059) -- (6.330,-8.028) -- (6.411,-7.992) -- (6.486,-7.950) -- (6.556,-7.903) -- (6.620,-7.850) -- (6.679,-7.792) -- (6.732,-7.727) -- (6.779,-7.657) -- (6.820,-7.582) -- (6.857,-7.501) -- (6.888,-7.415) -- (6.915,-7.325) -- (6.938,-7.230) -- (6.956,-7.132) -- (6.972,-7.030) -- (6.985,-6.926) -- (6.995,-6.819) -- (7.003,-6.710) -- (7.009,-6.599) -- (7.014,-6.487) -- (7.018,-6.373) -- (7.020,-6.259) -- (7.022,-6.144) -- (7.024,-6.028) -- (7.025,-5.921);
\node[reglabel] (w0) at (2.5,-6.267) {$6$};
\node[reglabel] (w1) at (6.436,-4.120) {$1$};
\node[reglabel] (w2) at (5.266,-6.267) {$2$};
\node[reglabel] (w3) at (2.946,-4.105) {$3$};
\node[reglabel] (w4) at (4.714,-2.977) {$4$};
\node[reglabel] (w5) at (3.974,-5.377) {$5$};
\end{tikzpicture}

%% file: figures/L3-mirror.tex
\begin{tikzpicture}[line width=1.00pt, line cap=round, every node/.style={font=\small}]
\draw (0.000,0.000) -- (0.600,0.000);
\draw (0.000,1.800) -- (0.600,1.800);
\draw (0.000,1.200) .. controls (0.300,1.200) and (0.300,0.600) .. (0.600,0.600);
\draw[white, line width=4.0pt, line cap=butt] (0.149,0.662) .. controls (0.271,0.778) and (0.329,1.022) .. (0.451,1.138);
\draw (0.000,0.600) .. controls (0.300,0.600) and (0.300,1.200) .. (0.600,1.200);
\draw (0.600,0.000) -- (1.200,0.000);
\draw (0.600,1.800) -- (1.200,1.800);
\draw (0.600,1.200) .. controls (0.900,1.200) and (0.900,0.600) .. (1.200,0.600);
\draw[white, line width=4.0pt, line cap=butt] (0.749,0.662) .. controls (0.871,0.778) and (0.929,1.022) .. (1.051,1.138);
\draw (0.600,0.600) .. controls (0.900,0.600) and (0.900,1.200) .. (1.200,1.200);
\draw (1.200,0.000) -- (1.800,0.000);
\draw (1.200,0.600) -- (1.800,0.600);
\draw (1.200,1.800) .. controls (1.500,1.800) and (1.500,1.200) .. (1.800,1.200);
\draw[white, line width=4.0pt, line cap=butt] (1.349,1.262) .. controls (1.471,1.378) and (1.529,1.622) .. (1.651,1.738);
\draw (1.200,1.200) .. controls (1.500,1.200) and (1.500,1.800) .. (1.800,1.800);
\draw (1.800,0.000) -- (2.400,0.000);
\draw (1.800,0.600) -- (2.400,0.600);
\draw (1.800,1.800) .. controls (2.100,1.800) and (2.100,1.200) .. (2.400,1.200);
\draw[white, line width=4.0pt, line cap=butt] (1.949,1.262) .. controls (2.071,1.378) and (2.129,1.622) .. (2.251,1.738);
\draw (1.800,1.200) .. controls (2.100,1.200) and (2.100,1.800) .. (2.400,1.800);
\draw (2.400,0.000) -- (3.000,0.000);
\draw (2.400,1.800) -- (3.000,1.800);
\draw (2.400,1.200) .. controls (2.700,1.200) and (2.700,0.600) .. (3.000,0.600);
\draw[white, line width=4.0pt, line cap=butt] (2.549,0.662) .. controls (2.671,0.778) and (2.729,1.022) .. (2.851,1.138);
\draw (2.400,0.600) .. controls (2.700,0.600) and (2.700,1.200) .. (3.000,1.200);
\draw (3.000,0.000) -- (3.600,0.000);
\draw (3.000,1.800) -- (3.600,1.800);
\draw (3.000,1.200) .. controls (3.300,1.200) and (3.300,0.600) .. (3.600,0.600);
\draw[white, line width=4.0pt, line cap=butt] (3.149,0.662) .. controls (3.271,0.778) and (3.329,1.022) .. (3.451,1.138);
\draw (3.000,0.600) .. controls (3.300,0.600) and (3.300,1.200) .. (3.600,1.200);
\draw (3.600,0.000) -- (4.200,0.000);
\draw (3.600,1.800) -- (4.200,1.800);
\draw (3.600,1.200) .. controls (3.900,1.200) and (3.900,0.600) .. (4.200,0.600);
\draw[white, line width=4.0pt, line cap=butt] (3.749,0.662) .. controls (3.871,0.778) and (3.929,1.022) .. (4.051,1.138);
\draw (3.600,0.600) .. controls (3.900,0.600) and (3.900,1.200) .. (4.200,1.200);
\draw (4.200,0.000) -- (4.800,0.000);
\draw (4.200,0.600) -- (4.800,0.600);
\draw (4.200,1.800) .. controls (4.500,1.800) and (4.500,1.200) .. (4.800,1.200);
\draw[white, line width=4.0pt, line cap=butt] (4.349,1.262) .. controls (4.471,1.378) and (4.529,1.622) .. (4.651,1.738);
\draw (4.200,1.200) .. controls (4.500,1.200) and (4.500,1.800) .. (4.800,1.800);
\draw (4.800,0.000) -- (5.400,0.000);
\draw (4.800,0.600) -- (5.400,0.600);
\draw (4.800,1.800) .. controls (5.100,1.800) and (5.100,1.200) .. (5.400,1.200);
\draw[white, line width=4.0pt, line cap=butt] (4.949,1.262) .. controls (5.071,1.378) and (5.129,1.622) .. (5.251,1.738);
\draw (4.800,1.200) .. controls (5.100,1.200) and (5.100,1.800) .. (5.400,1.800);
\draw (5.400,0.000) -- (6.000,0.000);
\draw (5.400,0.600) -- (6.000,0.600);
\draw (5.400,1.800) .. controls (5.700,1.800) and (5.700,1.200) .. (6.000,1.200);
\draw[white, line width=4.0pt, line cap=butt] (5.549,1.262) .. controls (5.671,1.378) and (5.729,1.622) .. (5.851,1.738);
\draw (5.400,1.200) .. controls (5.700,1.200) and (5.700,1.800) .. (6.000,1.800);
\draw (6.000,0.000) -- (6.600,0.000);
\draw (6.000,0.600) -- (6.600,0.600);
\draw (6.000,1.800) .. controls (6.300,1.800) and (6.300,1.200) .. (6.600,1.200);
\draw[white, line width=4.0pt, line cap=butt] (6.149,1.262) .. controls (6.271,1.378) and (6.329,1.622) .. (6.451,1.738);
\draw (6.000,1.200) .. controls (6.300,1.200) and (6.300,1.800) .. (6.600,1.800);
\draw (6.600,0.000) -- (7.200,0.000);
\draw (6.600,0.600) -- (7.200,0.600);
\draw (6.600,1.800) .. controls (6.900,1.800) and (6.900,1.200) .. (7.200,1.200);
\draw[white, line width=4.0pt, line cap=butt] (6.749,1.262) .. controls (6.871,1.378) and (6.929,1.622) .. (7.051,1.738);
\draw (6.600,1.200) .. controls (6.900,1.200) and (6.900,1.800) .. (7.200,1.800);
\draw (7.200,0.000) -- (7.800,0.000);
\draw (7.200,0.600) -- (7.800,0.600);
\draw (7.200,1.800) .. controls (7.500,1.800) and (7.500,1.200) .. (7.800,1.200);
\draw[white, line width=4.0pt, line cap=butt] (7.349,1.262) .. controls (7.471,1.378) and (7.529,1.622) .. (7.651,1.738);
\draw (7.200,1.200) .. controls (7.500,1.200) and (7.500,1.800) .. (7.800,1.800);
\draw (7.800,0.000) -- (8.400,0.000);
\draw (7.800,0.600) -- (8.400,0.600);
\draw (7.800,1.800) .. controls (8.100,1.800) and (8.100,1.200) .. (8.400,1.200);
\draw[white, line width=4.0pt, line cap=butt] (7.949,1.262) .. controls (8.071,1.378) and (8.129,1.622) .. (8.251,1.738);
\draw (7.800,1.200) .. controls (8.100,1.200) and (8.100,1.800) .. (8.400,1.800);
\draw (8.400,0.000) -- (9.000,0.000);
\draw (8.400,1.800) -- (9.000,1.800);
\draw (8.400,1.200) .. controls (8.700,1.200) and (8.700,0.600) .. (9.000,0.600);
\draw[white, line width=4.0pt, line cap=butt] (8.549,0.662) .. controls (8.671,0.778) and (8.729,1.022) .. (8.851,1.138);
\draw (8.400,0.600) .. controls (8.700,0.600) and (8.700,1.200) .. (9.000,1.200);
\draw (9.000,0.000) -- (9.600,0.000);
\draw (9.000,1.800) -- (9.600,1.800);
\draw (9.000,1.200) .. controls (9.300,1.200) and (9.300,0.600) .. (9.600,0.600);
\draw[white, line width=4.0pt, line cap=butt] (9.149,0.662) .. controls (9.271,0.778) and (9.329,1.022) .. (9.451,1.138);
\draw (9.000,0.600) .. controls (9.300,0.600) and (9.300,1.200) .. (9.600,1.200);
\draw (9.600,1.800) .. controls (10.000,1.800) and (10.000,1.200) .. (9.600,1.200);
\draw (9.600,0.600) .. controls (10.000,0.600) and (10.000,0.000) .. (9.600,0.000);
\draw (0.000,1.800) .. controls (-0.400,1.800) and (-0.400,1.200) .. (0.000,1.200);
\draw (0.000,0.600) .. controls (-0.400,0.600) and (-0.400,0.000) .. (0.000,0.000);
\end{tikzpicture}

%% file: figures/cc1028-1-10_28.tex
\begin{tikzpicture}[line width=1.0pt, line cap=round, line join=round]
\draw (2.817,-6.843) -- (2.883,-6.897) -- (2.955,-6.945) -- (3.032,-6.988) -- (3.115,-7.025) -- (3.203,-7.057) -- (3.295,-7.085) -- (3.392,-7.108) -- (3.493,-7.127) -- (3.597,-7.143) -- (3.704,-7.156) -- (3.814,-7.166) -- (3.925,-7.175) -- (4.039,-7.181) -- (4.154,-7.186) -- (4.270,-7.190) -- (4.387,-7.193) -- (4.505,-7.195) -- (4.624,-7.196) -- (4.742,-7.198) -- (4.862,-7.198) -- (4.981,-7.199) -- (5.101,-7.199) -- (5.220,-7.199) -- (5.340,-7.199) -- (5.460,-7.199) -- (5.580,-7.199) -- (5.699,-7.199) -- (5.819,-7.199) -- (5.938,-7.198) -- (6.058,-7.198) -- (6.176,-7.196) -- (6.295,-7.195) -- (6.413,-7.193) -- (6.530,-7.190) -- (6.646,-7.186) -- (6.761,-7.181) -- (6.875,-7.175) -- (6.986,-7.167) -- (7.096,-7.156) -- (7.203,-7.143) -- (7.307,-7.127) -- (7.408,-7.108) -- (7.505,-7.085) -- (7.597,-7.058) -- (7.685,-7.025) -- (7.767,-6.988) -- (7.845,-6.946) -- (7.916,-6.898) -- (7.981,-6.845) -- (8.040,-6.785) -- (8.093,-6.721) -- (8.140,-6.650) -- (8.181,-6.575) -- (8.215,-6.495) -- (8.243,-6.411) -- (8.266,-6.324) -- (8.283,-6.233) -- (8.294,-6.141) -- (8.299,-6.047) -- (8.299,-5.953) -- (8.294,-5.859) -- (8.283,-5.767) -- (8.266,-5.676) -- (8.243,-5.589) -- (8.215,-5.505) -- (8.181,-5.425) -- (8.140,-5.350) -- (8.093,-5.279) -- (8.040,-5.215) -- (7.981,-5.155) -- (7.916,-5.102) -- (7.845,-5.054) -- (7.767,-5.012) -- (7.685,-4.975) -- (7.597,-4.942) -- (7.505,-4.915) -- (7.408,-4.892) -- (7.307,-4.873) -- (7.203,-4.857) -- (7.096,-4.844) -- (6.986,-4.833) -- (6.875,-4.825) -- (6.761,-4.819) -- (6.646,-4.814) -- (6.530,-4.810) -- (6.413,-4.807) -- (6.295,-4.805) -- (6.176,-4.804) -- (6.058,-4.802) -- (5.938,-4.802) -- (5.819,-4.801) -- (5.699,-4.801) -- (5.580,-4.800) -- (5.460,-4.800) -- (5.340,-4.800) -- (5.220,-4.800) -- (5.100,-4.800) -- (4.980,-4.800) -- (4.860,-4.800) -- (4.740,-4.800) -- (4.620,-4.800) -- (4.500,-4.800) -- (4.380,-4.800) -- (4.260,-4.800) -- (4.140,-4.800) -- (4.020,-4.800) -- (3.901,-4.801) -- (3.781,-4.801) -- (3.662,-4.802) -- (3.542,-4.802) -- (3.424,-4.804) -- (3.305,-4.805) -- (3.187,-4.807) -- (3.070,-4.810) -- (2.954,-4.814) -- (2.839,-4.818) -- (2.726,-4.825) -- (2.615,-4.832) -- (2.506,-4.842) -- (2.399,-4.854) -- (2.296,-4.869) -- (2.197,-4.887) -- (2.102,-4.908) -- (2.013,-4.933) -- (1.929,-4.961) -- (1.851,-4.993) -- (1.780,-5.029) -- (1.716,-5.069) -- (1.661,-5.113) -- (1.614,-5.160) -- (1.576,-5.210) -- (1.547,-5.263) -- (1.527,-5.317) -- (1.517,-5.372) -- (1.517,-5.428) -- (1.527,-5.483) -- (1.546,-5.537) -- (1.575,-5.589) -- (1.612,-5.638) -- (1.659,-5.684) -- (1.713,-5.727) -- (1.775,-5.765) -- (1.844,-5.800) -- (1.919,-5.829) -- (1.999,-5.854) -- (2.084,-5.873) -- (2.172,-5.888) -- (2.263,-5.897) -- (2.355,-5.902) -- (2.449,-5.901) -- (2.542,-5.895) -- (2.634,-5.883) -- (2.724,-5.866) -- (2.811,-5.844) -- (2.895,-5.815) -- (2.975,-5.781) -- (3.050,-5.740) -- (3.121,-5.693) -- (3.185,-5.640) -- (3.245,-5.581) -- (3.298,-5.516) -- (3.346,-5.445) -- (3.388,-5.367) -- (3.425,-5.285) -- (3.458,-5.197) -- (3.485,-5.105) -- (3.508,-5.008) -- (3.527,-4.907) -- (3.543,-4.803) -- (3.556,-4.696) -- (3.567,-4.587) -- (3.575,-4.475) -- (3.581,-4.361) -- (3.586,-4.246) -- (3.590,-4.130) -- (3.593,-4.013) -- (3.596,-3.896) -- (3.598,-3.778) -- (3.599,-3.659) -- (3.601,-3.541) -- (3.602,-3.422) -- (3.604,-3.304) -- (3.607,-3.187) -- (3.610,-3.070) -- (3.613,-2.954) -- (3.618,-2.839) -- (3.624,-2.726) -- (3.632,-2.615) -- (3.642,-2.506) -- (3.654,-2.399) -- (3.669,-2.296) -- (3.687,-2.197) -- (3.708,-2.102) -- (3.733,-2.013) -- (3.761,-1.929) -- (3.793,-1.851) -- (3.829,-1.780) -- (3.869,-1.716) -- (3.913,-1.661) -- (3.960,-1.614) -- (4.010,-1.576) -- (4.063,-1.547) -- (4.117,-1.527) -- (4.172,-1.517) -- (4.228,-1.517) -- (4.284,-1.527) -- (4.338,-1.546) -- (4.391,-1.575) -- (4.441,-1.612) -- (4.489,-1.659) -- (4.534,-1.713) -- (4.576,-1.775) -- (4.614,-1.844) -- (4.649,-1.919) -- (4.681,-1.999) -- (4.711,-2.084) -- (4.738,-2.172) -- (4.764,-2.263) -- (4.790,-2.355) -- (4.815,-2.449) -- (4.840,-2.542) -- (4.867,-2.634) -- (4.897,-2.724) -- (4.929,-2.812) -- (4.965,-2.896) -- (5.005,-2.976) -- (5.049,-3.051) -- (5.098,-3.122) -- (5.153,-3.187) -- (5.213,-3.247) -- (5.278,-3.302) -- (5.349,-3.351) -- (5.424,-3.396) -- (5.504,-3.436) -- (5.589,-3.471) -- (5.676,-3.504) -- (5.767,-3.533) -- (5.859,-3.561) -- (5.953,-3.587) -- (6.047,-3.613) -- (6.141,-3.639) -- (6.233,-3.667) -- (6.324,-3.696) -- (6.411,-3.729) -- (6.495,-3.765) -- (6.575,-3.805) -- (6.650,-3.850) -- (6.721,-3.899) -- (6.785,-3.955) -- (6.845,-4.015) -- (6.898,-4.082) -- (6.946,-4.154) -- (6.988,-4.232) -- (7.025,-4.315) -- (7.058,-4.402) -- (7.085,-4.495) -- (7.108,-4.592) -- (7.127,-4.693) -- (7.143,-4.797) -- (7.156,-4.904) -- (7.166,-5.013) -- (7.175,-5.125) -- (7.181,-5.239) -- (7.186,-5.354) -- (7.190,-5.470) -- (7.192,-5.587) -- (7.194,-5.704) -- (7.195,-5.822) -- (7.196,-5.941) -- (7.196,-6.059) -- (7.195,-6.178) -- (7.194,-6.296) -- (7.192,-6.413) -- (7.190,-6.530) -- (7.186,-6.646) -- (7.182,-6.761) -- (7.175,-6.874) -- (7.168,-6.985) -- (7.158,-7.094) -- (7.146,-7.201) -- (7.131,-7.304) -- (7.113,-7.403) -- (7.092,-7.498) -- (7.067,-7.587) -- (7.039,-7.671) -- (7.007,-7.749) -- (6.971,-7.820) -- (6.931,-7.884) -- (6.887,-7.939) -- (6.840,-7.986) -- (6.790,-8.024) -- (6.737,-8.053) -- (6.683,-8.073) -- (6.628,-8.082) -- (6.572,-8.082) -- (6.517,-8.073) -- (6.463,-8.053) -- (6.410,-8.024) -- (6.360,-7.986) -- (6.313,-7.939) -- (6.269,-7.884) -- (6.229,-7.820) -- (6.193,-7.749) -- (6.161,-7.671) -- (6.133,-7.587) -- (6.108,-7.498) -- (6.087,-7.403) -- (6.069,-7.304) -- (6.054,-7.201) -- (6.042,-7.094) -- (6.032,-6.985) -- (6.025,-6.874) -- (6.018,-6.761) -- (6.014,-6.646) -- (6.010,-6.530) -- (6.007,-6.413) -- (6.005,-6.295) -- (6.004,-6.176) -- (6.002,-6.058) -- (6.002,-5.938) -- (6.001,-5.819) -- (6.001,-5.699) -- (6.000,-5.580) -- (6.000,-5.460) -- (6.000,-5.340) -- (6.000,-5.220) -- (5.999,-5.101) -- (5.999,-4.981) -- (5.998,-4.862) -- (5.998,-4.742) -- (5.996,-4.624) -- (5.995,-4.505) -- (5.993,-4.387) -- (5.990,-4.270) -- (5.986,-4.154) -- (5.981,-4.039) -- (5.975,-3.925) -- (5.966,-3.814) -- (5.956,-3.704) -- (5.943,-3.597) -- (5.927,-3.493) -- (5.908,-3.392) -- (5.885,-3.295) -- (5.857,-3.203) -- (5.825,-3.115) -- (5.788,-3.032) -- (5.745,-2.955) -- (5.697,-2.883) -- (5.643,-2.817) -- (5.583,-2.757) -- (5.517,-2.703) -- (5.445,-2.655) -- (5.368,-2.612) -- (5.285,-2.575) -- (5.198,-2.543) -- (5.105,-2.516) -- (5.009,-2.493) -- (4.908,-2.474) -- (4.805,-2.459) -- (4.699,-2.446) -- (4.590,-2.437) -- (4.480,-2.430) -- (4.368,-2.426) -- (4.256,-2.424) -- (4.144,-2.424) -- (4.032,-2.426) -- (3.920,-2.430) -- (3.810,-2.437) -- (3.701,-2.446) -- (3.595,-2.459) -- (3.492,-2.474) -- (3.391,-2.493) -- (3.295,-2.516) -- (3.202,-2.543) -- (3.115,-2.575) -- (3.032,-2.612) -- (2.955,-2.655) -- (2.883,-2.703) -- (2.817,-2.757) -- (2.757,-2.817) -- (2.703,-2.883) -- (2.655,-2.955) -- (2.612,-3.032) -- (2.575,-3.115) -- (2.543,-3.203) -- (2.515,-3.295) -- (2.492,-3.392) -- (2.473,-3.493) -- (2.457,-3.597) -- (2.444,-3.704) -- (2.434,-3.813) -- (2.425,-3.925) -- (2.419,-4.039) -- (2.414,-4.154) -- (2.410,-4.270) -- (2.408,-4.387) -- (2.406,-4.504) -- (2.405,-4.622) -- (2.404,-4.741) -- (2.404,-4.859) -- (2.405,-4.978) -- (2.406,-5.096) -- (2.408,-5.213) -- (2.410,-5.330) -- (2.414,-5.446) -- (2.419,-5.561) -- (2.425,-5.675) -- (2.434,-5.787) -- (2.444,-5.896) -- (2.457,-6.003) -- (2.473,-6.107) -- (2.492,-6.208) -- (2.515,-6.305) -- (2.543,-6.397) -- (2.575,-6.485) -- (2.612,-6.568) -- (2.655,-6.645) -- (2.703,-6.717) -- (2.757,-6.783) -- cycle;
\draw[white, line width=4.5pt, line cap=butt] (6.087,-7.403) -- (6.069,-7.304) -- (6.054,-7.201) -- (6.054,-7.198) -- (6.042,-7.094) -- (6.032,-6.985);
\draw[line cap=butt] (6.087,-7.403) -- (6.069,-7.304) -- (6.054,-7.201) -- (6.054,-7.198) -- (6.042,-7.094) -- (6.032,-6.985);
\draw[white, line width=4.5pt, line cap=butt] (6.986,-7.167) -- (7.096,-7.156) -- (7.151,-7.149) -- (7.203,-7.143) -- (7.307,-7.127);
\draw[line cap=butt] (6.986,-7.167) -- (7.096,-7.156) -- (7.151,-7.149) -- (7.203,-7.143) -- (7.307,-7.127);
\draw[white, line width=4.5pt, line cap=butt] (7.127,-4.693) -- (7.143,-4.797) -- (7.150,-4.850) -- (7.156,-4.904) -- (7.166,-5.013);
\draw[line cap=butt] (7.127,-4.693) -- (7.143,-4.797) -- (7.150,-4.850) -- (7.156,-4.904) -- (7.166,-5.013);
\draw[white, line width=4.5pt, line cap=butt] (6.176,-4.804) -- (6.058,-4.802) -- (5.998,-4.802) -- (5.938,-4.802) -- (5.819,-4.801);
\draw[line cap=butt] (6.176,-4.804) -- (6.058,-4.802) -- (5.998,-4.802) -- (5.938,-4.802) -- (5.819,-4.801);
\draw[white, line width=4.5pt, line cap=butt] (3.508,-5.008) -- (3.527,-4.907) -- (3.543,-4.803) -- (3.543,-4.802) -- (3.556,-4.696) -- (3.567,-4.587);
\draw[line cap=butt] (3.508,-5.008) -- (3.527,-4.907) -- (3.543,-4.803) -- (3.543,-4.802) -- (3.556,-4.696) -- (3.567,-4.587);
\draw[white, line width=4.5pt, line cap=butt] (2.615,-4.832) -- (2.506,-4.842) -- (2.404,-4.854) -- (2.399,-4.854) -- (2.296,-4.869) -- (2.197,-4.887);
\draw[line cap=butt] (2.615,-4.832) -- (2.506,-4.842) -- (2.404,-4.854) -- (2.399,-4.854) -- (2.296,-4.869) -- (2.197,-4.887);
\draw[white, line width=4.5pt, line cap=butt] (2.425,-5.675) -- (2.434,-5.787) -- (2.444,-5.896) -- (2.445,-5.901) -- (2.457,-6.003) -- (2.473,-6.107);
\draw[line cap=butt] (2.425,-5.675) -- (2.434,-5.787) -- (2.444,-5.896) -- (2.445,-5.901) -- (2.457,-6.003) -- (2.473,-6.107);
\draw[white, line width=4.5pt, line cap=butt] (3.810,-2.437) -- (3.701,-2.446) -- (3.648,-2.452) -- (3.595,-2.459) -- (3.492,-2.474);
\draw[line cap=butt] (3.810,-2.437) -- (3.701,-2.446) -- (3.648,-2.452) -- (3.595,-2.459) -- (3.492,-2.474);
\draw[white, line width=4.5pt, line cap=butt] (4.764,-2.263) -- (4.790,-2.355) -- (4.815,-2.449) -- (4.818,-2.460) -- (4.840,-2.542) -- (4.867,-2.634);
\draw[line cap=butt] (4.764,-2.263) -- (4.790,-2.355) -- (4.815,-2.449) -- (4.818,-2.460) -- (4.840,-2.542) -- (4.867,-2.634);
\draw[white, line width=4.5pt, line cap=butt] (5.966,-3.814) -- (5.956,-3.704) -- (5.943,-3.597) -- (5.941,-3.584) -- (5.927,-3.493) -- (5.908,-3.392);
\draw[line cap=butt] (5.966,-3.814) -- (5.956,-3.704) -- (5.943,-3.597) -- (5.941,-3.584) -- (5.927,-3.493) -- (5.908,-3.392);
\draw[red!75!black, line width=0.8pt] (6.054,-7.198) circle (0.512);
\end{tikzpicture}

%% file: figures/cc1028-2-afterR2.tex
\begin{tikzpicture}[line width=1.0pt, line cap=round, line join=round]
\draw (2.817,-6.843) -- (2.883,-6.897) -- (2.955,-6.945) -- (3.032,-6.988) -- (3.115,-7.025) -- (3.203,-7.057) -- (3.295,-7.085) -- (3.392,-7.108) -- (3.493,-7.127) -- (3.597,-7.143) -- (3.704,-7.156) -- (3.814,-7.166) -- (3.925,-7.175) -- (4.039,-7.181) -- (4.154,-7.186) -- (4.270,-7.190) -- (4.387,-7.193) -- (4.505,-7.195) -- (4.624,-7.196) -- (4.742,-7.198) -- (4.862,-7.198) -- (4.981,-7.199) -- (5.101,-7.199) -- (5.220,-7.199) -- (5.340,-7.199) -- (5.460,-7.199) -- (5.580,-7.199) -- (5.699,-7.199) -- (5.819,-7.199) -- (5.938,-7.198) -- (6.058,-7.198) -- (6.176,-7.196) -- (6.295,-7.195) -- (6.413,-7.193) -- (6.530,-7.190) -- (6.646,-7.186) -- (6.761,-7.181) -- (6.875,-7.175) -- (6.986,-7.167) -- (7.096,-7.156) -- (7.203,-7.143) -- (7.307,-7.127) -- (7.408,-7.108) -- (7.505,-7.085) -- (7.597,-7.058) -- (7.685,-7.025) -- (7.767,-6.988) -- (7.845,-6.946) -- (7.916,-6.898) -- (7.981,-6.845) -- (8.040,-6.785) -- (8.093,-6.721) -- (8.140,-6.650) -- (8.181,-6.575) -- (8.215,-6.495) -- (8.243,-6.411) -- (8.266,-6.324) -- (8.283,-6.233) -- (8.294,-6.141) -- (8.299,-6.047) -- (8.299,-5.953) -- (8.294,-5.859) -- (8.283,-5.767) -- (8.266,-5.676) -- (8.243,-5.589) -- (8.215,-5.505) -- (8.181,-5.425) -- (8.140,-5.350) -- (8.093,-5.279) -- (8.040,-5.215) -- (7.981,-5.155) -- (7.916,-5.102) -- (7.845,-5.054) -- (7.767,-5.012) -- (7.685,-4.975) -- (7.597,-4.942) -- (7.505,-4.915) -- (7.408,-4.892) -- (7.307,-4.873) -- (7.203,-4.857) -- (7.096,-4.844) -- (6.986,-4.833) -- (6.875,-4.825) -- (6.761,-4.819) -- (6.646,-4.814) -- (6.530,-4.810) -- (6.413,-4.807) -- (6.295,-4.805) -- (6.176,-4.804) -- (6.058,-4.802) -- (5.938,-4.802) -- (5.819,-4.801) -- (5.699,-4.801) -- (5.580,-4.800) -- (5.460,-4.800) -- (5.340,-4.800) -- (5.220,-4.800) -- (5.100,-4.800) -- (4.980,-4.800) -- (4.860,-4.800) -- (4.740,-4.800) -- (4.620,-4.800) -- (4.500,-4.800) -- (4.380,-4.800) -- (4.260,-4.800) -- (4.140,-4.800) -- (4.020,-4.800) -- (3.901,-4.801) -- (3.781,-4.801) -- (3.662,-4.802) -- (3.542,-4.802) -- (3.424,-4.804) -- (3.305,-4.805) -- (3.187,-4.807) -- (3.070,-4.810) -- (2.954,-4.814) -- (2.839,-4.818) -- (2.726,-4.825) -- (2.615,-4.832) -- (2.506,-4.842) -- (2.399,-4.854) -- (2.296,-4.869) -- (2.197,-4.887) -- (2.102,-4.908) -- (2.013,-4.933) -- (1.929,-4.961) -- (1.851,-4.993) -- (1.780,-5.029) -- (1.716,-5.069) -- (1.661,-5.113) -- (1.614,-5.160) -- (1.576,-5.210) -- (1.547,-5.263) -- (1.527,-5.317) -- (1.517,-5.372) -- (1.517,-5.428) -- (1.527,-5.483) -- (1.546,-5.537) -- (1.575,-5.589) -- (1.612,-5.638) -- (1.659,-5.684) -- (1.713,-5.727) -- (1.775,-5.765) -- (1.844,-5.800) -- (1.919,-5.829) -- (1.999,-5.854) -- (2.084,-5.873) -- (2.172,-5.888) -- (2.263,-5.897) -- (2.355,-5.902) -- (2.449,-5.901) -- (2.542,-5.895) -- (2.634,-5.883) -- (2.724,-5.866) -- (2.811,-5.844) -- (2.895,-5.815) -- (2.975,-5.781) -- (3.050,-5.740) -- (3.121,-5.693) -- (3.185,-5.640) -- (3.245,-5.581) -- (3.298,-5.516) -- (3.346,-5.445) -- (3.388,-5.367) -- (3.425,-5.285) -- (3.458,-5.197) -- (3.485,-5.105) -- (3.508,-5.008) -- (3.527,-4.907) -- (3.543,-4.803) -- (3.556,-4.696) -- (3.567,-4.587) -- (3.575,-4.475) -- (3.581,-4.361) -- (3.586,-4.246) -- (3.590,-4.130) -- (3.593,-4.013) -- (3.596,-3.896) -- (3.598,-3.778) -- (3.599,-3.659) -- (3.601,-3.541) -- (3.602,-3.422) -- (3.604,-3.304) -- (3.607,-3.187) -- (3.610,-3.070) -- (3.613,-2.954) -- (3.618,-2.839) -- (3.624,-2.726) -- (3.632,-2.615) -- (3.642,-2.506) -- (3.654,-2.399) -- (3.669,-2.296) -- (3.687,-2.197) -- (3.708,-2.102) -- (3.733,-2.013) -- (3.761,-1.929) -- (3.793,-1.851) -- (3.829,-1.780) -- (3.869,-1.716) -- (3.913,-1.661) -- (3.960,-1.614) -- (4.010,-1.576) -- (4.063,-1.547) -- (4.117,-1.527) -- (4.172,-1.517) -- (4.228,-1.517) -- (4.284,-1.527) -- (4.338,-1.546) -- (4.391,-1.575) -- (4.441,-1.612) -- (4.489,-1.659) -- (4.534,-1.713) -- (4.576,-1.775) -- (4.614,-1.844) -- (4.649,-1.919) -- (4.681,-1.999) -- (4.711,-2.084) -- (4.738,-2.172) -- (4.764,-2.263) -- (4.790,-2.355) -- (4.815,-2.449) -- (4.840,-2.542) -- (4.867,-2.634) -- (4.897,-2.724) -- (4.929,-2.812) -- (4.965,-2.896) -- (5.005,-2.976) -- (5.049,-3.051) -- (5.098,-3.122) -- (5.153,-3.187) -- (5.213,-3.247) -- (5.278,-3.302) -- (5.349,-3.351) -- (5.424,-3.396) -- (5.504,-3.436) -- (5.589,-3.471) -- (5.676,-3.504) -- (5.767,-3.533) -- (5.859,-3.561) -- (5.953,-3.587) -- (6.047,-3.613) -- (6.141,-3.639) -- (6.233,-3.667) -- (6.324,-3.696) -- (6.411,-3.729) -- (6.495,-3.765) -- (6.575,-3.805) -- (6.645,-3.855) -- (6.710,-3.910) -- (6.769,-3.971) -- (6.823,-4.037) -- (6.872,-4.108) -- (6.915,-4.185) -- (6.955,-4.265) -- (6.989,-4.351) -- (7.020,-4.440) -- (7.047,-4.533) -- (7.070,-4.630) -- (7.090,-4.729) -- (7.108,-4.831) -- (7.122,-4.935) -- (7.135,-5.042) -- (7.145,-5.149) -- (7.154,-5.258) -- (7.160,-5.366) -- (7.165,-5.475) -- (7.169,-5.583) -- (7.171,-5.690) -- (7.171,-5.794) -- (7.170,-5.896) -- (7.166,-5.994) -- (7.162,-6.088) -- (7.155,-6.176) -- (7.146,-6.259) -- (7.135,-6.335) -- (7.122,-6.405) -- (7.107,-6.467) -- (7.090,-6.521) -- (7.070,-6.569) -- (7.049,-6.609) -- (7.026,-6.643) -- (7.002,-6.671) -- (6.976,-6.694) -- (6.949,-6.711) -- (6.921,-6.725) -- (6.892,-6.735) -- (6.863,-6.743) -- (6.833,-6.749) -- (6.802,-6.753) -- (6.772,-6.756) -- (6.741,-6.757) -- (6.709,-6.759) -- (6.678,-6.759) -- (6.647,-6.760) -- (6.616,-6.760) -- (6.584,-6.760) -- (6.553,-6.760) -- (6.522,-6.759) -- (6.490,-6.759) -- (6.459,-6.757) -- (6.428,-6.755) -- (6.398,-6.753) -- (6.367,-6.749) -- (6.337,-6.743) -- (6.307,-6.735) -- (6.279,-6.725) -- (6.250,-6.711) -- (6.223,-6.693) -- (6.198,-6.671) -- (6.173,-6.643) -- (6.150,-6.609) -- (6.129,-6.568) -- (6.109,-6.520) -- (6.091,-6.465) -- (6.076,-6.402) -- (6.062,-6.332) -- (6.050,-6.255) -- (6.040,-6.171) -- (6.031,-6.081) -- (6.024,-5.985) -- (6.018,-5.884) -- (6.013,-5.780) -- (6.009,-5.671) -- (6.006,-5.560) -- (6.004,-5.447) -- (6.002,-5.333) -- (6.000,-5.217) -- (5.998,-5.100) -- (5.997,-4.982) -- (5.995,-4.864) -- (5.993,-4.746) -- (5.991,-4.629) -- (5.989,-4.511) -- (5.986,-4.394) -- (5.982,-4.278) -- (5.977,-4.163) -- (5.972,-4.048) -- (5.965,-3.935) -- (5.957,-3.823) -- (5.947,-3.713) -- (5.935,-3.605) -- (5.921,-3.499) -- (5.905,-3.395) -- (5.885,-3.295) -- (5.857,-3.203) -- (5.825,-3.115) -- (5.788,-3.032) -- (5.745,-2.955) -- (5.697,-2.883) -- (5.643,-2.817) -- (5.583,-2.757) -- (5.517,-2.703) -- (5.445,-2.655) -- (5.368,-2.612) -- (5.285,-2.575) -- (5.198,-2.543) -- (5.105,-2.516) -- (5.009,-2.493) -- (4.908,-2.474) -- (4.805,-2.459) -- (4.699,-2.446) -- (4.590,-2.437) -- (4.480,-2.430) -- (4.368,-2.426) -- (4.256,-2.424) -- (4.144,-2.424) -- (4.032,-2.426) -- (3.920,-2.430) -- (3.810,-2.437) -- (3.701,-2.446) -- (3.595,-2.459) -- (3.492,-2.474) -- (3.391,-2.493) -- (3.295,-2.516) -- (3.202,-2.543) -- (3.115,-2.575) -- (3.032,-2.612) -- (2.955,-2.655) -- (2.883,-2.703) -- (2.817,-2.757) -- (2.757,-2.817) -- (2.703,-2.883) -- (2.655,-2.955) -- (2.612,-3.032) -- (2.575,-3.115) -- (2.543,-3.203) -- (2.515,-3.295) -- (2.492,-3.392) -- (2.473,-3.493) -- (2.457,-3.597) -- (2.444,-3.704) -- (2.434,-3.813) -- (2.425,-3.925) -- (2.419,-4.039) -- (2.414,-4.154) -- (2.410,-4.270) -- (2.408,-4.387) -- (2.406,-4.504) -- (2.405,-4.622) -- (2.404,-4.741) -- (2.404,-4.859) -- (2.405,-4.978) -- (2.406,-5.096) -- (2.408,-5.213) -- (2.410,-5.330) -- (2.414,-5.446) -- (2.419,-5.561) -- (2.425,-5.675) -- (2.434,-5.787) -- (2.444,-5.896) -- (2.457,-6.003) -- (2.473,-6.107) -- (2.492,-6.208) -- (2.515,-6.305) -- (2.543,-6.397) -- (2.575,-6.485) -- (2.612,-6.568) -- (2.655,-6.645) -- (2.703,-6.717) -- (2.757,-6.783) -- cycle;
\draw[white, line width=4.5pt, line cap=butt] (7.070,-4.630) -- (7.090,-4.729) -- (7.108,-4.831) -- (7.110,-4.846) -- (7.122,-4.935) -- (7.135,-5.042);
\draw[line cap=butt] (7.070,-4.630) -- (7.090,-4.729) -- (7.108,-4.831) -- (7.110,-4.846) -- (7.122,-4.935) -- (7.135,-5.042);
\draw[white, line width=4.5pt, line cap=butt] (6.176,-4.804) -- (6.058,-4.802) -- (5.994,-4.802) -- (5.938,-4.802) -- (5.819,-4.801);
\draw[line cap=butt] (6.176,-4.804) -- (6.058,-4.802) -- (5.994,-4.802) -- (5.938,-4.802) -- (5.819,-4.801);
\draw[white, line width=4.5pt, line cap=butt] (3.508,-5.008) -- (3.527,-4.907) -- (3.543,-4.803) -- (3.543,-4.802) -- (3.556,-4.696) -- (3.567,-4.587);
\draw[line cap=butt] (3.508,-5.008) -- (3.527,-4.907) -- (3.543,-4.803) -- (3.543,-4.802) -- (3.556,-4.696) -- (3.567,-4.587);
\draw[white, line width=4.5pt, line cap=butt] (2.615,-4.832) -- (2.506,-4.842) -- (2.404,-4.854) -- (2.399,-4.854) -- (2.296,-4.869) -- (2.197,-4.887);
\draw[line cap=butt] (2.615,-4.832) -- (2.506,-4.842) -- (2.404,-4.854) -- (2.399,-4.854) -- (2.296,-4.869) -- (2.197,-4.887);
\draw[white, line width=4.5pt, line cap=butt] (2.425,-5.675) -- (2.434,-5.787) -- (2.444,-5.896) -- (2.445,-5.901) -- (2.457,-6.003) -- (2.473,-6.107);
\draw[line cap=butt] (2.425,-5.675) -- (2.434,-5.787) -- (2.444,-5.896) -- (2.445,-5.901) -- (2.457,-6.003) -- (2.473,-6.107);
\draw[white, line width=4.5pt, line cap=butt] (3.810,-2.437) -- (3.701,-2.446) -- (3.648,-2.452) -- (3.595,-2.459) -- (3.492,-2.474);
\draw[line cap=butt] (3.810,-2.437) -- (3.701,-2.446) -- (3.648,-2.452) -- (3.595,-2.459) -- (3.492,-2.474);
\draw[white, line width=4.5pt, line cap=butt] (4.764,-2.263) -- (4.790,-2.355) -- (4.815,-2.449) -- (4.818,-2.460) -- (4.840,-2.542) -- (4.867,-2.634);
\draw[line cap=butt] (4.764,-2.263) -- (4.790,-2.355) -- (4.815,-2.449) -- (4.818,-2.460) -- (4.840,-2.542) -- (4.867,-2.634);
\draw[white, line width=4.5pt, line cap=butt] (5.947,-3.713) -- (5.935,-3.605) -- (5.932,-3.581) -- (5.921,-3.499) -- (5.905,-3.395);
\draw[line cap=butt] (5.947,-3.713) -- (5.935,-3.605) -- (5.932,-3.581) -- (5.921,-3.499) -- (5.905,-3.395);
\end{tikzpicture}

%% file: figures/cc1028-3-8_8.tex
\begin{tikzpicture}[line width=1.0pt, line cap=round, line join=round]
\draw (5.640,-5.585) -- (5.693,-5.521) -- (5.740,-5.450) -- (5.781,-5.375) -- (5.815,-5.295) -- (5.843,-5.211) -- (5.866,-5.124) -- (5.883,-5.033) -- (5.894,-4.941) -- (5.899,-4.847) -- (5.899,-4.753) -- (5.894,-4.659) -- (5.883,-4.567) -- (5.866,-4.476) -- (5.843,-4.389) -- (5.815,-4.305) -- (5.781,-4.225) -- (5.740,-4.150) -- (5.693,-4.079) -- (5.640,-4.015) -- (5.581,-3.955) -- (5.516,-3.902) -- (5.445,-3.854) -- (5.368,-3.812) -- (5.285,-3.775) -- (5.197,-3.743) -- (5.105,-3.716) -- (5.009,-3.693) -- (4.908,-3.674) -- (4.805,-3.658) -- (4.699,-3.646) -- (4.590,-3.637) -- (4.480,-3.630) -- (4.368,-3.626) -- (4.256,-3.624) -- (4.144,-3.624) -- (4.032,-3.626) -- (3.920,-3.630) -- (3.810,-3.637) -- (3.701,-3.646) -- (3.595,-3.659) -- (3.492,-3.674) -- (3.391,-3.693) -- (3.295,-3.716) -- (3.202,-3.743) -- (3.115,-3.775) -- (3.032,-3.812) -- (2.955,-3.855) -- (2.883,-3.903) -- (2.817,-3.957) -- (2.757,-4.017) -- (2.703,-4.083) -- (2.655,-4.155) -- (2.612,-4.232) -- (2.575,-4.315) -- (2.543,-4.402) -- (2.516,-4.495) -- (2.493,-4.591) -- (2.474,-4.692) -- (2.458,-4.795) -- (2.446,-4.901) -- (2.437,-5.010) -- (2.430,-5.120) -- (2.426,-5.231) -- (2.424,-5.344) -- (2.424,-5.456) -- (2.426,-5.568) -- (2.430,-5.679) -- (2.436,-5.789) -- (2.445,-5.897) -- (2.456,-6.002) -- (2.470,-6.105) -- (2.488,-6.204) -- (2.508,-6.298) -- (2.533,-6.388) -- (2.561,-6.471) -- (2.593,-6.549) -- (2.630,-6.620) -- (2.670,-6.684) -- (2.713,-6.739) -- (2.760,-6.786) -- (2.810,-6.824) -- (2.863,-6.853) -- (2.917,-6.873) -- (2.972,-6.882) -- (3.028,-6.882) -- (3.083,-6.873) -- (3.137,-6.853) -- (3.190,-6.824) -- (3.240,-6.786) -- (3.287,-6.739) -- (3.331,-6.684) -- (3.371,-6.620) -- (3.407,-6.549) -- (3.439,-6.471) -- (3.467,-6.387) -- (3.492,-6.298) -- (3.513,-6.203) -- (3.531,-6.104) -- (3.546,-6.001) -- (3.558,-5.894) -- (3.568,-5.785) -- (3.575,-5.674) -- (3.582,-5.561) -- (3.586,-5.446) -- (3.590,-5.330) -- (3.593,-5.213) -- (3.595,-5.095) -- (3.596,-4.976) -- (3.598,-4.858) -- (3.598,-4.738) -- (3.599,-4.619) -- (3.599,-4.499) -- (3.600,-4.380) -- (3.600,-4.260) -- (3.600,-4.140) -- (3.600,-4.020) -- (3.601,-3.901) -- (3.601,-3.781) -- (3.602,-3.662) -- (3.602,-3.542) -- (3.604,-3.424) -- (3.605,-3.305) -- (3.607,-3.187) -- (3.610,-3.070) -- (3.614,-2.954) -- (3.619,-2.839) -- (3.625,-2.725) -- (3.634,-2.614) -- (3.644,-2.504) -- (3.657,-2.397) -- (3.673,-2.293) -- (3.692,-2.192) -- (3.715,-2.095) -- (3.743,-2.003) -- (3.775,-1.915) -- (3.812,-1.832) -- (3.855,-1.755) -- (3.903,-1.683) -- (3.957,-1.617) -- (4.017,-1.557) -- (4.083,-1.503) -- (4.155,-1.455) -- (4.232,-1.412) -- (4.315,-1.375) -- (4.402,-1.343) -- (4.495,-1.316) -- (4.591,-1.293) -- (4.692,-1.274) -- (4.795,-1.259) -- (4.901,-1.246) -- (5.010,-1.237) -- (5.120,-1.230) -- (5.232,-1.226) -- (5.344,-1.224) -- (5.456,-1.224) -- (5.568,-1.226) -- (5.680,-1.230) -- (5.790,-1.237) -- (5.899,-1.246) -- (6.005,-1.259) -- (6.108,-1.274) -- (6.209,-1.293) -- (6.305,-1.316) -- (6.398,-1.343) -- (6.485,-1.375) -- (6.568,-1.412) -- (6.645,-1.455) -- (6.717,-1.503) -- (6.783,-1.557) -- (6.843,-1.617) -- (6.897,-1.683) -- (6.945,-1.755) -- (6.988,-1.832) -- (7.025,-1.915) -- (7.057,-2.003) -- (7.085,-2.095) -- (7.108,-2.192) -- (7.127,-2.293) -- (7.143,-2.397) -- (7.156,-2.504) -- (7.166,-2.614) -- (7.175,-2.725) -- (7.181,-2.839) -- (7.186,-2.954) -- (7.190,-3.070) -- (7.193,-3.187) -- (7.195,-3.305) -- (7.196,-3.424) -- (7.198,-3.542) -- (7.198,-3.662) -- (7.199,-3.781) -- (7.199,-3.901) -- (7.199,-4.020) -- (7.199,-4.140) -- (7.199,-4.260) -- (7.199,-4.380) -- (7.199,-4.499) -- (7.199,-4.619) -- (7.198,-4.738) -- (7.198,-4.858) -- (7.196,-4.976) -- (7.195,-5.095) -- (7.193,-5.213) -- (7.190,-5.330) -- (7.186,-5.446) -- (7.181,-5.561) -- (7.175,-5.675) -- (7.167,-5.786) -- (7.156,-5.896) -- (7.143,-6.003) -- (7.127,-6.107) -- (7.108,-6.208) -- (7.085,-6.305) -- (7.058,-6.397) -- (7.025,-6.485) -- (6.988,-6.567) -- (6.946,-6.645) -- (6.898,-6.716) -- (6.845,-6.781) -- (6.785,-6.840) -- (6.721,-6.893) -- (6.650,-6.940) -- (6.575,-6.981) -- (6.495,-7.015) -- (6.411,-7.043) -- (6.324,-7.066) -- (6.233,-7.083) -- (6.141,-7.094) -- (6.047,-7.099) -- (5.953,-7.099) -- (5.859,-7.094) -- (5.767,-7.083) -- (5.677,-7.066) -- (5.589,-7.044) -- (5.505,-7.015) -- (5.425,-6.981) -- (5.350,-6.941) -- (5.281,-6.895) -- (5.216,-6.842) -- (5.158,-6.784) -- (5.105,-6.719) -- (5.059,-6.650) -- (5.019,-6.575) -- (4.985,-6.495) -- (4.956,-6.411) -- (4.934,-6.323) -- (4.917,-6.233) -- (4.906,-6.141) -- (4.901,-6.047) -- (4.901,-5.953) -- (4.906,-5.859) -- (4.917,-5.767) -- (4.934,-5.676) -- (4.957,-5.589) -- (4.985,-5.505) -- (5.019,-5.425) -- (5.060,-5.350) -- (5.107,-5.279) -- (5.160,-5.215) -- (5.219,-5.155) -- (5.284,-5.102) -- (5.355,-5.054) -- (5.432,-5.011) -- (5.515,-4.974) -- (5.603,-4.942) -- (5.695,-4.915) -- (5.791,-4.891) -- (5.892,-4.872) -- (5.995,-4.855) -- (6.101,-4.841) -- (6.210,-4.830) -- (6.320,-4.820) -- (6.432,-4.812) -- (6.544,-4.804) -- (6.656,-4.796) -- (6.768,-4.788) -- (6.880,-4.780) -- (6.990,-4.770) -- (7.099,-4.759) -- (7.205,-4.745) -- (7.308,-4.728) -- (7.409,-4.709) -- (7.505,-4.685) -- (7.597,-4.658) -- (7.685,-4.626) -- (7.768,-4.589) -- (7.845,-4.546) -- (7.916,-4.498) -- (7.981,-4.445) -- (8.040,-4.385) -- (8.093,-4.321) -- (8.140,-4.250) -- (8.181,-4.175) -- (8.215,-4.095) -- (8.243,-4.011) -- (8.266,-3.924) -- (8.283,-3.833) -- (8.294,-3.741) -- (8.299,-3.647) -- (8.299,-3.553) -- (8.294,-3.459) -- (8.283,-3.367) -- (8.266,-3.276) -- (8.243,-3.189) -- (8.215,-3.105) -- (8.181,-3.025) -- (8.140,-2.950) -- (8.093,-2.879) -- (8.040,-2.815) -- (7.981,-2.755) -- (7.916,-2.702) -- (7.845,-2.654) -- (7.767,-2.612) -- (7.685,-2.575) -- (7.597,-2.542) -- (7.505,-2.515) -- (7.408,-2.492) -- (7.307,-2.473) -- (7.203,-2.457) -- (7.096,-2.444) -- (6.986,-2.433) -- (6.875,-2.425) -- (6.761,-2.419) -- (6.646,-2.414) -- (6.530,-2.410) -- (6.413,-2.407) -- (6.295,-2.405) -- (6.176,-2.404) -- (6.058,-2.402) -- (5.938,-2.402) -- (5.819,-2.401) -- (5.699,-2.401) -- (5.580,-2.400) -- (5.460,-2.400) -- (5.340,-2.400) -- (5.220,-2.400) -- (5.100,-2.400) -- (4.980,-2.400) -- (4.860,-2.400) -- (4.740,-2.400) -- (4.620,-2.400) -- (4.500,-2.400) -- (4.380,-2.400) -- (4.260,-2.400) -- (4.140,-2.400) -- (4.020,-2.400) -- (3.901,-2.401) -- (3.781,-2.401) -- (3.662,-2.402) -- (3.542,-2.402) -- (3.424,-2.404) -- (3.305,-2.405) -- (3.187,-2.407) -- (3.070,-2.410) -- (2.954,-2.414) -- (2.839,-2.419) -- (2.725,-2.425) -- (2.614,-2.434) -- (2.504,-2.444) -- (2.397,-2.457) -- (2.293,-2.473) -- (2.192,-2.492) -- (2.095,-2.515) -- (2.003,-2.543) -- (1.915,-2.575) -- (1.832,-2.612) -- (1.755,-2.655) -- (1.683,-2.703) -- (1.617,-2.757) -- (1.557,-2.817) -- (1.503,-2.883) -- (1.455,-2.955) -- (1.412,-3.032) -- (1.375,-3.115) -- (1.343,-3.202) -- (1.316,-3.295) -- (1.293,-3.391) -- (1.274,-3.492) -- (1.259,-3.595) -- (1.246,-3.701) -- (1.237,-3.810) -- (1.230,-3.920) -- (1.226,-4.032) -- (1.224,-4.144) -- (1.224,-4.256) -- (1.226,-4.368) -- (1.230,-4.480) -- (1.237,-4.590) -- (1.246,-4.699) -- (1.259,-4.805) -- (1.274,-4.908) -- (1.293,-5.009) -- (1.316,-5.105) -- (1.343,-5.198) -- (1.375,-5.285) -- (1.412,-5.368) -- (1.455,-5.445) -- (1.503,-5.517) -- (1.557,-5.583) -- (1.617,-5.643) -- (1.683,-5.697) -- (1.755,-5.745) -- (1.832,-5.788) -- (1.915,-5.825) -- (2.003,-5.857) -- (2.095,-5.885) -- (2.192,-5.908) -- (2.293,-5.927) -- (2.397,-5.943) -- (2.504,-5.956) -- (2.613,-5.966) -- (2.725,-5.975) -- (2.839,-5.981) -- (2.954,-5.986) -- (3.070,-5.990) -- (3.187,-5.992) -- (3.304,-5.994) -- (3.422,-5.995) -- (3.541,-5.996) -- (3.659,-5.996) -- (3.778,-5.995) -- (3.896,-5.994) -- (4.013,-5.992) -- (4.130,-5.990) -- (4.246,-5.986) -- (4.361,-5.981) -- (4.475,-5.975) -- (4.587,-5.966) -- (4.696,-5.956) -- (4.803,-5.943) -- (4.907,-5.927) -- (5.008,-5.908) -- (5.105,-5.885) -- (5.197,-5.858) -- (5.285,-5.825) -- (5.367,-5.788) -- (5.445,-5.746) -- (5.516,-5.698) -- (5.581,-5.645) -- cycle;
\draw[white, line width=4.5pt, line cap=butt] (5.883,-5.033) -- (5.894,-4.941) -- (5.898,-4.871) -- (5.899,-4.847) -- (5.899,-4.753) -- (5.894,-4.659);
\draw[line cap=butt] (5.883,-5.033) -- (5.894,-4.941) -- (5.898,-4.871) -- (5.899,-4.847) -- (5.899,-4.753) -- (5.894,-4.659);
\draw[white, line width=4.5pt, line cap=butt] (3.601,-3.781) -- (3.602,-3.662) -- (3.602,-3.658) -- (3.602,-3.542) -- (3.604,-3.424);
\draw[line cap=butt] (3.601,-3.781) -- (3.602,-3.662) -- (3.602,-3.658) -- (3.602,-3.542) -- (3.604,-3.424);
\draw[white, line width=4.5pt, line cap=butt] (2.436,-5.789) -- (2.445,-5.897) -- (2.450,-5.950) -- (2.456,-6.002) -- (2.470,-6.105);
\draw[line cap=butt] (2.436,-5.789) -- (2.445,-5.897) -- (2.450,-5.950) -- (2.456,-6.002) -- (2.470,-6.105);
\draw[white, line width=4.5pt, line cap=butt] (3.422,-5.995) -- (3.541,-5.996) -- (3.546,-5.996) -- (3.659,-5.996) -- (3.778,-5.995);
\draw[line cap=butt] (3.422,-5.995) -- (3.541,-5.996) -- (3.546,-5.996) -- (3.659,-5.996) -- (3.778,-5.995);
\draw[white, line width=4.5pt, line cap=butt] (3.781,-2.401) -- (3.662,-2.402) -- (3.656,-2.402) -- (3.542,-2.402) -- (3.424,-2.404);
\draw[line cap=butt] (3.781,-2.401) -- (3.662,-2.402) -- (3.656,-2.402) -- (3.542,-2.402) -- (3.424,-2.404);
\draw[white, line width=4.5pt, line cap=butt] (7.127,-2.293) -- (7.143,-2.397) -- (7.150,-2.450) -- (7.156,-2.504) -- (7.166,-2.614);
\draw[line cap=butt] (7.127,-2.293) -- (7.143,-2.397) -- (7.150,-2.450) -- (7.156,-2.504) -- (7.166,-2.614);
\draw[white, line width=4.5pt, line cap=butt] (6.990,-4.770) -- (7.099,-4.759) -- (7.198,-4.746) -- (7.205,-4.745) -- (7.308,-4.728) -- (7.409,-4.709);
\draw[line cap=butt] (6.990,-4.770) -- (7.099,-4.759) -- (7.198,-4.746) -- (7.205,-4.745) -- (7.308,-4.728) -- (7.409,-4.709);
\draw[white, line width=4.5pt, line cap=butt] (4.906,-6.141) -- (4.901,-6.047) -- (4.901,-5.953) -- (4.902,-5.928) -- (4.906,-5.859) -- (4.917,-5.767);
\draw[line cap=butt] (4.906,-6.141) -- (4.901,-6.047) -- (4.901,-5.953) -- (4.902,-5.928) -- (4.906,-5.859) -- (4.917,-5.767);

\end{tikzpicture}

%% file: figures/Th-podd-tait-target.tex
\begin{tikzpicture}[tait/.style={line width=1.1pt, black, line cap=round},
  vlab/.style={font=\small, inner sep=1pt}]
\draw[tait] (-1.126,0.650) arc[start angle=150, end angle=90, radius=1.300];
\draw[tait] (0.000,1.300) arc[start angle=90, end angle=30, radius=1.300];
\draw[tait] (1.126,0.650) arc[start angle=30, end angle=-30, radius=1.300];
\draw[tait] (1.126,-0.650) arc[start angle=-30, end angle=-82, radius=1.300];
\draw[tait] (-1.126,0.650) arc[start angle=150, end angle=202, radius=1.300];
\fill (-0.865,-0.971) circle (0.03);
\fill (-0.650,-1.126) circle (0.03);
\fill (-0.408,-1.234) circle (0.03);
\draw[tait] (-1.992,1.150) arc[start angle=150, end angle=90, radius=2.300];
\draw[tait] (0.000,2.300) arc[start angle=90, end angle=30, radius=2.300];
\draw[tait] (1.992,1.150) arc[start angle=30, end angle=-30, radius=2.300];
\draw[tait] (1.992,-1.150) arc[start angle=-30, end angle=-82, radius=2.300];
\draw[tait] (-1.992,1.150) arc[start angle=150, end angle=202, radius=2.300];
\fill (-1.372,-1.846) circle (0.03);
\fill (-1.150,-1.992) circle (0.03);
\fill (-0.913,-2.111) circle (0.03);
\draw[tait] (-3.031,1.750) arc[start angle=150, end angle=90, radius=3.500];
\draw[tait] (0.000,3.500) arc[start angle=90, end angle=30, radius=3.500];
\draw[tait] (3.031,1.750) arc[start angle=30, end angle=-30, radius=3.500];
\draw[tait] (3.031,-1.750) arc[start angle=-30, end angle=-82, radius=3.500];
\draw[tait] (-3.031,1.750) arc[start angle=150, end angle=202, radius=3.500];
\fill (-1.975,-2.890) circle (0.03);
\fill (-1.750,-3.031) circle (0.03);
\fill (-1.515,-3.155) circle (0.03);
\draw[tait] (0,0) -- (0.000,1.300);
\draw[tait] (0.000,1.300) -- (0.000,2.300);
\draw[tait] (0.000,2.300) -- (0.000,2.620);
\draw[tait] (0.000,3.180) -- (0.000,3.500);
\fill (0.000,2.750) circle (0.03);
\fill (0.000,2.900) circle (0.03);
\fill (0.000,3.050) circle (0.03);
\draw[tait] (0,0) -- (1.126,0.650);
\draw[tait] (1.126,0.650) -- (1.992,1.150);
\draw[tait] (1.992,1.150) -- (2.269,1.310);
\draw[tait] (2.754,1.590) -- (3.031,1.750);
\fill (2.382,1.375) circle (0.03);
\fill (2.511,1.450) circle (0.03);
\fill (2.641,1.525) circle (0.03);
\draw[tait] (0,0) -- (1.126,-0.650);
\draw[tait] (1.126,-0.650) -- (1.992,-1.150);
\draw[tait] (1.992,-1.150) -- (2.269,-1.310);
\draw[tait] (2.754,-1.590) -- (3.031,-1.750);
\fill (2.382,-1.375) circle (0.03);
\fill (2.511,-1.450) circle (0.03);
\fill (2.641,-1.525) circle (0.03);
\draw[tait] (0,0) -- (-1.126,0.650);
\draw[tait] (-1.126,0.650) -- (-1.992,1.150);
\draw[tait] (-1.992,1.150) -- (-2.269,1.310);
\draw[tait] (-2.754,1.590) -- (-3.031,1.750);
\fill (-2.382,1.375) circle (0.03);
\fill (-2.511,1.450) circle (0.03);
\fill (-2.641,1.525) circle (0.03);
\draw[tait, dash pattern=on 0pt off 2.6pt, line width=1.5pt] (0,0) -- (-0.525,-0.909);
\fill (0,0) circle (0.070);
\node[vlab] at (0.175,-0.62) {$v_{mq+1}$};
\fill (0.000,1.300) circle (0.070);
\node[vlab] at (-0.337,1.503) {$v_1$};
\fill (1.126,0.650) circle (0.070);
\node[vlab] at (1.133,1.044) {$v_2$};
\fill (1.126,-0.650) circle (0.070);
\node[vlab] at (1.470,-0.459) {$v_3$};
\fill (-1.126,0.650) circle (0.070);
\node[vlab] at (-1.470,0.459) {$v_q$};
\fill (0.000,2.300) circle (0.070);
\node[vlab] at (-0.443,2.501) {$v_{q+1}$};
\fill (1.992,1.150) circle (0.070);
\node[vlab] at (2.085,1.634) {$v_{q+2}$};
\fill (1.992,-1.150) circle (0.070);
\node[vlab] at (2.487,-0.987) {$v_{q+3}$};
\fill (-1.992,1.150) circle (0.070);
\node[vlab] at (-2.5,0.95) {$v_{2q}$};
\fill (0.000,3.500) circle (0.070);
\node[vlab] at (0.000,3.790) {$v_{(m-1)q+1}$};
\fill (3.031,1.750) circle (0.070);
\node[vlab] at (3.605,2.2) {$v_{(m-1)q+2}$};
\fill (3.031,-1.750) circle (0.070);
\node[vlab] at (3.605,-2.2) {$v_{(m-1)q+3}$};
\fill (-3.031,1.750) circle (0.070);
\node[vlab] at (-3.368,1.945) {$v_{mq}$};
\end{tikzpicture}

%% file: figures/Th-p2-crossings-flat-labelled.tex
\begin{tikzpicture}[line width=1.00pt, line cap=round, line join=round, every node/.style={font=\small}]
\draw (0,0) -- (1.3,0);
\draw (1.9,0) -- (5.05,0);
\draw (5.65,0) -- (8.8,0);
\draw (0,-0.42) -- (1.3,-0.42);
\draw (2.5,-0.42) -- (5.05,-0.42);
\draw (5.65,-0.42) -- (5.65,-0.42);
\draw (6.25,-0.42) -- (8.8,-0.42);
\draw (0,-0.84) -- (1.9,-0.84);
\draw (2.5,-0.84) -- (2.7,-0.84);
\draw (3.3,-0.84) -- (5.65,-0.84);
\draw (6.25,-0.84) -- (6.45,-0.84);
\draw (7.05,-0.84) -- (8.8,-0.84);
\draw (0,-1.68) -- (2.7,-1.68);
\draw (3.3,-1.68) -- (3.5,-1.68);
\draw (4.1,-1.68) -- (6.45,-1.68);
\draw (7.05,-1.68) -- (7.25,-1.68);
\draw (7.85,-1.68) -- (8.8,-1.68);
\draw (0,-2.1) -- (3.5,-2.1);
\draw (4.7,-2.1) -- (7.25,-2.1);
\draw (8.45,-2.1) -- (8.8,-2.1);
\draw (0,-2.52) -- (4.1,-2.52);
\draw (4.7,-2.52) -- (7.85,-2.52);
\draw (8.45,-2.52) -- (8.8,-2.52);
\draw (1.3,-0.42) .. controls (1.6,-0.42) and (1.6,0) .. (1.9,0);
\draw[white, line width=4.0pt, line cap=butt, shorten <=0.12cm, shorten >=0.12cm] (1.3,0) .. controls (1.6,0) and (1.6,-0.42) .. (1.9,-0.42);
\draw (1.3,0) .. controls (1.6,0) and (1.6,-0.42) .. (1.9,-0.42);
\draw (1.9,-0.42) .. controls (2.2,-0.42) and (2.2,-0.84) .. (2.5,-0.84);
\draw[white, line width=4.0pt, line cap=butt, shorten <=0.12cm, shorten >=0.12cm] (1.9,-0.84) .. controls (2.2,-0.84) and (2.2,-0.42) .. (2.5,-0.42);
\draw (1.9,-0.84) .. controls (2.2,-0.84) and (2.2,-0.42) .. (2.5,-0.42);
\draw (3.5,-2.1) .. controls (3.8,-2.1) and (3.8,-1.68) .. (4.1,-1.68);
\draw[white, line width=4.0pt, line cap=butt, shorten <=0.12cm, shorten >=0.12cm] (3.5,-1.68) .. controls (3.8,-1.68) and (3.8,-2.1) .. (4.1,-2.1);
\draw (3.5,-1.68) .. controls (3.8,-1.68) and (3.8,-2.1) .. (4.1,-2.1);
\draw (4.1,-2.1) .. controls (4.4,-2.1) and (4.4,-2.52) .. (4.7,-2.52);
\draw[white, line width=4.0pt, line cap=butt, shorten <=0.12cm, shorten >=0.12cm] (4.1,-2.52) .. controls (4.4,-2.52) and (4.4,-2.1) .. (4.7,-2.1);
\draw (4.1,-2.52) .. controls (4.4,-2.52) and (4.4,-2.1) .. (4.7,-2.1);
\draw (5.05,-0.42) .. controls (5.35,-0.42) and (5.35,0) .. (5.65,0);
\draw[white, line width=4.0pt, line cap=butt, shorten <=0.12cm, shorten >=0.12cm] (5.05,0) .. controls (5.35,0) and (5.35,-0.42) .. (5.65,-0.42);
\draw (5.05,0) .. controls (5.35,0) and (5.35,-0.42) .. (5.65,-0.42);
\draw (5.65,-0.42) .. controls (5.95,-0.42) and (5.95,-0.84) .. (6.25,-0.84);
\draw[white, line width=4.0pt, line cap=butt, shorten <=0.12cm, shorten >=0.12cm] (5.65,-0.84) .. controls (5.95,-0.84) and (5.95,-0.42) .. (6.25,-0.42);
\draw (5.65,-0.84) .. controls (5.95,-0.84) and (5.95,-0.42) .. (6.25,-0.42);
\draw (7.25,-2.1) .. controls (7.55,-2.1) and (7.55,-1.68) .. (7.85,-1.68);
\draw[white, line width=4.0pt, line cap=butt, shorten <=0.12cm, shorten >=0.12cm] (7.25,-1.68) .. controls (7.55,-1.68) and (7.55,-2.1) .. (7.85,-2.1);
\draw (7.25,-1.68) .. controls (7.55,-1.68) and (7.55,-2.1) .. (7.85,-2.1);
\draw (7.85,-2.1) .. controls (8.15,-2.1) and (8.15,-2.52) .. (8.45,-2.52);
\draw[white, line width=4.0pt, line cap=butt, shorten <=0.12cm, shorten >=0.12cm] (7.85,-2.52) .. controls (8.15,-2.52) and (8.15,-2.1) .. (8.45,-2.1);
\draw (7.85,-2.52) .. controls (8.15,-2.52) and (8.15,-2.1) .. (8.45,-2.1);
\fill (2.84,-1.008) circle (0.035);
\fill (3,-1.26) circle (0.035);
\fill (3.16,-1.512) circle (0.035);
\fill (6.59,-1.008) circle (0.035);
\fill (6.75,-1.26) circle (0.035);
\fill (6.91,-1.512) circle (0.035);
\node[fill=white, inner sep=0.6pt, font=\scriptsize] at (0.6,0) {$1$};
\node[fill=white, inner sep=0.6pt, font=\scriptsize] at (0.6,-0.42) {$2$};
\node[fill=white, inner sep=0.6pt, font=\scriptsize] at (0.6,-0.84) {$3$};
\node[fill=white, inner sep=0.6pt, font=\scriptsize] at (0.6,-1.68) {$p-2$};
\node[fill=white, inner sep=0.6pt, font=\scriptsize] at (0.6,-2.1) {$p-1$};
\node[fill=white, inner sep=0.6pt, font=\scriptsize] at (0.6,-2.52) {$p$};
\fill (0.6,-1.025) circle (0.035);
\fill (0.6,-1.26) circle (0.035);
\fill (0.6,-1.495) circle (0.035);
\draw (8.8,0) arc[start angle=-90, end angle=90, x radius=0.45, y radius=0.5] -- (0,1) arc[start angle=90, end angle=270, x radius=0.45, y radius=0.5];
\draw (8.8,-0.42) arc[start angle=-90, end angle=90, x radius=0.744, y radius=0.815] -- (0,1.21) arc[start angle=90, end angle=270, x radius=0.744, y radius=0.815];
\draw (8.8,-2.1) arc[start angle=-90, end angle=90, x radius=1.92, y radius=2.075] -- (0,2.05) arc[start angle=90, end angle=270, x radius=1.92, y radius=2.075];
\draw (8.8,-2.52) arc[start angle=-90, end angle=90, x radius=2.214, y radius=2.39] -- (0,2.26) arc[start angle=90, end angle=270, x radius=2.214, y radius=2.39];
\draw (8.8,-0.84) arc[start angle=-90, end angle=-60, x radius=1.04, y radius=1.13];
\fill (9.496,-0.55) circle (0.03);
\fill (9.62,-0.406) circle (0.03);
\fill (9.718,-0.241) circle (0.03);
\draw (0,-0.84) arc[start angle=270, end angle=240, x radius=1.04, y radius=1.13];
\fill (-0.696,-0.55) circle (0.03);
\fill (-0.82,-0.406) circle (0.03);
\fill (-0.918,-0.241) circle (0.03);
\draw (8.8,-1.68) arc[start angle=-90, end angle=-60, x radius=1.62, y radius=1.76];
\fill (9.884,-1.228) circle (0.03);
\fill (10.077,-1.004) circle (0.03);
\fill (10.23,-0.746) circle (0.03);
\draw (0,-1.68) arc[start angle=270, end angle=240, x radius=1.62, y radius=1.76];
\fill (-1.084,-1.228) circle (0.03);
\fill (-1.277,-1.004) circle (0.03);
\fill (-1.43,-0.746) circle (0.03);
\fill (4.4,1.51) circle (0.03);
\fill (4.4,1.63) circle (0.03);
\fill (4.4,1.75) circle (0.03);
\draw[thin, dashed] (1.6,-0.21) circle (0.36);
\draw[thin, dashed] (4.4,-2.31) circle (0.36);
\draw[thin, decorate, decoration={brace, amplitude=4pt}] (-2.514,-2.520) -- (-2.514,0) node[midway, left=5pt] {$p$};
\draw[thin, dashed] (1.6,-0.21) circle (0.36) node[above=0.36cm, inner sep=1pt] {$+$};
\draw[thin, dashed] (4.4,-2.31) circle (0.36) node[below=0.36cm, inner sep=1pt] {$-$};
\end{tikzpicture}

%% file: references.bib
@article{AMMMPS21,
  author        = {Aceto, Paolo and Meier, Jeffrey and Miller, Allison N. and Miller, Maggie and Park, JungHwan and Stipsicz, Andr{\'a}s I.},
  title         = {Branched covers bounding rational homology balls},
  journal       = {Algebr. Geom. Topol.},
  volume        = {21},
  number        = {7},
  year          = {2021},
  pages         = {3569--3599},
  doi           = {10.2140/agt.2021.21.3569},
  eprint        = {2002.10324},
  archivePrefix = {arXiv},
  primaryClass  = {math.GT},
}

@article{CHH13,
  author        = {Cochran, Tim D. and Harvey, Shelly and Horn, Peter},
  title         = {Filtering smooth concordance classes of topologically slice knots},
  journal       = {Geom. Topol.},
  volume        = {17},
  number        = {4},
  year          = {2013},
  pages         = {2103--2162},
  doi           = {10.2140/gt.2013.17.2103},
  eprint        = {1201.6283},
  archivePrefix = {arXiv},
  primaryClass  = {math.GT},
}

@article{Coh96,
  author = {Cohn, J. H. E.}, title = {Perfect {Pell} powers},
  journal = {Glasgow Mathematical Journal},
  volume = {38}, number = {1}, pages = {19--20}, year = {1996},
  doi = {10.1017/S0017089500031207}
}

@incollection{Con21,
  author        = {Conway, Anthony},
  title         = {The {Levine--Tristram} signature: a survey},
  booktitle     = {2019--20 {MATRIX} Annals},
  series        = {MATRIX Book Ser.},
  volume        = {4},
  publisher     = {Springer},
  address       = {Cham},
  year          = {2021},
  pages         = {31--56},
  doi           = {10.1007/978-3-030-62497-2_2},
  eprint        = {1903.04477},
  archivePrefix = {arXiv},
}

@article{DISST25,
  author        = {Daemi, Aliakbar and Imori, Hayato and Sato, Kouki and Scaduto, Christopher and Taniguchi, Masaki},
  title         = {Instantons, special cycles and knot concordance},
  journal       = {Geom. Topol.},
  volume        = {29},
  number        = {8},
  year          = {2025},
  pages         = {4189--4298},
  doi           = {10.2140/gt.2025.29.4189},
  eprint        = {2209.05400},
  archivePrefix = {arXiv},
  primaryClass  = {math.GT},
}

@article{DS24,
  author        = {Daemi, Aliakbar and Scaduto, Christopher},
  title         = {{Chern--Simons} functional, singular instantons, and the four-dimensional clasp number},
  journal       = {J. Eur. Math. Soc. (JEMS)},
  volume        = {26},
  number        = {6},
  year          = {2024},
  pages         = {2127--2190},
  doi           = {10.4171/JEMS/1320},
  eprint        = {2007.13160},
  archivePrefix = {arXiv},
  primaryClass  = {math.GT},
}

@article{DPLS25,
    title = {Every negative amphichiral knot is rationally slice},
    journal = {Advances in Mathematics},
    volume = {503},
    pages = {111267},
    year = {2026},
    issn = {0001-8708},
    doi = {10.1016/j.aim.2026.111267},
    author = {Alessio {Di Prisa} and Jaewon Lee and O{\u{g}}uz {\c{S}}avk},
  eprint        = {2509.21140},
  archivePrefix = {arXiv},
  primaryClass  = {math.GT},
}

@article{DPS26,
  author        = {Alessio {Di Prisa} and O{\u{g}}uz {\c{S}}avk},
  title         = {{Turk's head} knots and links: a survey},
  journal       = {Expositiones Mathematicae},
  volume        = {44},
  number        = {3},
  year          = {2026},
  pages         = {125777},
  doi           = {10.1016/j.exmath.2026.125777},
  eprint        = {2409.20106},
  archivePrefix = {arXiv},
  primaryClass  = {math.GT},
}

@article{Don87,
  author  = {Donaldson, Simon K.},
  title   = {The orientation of {Yang--Mills} moduli spaces and $4$-manifold topology},
  journal = {J. Differential Geom.},
  volume  = {26},
  number  = {3},
  year    = {1987},
  pages   = {397--428},
  doi     = {10.4310/jdg/1214441485},
}

@article{DMMS10,
  author        = {Dowdall, Nicholas E. and Mattman, Thomas W. and Meek, Kevin and Solis, Pablo R.},
  title         = {On the {Harary--Kauffman} conjecture and {Turk's head} knots},
  journal       = {Kobe J. Math.},
  volume        = {27},
  number        = {1--2},
  year          = {2010},
  pages         = {1--20},
  eprint        = {0811.0044},
  archivePrefix = {arXiv},
  primaryClass  = {math.GT},
}

@misc{DG,
  author        = {Dunfield, Nathan M. and Gong, Sherry},
  title         = {Ribbon concordances and slice obstructions: experiments and examples},
  year          = {2025},
  eprint        = {2512.21825},
  archivePrefix = {arXiv},
  primaryClass  = {math.GT},
  note          = {Preprint, arXiv:2512.21825},
}

@article{DLS26,
  author        = {Dunfield, Nathan M. and Lipshitz, Robert and Sch{\"u}tz, Dirk},
  title         = {Local equivalence and refinements of {Rasmussen}'s {$s$}-invariant},
  journal       = {J. Topol.},
  volume        = {19},
  number        = {1},
  year          = {2026},
  pages         = {e70056},
  doi           = {10.1112/topo.70056},
  eprint        = {2312.09114},
  archivePrefix = {arXiv},
  primaryClass  = {math.GT},
  url           = {https://arxiv.org/abs/2312.09114}
}

@article{FGMW10,
  author        = {Freedman, Michael H. and Gompf, Robert E. and Morrison, Scott and Walker, Kevin},
  title         = {Man and machine thinking about the smooth 4-dimensional {Poincar\'e} conjecture},
  journal       = {Quantum Topol.},
  volume        = {1},
  number        = {2},
  year          = {2010},
  pages         = {171--208},
  doi           = {10.4171/QT/5},
  eprint        = {0906.5177},
  archivePrefix = {arXiv},
}

@misc{GP26,
  author        = {Golla, Marco and Pinz{\'o}n-Caicedo, Juanita},
  title         = {The {Conway} knot has infinite order in smooth concordance},
  year          = {2026},
  eprint        = {2609.22378},
  archivePrefix = {arXiv},
  primaryClass  = {math.GT},
  note          = {Preprint, arXiv:2609.22378},
  url           = {https://arxiv.org/abs/2609.22378}
}

@book{GS99,
  author    = {Gompf, Robert E. and Stipsicz, Andr{\'a}s I.},
  title     = {$4$-Manifolds and {Kirby} Calculus},
  series    = {Graduate Studies in Mathematics},
  volume    = {20},
  publisher = {American Mathematical Society},
  address   = {Providence, RI},
  year      = {1999},
  doi       = {10.1090/gsm/020},
  isbn      = {978-0-8218-0994-5},
  mrnumber  = {1707327}
}

@article{Gong,
  author        = {Gong, Sherry},
  title         = {On the {Kronheimer--Mrowka} concordance invariant},
  journal       = {J. Topol.},
  volume        = {14},
  number        = {1},
  year          = {2021},
  pages         = {1--28},
  doi           = {10.1112/topo.12175},
  eprint        = {1908.05018},
  archivePrefix = {arXiv},
  primaryClass  = {math.GT},
}

@article{GL78,
  author  = {Gordon, C. McA. and Litherland, Richard A.},
  title   = {On the signature of a link},
  journal = {Invent. Math.},
  volume  = {47},
  number  = {1},
  year    = {1978},
  pages   = {53--69},
  doi     = {10.1007/BF01609479},
}

@article{Greene13Spanning,
  author        = {Greene, Joshua Evan},
  title         = {A spanning tree model for the {Heegaard Floer}
                   homology of a branched double-cover},
  journal       = {J. Topol.},
  volume        = {6},
  number        = {2},
  year          = {2013},
  pages         = {525--567},
  doi           = {10.1112/jtopol/jtt007},
  eprint        = {0805.1381},
  archivePrefix = {arXiv},
  primaryClass  = {math.GT}
}

@article{Hom14Cable,
  author        = {Hom, Jennifer},
  title         = {Bordered {Heegaard Floer} homology and the
                   tau-invariant of cable knots},
  journal       = {J. Topol.},
  volume        = {7},
  number        = {2},
  year          = {2014},
  pages         = {287--326},
  doi           = {10.1112/jtopol/jtt030},
  eprint        = {1202.1463},
  archivePrefix = {arXiv},
  primaryClass  = {math.GT}
}

@article{Iss18,
  author = {Issa, Ahmad},
  title = {The classification of quasi-alternating {Montesinos} links},
  journal = {Proceedings of the American Mathematical Society},
  volume = {146},
  number = {9},
  year = {2018},
  pages = {4047--4057},
  doi = {10.1090/proc/14053},
  eprint = {1701.08425},
  archivePrefix = {arXiv},
  primaryClass = {math.GT}
}

@article{IM18,
  author        = {Issa, Ahmad and McCoy, Duncan},
  title         = {On {Seifert} fibered spaces bounding definite manifolds},
  journal       = {Pacific J. Math.},
  volume        = {304},
  number        = {2},
  year          = {2020},
  pages         = {463--480},
  doi           = {10.2140/pjm.2020.304.463},
  eprint        = {1807.10310},
  archivePrefix = {arXiv},
  primaryClass  = {math.GT},
}

@article{Jab12,
  author = {Jabuka, Stanislav},
  title = {Concordance invariants from higher order covers},
  journal = {Topology and its Applications},
  volume = {159},
  number = {10-11},
  year = {2012},
  pages = {2694--2710},
  doi = {10.1016/j.topol.2012.03.014}
}

@article{Jiang81,
  author  = {Jiang, Bo Ju},
  title   = {A simple proof that the concordance group of
             algebraically slice knots is infinitely generated},
  journal = {Proceedings of the American Mathematical Society},
  volume  = {83},
  number  = {1},
  year    = {1981},
  pages   = {189--192},
  doi     = {10.1090/S0002-9939-1981-0620010-7}
}

@article{KL04,
  author  = {Kauffman, Louis H. and Lambropoulou, Sofia},
  title   = {On the classification of rational tangles},
  journal = {Advances in Applied Mathematics},
  volume  = {33},
  number  = {2},
  year    = {2004},
  pages   = {199--237},
  doi     = {10.1016/j.aam.2003.06.002},
  eprint  = {math/0311499},
  archivePrefix = {arXiv}
}

@article{Kaw09,
  author  = {Kawauchi, Akio},
  title   = {Rational-slice knots via strongly negative-amphicheiral knots},
  journal = {Commun. Math. Res.},
  volume  = {25},
  number  = {2},
  year    = {2009},
  pages   = {177--192},
}

@article{KimOh05,
  author  = {Kim, Myung-Hwan and Oh, Byeong-Kweon},
  title   = {Representations of integral quadratic forms by sums of squares},
  journal = {Mathematische Zeitschrift},
  volume  = {250},
  number  = {2},
  pages   = {427--442},
  year    = {2005},
  doi     = {10.1007/s00209-005-0761-5}
}

@article{KMRS24,
  author        = {Kjuchukova, Alexandra and Miller, Allison N. and Ray, Arunima and Sakall{\i}, S{\"u}meyra},
  title         = {Slicing knots in definite $4$-manifolds},
  journal       = {Trans. Amer. Math. Soc.},
  volume        = {377},
  number        = {8},
  year          = {2024},
  pages         = {5905--5946},
  eprint        = {2112.14596},
  archivePrefix = {arXiv},
  primaryClass  = {math.GT},
}

@misc{KnotInfo,
label = {KnotInfo},
  author       = {Livingston, Charles and Moore, Allison H.},
  title        = {{KnotInfo}: Table of Knot Invariants},
  howpublished = {Available at \url{https://knotinfo.org}},
  year         = {2026},
  month        = {October},
}

@article{KM13,
  author        = {Kronheimer, Peter B. and Mrowka, Tomasz S.},
  title         = {Gauge theory and {Rasmussen's} invariant},
  journal       = {J. Topol.},
  volume        = {6},
  number        = {3},
  year          = {2013},
  pages         = {659--674},
  doi           = {10.1112/jtopol/jtt008},
  eprint        = {1110.1297},
  archivePrefix = {arXiv},
  primaryClass  = {math.GT},
  note          = {Corrigendum available in arXiv:1110.1297v3},
}

@misc{Lee24,
  author        = {Lee, Jaewon},
  title         = {Obstructing two-torsion in the rational knot concordance group},
  year          = {2024},
  eprint        = {2406.12761},
  archivePrefix = {arXiv},
  primaryClass  = {math.GT},
  note          = {Preprint, arXiv:2406.12761}
}

@article{LT88,
  author  = {Lickorish, W. B. R. and Thistlethwaite, M. B.},
  title   = {Some links with non-trivial polynomials and their crossing-numbers},
  journal = {Commentarii Mathematici Helvetici},
  volume  = {63},
  pages   = {527--539},
  year    = {1988},
  doi     = {10.1007/BF02566777}
}

@article{Lis07,
  author        = {Lisca, Paolo},
  title         = {Lens spaces, rational balls and the ribbon conjecture},
  journal       = {Geom. Topol.},
  volume        = {11},
  year          = {2007},
  pages         = {429--472},
  doi           = {10.2140/gt.2007.11.429},
  eprint        = {math/0701610},
  archivePrefix = {arXiv},
  primaryClass  = {math.GT},
}

@article{Lis17,
  author        = {Lisca, Paolo},
  title         = {On $3$-braid knots of finite concordance order},
  journal       = {Transactions of the American Mathematical Society},
  volume        = {369},
  number        = {7},
  pages         = {5087--5112},
  year          = {2017},
  doi           = {10.1090/tran/6888},
  eprint        = {1504.03129},
  archivePrefix = {arXiv},
  primaryClass  = {math.GT}
}

@article{LO15,
  author   = {Lisca, Paolo and Owens, Brendan},
  title    = {Signatures, {H}eegaard {F}loer correction terms and quasi-alternating links},
  journal  = {Proceedings of the American Mathematical Society},
  volume   = {143},
  number   = {2},
  year     = {2015},
  pages    = {907--914},
  doi      = {10.1090/S0002-9939-2014-12265-9},
  mrnumber = {3283677},
  eprint   = {1302.3190},
  archivePrefix = {arXiv}
}

@misc{LinkInfo,
  label = {LinkInfo},
  author       = {Livingston, Charles and Moore, Allison H.},
  title        = {{LinkInfo}: Table of Link Invariants},
  howpublished = {Available at \url{https://knotinfo.org/linkinfo}},
  year         = {2026},
  month        = {October}
}

@inproceedings{Man26,
  author        = {Manolescu, Ciprian},
  title         = {From knots to four-manifolds},
  booktitle     = {Proceedings of the {I}nternational {C}ongress of {M}athematicians 2026},
  volume        = {2},
  pages         = {154--179},
  year          = {2026},
  eprint        = {2601.05425},
  archivePrefix = {arXiv},
  primaryClass  = {math.GT}
}

@article{MMSW23,
  author        = {Manolescu, Ciprian and Marengon, Marco and Sarkar, Sucharit and Willis, Michael},
  title         = {A generalization of {Rasmussen's} invariant, with applications to surfaces in some four-manifolds},
  journal       = {Duke Math. J.},
  volume        = {172},
  number        = {2},
  year          = {2023},
  pages         = {231--311},
  eprint        = {1910.08195},
  archivePrefix = {arXiv},
  primaryClass  = {math.GT},
}

@incollection{ManolescuOzsvath08,
  author        = {Manolescu, Ciprian and Ozsv{\'a}th, Peter},
  title         = {On the {Khovanov} and knot {Floer} homologies of quasi-alternating links},
  booktitle     = {Proceedings of G{\"o}kova Geometry--Topology Conference 2007},
  publisher     = {G{\"o}kova Geometry/Topology Conference (GGT)},
  address       = {G{\"o}kova},
  year          = {2008},
  pages         = {60--81},
  eprint        = {0708.3249},
  archivePrefix = {arXiv},
  primaryClass  = {math.GT},
}

@article{MP23,
  author        = {Manolescu, Ciprian and Piccirillo, Lisa},
  title         = {From zero surgeries to candidates for exotic definite $4$-manifolds},
  journal       = {J. Lond. Math. Soc. (2)},
  volume        = {108},
  number        = {5},
  year          = {2023},
  pages         = {2001--2036},
  doi           = {10.1112/jlms.12800},
  eprint        = {2102.04391},
  archivePrefix = {arXiv},
  primaryClass  = {math.GT},
}

@article{MS09,
  author  = {Mattman, Thomas W. and Solis, Pablo},
  title   = {A proof of the {Kauffman--Harary} conjecture},
  journal = {Algebraic \& Geometric Topology},
  volume  = {9},
  number  = {4},
  pages   = {2027--2039},
  year    = {2009},
  doi     = {10.2140/agt.2009.9.2027}
}

@article{Men84,
  author  = {Menasco, William},
  title   = {Closed incompressible surfaces in alternating knot and link complements},
  journal = {Topology},
  volume  = {23},
  number  = {1},
  year    = {1984},
  pages   = {37--44},
  doi     = {10.1016/0040-9383(84)90023-5}
}

@article{MWW22,
  author        = {Morrison, Scott and Walker, Kevin and Wedrich, Paul},
  title         = {Invariants of 4-manifolds from {Khovanov--Rozansky} link homology},
  journal       = {Geom. Topol.},
  volume        = {26},
  number        = {8},
  year          = {2022},
  pages         = {3367--3420},
  doi           = {10.2140/gt.2022.26.3367},
  eprint        = {1907.12194},
  archivePrefix = {arXiv},
}

@article{Mur65,
  author   = {Murasugi, Kunio},
  title    = {On a certain numerical invariant of link types},
  journal  = {Trans. Amer. Math. Soc.},
  volume   = {117},
  year     = {1965},
  pages    = {387--422},
  doi      = {10.1090/S0002-9947-1965-0171275-5},
  mrnumber = {0171275}
}

@article{Mye71,
  author  = {Myers, B. R.},
  title   = {Number of spanning trees in a wheel},
  journal = {IEEE Trans. Circuit Theory},
  volume  = {18},
  number  = {2},
  year    = {1971},
  pages   = {280--282},
  doi     = {10.1109/TCT.1971.1083273},
}

@article{Nak23,
  author  = {Nakamura, Kai},
  title   = {Trace embeddings from zero surgery homeomorphisms},
  journal = {Journal of Topology},
  volume  = {16},
  number  = {4},
  year    = {2023},
  pages   = {1641--1664},
  doi     = {10.1112/topo.12319}
}

@incollection{NR,
  author    = {Neumann, Walter D. and Raymond, Frank},
  title     = {{Seifert} manifolds, plumbing, {$\mu$}-invariant
               and orientation reversing maps},
  booktitle = {Algebraic and geometric topology},
  series    = {Lecture Notes in Math.},
  volume    = {664},
  pages     = {163--196},
  publisher = {Springer},
  address   = {Berlin},
  year      = {1978},
  doi       = {10.1007/BFb0061699}
}

@article{NW15,
  author  = {Ni, Yi and Wu, Zhongtao},
  title   = {Cosmetic surgeries on knots in {$S^3$}},
  journal = {Journal f\"{u}r die reine und angewandte Mathematik},
  volume  = {706},
  year    = {2015},
  pages   = {1--17},
  doi     = {10.1515/crelle-2013-0067},
  eprint  = {1009.4720},
  archivePrefix = {arXiv}
}

@article{NST24,
  author  = {Nozaki, Yuta and Sato, Kouki and Taniguchi, Masaki},
  title   = {Filtered instanton {F}loer homology and the homology cobordism group},
  journal = {Journal of the European Mathematical Society},
  volume  = {26},
  year    = {2024},
  pages   = {4699--4761},
  doi     = {10.4171/JEMS/1371},
  eprint  = {1905.04001},
  archivePrefix = {arXiv}
}

@article{Owe10,
  author        = {Owens, Brendan},
  title         = {On slicing invariants of knots},
  journal       = {Trans. Amer. Math. Soc.},
  volume        = {362},
  number        = {6},
  year          = {2010},
  pages         = {3095--3106},
  eprint        = {0802.2109},
  archivePrefix = {arXiv},
  primaryClass  = {math.GT},
}

@article{Owe20,
  author        = {Owens, Brendan},
  title         = {Smooth, nonsymplectic embeddings of rational balls in the complex projective plane},
  journal       = {Q. J. Math.},
  volume        = {71},
  number        = {3},
  year          = {2020},
  pages         = {997--1007},
  eprint        = {1906.05913},
  archivePrefix = {arXiv},
  primaryClass  = {math.GT},
}

@article{OwS,
  author        = {Owens, Brendan and Strle, Sa{\v{s}}o},
  title         = {Rational homology spheres and the four-ball genus of knots},
  journal       = {Advances in Mathematics},
  volume        = {200},
  number        = {1},
  pages         = {196--216},
  year          = {2006},
  doi           = {10.1016/j.aim.2004.12.007},
  eprint        = {math/0308073},
  archivePrefix = {arXiv},
  primaryClass  = {math.GT}
}

@article{Pet13,
  author  = {Petkova, Ina},
  title   = {Cables of thin knots and bordered {Heegaard Floer} homology},
  journal = {Quantum Topology},
  volume  = {4},
  number  = {4},
  year    = {2013},
  pages   = {377--409},
  doi     = {10.4171/QT/43}
}

@article{Picc,
  author        = {Piccirillo, Lisa},
  title         = {The {Conway} knot is not slice},
  journal       = {Ann. of Math. (2)},
  volume        = {191},
  number        = {2},
  year          = {2020},
  pages         = {581--591},
  doi           = {10.4007/annals.2020.191.2.5},
  eprint        = {1808.02923},
  archivePrefix = {arXiv},
  primaryClass  = {math.GT},
}

@article{Qin25RBG,
  author  = {Qin, Qianhe},
  title   = {An {RBG} construction of integral surgery homeomorphisms},
  journal = {Algebraic \& Geometric Topology},
  volume  = {25},
  number  = {6},
  year    = {2025},
  pages   = {3755--3774},
  doi     = {10.2140/agt.2025.25.3755}
}

@misc{Qin24,
  author        = {Qin, Qianhe},
  title         = {Slicing degree of knots},
  year          = {2024},
  eprint        = {2404.15991},
  archivePrefix = {arXiv},
  primaryClass  = {math.GT},
  url           = {https://arxiv.org/abs/2404.15991},
  note          = {Preprint, arXiv:2404.15991},
}

@article{Ras10,
  author = {Rasmussen, Jacob},
  title = {{Khovanov} homology and the slice genus},
  journal = {Inventiones Mathematicae},
  volume = {182}, number = {2}, pages = {419--447}, year = {2010},
  doi = {10.1007/s00222-010-0275-6}
}

@misc{Sau26,
   author        = {Saurabh, Suman},
   title         = {Equivariant rational sliceness and {Klein} amphichirality of odd-stranded {Turk's head} knots},
   year          = {2026},
   eprint        = {2607.14638},
   archivePrefix = {arXiv},
   primaryClass  = {math.GT},
   note          = {Preprint, arXiv:2607.14638},
 }

@inproceedings{Sed70,
  author    = {Sedl{\'a}\v{c}ek, J.},
  title     = {On the skeletons of a graph or digraph},
  booktitle = {Combinatorial Structures and Their Applications (Proc. Calgary Internat. Conf., Calgary, Alta., 1969)},
  publisher = {Gordon and Breach},
  address   = {New York},
  year      = {1970},
  pages     = {387--391},
}

@misc{Sta26,
  author        = {Stanggassinger, Werner Alois},
  title         = {{Conway} polynomials of {$Th(p,q)$} and {$Th(q,p)$}: a proof of {Takemura}'s conjecture},
  year          = {2026},
  eprint        = {2610.32269},
  archivePrefix = {arXiv},
  primaryClass  = {math.GT},
  url           = {https://arxiv.org/abs/2610.32269},
  note          = {Preprint, arXiv:2610.32269},
}

@article{T87,
  author  = {Thistlethwaite, M. B.},
  title   = {A spanning tree expansion of the {Jones} polynomial},
  journal = {Topology},
  volume  = {26},
  number  = {3},
  pages   = {297--309},
  year    = {1987},
  doi     = {10.1016/0040-9383(87)90003-6}
}

@article{T91,
  author  = {Thistlethwaite, Morwen B.},
  title   = {On the algebraic part of an alternating link},
  journal = {Pacific Journal of Mathematics},
  volume  = {151},
  number  = {2},
  pages   = {317--333},
  year    = {1991},
  doi     = {10.2140/pjm.1991.151.317}
}

@misc{Wil08,
  author        = {Williams, Luke},
  title         = {Obstructing sliceness in a family of {Montesinos} knots},
  year          = {2008},
  eprint        = {0809.1247},
  archivePrefix = {arXiv},
  primaryClass  = {math.GT},
  note          = {Preprint, arXiv:0809.1247},
}
